\documentclass[12pt,reqno]{amsbook}
\usepackage{amsthm,amsfonts,amssymb,euscript}
\usepackage{graphicx}

\newcommand{\be}{\begin{equation}}
\newcommand{\ee}{\end{equation}}
\newcommand{\ba}{\begin{array}}
\newcommand{\ea}{\end{array}}
\newcommand{\bea}{\begin{eqnarray}}
\newcommand{\eea}{\end{eqnarray}}
\newcommand{\beaa}{\begin{eqnarray*}}
\newcommand{\eeaa}{\end{eqnarray*}}
\newcommand{\bee}{\begin{eqnarray*}}
\newcommand{\eee}{\end{eqnarray*}}
\newcommand{\lli}[1]{L^{#1}(\Sigma)}

\newcommand{\lab}{\label}
\newcommand{\und}{\underline}

\newcommand{\ds}{\displaystyle}
\newcommand{\nn}{\nonumber}
\providecommand{\norm}[1]{\lVert#1\rVert}
\providecommand{\normm}[1]{\left\lVert#1\right\rVert}
\renewcommand{\c}{\cdot}
\newcommand{\les}{\lesssim}

\def\a{{\alpha}}

\def\b{{\beta}}
\def\be{{\beta}}
\def\ga{\gamma}

\def\de{\delta}
\def\De{\Delta}
\def\ep{\varepsilon}

\def\la{\lambda}

\def\si{\sigma}
\def\Si{\Sigma}
\def\om{\omega}

\def\th{\theta}

\def\nab{\nabla}

\def\pr{{\partial}}

\def\les{\lesssim}
\def\c{\cdot}

\def\BB{{\mathcal B}}

\def\CC{{\mathcal C}}
\def\MM{{\mathcal M}}

\def\II{{\mathcal I}}

\def\FF{{\mathcal F}}
\def\EE{{\mathcal E}}
\def\HH{{\mathcal H}}
\def\LL{{\mathcal L}}

\def\PP{{\mathcal P}}

\def\A{{\bf A}}

\def\D{{\bf D}}
\def\F{{\bf F}}
\def\H{{\bf H}}

\def\J{{\bf J}}

\def\N{{\bf N}}
\def\L{{\bf L}}
\def\O{{\bf O}}
\def\Q{{\bf Q}}
\def\R{{\bf R}}

\def\U{{\bf U}}

\def\g{{\bf g}}
\def\m{{\bf m}}

\def\p{{\bf p}}
\def\SSS{{\mathbb S}}
\def\RRR{{\mathbb R}}
\def\NNN{{\mathbb N}}

\def\f12{{\frac 1 2}}

\def\dual{{\,\,^*}}
\def\div{{\mbox div\,}}
\def\curl{{\mbox curl\,}}
\def\lot{\mbox{ l.o.t.}}

\def\lb{{\,\underline{L}}}

\def\Lb{{\,\underline{L}}}

\def\tr{\mbox{tr}}

\def\trchb{{\mbox tr}\, \chib}

\def\bb{{\underline{\b}}}

\def\xib{{\underline{\xi}}}
\def\th{\theta}

\def\f{\widetilde{f}}

\def\lap{\De}

\def\Da{{^{(\A)}\hskip-.15 pc \D}}

\newcommand{\nabb}{{\bf \nab} \mkern-13mu /\,}

\providecommand{\norm}[1]{\lVert#1\rVert}
\newcommand{\lsit}[1]{L^{\infty}_tL^{#1}(\Si_t)}

\newcommand{\lsitt}[2]{L^{#1}_tL^{#2}(\Si_t)}
\newcommand{\lu}[1]{L^\infty_uL^{#1}(\mathcal{H}_u)}
\newcommand{\luom}[1]{L^\infty_{\uom}L^{#1}(\mathcal{H}_{\uom})}

\newcommand{\lhs}[2]{L^{#1}_uH^{#2}(P_u)}
\def\uom{{ \, \, ^{\om}  u}}
\def\Lom{ {\,\,  ^{\om}  L}}
\def\Nom{ {\,\,  ^{\om}  N}}

\def\bom{{ \, \, ^{\om}  b}}
\def\eom{{ \, \, ^{\om}  e}}
\def\uoms{{ \, \, ^{\om,s}  u}}
\def\prb{\boldsymbol{\pr}}

\newtheorem{theorem}{Theorem}[chapter]
\newtheorem{lemma}[theorem]{Lemma}
\newtheorem{proposition}[theorem]{Proposition}
\newtheorem{corollary}[theorem]{Corollary}
\newtheorem{definition}[theorem]{Definition}
\newtheorem{remark}[theorem]{Remark}

\numberwithin{section}{chapter}
\numberwithin{equation}{chapter}

\begin{document}

\title[The Bounded $L^2$ Curvature Conjecture]{Overview of the proof of the Bounded $L^2$ Curvature Conjecture}

\author{Sergiu Klainerman}
\address{Department of Mathematics, Princeton University,
 Princeton NJ 08544}
\email{ seri@math.princeton.edu}
\author{Igor Rodnianski}
\address{Department of Mathematics, Princeton University, 
Princeton NJ 08544}
\email{ irod@math.princeton.edu}
%\subjclass{35J10\newline\newline}
\author{Jeremie Szeftel}
\address{DMA, Ecole Normale Superieure, Paris 75005}
\email{Jeremie.Szeftel@ens.fr}
\vspace{-0.3in}

\maketitle

\frontmatter

%%%%%%

\chapter*{Summary}

This memoir contains an overview of the proof of the   bounded $L^2$ curvature conjecture.      More precisely we show that  the   time
    of existence   of a classical solution to    the Einstein-vacuum
     equations  depends   only on the  $L^2$-norm
      of the  curvature    and  a lower bound of the volume radius   of the corresponding  initial data set.  We note that  though  the result is not  optimal   with respect to the  standard  scaling  of the Einstein equations, it is  nevertheless critical   with respect to  another, more subtle,   scaling tied to   its causal   geometry. Indeed, $L^2 $ bounds on the curvature    is     the minimum  requirement   necessary to  obtain lower bounds on   the radius of  injectivity of  causal  boundaries.   We note also  that, while   the first   nontrivial improvements  for well posedness     for quasilinear  hyperbolic systems in spacetime dimensions greater than $1+1$ (based on Strichartz estimates)  were obtained in   \cite{Ba-Ch1}, \cite{Ba-Ch2}, \cite{Ta1}, \cite{Ta2},   \cite{Kl-R1}  and optimized in      \cite{Kl-R2},  \cite{Sm-Ta},    the  result   we present  here   is the first   in which the  full    structure of the quasilinear hyperbolic   system, not just its principal part,  plays  a  crucial  role.  

\vspace{0.3cm}

The   entire  proof of the conjecture is contained in the following sequence of papers

\vspace{0.3cm}

\noindent\textit{S. Klainerman, I. Rodnianski, J. Szeftel, The bounded $L^2$ curvature conjecture. arXiv:1204.1767, 91 pp.}  This is the main part of the series  in which the proof is completed based on the results  of the papers below.  \\
        
\noindent\textit{J. Szeftel, Parametrix for wave equations on a rough background I: regularity of the phase at initial time. arXiv:1204.1768, 145 pp.}\\

\noindent\textit{J. Szeftel, Parametrix for wave equations on a rough background II: control of the parametrix at initial time. arXiv:1204.1769, 84 pp.}\\

\noindent\textit{J. Szeftel, Parametrix for wave equations on a rough background III: space-time regularity of the phase. arXiv:1204.1770, 276 pp.}\\

\noindent\textit{J. Szeftel, Parametrix for wave equations on a rough background IV: Control of the error term. arXiv:1204.1771, 284 pp.}\\

\noindent\textit{J. Szeftel, Sharp Strichartz estimates for the wave equation on a rough background. arXiv:1301.0112, 30 pp.}

%%%%%%%%%%%%%%%%%%

\tableofcontents

%%%%%%%%%%%%%%%%%%%%

\mainmatter

%%%%%%%%%%%%%%%%%%%%%%%%

\chapter{Introduction}\lab{chap:intro}

\section{General Introduction}
We  present a summary of  our  proof of the bounded  $L^2$-curvature
 conjecture in General Relativity.   According to  the conjecture  
     the   time    of existence   of a classical solution to    the Einstein-vacuum
     equations  depends   only on the  $L^2$-norm
      of the  curvature  and a   lower bound of the volume radius  
       of the corresponding  initial data set.  At a deep level   the  $L^2$ curvature conjecture  
   concerns   the  relationship
between the curvature tensor  and the causal  geometry of an Einstein vacuum  
space-time.  Thus,  though  the result is not  optimal   with respect to the  standard  scaling  of the Einstein equations, it is  nevertheless critical  with respect    to    a different   scaling, which we call
 \textit{null scaling},     tied  to    its      causal   properties.    More precisely,      $L^2 $   curvature       bounds    are strictly    necessary to  obtain lower bounds on   the radius of  injectivity  of  causal  boundaries.  These lower bounds  turn out  to be  crucial  for the construction of parametrices   and derivation of bilinear and trilinear spacetime estimates  for solutions  to  scalar wave equations.  We note also  that, while   the first   nontrivial improvements  for well posedness     for quasilinear  hyperbolic systems in spacetime dimensions greater than $1+1$ (based on Strichartz estimates)  were obtained in   \cite{Ba-Ch1}, \cite{Ba-Ch2}, \cite{Ta1}, \cite{Ta2},   \cite{Kl-R1}  and optimized in      \cite{Kl-R2},  \cite{Sm-Ta},    the  result   we present  here   is the first   in which the  full    structure of the quasilinear hyperbolic   system, not just its principal part,  plays  a  crucial  role.

 \subsection{Initial value problem}
  
We consider the Einstein vacuum equations (EVE),
\begin{equation}\lab{EVE}
{\bf Ric}_{\alpha\beta}=0
\end{equation}
where ${\bf Ric}_{\alpha\beta}$
denotes the  Ricci curvature tensor  of  a four dimensional Lorentzian space time  $(\mathcal{M},\,  {\bf g})$. 
An  initial data  set for \eqref{EVE}  consists of a three dimensional   $3$-surface 
$\Si_0$   together with a    Riemannian  metric $g$ and a symmetric  $2$-tensor $k$  verifying the constraint equations,
\begin{equation}\lab{const}
\left\{\begin{array}{l}
\nabla^j k_{ij}-\nabla_i \textrm{tr}k=0,\\
 {R_{scal}}-|k|^2+(\textrm{tr}k)^2=0, 
\end{array}\right.
\end{equation}
where the covariant derivative $\nabla$ is defined with respect to the metric $g$, $R_{scal}$ is the scalar curvature of $g$, and $\textrm{tr}k$ is the trace of $k$ with respect to the metric $g$.  In this  work  we restrict ourselves 
 to asymptotically flat  initial data sets with one end. 
 For a given initial data set the Cauchy problem  consists in finding a metric ${\bf g}$ satisfying \eqref{EVE}   and an embedding of $\Si_0$  in $\MM$ such that the metric induced by ${\bf g}$ on $\Si_0$ coincides with $g$ and the 2-tensor $k$ is the second fundamental form of the hypersurface $\Si_0\subset \MM$. 
 The  first   local existence and uniqueness   result for (EVE)  was   established
by Y.C. Bruhat,  see \cite{Br},  with the help of 
 wave coordinates  which
 allowed her to cast
the Einstein vacuum  equations in the form of a system of nonlinear wave 
equations to which one can apply\footnote{The original proof in 
\cite{Br} relied  on  representation formulas, following an approach pioneered by Sobolev,   see \cite{Sob}.  }   the standard theory of nonlinear  hyperbolic systems.  The optimal,
classical\footnote{Based only on energy estimates and classical Sobolev inequalities.} result states the following,

\begin{theorem}[Classical local existence \cite{FM} \cite{HKM}]
\label{thm:Bruhat} Let $(\Si_0, g, k)$ be an initial data set
for the Einstein vacuum equations \eqref{EVE}. Assume that $\Si_{0}$ can
be covered by a locally finite system of coordinate charts,
 related to each other by $C^1$ diffeomorphisms, such that
$(g,\, k )\in H^s_{loc}(\Si_0)\times H^{s-1}_{loc}(\Si_0)$
with $s>\frac{5}{2}$. Then there exists a unique\footnote{The original proof in \cite{FM}, \cite{HKM} actually requires  one more derivative for the uniqueness. The fact that uniqueness holds at the same level of regularity than the existence has been obtained in \cite{PlRo}} (up to an isometry)
globally hyperbolic   development
$(\MM, \g)$, verifying   \eqref{EVE},   for which
$\Si_0$ is a Cauchy hypersurface\footnote{That is any past directed, in-extendable  causal curve in $\MM$ intersects $\Si_0$.}.
\end{theorem}

\subsection{Bounded $L^2$ curvature conjecture}
The classical   exponents $s>5/2$ are  clearly not   optimal. 
 By  straightforward  scaling considerations one might  expect to make sense of the initial
value problem for    $s\ge s_c=3/2$, with $s_c$ the natural scaling exponent   for  $L^2$ based Sobolev norms.  Note that   for  $s=s_c=3/2$ a local in time  existence result, for sufficiently small  data,  would be  equivalent to  a global result.  More precisely any  smooth initial data, small in the corresponding  critical norm,  would be  globally smooth. 
  Such a     well-posedness   (WP)  result   would  be thus  comparable with  the so called $\epsilon$- regularity results for   nonlinear elliptic and parabolic problems, which play such a fundamental     role in  the global  regularity properties
  of general solutions.   For quasilinear  hyperbolic problems     critical  WP  results have only been established in the case
  of  $1+1$ dimensional systems, or  spherically symmetric solutions  of higher dimensional problems, 
  in which case  the  $L^2$- Sobolev norms  can be replaced  by bounded variation (BV) type norms\footnote{Recall  that the entire theory of  shock waves  for  1+1 systems of conservation laws  is based on  BV norms, which are \textit{critical} with respect to the scaling of the equations.  Note also  that these   BV norms are not, typically,  conserved   and   that Glimm's famous  existence  result  \cite{Gl}
   can be interpreted as   a global well posedness result for   initial data with small   BV norms.  }.     A  particularly important     example of this type is the critical  BV    well-posedness  result  established by  Christodoulou    for spherically symmetric solutions of the  Einstein equations   coupled with a scalar field, see  \cite{Ch1}. The result  played a  crucial  role  in his celebrated work on Weak  Cosmic Censorship   for the same model, see  \cite{Ch2}.   As well known, unfortunately,   the   BV-norms   are completely   inadequate in higher dimensions;  the only norms which  can propagate   the regularity properties    of the data are necessarily    $L^2$ based.

The quest  for   optimal well-posedness   in  higher dimensions   has been  one of the major themes   in non-linear   hyperbolic   PDE's   in the last twenty years.    Major   advances  have been made in the particular    case of 
 semi-linear  wave equations.    In the case  of geometric   wave  equations such as Wave Maps and Yang-Mills,  which possess a well understood  null structure,
  well-posedness    holds true  for all  exponents larger 
 than the corresponding critical exponent.   For example, 
in the case of    Wave Maps  defined  from the Minkowski space $\RRR^{n+1}$ to a complete Riemannian manifold,    
 the critical  scaling exponents is $s_c=n/2$  and   well-posedness is known to hold all the way down  to $s_c$  for all  dimensions $n\ge 2$.  This  critical    well-posedness result, for $s=n/2$, plays a fundamental   role  in the recent,  large data,       global  results of  \cite{Tao}, \cite{St-Ta1}, \cite{St-Ta2} and \cite{Kr-S} for $2+1$ dimensional wave maps. \\
 
  The   role played  by   critical  exponents   for quasi-linear   equations  is  much less understood.
  The first well  posedness   results, on any     (higher dimensional)  quasilinear  hyperbolic system,  which go  beyond the classical Sobolev exponents,      obtained  in    \cite{Ba-Ch1}, \cite{Ba-Ch2}, and \cite{Ta1}, \cite{Ta2} and   \cite{Kl-R1},   do not take into account the 
  specific (null)  structure   of the   equations.  Yet the presence of such   structure   was  crucial    in the derivation of the optimal  results mentioned above,   for geometric  semilinear equations.  In the case of the Einstein equations it is not at all clear what such structure should be, if there is one at all.  Indeed,  the only specific structural condition, known for (EVE),    discovered    in \cite{wNC} under the name  of  the \textit{weak null condition},   is    not at all adequate   for   improved  well posedness results, see remark  \ref{rem:weekN}.  It is known however, see \cite{L}, that without such a structure  one cannot have well posedeness   for  exponents\footnote{Note that the dimension here is $n=3$.}  $s\le 2$. Yet (EVE)  are of  fundamental  importance   and as such    it is not unreasonable  to expect that such a structure must exist.

   Even assuming     such a structure,  a result of well-posedness  for the Einstein equations at, or near,  the critical regularity  $s_c=3/2$
is not only completely out of reach  but may in fact be  wrong.    This is due to the presence 
  of a different scaling  connected to  the geometry of boundaries of causal domains.  It is  because of  this  more subtle  scaling that we  need   at least $L^2$-bounds for the     curvature  to derive a lower bound on the radius of injectivity  of    null    hypersurfaces   and thus control their local regularity properties. This imposes a crucial   obstacle  to    well posedness below  $s=2$.  Indeed, as we will  show
   in the next subsection, any    such    result  would require, crucially,   bilinear and even trillinear  estimates  for solutions to wave equations of the form  $\square_\g \phi=F$. Such  estimates, however, depend   on    Fourier   integral representations, with a   phase function
    $u$  which   solves the     eikonal equation  $\g^{\a\b}\pr_\a  u \pr_\b u=0$.   Thus    the much needed
      bilinear estimates depend, ultimately, on the regularity properties  of the level hypersurfaces of the phase $u$ which are, of course, null.   The catastrophic        breakdown of    the regularity of these null  hypersurfaces,   in the absence of a lower bound for the injectivity  radius,  would     make     these  Fourier  integral representations    entirely  useless.
   
 These considerations  lead one to conclude that, the  following conjecture,   proposed in \cite{PDE},  is  most probably 
sharp in so far as the minimal     number of derivatives    in $L^2$  is concerned:

\noindent {\bf Conjecture} [Bounded $L^2$ Curvature Conjecture (BCC)]\quad 
 {\it The Einstein- vacuum equations  admit   local    Cauchy developments  for initial data sets   $(\Si_0, g, k)$
with locally  finite $L^2$ curvature and 
locally finite
$L^2$ norm of the  first covariant derivatives of $k$\footnote{As we shall see,  from the precise theorem     stated below, other  weaker conditions, such as  a lower bound
on the volume radius, are needed.}.}\\

\begin{remark}
It is important  to emphasize here that the conjecture
 should be primarily  interpreted as a continuation argument
for the Einstein equations; that is  the space-time
constructed by evolution from
smooth data can be smoothly continued,
together with a time foliation, 
as long as the curvature  of the foliation
and the  first 
 covariant derivatives of its second fundamental form
remain $L^2$- bounded on the leaves of the foliation. 
In particular the conjecture implies the  break-down criterion
 previously   obtained  in \cite{conditional} and improved 
 in \cite{Pa}, \cite{W}. According to that criterion a vacuum  space-time,
  endowed with   a constant mean curvature   (CMC) foliation  $\Si_t$,  can be extended, together 
  with the foliation, as long as the $L_t^1 L^\infty(\Si_t) $    norm
  of the deformation tensor of the  future unit normal to the 
  foliation remains  bounded.    It is straightforward to see, by standard energy estimates,  that   this  condition implies  bounds   for the $L_t^\infty L^2(\Si_t)$  norm of the  space-time   curvature  from which   one can derive bounds
  for the  induced curvature tensor $R$ and the  first derivatives
  of the second fundamental form $k$. Thus, 
  if  we can ensure that the time of  existence  of  a space-time  foliated by   $\Si_t$   depends only on the  $L^2$ norms of    $R$ and  first covariant  derivatives of $k$,  we can extend the space-time  indefinitely.
   \end{remark}

   \subsection{Brief history}    
 The conjecture has its roots in the  remarkable   developments  
 of the last twenty years centered around  the  issue of optimal  well-posedness for semilinear wave equations. 
 The case of the Einstein equations  turns out  to   be a lot more complicated    due to the quasilinear character of the equations.
 To make  the discussion more tangible   it is  worthwhile to recall  the  form   of the Einstein vacuum equations in the wave gauge. 
 Assuming  given coordinates $x^\a$,  verifying  $\square_\g x^\a=0$,
  the  metric coefficients $g_{\a\b}=\g(\pr_\a, \pr_\b)$,  with respect to these coordinates,  verify the system of quasilinear wave equations,
 \bea
 \label{E-wave.coord}
 g^{\mu\nu}\pr_\mu\pr_\nu g_{\a\b} =F_{\a\b}(g, \pr g)
 \eea
 where $F_{\a\b}$ are   quadratic functions of   $\pr g$, i.e. the derivatives  of $g$ with respect to the coordinates $x^\a$.
  In a first approximation we  may   compare   \eqref{E-wave.coord}
 with the semilinear wave equation,
 \bea
 \label{semil}
 \square\phi=F(\phi, \pr \phi) 
 \eea
 with $F$ quadratic in $\pr\phi$. Using standard energy estimates,
  one can prove an estimate, roughly,   of the form:
  \beaa
  \|\phi(t)\|_{s}\les   \|\phi(0)\|_{s}\exp\left(C_s\int_0^t \|\pr \phi(\tau)\|_{L^{\infty}} d\tau\right) .
  \eeaa
 The classical exponent $s >3/2+1$  arises  simply 
 from the  Sobolev embedding of $H^r$, $r>3/2$ 
  into  $L^\infty$.  To go beyond  the classical exponent, see   \cite{Po-Si}, 
 one  has to replace   Sobolev    inequalities with  
 Strichartz estimates of, roughly,  the following type,
 \beaa
\left( \int_0^t   \|\pr \phi(\tau)\|_{L^{\infty}}  ^2d\tau\right)^{1/2}
\les C\left(\|\pr \phi(0)\|_{H^{1+\ep}}+\int_0^t \|\square \phi (\tau)\|_{H^{1+\ep}}\right)
 \eeaa 
 where  $\ep>0$ can be chosen arbitrarily small. This leads to
  a gain of $1/2$ derivatives, i.e. we can prove well-posedness 
  for equations  of type  \eqref{semil}  for any exponent $s>2$.
   
 The same  type of improvement in the case of quasilinear equations requires a highly non-trivial extension of such
  estimates  for wave operators with non-smooth coefficients.
   The first improved regularity results  for  quasilinear wave equations of the type,
   \bea
   g^{\mu\nu }(\phi) \pr_\mu\pr_\nu\phi= F(\phi, \pr \phi)\label{eq:wave-intr}
   \eea
    with $g^{\mu\nu}(\phi)$    a non-linear perturbation of 
    the Minkowski metric  $m^{\mu\nu}$, are due to 
  \cite{Ba-Ch1}, \cite{Ba-Ch2}, and \cite{Ta1}, \cite{Ta2} and 
  \cite{Kl-R1}.  
   The best known  results  for  equations  of type \eqref{E-wave.coord}  were obtained 
  in 
   \cite{Kl-R2} and \cite{Sm-Ta}.  According to them
    one can lower the Sobolev exponent  $s>5/2$ in 
    Theorem \ref{thm:Bruhat} to $s>2$.  It turns out,  see \cite{L},  that 
      these  results are sharp    in the  general class of  quasilinear wave equations    of type   \eqref{E-wave.coord}. To do better one needs   to take into account the special structure of the  Einstein      equations and rely     on  a class of  estimates  which go  beyond   Strichartz, namely  the so called bilinear estimates\footnote{Note that   no  such   result,   i.e.  well-posedness  for $s=2$,  is  presently  known  for   either   scalar equations of the   form   \eqref{eq:wave-intr}      or systems of the form \eqref{E-wave.coord}.}.

      In the case of     semilinear wave equations,
      such as Wave Maps,  Maxwell-Klein-Gordon and  Yang-Mills,  
      the first results  which make use     of bilinear estimates     
       go back to     \cite{Kl-Ma1}, \cite{Kl-Ma2}, \cite{Kl-Ma3}. 
       In the particular case of the Maxwell-Klein-Gordon and  Yang-Mills equation   the
       main observation  was that,  after the choice of a special gauge (Coulomb gauge),  the   most dangerous 
       nonlinear  terms exhibit a special,  null   structure  
      for   which one can apply the  bilinear  estimates 
         derived in   \cite{Kl-Ma1}.           
          With the help of  these  
         estimates  one was able to derive a well posedness 
         result, in  the flat Minkowski space  $\RRR^{1+3}$, 
         for the exponent   $s=s_c+1/2=1$, where $s_c=1/2$         is the critical Sobolev exponent   in that case\footnote{ This   corresponds precisely    to  the $s=2$  exponent in the case of          the Einstein-vacuum equations}.

                          To  carry out   a similar program  in the case of the  Einstein equations  one would need, at the very  least, the following crucial ingredients:\\
                         
                         {\em\begin{enumerate}
\item[{\bf A}.] Provide a   coordinate   condition, relative to which the Einstein vacuum equations verifies an appropriate version of the null condition.
\item[{\bf B}.]  Provide an appropriate    geometric  framework  for  deriving          bilinear estimates     for the null quadratic terms   appearing in the previous step.  
 \item[{\bf C}.] Construct an effective   progressive wave representation    $\Phi_F$   (parametrix)   for  solutions to 
the scalar linear wave equation  $\square_\g\phi=F$,    derive  appropriate   bounds for  both the parametrix and the corresponding error term   $E=F-\square_\g\Phi_F$ and use them to derive the desired bilinear estimates.
\end{enumerate}
    }
\noindent As it turns out, the proof of several bilinear estimates of Step B reduces to the proof of sharp $L^4(\MM)$ Strichartz estimates  for a localized version of the parametrix of step C. Thus we will also need the following fourth ingredient. 
                         {\em\begin{enumerate}
\item[{\bf D}.] Prove sharp $L^4(\MM)$ Strichartz estimates for a localized version of the parametrix of step C. \\
\end{enumerate}
    }

    Note that   the last three   steps need  to be implemented using only   hypothetical $L^2$ bounds for the space-time curvature tensor,  consistent with the conjectured result. 
    To start with,   it is  not at all  clear    what should be the correct     coordinate  condition, or even if there is one for that matter.
    \begin{remark}\label{rem:weekN}
  As mentioned above,  the only known   structural condition  related  to the classical  null condition, called the \textit{weak null condition}, tied to  wave coordinates,
 fails the test. Indeed, the following  simple  system  in Minkowski space  verifies  the  weak null condition  and yet, according to \cite{L},  it is ill posed for $s=2$.
 \beaa \Box \phi=0, \qquad \Box\psi= \phi\c \De\phi.
 \eeaa
 Coordinate conditions, such as    spatial harmonic\footnote{Maximal foliation  together with  spatial harmonic  coordinates   on  the leaves of the foliation would be the  coordinate condition closest in spirit to    the Coulomb gauge.  },  also  do not  seem to work. 
\end{remark}
  We rely  instead     on a      Coulomb type condition,  for   orthonormal frames,  adapted  to a maximal foliation.  Such a gauge condition 
 appears naturally if we adopt  a  Yang-Mills description    of the Einstein    field equations using
  Cartan's formalism of moving frames\footnote{ We would like to thank L. Anderson  for pointing     out to  us the possibility of using       such a  formalism as a  potential  bridge  to  \cite{Kl-Ma2} .}, see \cite{Ca}.   It is important  to note   nevertheless
  that    it is not at all a priori   clear   that such a     choice  would  do the job. Indeed, the null  form nature 
  of the  Yang-Mills equations  in the Coulomb  gauge   is only revealed once we  commute the   resulting   equations   with     the projection operator $\PP$ on the divergence free vectorfields.  
  Such an operation is natural in that case,  since   $\PP$ commutes with the flat  d'Alembertian. 
     In the case  of the     Einstein   equations,  however,  the  corresponding   commutator  term  $[\square_\g,  \PP]$     generates\footnote{Note also that additional error terms  are generated   by projecting the    equations   on the components of the frame. } a whole host of new terms  and it is  quite a miracle that  they  can all be treated  by    an  extended version of bilinear estimates.      At an even more  fundamental level, the flat Yang-Mills equations  possess natural energy estimates  based on the  time symmetry of the Minkowski space.
 There are  no such timelike   Killing   vectorfield  in curved space.  We  have  to rely instead  on   the future unit normal to   the maximal foliation $\Si_t$     whose deformation tensor 
 is non-trivial.  This leads to another class 
 of    nonlinear terms which  have to be  treated    by a  novel trilinear estimate.

  We   will make more comments 
   concerning the implementations of   all  four  ingredients   later on,  in the   section \ref{sec:strategyproof}.
   
   \begin{remark}
   In addition to the  ingredients  mentioned  above, we also need  a mechanism of reducing     the proof of the  conjecture to  small  data,  in an appropriate sense. Indeed, even in the  flat case,  
   the Coulomb gauge condition cannot be   globally imposed for large data. In fact    \cite{Kl-Ma3}
    relied on a cumbersome technical  device based on   local Coulomb gauges, defined on domain of dependence of small  balls.      Here we rely instead on    a variant  of   the gluing construction of    \cite{Cor},  \cite{CorSch}, see section \ref{sec:reductionsmall}.         
   \end{remark}
\section{Statement of the main results}

\subsection{Maximal foliations}\lab{sec:maxfoliation}

In this section, we recall some well-known facts about maximal foliations (see for example the introduction in \cite{ChKl}). We assume the space-time $(\MM, \g)$ to be foliated by  the level surfaces $\Si_t$ of a time function $t$. Let $T$ denote the unit normal to $\Si_t$, and let $k$ the the second fundamental form of  $\Si_t$, i.e. $k_{ab}=-\g(\D_aT,e_b)$, where $e_a, a=1, 2, 3$ denotes an arbitrary frame on $\Si_t$ and $\D_aT=\D_{e_a}T$. We assume that the $\Si_t$ foliation is maximal, i.e. we have:
\bea
\label{maxfoliation}
\tr_g k=0
\eea
where $g$ is the induced metric  on $\Si_t$. The constraint equations on $\Si_t$ for a maximal 
foliation are given by:
\bea\label{constk}
\nabla^ak_{ab}=0,
\eea
where $\nabla$ denotes the induced covariant derivative on $\Si_t$, and
\begin{equation}\lab{constR}
R_{scal}=|k|^2.
\end{equation} 
Also, we denote by $n$ the lapse of the $t$-foliation, i.e. $n^{-2}=-\g(\D t, \D
t)$. $n$ satisfies the following elliptic equation on $\Si_t$:
\bea\lab{eqlapsen}
\Delta n=n|k|^2.
\eea
Finally, we recall the structure equations of the maximal foliation:
\bea\lab{eq:structfol1}
\nabla_0k_{ab}= \R_{a\,0\,b\,0}-n^{-1}\nabla_a\nabla_bn-k_{ac}k_b\,^c,
\eea
\bea\lab{eq:structfol2}
\nabla_ak_{bc}-\nabla_bk_{ac}=\R_{c0ab}
\eea
and:
\bea\lab{eq:structfol3}
R_{ab}-k_{ac}k^c\,_b=\R_{a0b0}.
\eea

\subsection{Main Theorem}

We recall below the definition of the volume radius on a general Riemannian manifold $M$.
\begin{definition}
Let $B_r(p)$ denote the geodesic ball of center $p$ and radius $r$. The volume radius $r_{vol}(p,r)$ at a point $p\in M$ and scales $\leq r$ is defined by
$$r_{vol}(p,r)=\inf_{r'\leq r}\frac{|B_{r'}(p)|}{r^3},$$
with $|B_r|$ the volume of $B_r$ relative to the metric on $M$. The volume radius $r_{vol}(M,r)$ of $M$ on scales $\leq r$ is the infimum of $r_{vol}(p,r)$ over all points $p\in M$.
\end{definition}

Our main result is the following:
\begin{theorem}[Main theorem]\lab{th:main}
Let $(\MM, {\bf g})$ an asymptotically flat solution to the Einstein vacuum equations \eqref{EVE} together with a maximal foliation by space-like hypersurfaces $\Si_t$ defined as level hypersurfaces of a time function $t$. Assume that the initial slice $(\Si_0,g,k)$ is such that  the Ricci curvature  $\mbox{Ric} \in L^2(\Si_0)$, $\nabla k\in L^2(\Si_0)$, and $\Si_0$ has a strictly positive volume radius on scales $\leq 1$, i.e. $r_{vol}(\Si_0,1)>0$. Then,
\begin{enumerate}
\item \textbf{$L^2$ regularity.} There exists a time
$$T=T(\norm{\mbox{Ric}}_{L^2(\Si_0)}, \norm{\nabla k}_{L^2(\Si_0)}, r_{vol}(\Si_0,1))>0$$
and a constant 
$$C=C(\norm{\mbox{Ric}}_{L^2(\Si_0)}, \norm{\nabla k}_{L^2(\Si_0)}, r_{vol}(\Si_0,1))>0$$
such that the following control holds on $0\leq t\leq T$:
$$\norm{\R}_{L^\infty_{[0,T]}L^2(\Si_t)}\leq C,\,\norm{\nabla k}_{L^\infty_{[0,T]}L^2(\Si_t)}\leq C\textrm{ and }\inf_{0\leq t\leq T}r_{vol}(\Si_t,1)\geq \frac{1}{C}.$$
\item \textbf{Higher regularity.} Within the same time interval as in
  part (1) we also have the higher derivative  estimates\footnote{Assuming that the initial has more regularity so that the right-hand side of \eqref{ffllffll} makes sense.},
\bea\lab{ffllffll}
\sum_{|\a|\le m}\|\D^{(\a)}\R\|_{L^\infty_{[0,T]}L^2(\Sigma_t)}\le  C_m  \sum_{|i|\le m} \bigg [\|\nab^{(i)}\textrm{Ric}\|_{L^2(\Sigma_0)} + \|\nab^{(i)} \nab k\|_{L^2(\Sigma_0)}\bigg],
\eea
where $C_m$  depends only on  the   previous $C$ and $m$. 
\end{enumerate}
\end{theorem}
\begin{remark}
Since  the core of the  main theorem is local in nature  we do not need to be very precise here with  our asymptotic flatness assumption.    We may thus  assume the existence of a coordinate system
 at infinity,   relative to which the metric     has two derivatives bounded in $L^2$, with
  appropriate asymptotic decay. Note that such bounds could be deduced from 
   weighted $L^2$ bounds  assumptions for $\mbox{Ric}$ and $\nab k$.
\end{remark}
\begin{remark} Note  that the dependence  on
 $ \norm{\mbox{Ric}}_{L^2(\Si_0)}    , \norm{\nabla k}_{L^2(\Si_0)}$    in the main theorem  can be replaced by   dependence on $  \norm{\R}_{L^2(\Si_0)}$ where $\R$ denotes the space-time curvature tensor\footnote{Here and in what follows the notations $R, {\bf R}$ will stand for the Riemann curvature tensors of $\Sigma_t$ and ${\mathcal M}$, while 
 $Ric$, ${\bf {Ric}}$ and $R_{scal}, {\bf {R}}_{scal}$ will denote the corresponding Ricci and scalar curvatures.}.  Indeed this follows from the following well known   $L^2$   estimate (see section 8 in \cite{conditional}).
 
\bea\label{eq:L2-estim-k}
\int_{\Si_0}|\nab k|^2+\frac{1}{4} |k|^4\le\int_{\Si_0} |\R|^2.
\eea
and the Gauss equation relating $\mbox{Ric}$ to $\R$.
\end{remark}

\subsection{Reduction to small initial data}\lab{sec:reductionsmall}

We first need an appropriate covering of $\Si_0$ by harmonic coordinates. This is obtained using the following general result based on Cheeger-Gromov convergence of Riemannian manifolds. 

\begin{theorem}[\cite{An} or Theorem 5.4 in \cite{Pe}]\label{th:coordharm}
Given $c_1>0, c_2>0, c_3>0$, there exists $r_0>0$ such that any 3-dimensional, complete, Riemannian manifold $(M,g)$ with $\norm{\mbox{Ric}}_{L^2(M)}\leq c_1$ and volume radius at scales $\leq 1$ bounded from below by $c_2$, i.e. $r_{vol}(M,1)\geq c_2$, verifies the following property:

Every geodesic ball $B_r(p)$ with $p\in M$ and $r\leq r_0$ admits a system of harmonic coordinates $x=(x_1,x_2,x_3)$ relative to which we have
\begin{equation}\label{coorharmth1}
(1+c_3)^{-1}\de_{ij}\leq g_{ij}\leq (1+c_3)\de_{ij},
\end{equation}
and
\begin{equation}\label{coorharmth2}
r\int_{B_r(p)}|\partial^2g_{ij}|^2\sqrt{|g|}dx\leq c_3.
\end{equation}
\end{theorem}

We consider $\ep>0$ which will be chosen as a small universal constant. We apply theorem \ref{th:coordharm} to the Riemannian manifold $\Si_0$. Then, there exists a constant:
$$r_0=r_0(\norm{\mbox{Ric}}_{L^2(\Si_0)}, \norm{\nabla k}_{L^2(\Si_0)}, r_{vol}(\Si_0,1),\ep)>0$$
such that every geodesic ball $B_r(p)$ with $p\in \Si_0$ and $r\leq r_0$ admits a system of harmonic coordinates $x=(x_1,x_2,x_3)$ relative to which we have:
$$(1+\ep)^{-1}\de_{ij}\leq g_{ij}\leq (1+\ep)\de_{ij},$$
and
$$r\int_{B_r(p)}|\partial^2g_{ij}|^2\sqrt{|g|}dx\leq \ep.$$

Now, by the asymptotic flatness of $\Si_0$,  the complement of its end   can be covered by the union of a finite number of geodesic balls of radius $r_0$, where the number $N_0$ of geodesic balls required only depends on $r_0$. In particular, it is therefore enough to obtain the control of $\R$, $k$ and $r_{vol}(\Si_t,1)$ of Theorem \ref{th:main} when one restricts to the domain of dependence of one such ball. Let us denote this ball by $B_{r_0}$. Next, we rescale the metric of this geodesic ball by:
$$g_\la(t,x)=g(\la t,\la x),\, \la=\min\left(\frac{\ep^2}{\norm{R}^2_{L^2(B_{r_0})}},\,\frac{\ep^2}{\norm{\nabla k}^2_{L^2(B_{r_0})}},\, r_0\ep\right)>0.$$
Let\footnote{Since in  what follows there is no danger to confuse  the Ricci curvature  $\mbox{Ric}$ with the scalar  curvature $R$
we use the short hand  $R$    to denote the full curvature  tensor   $\mbox{Ric}$.  }    $R_\la, k_\la$ and $B_{r_0}^\la$ be the rescaled versions of $R$, $k$ and $B_{r_0}$. Then, in view of our choice for $\la$, we have:
$$\norm{R_\la}_{L^2(B_{r_0}^\la)}=\sqrt{\la} \norm{R}_{L^2(B_{r_0})}\leq \ep,$$
$$\norm{\nabla k_\la}_{L^2(B_{r_0}^\la)}=\sqrt{\la} \norm{\nabla k}_{L^2(B_{r_0})}\leq \ep,$$
and 
$$\norm{\partial^2g_\la}_{L^2(B_{r_0}^\la)}=\sqrt{\la} \norm{\partial^2g}_{L^2(B_{r_0})}\leq \sqrt{\frac{\la\ep}{r_0}}\leq \ep.$$
Note that $B_{r_0}^\la$ is the rescaled version of $B_{r_0}$. Thus, it is a geodesic ball for $g_\la$ of radius $\frac{r_0}{\la}\geq \frac{1}{\ep}\geq 1$. Now, considering $g_\la$ on $0\leq t\leq 1$ is equivalent to considering $g$ on $0\leq t\leq \la$. Thus, since $r_0$, $N_0$ and $\la$ depend only on $\norm{R}_{L^2(\Si_0)}$, $\norm{\nabla k}_{L^2(\Si_0)}$, $r_{vol}(\Si_0,1)$ and $\ep$, Theorem \ref{th:main} is equivalent to the following theorem:

\begin{theorem}[Main theorem, version 2]\lab{th:mainbis}
Let $(\MM, {\bf g})$ an asymptotically flat solution to the Einstein vacuum equations \eqref{EVE} together with a maximal foliation by space-like hypersurfaces $\Si_t$ defined as level hypersurfaces of a time function $t$. Let $B$ a geodesic ball of radius one in $\Si_0$, and let $D$ its domain of dependence. Assume that the initial slice $(\Si_0,g,k)$ is such that:
$$\norm{R}_{L^2(B)}\leq \ep,\,\norm{\nabla k}_{L^2(B)}\leq \ep\textrm{ and }r_{vol}(B,1)\geq \frac{1}{2}.$$
Let $B_t=D\cap \Si_t$ the slice of $D$ at time $t$. Then:
\begin{enumerate}
\item\textbf{$L^2$ regularity.} There exists a small universal constant $\ep_0>0$ such that if $0<\ep<\ep_0$, then the following control holds on $0\leq t\leq 1$:
$$\norm{\R}_{L^\infty_{[0,1]}L^2(B_t)}\lesssim \ep,\,\norm{\nabla k}_{L^\infty_{[0,1]}L^2(B_t)}\lesssim \ep\textrm{ and }\inf_{0\leq t\leq 1}r_{vol}(B_t,1)\geq \frac{1}{4}.$$
\item \textbf{Higher regularity.} The following bounds hold on $0\leq t\leq 1$:
\bea
\sum_{|\a|\le m}\|\D^{(\a)} \R\|_{L^\infty_{[0,1]}L^2(B_t)}  \les  
\sum_{|i|\le m}  \|\nab^{(i)} \textrm{Ric}\|_{L^2(B)} + \|\nab^{(i)} \nab k\|_{L^2(B)}.
\eea
\end{enumerate}
\end{theorem}

\noindent{\bf Notation:}\,\, In the statement of Theorem \ref{th:mainbis}, and in the rest of the paper, the notation $f_1\les f_2$ for two real positive scalars $f_1, f_2$ means that there exists a universal constant $C>0$ such that:
$$f_1\leq C f_2.$$\\ 

Theorem \ref{th:mainbis} is not  yet   in  a suitable form for our proof since some of our constructions will be global in space and may not be carried out on a subregion $B$ of $\Si_0$. Thus, we glue a smooth asymptotically flat solution of the constraint equations \eqref{const} outside of $B$, where the gluing takes place in an annulus just outside  $B$. This can be achieved using the construction  in \cite{Cor},  \cite{CorSch}. We finally get an asymptotically flat solution to the constraint equations,  defined everywhere on $\Si_0$,  which  agrees with our original data set $(\Si_0,g,k)$ inside $B$. We still denote this data set by $(\Si_0,g,k)$. It satisfies the bounds: 
$$\norm{R}_{L^2(\Si_0)}\leq 2\ep,\,\norm{\nabla k}_{L^2(\Si_0)}\leq 2\ep\textrm{ and }r_{vol}(\Si_0,1)\geq \frac{1}{4}.$$

\begin{remark}
Notice that the gluing process in \cite{Cor}--\cite{CorSch} requires the kernel of a certain linearized operator to be trivial. This is achieved by conveniently choosing the asymptotically flat solution to \eqref{const} that is glued outside of $B$ to our original data set. This choice is always possible since the metrics for which the kernel is nontrivial are non generic (see \cite{BCS}).
\end{remark}

\begin{remark}
Assuming only $L^2$ bounds on $R$ and $\nabla k$ is not enough to carry out the construction in the above mentioned  results. However, the  problem solved  there  remains  subcritical  at   our desired   level of regularity      and  thus we   believe that a closer look at the construction  in  \cite{Cor}--\cite{CorSch}, or an alternative construction, should  be able to provide
the desired result. This is an  open problem.
\end{remark}

\begin{remark}
Since  $\|k\|_{L^4(\Si_0)} ^2\le \|\mbox{Ric}\|_{L^2} $
 we  deduce   that 
$\norm{k}_{L^2(B)}\lesssim \ep^{1/2}$
 on the geodesic ball $B$ of radius one.
Furthermore, asymptotic flatness is compatible with a decay of $|x|^{-2}$ at infinity, and in particular with $k$ in $L^2(\Si_0)$. So we may assume that the gluing process is such that the resulting $k$ satisfies:
$$\norm{k}_{L^2(\Si_0)}\lesssim \ep.$$
\end{remark}

Finally, we have reduced Theorem \ref{th:main} to the case of a small initial data set:
\begin{theorem}[Main theorem, version 3]\lab{th:mainter}
Let $(\MM, {\bf g})$ an asymptotically flat solution to the Einstein vacuum equations \eqref{EVE} together with a maximal foliation by space-like hypersurfaces $\Si_t$ defined as level hypersurfaces of a time function $t$. Assume that the initial slice $(\Si_0,g,k)$ is such that:
$$\norm{R}_{L^2(\Si_0)}\leq \ep,\,\norm{k}_{L^2(\Si_0)}+\norm{\nabla k}_{L^2(\Si_0)}\leq \ep\textrm{ and }r_{vol}(\Si_0,1)\geq \frac{1}{2}.$$
Then:
\begin{enumerate}
\item\textbf{$L^2$ regularity.} There exists a small universal constant $\ep_0>0$ such that if $0<\ep<\ep_0$,  the following control holds on $0\leq t\leq 1$:
$$\norm{\R}_{L^\infty_{[0,1]}L^2(\Si_t)}\lesssim \ep,\,\norm{k}_{L^\infty_{[0,1]}L^2(\Si_t)}+\norm{\nabla k}_{L^\infty_{[0,1]}L^2(\Si_t)}\lesssim \ep\textrm{ and }\inf_{0\leq t\leq 1}r_{vol}(\Si_t,1)\geq \frac{1}{4}.$$
\item \textbf{Higher regularity.} The following control holds on $0\leq t\leq 1$:
\bea
\sum_{|\a|\le m}\|\D^{(\a)} \R\|_{L^\infty_{[0,1]}L^2(\Sigma_t)}  \les  
 \sum_{|i|\le m}\ \|\nab^{(i)} \textrm{Ric}\|_{L^2(\Sigma_0)} + \|\nab^{(i)} \nab k\|_{L^2(\Sigma_0)}.
\eea
\end{enumerate}
\end{theorem}

The rest of this paper is devoted to the proof of Theorem \ref{th:mainter}.

\subsection{Strategy of the proof}\lab{sec:strategyproof}

The proof of Theorem \ref{th:mainter} consists of four steps. \\

\noindent{\bf Step A (Yang-Mills formalism)} We  first cast the Einstein-vacuum equations  in a Yang-Mills form. This relies on  the Cartan formalism of moving frames. The idea 
 is to give up on    a choice of coordinates and instead express 
 the Einstein vacuum equations  in terms of the connection  $1$-forms    associated to  moving orthonormal  frames,  i.e.  vectorfields  $e_\a$,   
 which verify, 
 \beaa
\g(e_\a, e_\b)=\m_{\a\b}=\mbox{diag}(-1,1,1,1).
\eeaa
The connection $1$-forms (they are to be interpreted 
as $1$-forms with respect to the  external index $\mu$  with values  in the Lie algebra of $so(3,1)$),  defined  by the formulas,
\bea
(\A_\mu)_{\a\b}=\g(\D_{\mu}e_\b,e_\a)\label{intr.A}
\eea
verify the equations,
\bea
\label{intr-YM}
\D^\mu \F_{\mu\nu} + [ \A^\mu, \F_{\mu\nu}]=0
\eea
where,  denoting $ (\F_{\mu\nu})_{\a\b}:=\R_{\a\b\mu\nu}$,
\bea
\label{intr-YM1}
( \F_{\mu\nu})_{\a\b}=\big(\D_\mu \A_\nu-\D_\nu \A_\mu
-[\A_\mu,
\A_\nu]\big)_{\a\b}.
\eea
In other words  we can interpret   the curvature   tensor    as the curvature of the $so(3,1)$-valued  connection  1-form  $\A$. Note also that the covariant derivatives    are taken only with respect to the \textit{external indices} $\mu, \nu$ and  do not affect the \textit{internal indices} $\a,\b$.
 We can rewrite \eqref{intr-YM} in the form,
 \bea
 \label{intr-YM2}
 \square_\g \A_\nu-\D_\nu(\D^\mu \A_\mu) =\J_\nu(\A, \D \A)
 \eea
 where,
 \beaa
 \J_\nu=\D^\mu([\A_\mu, \A_\nu])-[\A_\mu, \F_{\mu\nu}].
 \eeaa
 Observe that the equations  \eqref{intr-YM}-\eqref{intr-YM1}
 look just like the Yang-Mills equations    on a fixed Lorentzian
 manifold $(\MM, \g)$ except, of course,   that  in our case
 $\A$ and $\g$ are  not independent but  rather connected  by  \eqref{intr.A},
  reflecting the quasilinear  structure of the Einstein equations.
  Just as in the case   of   \cite{Kl-Ma1},  which  establishes the well-posedness of the  Yang-Mills equation  in Minkowski space 
  in   the    energy norm (i.e. $s=1$),    we    rely   in an essential manner  on   a    Coulomb type  gauge condition.   More precisely,  we  take 
   $e_0$  to be  the  future  unit normal to the 
   $\Si_t$ foliation and choose  $e_1, e_2, e_3$ an 
   orthonormal basis  to  $\Si_t$,   in such a way
   that we have, essentially (see precise discussion in 
   section \ref{sect:compatible}),  $\div A=\nab^i A_i=0$, where 
   $A$ is the spatial component of $\A$.   It turns out that $A_0$ satisfies
   an elliptic equation  while each  component  $A_i=\g(\A, e_i)$,  $ i=1,2,3$
   verifies an equation of the form,
   \bea
     \square_\g A_i  &=&-\pr_i (\pr_0 A_0)+ A^j \pr_j A_i +A^j\pr_i A_j+\lot\label{intr:null}
\eea
with $\lot$ denoting nonlinear  terms which  can be treated by more elementary techniques (including non sharp Strichartz estimates). \\     

\noindent{\bf Step B (Bilinear and trilinear estimates)} To eliminate $\pr_i (\pr_0 A_0)$  in \eqref{intr:null},  we need   to project \eqref{intr:null} onto  divergence free vectorfields   with the help of  a non-local operator which we denote  by $\PP$.      In  the  case of the   flat  Yang-Mills 
equations, treated in   \cite{Kl-Ma1},   this leads to an equation 
of the form,
\beaa
\square  A_i&=&\PP(A^j \pr_j A_i ) +\PP(A^j\pr_i A_j)+\lot
\eeaa
where  both terms on the right can be handled  by bilinear 
 estimates.    In our case  we encounter  however three 
 fundamental differences with the  flat  situation of    \cite{Kl-Ma1}.
 \begin{itemize}
 \item To start with the operator $\PP$ does not commute 
 with $\square_\g$.   It turns out,  fortunately,   that  the terms
 generated by  commutation     can still be estimated  by 
 an extended  class of bilinear estimates which includes  contractions 
 with the curvature tensor, see section \ref{sec:proofconj}. 
 \item All  energy estimates  used in   \cite{Kl-Ma1}
 are based on the standard    timelike   Killing vectorfield $\pr_t$. In our case the corresponding 
 vectorfield  $e_0=T$ ( the future unit   normal to $\Si_t$) is not Killing.
 This leads  to another class of trilinear  error terms  which 
  we  discuss in sections    \ref{sec:bobo6} and \ref{sec:proofconj}.
  \item The main difference 
   with  \cite{Kl-Ma1}  is that we now need bilinear and trilinear  estimates
    for solutions of wave equations on    background metrics 
   which   possess only limited  regularity.  
\end{itemize} 
    This last  item is a major   problem,  both conceptually and technically.
   On the conceptual side we need to rely on  a more geometric    proof of  bilinear estimates  based on   a plane    wave representation    formula\footnote{We follow   the proof   of the  bilinear estimates  outlined   in  \cite{Kl-R3}
    which differs substantially  from that of  \cite{Kl-Ma1} and
    is reminiscent of the  null frame space strategy used by   Tataru  in his      fundamental paper \cite{Ta3}. } for solutions of  scalar wave equations,
   \beaa
   \square_\g \phi=0.
   \eeaa 
   The proof of the bilinear estimates rests on the representation formula\footnote{\eqref{parametrix.intr} actually corresponds to the representation formula for a half-wave. The full representation formula corresponds to the sum of two half-waves (see section \ref{sec:parametrix})}
   \bea
   \label{parametrix.intr}
   \phi_f(t,x)=\int_{\SSS^2} \int_0^\infty e^{i\la \uom(t,x)} \,  f(\la\om) \la^2 d\la  d\om
   \eea
    where $f$ represents schematically the initial data\footnote{Here $f$ is in fact at the level of  the Fourier transform of the initial data  and the norm  $\|\la f\|_{L^2(\RRR^3)  }  $  corresponds, roughly,  to the $H^1$ norm of the data .},    and where  $\uom$ is a solution of the eikonal equation\footnote{In the flat Minkowski space   $\uom(t,x)=t\pm x\c\om$.},
   \bea
   \g^{\a\b}\pr_\a\uom\,  \pr_\b \uom=0,\label{eikonal-intr}
   \eea
   with appropriate initial conditions on $\Si_0$ and $d\om$
    the area element of the standard sphere in $\RRR^3$.\\
    
    \begin{remark}\label{rem:scalarize}
    Note that \eqref{parametrix.intr} is a parametrix for a scalar wave equation. The lack of a good parametrix for a covariant wave equation forces us to develop a strategy based on writing the main equation in components relative to a frame, i.e. instead of dealing with the tensorial wave equation \eqref{intr-YM2} directly, we consider the system of scalar wave equations \eqref{intr:null}. Unlike the flat case, this ``scalarization" procedure produces several terms which are potentially dangerous, and it is fortunate, as in yet
    another manifestation of a hidden null structure of the Einstein equations, that they can still be controlled by the use of an extended\footnote{involving  contractions  between the Riemann curvature tensor and   derivatives of  solutions of   scalar  wave equations.} class of  bilinear estimates.
    \end{remark}
    
 \vspace{0.3cm}   
 
    \noindent{\bf Step C (Control of the parametrix)} To prove the bilinear and trilinear estimates of Step B, we need in particular to control the parametrix at initial time (i.e. restricted to the initial slice $\Sigma_0$)
   \bea
   \label{parametrixinit.intr}
   \phi_f(0,x)=\int_{\SSS^2} \int_0^\infty e^{i\la \uom(0,x)} \,  f(\la\om) \la^2 d\la  d\om
   \eea
 and the error term corresponding to \eqref{parametrix.intr}    
   \bea
   \label{error-param}
  Ef(t,x)=\square_\g  \phi_f(t,x)=i \int_{\SSS^2} \int_0^\infty e^{i\la \uom(t,x)}   
  \,(\square_\g \uom )
   f(\la\om)  \la^3d\la d\om
   \eea
   i.e.  $\phi_f$   is an exact solution of  $\square_\g\phi=0$  only in flat space
   in which case $ \square_\g \uom=0$. This requires the following four sub steps
{\em\begin{enumerate}
\item[{\bf C1}] Make an appropriate choice for the equation satisfied by $\uom(0,x)$ on $\Sigma_0$, and control the geometry of the foliation of $\Sigma_0$ by the level surfaces of $\uom(0,x)$.

\item[{\bf C2}] Prove that the parametrix at $t=0$ given by \eqref{parametrixinit.intr} is bounded in $\mathcal{L}(L^2(\mathbb{R}^3), L^2(\Sigma_0))$ using the estimates for $\uom(0,x)$ obtained in {\bf C1}.

\item[{\bf C3}] Control the geometry of the foliation of $\mathcal{M}$ given by the level hypersurfaces of $\uom$.

\item[{\bf C4}] Prove that the error term \eqref{error-param} satisfies the estimate $\norm{Ef}_{L^2(\mathcal{M})}\leq C\norm{\lambda f}_{L^2(\mathbb{R}^3)}$ using the estimates for $\uom$ and $\square_{\g}\uom$ proved in {\bf C3}.
\end{enumerate}
}
   
To achieve Step C3 and Step C4, we need, at the very  least,  to control  $\square_\g \uom$
in $L^\infty$. This issue   was first addressed in 
 the sequence of papers \cite{Kl-R4}--\cite{Kl-R6} where an $L^\infty$ bound for  $\square_\g \uom$ was established, 
 depending only on the $L^2$ norm of the curvature flux along null hypersurfaces.      The proof required   an interplay between 
 both  geometric and analytic techniques and  had all the appearances 
 of being sharp, i.e.    we don't expect an $L^\infty$ bound for $\square_\g \uom$  which requires bounds on  less
 than two derivatives  in $L^2$ for the metric\footnote{classically, this requires, at the very  least, the control of $\R$ in $L^\infty$}.

 To obtain the $L^2$ bound for the Fourier integral operator $E$ defined in \eqref{error-param}, we 
 need, of course, to  go  beyond uniform   estimates 
 for $\square_\g \uom$. The classical $L^2$ bounds 
 for Fourier integral operators of the form \eqref{error-param} 
 are not at all economical in terms of the number of integration by parts which are needed.   In our case the total  number of     
 such  integration by parts  is limited by  the   regularity properties of the function $\square_\g\uom$.  To
   get an $L^2$ bound for the parametrix at initial time \eqref{parametrixinit.intr}  and the error term \eqref{error-param} within such restrictive regularity properties 
   we  need, in particular:
   \begin{itemize}
   \item In Step C1 and Step C3, a precise  control of     derivatives of $\uom$ and $\square_\g \uom$ 
   with respect to both $\om$ as well as with respect to various 
   directional derivatives\footnote{Taking  into account 
   the   different    behavior in tangential and transversal directions
    with respect to the level surfaces of $\uom$.}.   
   To get optimal control  we need,
   in particular,  a  very careful   construction of   the initial condition for    $\uom $ on $\Si_0$  and then sharp space-time estimates  of Ricci   coefficients, and their  derivatives,     associated     to the foliation induced by  $\uom$.  
   
\item  In Step C2 and Step C4, a careful decompositions of    the Fourier integral operators \eqref{parametrixinit.intr} and  \eqref{error-param} in both  $\la$  and $\om$, 
 similar to the first and second dyadic decomposition in harmonic analysis, see 
 \cite{St}, as well as a third    decomposition, which in the case of  \eqref{error-param} is done with respect 
 to the space-time variables    relying on  the geometric Littlewood-Paley  theory developed in  \cite{Kl-R6}.
\end{itemize}

Below, we make further  comments on Steps C1-C4:
\begin{enumerate}
\item \textit{The choice of $u(0,x,\om)$ on $\Sigma_0$ in Step C1.} Let us note that the typical choice $u(0,x,\om)=x\c\om$ in a given coordinate system would not work for us, since we don't have enough control on the regularity of a given coordinate system within our framework. Instead, we need to find a geometric definition of $u(0,x,\om)$. A natural choice would be
$$\square_\g u=0\textrm{ on }\Sigma_0$$
which by a simple computation turns out to be the following simple variant of the minimal surface equation\footnote{In the time symmetric case $k=0$, this is exactly the minimal surface equation}
$$\div\left(\frac{\nabla u}{|\nabla u|}\right)=k\left(\frac{\nabla u}{|\nabla u|}, \frac{\nabla u}{|\nabla u|}\right)\textrm{ on }\Sigma_0.$$
Unfortunately, this choice does not allow us to have enough control of the derivatives of $u$ in the normal direction to the level surfaces of $u$. This forces us to look for an alternate equation for $u$:
$$\div\left(\frac{\nabla u}{|\nabla u|}\right)=1-\frac{1}{|\nabla u|}+k\left(\frac{\nabla u}{|\nabla u|}, \frac{\nabla u}{|\nabla u|}\right)\textrm{ on }\Sigma_0.$$
This equation turns out to be parabolic in the normal direction to the level surfaces of $u$, and allows us to obtain the desired regularity in Step C1. A closer inspection  
reveals its relation to the mean curvature flow on $\Si_0$.

\item \textit{How to achieve Step C3.} The regularity obtained in Step C1, together with null transport equations tied to the eikonal equation, elliptic systems of Hodge type, the geometric Littlewood-Paley theory of \cite{Kl-R6}, sharp trace theorems, and an extensive use of the structure of the Einstein equations, allows us to propagate the regularity on $\Sigma_0$ to the space-time, thus achieving Step C3.

\item \textit{The regularity with respect to $\om$ in Steps C1 and C3.} The regularity with respect to $x$ for $u$ is clearly limited as a consequence of the fact that we only assume $L^2$ bounds on $\R$. On the other hand, $\R$ is independent of the parameter $\om$, and one might infer that $u$ is smooth with respect to $\om$. Surprisingly, this is not at all the case. Indeed,  the regularity in $x$ obtained for $u$ in Steps C1 and C3 is better in directions tangent to the level hypersurfaces of $u$. Now, the $\om$ derivatives of the tangential directions have non zero normal components. Thus, when differentiating the structure equations with respect to $\om$, tangential derivatives to the level surfaces of $u$ are transformed to non tangential derivatives which in turn severely limits the regularity in $\om$ obtained in Steps C1 and C3.

\item \textit{How to achieve Steps C2 and C4.} The classical arguments for proving $L^2$ bounds for Fourier operators are based either on a $T T^*$ argument, or a $T^* T$ argument, which requires several integration by parts either with respect to $x$ for $T^*T$, of with respect to $(\la, \om)$ for $TT^*$. Both methods would fail by far within the regularity for $u$ obtained in Step C1 and Step C3. This forces us to design a method which allows to take advantage both of the regularity in $x$ and $\om$. This is achieved using in particular the following ingredients:
\begin{itemize}
\item geometric integrations by parts taking full advantage of the better regularity properties in directions tangent to the level hypersurfaces of $u$,

\item the standard first and second dyadic decomposition in frequency  space, with  respect to  both size and angle (see \cite{St}), an additional decomposition in physical space relying on the geometric Littlewood-Paley projections of \cite{Kl-R6} for Step C4, as well as another decomposition involving frequency and angle for Step C2.
\end{itemize}
Even with these precautions, at several places in the proof, one encounters log-divergences which have to be tackled by ad-hoc techniques,  taking full advantage of the structure of the Einstein equations. 
\end{enumerate}

\vspace{0.5cm}

\noindent{\bf Step D (Sharp $L^4(\MM)$ Strichartz estimates)} Recall that the parametrix constructed   in Step C needs also to be used to prove  sharp $L^4(\MM)$ Strichartz estimates. Indeed the proof of several bilinear estimates of Step B reduces to the proof of sharp $L^4(\MM)$ Strichartz estimates for the parametrix \eqref{parametrix.intr} with $\la$ localized in a dyadic shell. 

More precisely, let $j\geq 0$, and let $\psi$ a smooth function on $\RRR^3$ supported in 
$$\frac 1 2 \leq |\xi|\leq 2.$$
Let $\phi_{f,j}$ the parametrix \eqref{parametrix.intr} with a additional frequency localization $\la\sim 2^j$
\begin{equation}\lab{paraml}
\phi_{f,j}(t,x)=\int_{\SSS^2} \int_0^\infty  e^{i \la \uom(t,x)}\psi(2^{-j}\la)f(\la\om)\la^2d\la d\om.
\end{equation}
We will need the sharp\footnote{Note in particular that the corresponding estimate in the flat case is sharp.} $L^4(\MM)$ Strichartz estimate
\begin{equation}\lab{strichgeneralintro}
\norm{\phi_{f,j}}_{L^4(\MM)}\les 2^{\frac{j}{2}}\norm{\psi(2^{-j}\la)f}_{L^2(\RRR^3)}.
\end{equation}
The standard procedure for proving\footnote{Note that the procedure we  describe would prove not only   \eqref{strichgeneralintro}   but the full   range  of mixed Strichartz estimates.}          \eqref{strichgeneralintro} is  based  on a  $TT^*$ argument  which  reduces  it  to an  $L^\infty$ estimate for an oscillatory integral with a phase involving $\uom$. This is then achieved  by  the method of  stationary phase which  requires  quite a  few integrations by parts.    In fact   the standard  argument   would require, at the very least\footnote{The regularity \eqref{staphase} is necessary to make sense of the change of variables involved in the stationary phase method.}, that the phase function
$u=\uom$ verifies,
\begin{equation}\lab{staphase}
\partial_{t,x} u  \in L^\infty,\,\partial_{t,x}\partial_\omega^2 u\in L^\infty.
\end{equation}
This level of regularity is, unfortunately,  incompatible with the regularity  properties    of solutions
  to our  eikonal equation  \eqref{eikonal-intr}. 
 In fact,  based  on the estimates for $\uom$  derived in step C3, 
  we  are only allowed      to assume
\begin{equation}\lab{reguintro}
\partial_{t,x}  u  \in L^\infty,\, \partial_{t,x}\partial_\omega u \in L^\infty.
\end{equation}
 We are  thus forced to follow an alternative approach\footnote{ We  refer  to  the approach  based on   the  overlap estimates for wave packets derived in \cite{Sm} and \cite{Sm-Ta} in the context of Strichartz estimates respectively for $C^{1,1}$ and $H^{2+\ep}$ metrics.  Note however that our approach    does not require a wave packet  decomposition. } to the stationary phase  method inspired by    \cite{Sm} and    \cite{Sm-Ta} .

\begin{remark}
Note that apart from the results of Chapter \ref{part:yangmills} which require the projection of various tensors on a  frame, the computations and estimates in all the other chapters are covariant. 
\end{remark}

\subsection{Structure of the paper} 

In Chapter \ref{part:yangmills}, we perform Step A and Step B, i.e. we recast the Einstein equations as a quasilinear Yang-Mills type system, we prove bilinear estimates, and we reduce the proof of Theorem \ref{th:mainter} to Step C and Step D. Next, we perform Step C on the control of the plane wave parametrix \eqref{parametrix.intr}. More precisely, in Chapter \ref{part:paramtime}, we perform Step C4 on the control of the error term \eqref{error-param}.  Next, in Chapter \ref{part:spacetimeu}, we perform Step C3 on the space-time control of the optical function $\uom$. Then, in Chapter \ref{part:paraminit}, we perform Step C2 on the control of the parametrix at initial time \eqref{parametrixinit.intr}. In Chapter \ref{part:initialu}, we perform Step C1 on the control of the optical function $\uom$ on the initial slice $\Sigma_0$. Finally, in Chapter \ref{part:strich}, we prove sharp $L^4(\MM)$ Strichartz estimates localized in frequency which corresponds to Step D.

\begin{remark}
Chapter \ref{part:yangmills} summarizes the results obtained in \cite{KRS}. Chapter \ref{part:paramtime} summarizes the results obtained in \cite{param4}. Chapter \ref{part:spacetimeu} summarizes the results obtained in \cite{param3}. 
Chapter \ref{part:paraminit} summarizes the results obtained in \cite{param2}. Chapter \ref{part:initialu} summarizes the results obtained in \cite{param1}. Finally, Chapter \ref{part:strich}, summarizes the results obtained in \cite{bil2}.
\end{remark}

\begin{remark}
The structure of this overview is such that each part motivates the next one. In particular, Chapter \ref{part:yangmills} relies on the control of the parametrix \eqref{parametrix.intr} (see Theorem \ref{prop:estparam} in Chapter \ref{part:yangmills}), and thus motivates Chapters \ref{part:paramtime}, \ref{part:spacetimeu}, \ref{part:paraminit} and \ref{part:initialu} which precisely deal with the control of that parametrix. Next, in order to control the error term \eqref{error-param} in Chapter \ref{part:paramtime}, we rely on estimates for the optical function $\uom$, which motivates Chapter \ref{part:spacetimeu} where these estimates are proved. In turn, the space-time estimates for $\uom$ in Chapter \ref{part:spacetimeu} are obtained in particular using transport equations, and we need the corresponding control for $\uom$ on the initial slice $\Sigma_0$, which motivates Chapter \ref{part:initialu}. Finally, in order to control the parametrix at initial time \eqref{parametrixinit.intr} in Chapter \ref{part:paraminit}, we rely on estimates for the function $\uom$ on $\Sigma_0$, which motivates again Chapter \ref{part:initialu}. Finally, Chapter \ref{part:yangmills} also relies on sharp $L^4(\MM)$ Strichartz estimates localized in frequency (see Proposition \ref{prop:L4strichartz} in Chapter \ref{part:yangmills}), and thus motivates Chapter \ref{part:strich}. 
\end{remark}

\subsection{Conclusion}

Though this result does not achieve the  crucial  goal of finding a scale invariant  well-posedness  criterion in GR, it is clearly optimal in terms of all currently available  ideas and  techniques. Indeed, within our current understanding,   a better  result would require enhanced   bilinear estimates, which in turn would rely heavily on parametrices. On the other hand, parametrices are based on solutions to the eikonal equation whose control 
  requires, at least, 
$L^2$ bounds for the curvature tensor, as can be seen in many instances in our work.  Thus,  if we are   
 to ultimately   find    a    scale invariant  well-posedness  criterion, it is clear that an entirely new circle        of ideas is  needed. Such a goal is clearly of fundamental importance not just to GR, but also to any physically relevant quasilinear hyperbolic system.

\subsection{\bf Acknowledgements}
This work would be inconceivable without the extraordinary  advancements made 
on nonlinear wave equations  in the last twenty years in which so many have
 participated.
We would like to  single out the contributions of those who have affected this  work  in a more direct fashion,   either through their papers or through  relevant  discussions,  in  various stages   of its long gestation.  D.  Christodoulou's     seminal work \cite{Ch2}  on the    weak cosmic censorship conjecture   had a direct motivating   role on  our program, starting with  a series of papers  of the first author and M. Machedon, in which spacetime  bilinear estimates     were first introduced and used    to take advantage of   the null structure of geometric semilinear  equations such as Wave Maps and Yang-Mills.  The works of  Bahouri- Chemin  \cite {Ba-Ch1}-\cite{Ba-Ch2}   and     D.Tataru \cite{Ta2}   were       the first  to go below  the classical  Sobolev exponent  $s=5/2$,   for any quasilinear system  in  higher dimensions. This   was, at the time,  a major psychological  and technical   breakthrough which opened the way for future developments.  Another major breakthrough  of the period,   with direct   influence on   our  approach to bilinear estimates in curved spacetimes,  is   D. Tataru's work \cite{Ta3}  on  critical well posedness      for Wave Maps, in which null frame spaces were first introduced.    His  joint work  with H. Smith \cite{Sm-Ta}  which, together with \cite{Kl-R2}   is the first to reach   optimal     well-posedness without  bilinear estimates, has also   influenced our approach on parametrices and Strichartz estimates.  
 The authors  would also  like  to acknowledge fruitful conversations with  L. Anderson,  and   J. Sterbenz.

\chapter{Einstein vacuum equations as a Yang-Mills gauge theory}\lab{part:yangmills}

%%%%%%%%%%%%%%%%%%%%%%%%%%%%%%%%%%%%%%%

Recall Steps A, B, C and D introduced in section \ref{sec:strategyproof}. In this chapter, we perform Step A and Step B, i.e. we recast the Einstein equations as a quasilinear Yang-Mills type system and we prove bilinear estimates. This allows us to reduce the proof of Theorem \ref{th:mainter} to Step C and Step D. Here, we only outline the main ideas, and we refer to \cite{KRS} for the details.

\section{Yang-Mills formalism}
\subsection{Cartan formalism} \label{sect:Cartan} Consider an Einstein vacuum space-time $(\MM, \g)$.
  We denote the  covariant differentiation by $\D$. 
   Let $e_{\a}$ be an orthonormal frame
on $\MM$, i.e. 
\beaa
\g(e_\a, e_\b)=\m_{\a\b}=\mbox{diag}(-1,1,\ldots,1).
\eeaa
Consistent with the Cartan formalism we  define the connection 1 form, 
\bea
\label{eq:YM1abstract}
(\A)_{\a\b}(X)=\g(\D_{X}e_\b,e_\a)
\eea
where $X$ is an arbitrary vectorfield in  $T(\MM)$.
Observe that,
\beaa
(\A)_{\a\b}(X)=-(\A)_{\b\a}(X)
\eeaa
i.e.  the $1$ -form $\A_\mu dx^\mu$  takes values  in 
 the Lie  algebra of $SO(1,3)$. We separate the internal indices
 $\a,\b$ from the  external indices $\mu$  according to the following 
 notation.
\bea
\label{eq:YM1}
(\A_\mu)_{\a\b}:=  (\A)_{\a\b}(\pr_\mu)      =       \g(\D_{\mu}e_\b ,e_\a)
\eea
% Note change of definition

The Riemann curvature tensor is defined by
$
\R(X, Y,  U, V)=\g\big(X, \big[ \D_U \D_V-\D_V \D_U -\D_{[U,V]} Y\big]\big)
$
with $X, Y, U, V$ arbitrary vectorfields in $T(\MM)$. 
Thus, taking $U=\pr_\mu, V=\pr_\nu$,   coordinate vector-fields,
\bea\lab{linkRA}
\R(e_\a, e_\b , \pr_\mu, \pr_\nu)&=&\pr_\mu (\A_\nu)_{\a\b}-
\pr_\nu (\A_\mu)_{\a\b}+(\A_\nu)_\a\,^\la  (\A_\mu)_{\la\b}-
(\A_\mu)_\a\,^\la  (\A_\nu)_{\la\b}.
\eea
Defining the Lie bracket,
\bea
\label{eq:YM2}
([\A_\mu, \A_\nu])_{\a\b}=(\A_\mu)_{\a}{\,^\ga} \,( \A_\nu)_{\ga\b}-
(\A_\nu)_{\a}{\,^\ga} \,( \A_\mu)_{\ga\b}
\eea
we obtain:
$$
\R_{\a\b\mu\nu}=\pr_\mu (\A_\nu)_{\a\b}-\pr_\nu (\A_\mu)_{\a\b}-([\A_\mu, \A_\nu])_{\a\b},
$$
or, since $\pr_\mu (\A_\nu)-\pr_\nu
(\A_\mu)=\D_\mu \A_\nu-\D_\nu \A_\mu$
\bea
\label{eq:YM3}
( \F_{\mu\nu})_{\a\b}=\R_{\a\b\mu\nu}=\big(\D_\mu \A_\nu-\D_\nu \A_\mu
-[\A_\mu,
\A_\nu]\big)_{\a\b}.
\eea
where   interpret   $\F$  is  the curvature of the connection  $\A$.

The  usual covariant derivative of the Riemann curvature tensor can be expressed as follows:
\bea
\label{eq:YM4}
\D_\si \R_{\a\b\mu\nu}=\Da_\si F_{\mu\nu}:= \D_\si \F_{\mu\nu}+[\A_\si, \F_{\mu\nu}]
\eea
where we denote by $\Da$ the  covariant derivative
on the corresponding vector bundle. More precisely
if $\U=\U_{\mu_1\mu_2\ldots \mu_k}$ is any $k$-tensor
 on $\MM$ with values on the Lie algebra of $SO(3,1)$,
\bea
\label{eq:YM4'}
\Da_\si  \U=
 \D_\si  \U+ [\A_\si, \U]. 
\eea
\begin{remark}
Recall that in $(\A_\mu)_{\a\b}$, $\a, \b$ are called the internal indices, while $\mu$ are called the  external indices. Now, the internal indices  are mostly irrelevant  in our work. Thus, from now on, we will drop them, except for rare instances where we will need to distinguish between internal indices of the type $ij$ and internal indices of the type $0i$.
\end{remark}

The Bianchi identities for $\R_{\a\b\mu\nu}$
 take the form 
\bea
\label{eq: YM5}
\Da_\si \F_{\mu\nu}+\Da_\mu F_{\nu\si}+\Da_\nu \F_{\si\mu}=0.
\eea
As it is well known the Einstein vacuum equations $\R_{\a\b}=0$
imply 
$ \D^\mu \R_{\a\b\mu\nu}=0.$
Thus, in view of   equation    \eqref{eq:YM4},
\bea
\label{eq:YM6}
0=\Da^\mu\, \F_{\mu\nu}= \D^\mu \F_{\mu\nu} + [ \A^\mu, \F_{\mu\nu}]
\eea
or, in view of \eqref{eq:YM3} and  the vanishing of the Ricci curvature of $\g$,.
\bea
\label{eq:YM7}
\square \A_\nu-\D_\nu(\D^\mu \A_\mu) =\J_\nu
\eea
where
\bea
\label{eq:YM8}
\J_\nu=\D^\mu([\A_\mu, \A_\nu])-[\A_\mu, \F_{\mu\nu}].
\eea
Using again the vanishing of the Ricci curvature it is easy to check,
\bea
\label{eq:YM9}
\D^\nu \J_\nu=0.
\eea

Finally we recall the general formula of transition between two different orthonormal frames  $e_\a$ and $ \widetilde{e}_\a$ on $\MM$, related by, 
\beaa
 \widetilde{e}_\a=\O_{\a}^\ga e_\ga
\eeaa
where
$
\m_{\a\b}=\O_\a^\ga \O_\b^\de \, \m_{\ga\de}
$,
i.e. $\O$ is a smooth map from $\MM$ to the Lorentz group $O(3,1)$.
In other words,  raising and lowering indices with respect to $\m$,
\bea
\O_{\a\la}\O^{\b\la}=\de_{\a}^\b
\eea
Now, $(\widetilde{\A}_\mu)_{\a\b}=\g(\D_\mu  \widetilde{e}_\b, \widetilde{e}_\a)$. Therefore,
\bea
\label{gauge.tr}
(\widetilde{\A}_\mu)_{\a\b}=\O_\a^\ga \O_\b^\de (\A_\mu)_{\ga\de}+\pr_\mu (\O_\a^\ga) \, \O_{\b}^\de\,  \m_{\ga\de}
\eea
\subsection{Compatible frames}
\label{sect:compatible}

Recall that our space-time is assumed to be foliated by  the level surfaces $\Si_t$ of a 
time function $t$,  which are maximal, i.e.  denoting by $k$
the second fundamental form of  $\Si_t$ we have,
\bea\lab{eq:YM9bis}
\tr_g k=0
\eea
where $g$ is the induced metric  on $\Si_t$. Let us choose $e_{(0)}=T$, the future unit normal to the $\Si_t$  foliation, and $e_{(i)}$, $i=1,2,3$  an orthonormal frame tangent  to $\Si_t$.  We call this   a frame compatible with our $\Si_t$ foliation.  We consider the connection coefficients \eqref{eq:YM1}
with respect to this frame. Thus, in particular, denoting   by  $A_0$,   respectively $A_i$, the temporal and spatial components
of $\A_\mu$
\bea
(A_i)_{0j}&=&(A_j)_{0i}=-k_{ij} ,\qquad i,j=1,2,3\label{eq:YM10}  \\
(A_0)_{0i}   &=&-n^{-1}\nabla_in  \qquad i=1,2,3\ \label{eq:YM11} 
\eea
where $n$ denotes the lapse of the $t$-foliation, i.e. $n^{-2}=-\g(\D t, \D
t)$. With this notation  we note  that,
\beaa
\nabla_lk_{ij}=\nabla_l(k_i)_j+k_{in}(A_l)_j\,^n=\nabla^l(A_i)_{0j}+k_{in}(A_l)_j\,^n
\eeaa
 where, as  before, the notation  $\nabla_l(k_i)_j$ or    $ \nabla^l(A_i)_{0j}   $,  is meant to suggest that
 the covariant differentiation affects only the external index $i$. Recalling from \eqref{constk} that $k$ verifies the constraint equations,
$$\nabla^ik_{ij}=0,$$
 we derive,
 \bea\lab{coulomblike1}
 \nab^i( A_i)_{0j}=k_i\,^m (A_i)_{mj}.
 \eea
 
  Besides the choice of $e_0$ we are still free to make  a choice
  for the  spatial elements of the frame $e_1, e_2, e_3$.  In other words we consider  frame transformations which keep $e_0$ fixed, i.e transformations
  of the type, 
  \beaa
  \tilde{e}_i=O_i^j e_j
  \eeaa
  with  $O$ in the orthogonal group $O(3)$. 
  We now have,  according to \eqref{gauge.tr},
  \beaa
  (\tilde{A}_m)_{ij}=O_i^k O_j^l  (A_m)_{kl}+\pr_m(O_i^k) O_j^l \de_{kl}
  \eeaa
  or, schematically,
  \bea
  \tilde{A}_m &=& O A_m O^{-1}+(\pr_m O) O^{-1}
  \eea
  formula in which we understand that   only the   spatial internal indices 
  are involved. We shall use this freedom later to exhibit a frame $e_1, e_2, e_3$ 
  such that the corresponding connection $A$ satisfies the coulomb gauge condition 
  $\nabla^l(A_l)_{ij}=0$ (see Lemma \ref{lemma:uhlenbeck}).
 
 \subsection{Notations} We use greek indices to denote general indices on $\MM$ which do not refer to the particular frame $(e_0, e_1, e_2, e_3)$. The letters $a, b, c, d$ will be used to denote general indices on $\Si_t$ which do not refer to the particular frame $(e_1, e_2, e_3)$. Finally, the letters $i, j, l, m, n$ will only denote indices relative to the frame $(e_1, e_2, e_3)$. Also, recall that $\D$ denotes the covariant derivative on $\MM$, while $\nabla$ denotes the induced covariant derivative on $\Si_t$. Furthermore, $\prb$ will always refer to the derivative of a scalar quantity relative to one component of the frame $(e_0, e_1, e_2, e_3)$, while $\pr$ will always refer to the derivative of a scalar quantity relative to one component of the the frame $(e_1, e_2, e_3)$, so that $\prb=(\pr_0, \pr)$. For example, 
$\pr A$ may be any term of the form $\pr_i(A_j)$, $\pr_0(A)$ may be any term of the form $\pr_0(A_j)$, 
$\pr(A_0)$ may be any term of the form $\pr_j(A_0)$, and $\prb\A=(\prb A, \prb(A_0))=\big(\pr_0(A_0), \pr(A_0), \pr_0(A), \pr( A) \big)$. Note that we use  brackets such as  $(A_j )$   to emphasize that  we are dealing with  $su(3,1)$ objects.
Often,  however,   we will simply   drop them. 

We introduce the curl operator $\curl$ defined for any $su(3,1)$-valued triplet $(\om_1, \om_2, \om_3)$ of functions on $\Si_t$  as  follows:
\bea\lab{def:curl}
(\curl \om)_i=\in_i\,^{jl}\partial_j(\om_l),
\eea
where $\in_{ijl}$ is fully antisymmetric and such that $\in_{123}=1$. We also introduce the divergence operator $\div$ defined for any $su(3,1)$-valued tensor $A$ on $\Si_t$  as  follows:
\bea\lab{def:div}
\div A=\nabla^l(A_l)=\pr^l(A_l)+A^2.
\eea
\begin{remark}
The term $A^2$ in \eqref{def:div} corresponds to a quadratic expression in components of $A$, 
where the particular indices do not matter. In the rest of this part, we will adopt this schematic notation for lower order terms (e.g. terms of the type $A^2$ and $A^3$) where the particular indices do not matter. 
\end{remark}

Finally, $\square A_0$ and $\square A_i$ will always be understood as $\square(A_0)$ and $\square(A_i)$, while $(\square A)_\a$ refers to the tensorial wave equation. Also, $\Delta A_0$ will always refer to $\Delta(A_0)$.
\begin{remark}
Since $\pr_0$ and $\pr_j$ are not coordinate derivatives, note that the commutators $[\pr_j, \pr_0]$ and $[\pr_j, \pr_l]$ do not vanish. In fact we have,  schematically,
\bea\lab{notcoordinate}
[\pr_i, \pr_j]\phi=A\prb\phi\textrm{ and }[\pr_j, \pr_0]\phi=\A\prb\phi,
\eea
for any scalar function $\phi$ on $\MM$.
\end{remark}

  \subsection{ Main equations  for $(A_0, A)$} 
 Using the conventions above   one can prove  the  following proposition.
\begin{proposition}
\label{prop:decomp}
Consider  an orthonormal frame $e_\a$  compatible with  a    maximal  $\Si_t$ foliation of the space-time $\MM$ 
with connection coefficients    $\A_\mu$   defined by \eqref{eq:YM1}, their decomposition  $\A=(A_0, A)$ relative to the  same frame  $e_\a$,
and Coulomb- like condition on the frame,
\beaa
\div A=A^2.
\eeaa
 In such a frame the  Einstein-vacuum equations take the form,
\bea
\De A_0 &=& \A\prb A+\A\pr (A_0)+\A^3,
\label{eq:YM21}
\eea
\bea
\label{eq:YM22}
\square A_i+\pr_i (\pr_0 A_0)   &=& A^j\pr_j A_i+A^j\pr_i A_j+A_0\prb\A+A\prb(A_0) +\A^3.
\eea
\end{proposition}

 \begin{remark}
 It is extremely important to our strategy that we have reduced the covariant wave equation \eqref{eq:YM7} to the system of scalar equations \eqref{eq:YM21} \eqref{eq:YM22} (see remark \ref{rem:scalarize}).    
 \end{remark}

\section{Strategy of the proof of the main Theorem}

In this section, we discuss the strategy of the proof of the main theorem after reduction to small initial data, i.e. Theorem  \ref{th:mainter}.

\subsection{The Uhlenbeck type lemma}

In order to exhibit a frame $e_1, e_2, e_3$ such that together with $e_0=T$ we obtain a connection $\A$ satisfying our Coulomb type gauge on the slice $\Si_t$, we will need the following result in the spirit of the Uhlenbeck lemma \cite{Uhl}.
\begin{lemma}\lab{lemma:uhlenbeck}
Let $(M,g)$ a 3 dimensional Riemannian asymptotically flat manifold. Let $R$ denote its curvature tensor and $r_{vol}(M,1)$ its volume radius on scales $\leq 1$. Let $\tilde{A}$ a connection on $M$ corresponding to an orthonormal frame. Assume the following bounds:

\beaa
\|\tilde{A}\|_{L^2(M)}+ \|\nab \tilde{A}\|_{L^2(M) }+ \norm{R}_{L^2(M)}\leq \delta\qquad \textrm{ and}\,\, r_{vol}(M,1)\geq \frac{1}{4},
\eeaa
where $\delta>0$ is a small enough constant. Assume also that $\tilde{A}$ and $\nabla\tilde{A}$ belong to $L^2(M)$. Then, there is another connection $A$ on $M$ satisfying he Coulomb like gauge condition, and such that
\beaa
\|\tilde{A}\|_{L^2(M)}+ \|\nab \tilde{A}\|_{L^2(M)} \leq \de
\eeaa
 Furthermore, if $\nabla^2\tilde{A}$ belongs to $L^2(M)$, then $\nabla^2A$ belongs to $L^2(M)$.
\end{lemma}

The proof of Lemma \ref{lemma:uhlenbeck} is a straightforward adaptation, in a simpler situation,  of   \cite{Uhl}. 

\subsection{Classical local existence}

We  rely on  the following standard  well-posedness result for the Cauchy problem for the Einstein equations \eqref{EVE} in the maximal foliation.

\begin{theorem}[Well-posedness for the Einstein equation in the maximal foliation]\lab{th:classicalwp}
Let $(\Si_0 ,g , k)$ be asymptotically flat and satisfying the constraint equations \eqref{const}, with $\textrm{Ric}$, $\nabla \textrm{Ric}$, $k$, $\nabla k$ and  $\nabla^2k$ in $L^2(\Si_0)$, and $r_{vol}(\Si_0,1)>0$. Then, there exists a unique asymptotically flat solution $(\MM, {\bf g})$ to the Einstein vacuum equations \eqref{EVE} corresponding to this initial data set, together with a maximal foliation by space-like hypersurfaces $\Si_t$ defined as level hypersurfaces of a time function $t$. Furthermore, there exists a time 
$$T_*=T_*(\norm{\nabla^{(l)}\textrm{Ric}}_{L^2(\Si_0)}, 0\leq l\leq 1, \norm{\nabla^{(j)}k}_{L^2(\Si_0)}, 0\leq j\leq 2, r_{vol}(\Si_0,1))>0$$
such that the maximal foliation exists for on $0\leq t\leq T_*$ with a corresponding control in $L^\infty_{[0,T_*]}L^2(\Si_t)$ for $\textrm{Ric}$, $\nabla\textrm{Ric}$, $k$, $\nabla k$ and $\nabla^2k$. 
\end{theorem}

Theorem \ref{th:classicalwp} requires two more derivatives both for $R$ and $k$ with respect to the main Theorem \ref{th:main}. Its proof is standard and relies solely on energy estimates (as opposed to Strichartz or  bilinear estimates). We refer the reader to \cite{ChKl} chapter 10 for a related statement.

\begin{remark} In the  proof  of  our main theorem  
 the result above will be used only    as an extension   tool (see steps 1 and 3 below), only      for very  tiny values of 
  the time interval.
\end{remark}

\subsection{ Weakly regular  null  hypersurfaces}  We shall be working with  null hyper surfaces in $\MM$ verifying 
 a set of     reasonable     assumptions, described below. These assumptions will be easily verified  by the level hyper surfaces  $\HH_u$   solutions $u$ of the eikonal equation $\g^{\mu\nu} \pr_\mu\pr_\nu=0$ discussed in  section \ref{sec:parametrix}. The regularity of the eikonal  equation is  studied in detail  in \cite{param3} (see also Chapter \ref{part:spacetimeu}).

\begin{definition} 
Let $\HH$ be a  null hypersurface  with future  null    normal  $L$ verifying  ${\bf g}(L,T)=-1$. Let also $N=L-T$.  We denote by $\nabb$ the induced connection along the $2$-surfaces $\HH\cap\Si_t$. We say that $\HH$ is 
 weakly  regular provided that, 
\bea\lab{assumptionH1}
\norm{\D L}_{L^3(\HH)}+\norm{\D N}_{L^3(\HH)}\les 1,
\eea
and the following Sobolev embedding holds for any scalar function $f$ on $\HH$:
\bea\lab{assumptionH2}
\norm{f}_{L^6(\HH)}\les \norm{\nabb f}_{L^2(\HH)}+\norm{L(f)}_{L^2(\HH)}+\norm{f}_{L^2(\HH)}.
\eea
\end{definition}

\subsection{Main bootstrap assumptions}
Let $M\geq 1$ a large enough constant to be chosen later in terms only of universal constants. By choosing $\ep>0$ sufficiently small, we can also ensure $M\ep$ is small enough.     From now on, we assume the following bootstrap assumptions  hold true on a fixed  interval  $[0, T^*]$,  for some $0<T^*\le 1$.  Note that  $\HH$ denotes an 
 arbitrary  weakly regular null hypersurface  with  future  normal  $L$, normalized by the condition ${\bf g}(L,T)=-1$.

\begin{itemize}
\item   \textit{Bootstrap curvature assumptions}
\bea\label{bootR}
\norm{\R}_{\lsit{2}}\leq M\ep.
\eea
Also,
\bea\lab{bootcurvatureflux}
\norm{\R\c L}_{L^2(\HH)}\leq M\ep,
\eea
where $\R\c L$ denotes any component of $\R$ such 
that at least one index is contracted with $L$. 
\item \textit{Bootstrap  assumptions for the connection}
We  also assume  that there exist  $\A=(A_0, A)$   verifying   our Coulomb type condition on  $ [0,T^*]$ ,   
such that,
\bea\label{bootA}
\norm{A}_{\lsit{2}}+\norm{\prb A}_{\lsit{2}}+\norm{A}_{\lsitt{2}{7}}\leq M\ep,
\eea
and:
\bea\label{bootA0}
\nn\norm{A_0}_{\lsit{2}}+\norm{\prb A_0}_{\lsit{2}}+\norm{A_0}_{\lsitt{2}{\infty}}+\norm{\prb A_0}_{\lsit{3}}&&\\
+\norm{\pr\prb A_0}_{\lsit{\frac{3}{2}}}&\leq& M\ep.
\eea
\end{itemize}
\begin{remark}
Together with the estimates in \cite{param3} (see section 4.4 in that paper, and also section \ref{ch4:sec:estn}),  the bootstrap  assumption \eqref{bootR}
 yields:
\bea\label{bootk}
\norm{k}_{\lsit{2}}+\norm{\nabla k}_{\lsit{2}}\lesssim M\ep.
\eea
Furthermore, the bootstrap assumption \eqref{bootcurvatureflux} together with the estimates in \cite{param3} (see section 4.2 in that paper, and also section \ref{ch4:sec:lowerboundvolrad}) yields:
\bea\lab{bootvolumeradius}
\inf_tr_{vol}(\Si_t,1)\geq \frac{1}{4}.
\eea
\end{remark}

In addition we make the following bilinear estimates assumptions for $\A$ and $\R$.:
\begin{itemize}
\item \textit{Bilinear assumptions I}.
Assume,
\bea\lab{bil1}
\norm{A^j\pr_jA}_{L^2(\MM)}\lesssim M^3\ep^2.
\eea
Also, let $B=(-\Delta)^{-1}\curl  A $ (see   \eqref{defineB}  and  the  accompanying explanations). Then, we have:
\bea\lab{bil5}
\norm{A^j\pr_j(\prb B)}_{L^2(\MM)}\lesssim M^3\ep^2,
\eea
and:
\bea\lab{bil6}
\norm{\R_{\c\,\c\,j\,0}\pr^j B}_{L^2(\MM)}\lesssim M^3\ep^2.
\eea
Finally,  for any weakly regular  null hypersurface $\HH$ and  any smooth  scalar function $\phi$ on $\MM$,
\bea\lab{bil7}
\norm{k_{j\,\c}\pr^j\phi}_{L^2(\MM)}\les M^2\ep\sup_{\HH}\norm{\nabb\phi}_{L^2(\HH)},
\eea
and
\bea\lab{bil8}
\norm{A^j\pr_j\phi}_{L^2(\MM)}\les M^2\ep\sup_{\HH}\norm{\nabb\phi}_{L^2(\HH)},
\eea
where the supremum is taken over all null hypersurfaces $\HH$.
\item  \textit{Bilinear assumptions II}. We assume,
\bea\lab{bil3}
\norm{(-\Delta)^{-\frac12}(Q_{ij}(A,A))}_{L^2(\MM)}\lesssim M^3\ep^2,
\eea
where the bilinear form $Q_{ij}$ is given by $Q_{ij}(\phi,\psi)=\pr_i\phi\pr_j\psi-\pr_j\phi\pr_i\psi$. Furthermore, we also have:
\bea\lab{bil4}
\norm{(-\Delta)^{-\frac12}(\pr(A^l)\pr_l  A)}_{L^2(\MM)}\lesssim M^3\ep^2.
\eea
\item \textit{Non-sharp Strichartz assumption}
\bea\lab{bootstrich}
\norm{A}_{\lsitt{2}{7}}\les M^2 \ep.
\eea
and, for $B=(-\De)^{-1}\curl A$, (see   \eqref{defineB}  and  the  accompanying explanations).
\bea
\norm{\pr B}_{\lsitt{2}{7}}\lesssim M^2\ep. \lab{bootstrichB}
\eea
\end{itemize}

\begin{remark}
Note that the Strichartz estimate for  $\norm{A}_{\lsitt{2}{7}}$ is far from being sharp. Nevertheless, this estimate will be sufficient for the proof as it will only be used to deal with lower order terms.
\end{remark}

Finally we also need a trilinear bootstrap  assumption.  For  this we need to introduce the Bell Robinson tensor,
 \begin{equation}\label{def:bellrobison}
 Q_{\a\b\ga\de}=\R_\a\,^\la\,\ga\,^\si \R_{\b\,\la\,\de\,\si}+\dual \R_\a\,^\la\,\ga\,^\si\dual \R_{\b\,\la\,\de\,\si}
 \ee
\begin{itemize}
\item
\textit{Trilinear bootstrap assumption.} 
  We assume the following,
\bea
\left|\int_{\MM}Q_{ij\ga\de}k^{ij}e_0^\ga e_0^\de\right| &\les& M^4 \ep^3.\label{trilinearboot}
\eea
\end{itemize}

Let us conclude this section by remarking  that the  bootstrap assumptions are verified  for a sufficiently small  final value $T^*$.
\begin{proposition}\lab{prop:continuity}
The above bootstrap assumptions are verified on $0\leq t\leq T^*$ for a sufficiently small $T^*>0$.
\end{proposition} 
 
The   only challenge  in the proof of Proposition \ref{prop:continuity} is to show the existence of   the desired  connection $\A$ using in particular the Uhlenbeck type Lemma \ref{lemma:uhlenbeck}. All  other  estimates    follow  trivially    from   our initial  bounds and    the local existence theorem above, for sufficiently small  $T^*$. We refer to Proposition 4.6 in \cite{KRS}. 

 \subsection{Proof of the bounded $L^2$ curvature conjecture}\lab{sec:proofconj}

In the following two propositions, we state the improvement of our bootstrap assumptions. 
\begin{proposition}\label{prop:improve1}
Let us assume that all bootstrap assumptions of the previous section hold for $0\leq t\leq T^*$. If $\ep>0$ is sufficiently small, then the following improved estimates hold true on $0\leq t\leq T^*$:
\bea\label{bootRimp}
\norm{\R}_{\lsit{2}}\les \ep+M^2\ep^{\frac{3}{2}}+M^3\ep^2,
\eea
\bea\lab{bootcurvaturefluximp}
\norm{\R\c L}_{L^2(\HH)}\les \ep+M^2\ep^{\frac{3}{2}}+M^3\ep^2,
\eea
\bea\label{bootAimp}
\norm{A}_{\lsit{2}}+\norm{\prb A_i}_{\lsit{2}}\les \ep+M^2\ep^{\frac{3}{2}}+M^3\ep^2,
\eea
\bea\label{bootA0imp}
\nn\norm{A_0}_{\lsit{2}}+\norm{\prb A_0}_{\lsit{2}}+\norm{A_0}_{\lsitt{2}{\infty}}&&\\
+\norm{\prb A_0}_{\lsit{3}}+\norm{\pr\prb A_0}_{\lsit{\frac{3}{2}}}&\les& \ep+M^2\ep^{\frac{3}{2}}+M^3\ep^2,
\eea
\end{proposition}

\begin{proposition}\label{prop:improve2}
Let us assume that all bootstrap assumptions of the previous section hold for $0\leq t\leq T^*$. If $\ep>0$ is sufficiently small, then the following improved estimates hold true on $0\leq t\leq T^*$:
\bea\lab{bil1imp}
\norm{A^j\pr_jA}_{L^2(\MM)}\lesssim M^2\ep^2,
\eea
\bea\lab{bil5imp}
\norm{A^j\pr_j(\prb B )}_{L^2(\MM)}\lesssim M^2\ep^2,
\eea
and
\bea\lab{bil6imp}
\norm{\R_{\c\,\c\,j\,0}\pr^j  B }_{L^2(\MM)}\lesssim M^2\ep^2.
\eea
Also, for any scalar function $\phi$ on $\MM$, we have: 
\bea\lab{bil7imp}
\norm{k_{j\,\c}\pr^j\phi}_{L^2(\MM)}\les M\ep\sup_{\HH}\norm{\nabb\phi}_{L^2(\HH)},
\eea
and
\bea\lab{bil8imp}
\norm{A^j\pr_j\phi}_{L^2(\MM)}\les M\ep\sup_{\HH}\norm{\nabb\phi}_{L^2(\HH)},
\eea
where the supremum is taken over all null hypersurfaces $\HH$. Finally, we have:
\bea\lab{bil3imp}
\norm{(-\Delta)^{-\frac12}(Q_{ij}(A,A))}_{L^2(\MM)}\lesssim M^2\ep^2,
\eea
\bea\lab{bil4imp}
\norm{(-\Delta)^{-\frac12}(\pr    A^l   \pr_l A )}_{L^2(\MM)}&\lesssim &M^2\ep^2.
\eea
\bea\lab{bootstrichimp}
\norm{A}_{\lsitt{2}{7}}&\les& M \ep,\\
\norm{\pr B}_{\lsitt{2}{7}}&\lesssim& M\ep. \lab{bootstrichBimp}
\eea
and
\bea
\left|\int_{\MM}Q_{ij\ga\de}k^{ij}e_0^\ga e_0^\de\right| &\les& M^3 \ep^3.\label{trilinearbootimp}
\eea
\end{proposition}

The proof of Proposition \ref{prop:improve1} is postponed to section \ref{sec:improve1}, while the proof of Proposition \ref{prop:improve2} is postponed to section \ref{sec:improve2}. We also need a proposition on the propagation of higher regularity.  
\begin{proposition}\lab{prop:propagreg}
Let us assume that the estimates corresponding to all bootstrap assumptions of the previous section hold for $0\leq t\leq T^*$ with a universal constant $M$. Then for any $t\in [0,T^*)$ and for $\ep>0$ sufficiently small, the following propagation of higher regularity holds:
$$
\|\D\R\|_{L^\infty_tL^2(\Sigma_t)}\le 2\left(\norm{\textrm{Ric}}_{L^2(\Sigma_0)}+\norm{\nabla\textrm{Ric}}_{L^2(\Sigma_0)}+\norm{k}_{L^2(\Sigma_0)}+\norm{\nabla k}_{L^2(\Sigma_0)}+\norm{\nabla^2k}_{L^2(\Sigma_0)}\right).$$
\end{proposition}

The proof of Proposition \ref{prop:propagreg} follows along the same lines as the proof Proposition \ref{prop:improve1} and Proposition \ref{prop:improve2}, and we refer to \cite{KRS} for its proof. Next, let us show how Propositions  \ref{prop:continuity}, \ref{prop:improve1}, \ref{prop:improve2} and \ref{prop:propagreg} imply our main theorem \ref{th:mainter}. We proceed, by the standard  bootstrap  method , along the following steps:
\begin{itemize}
\item[\textit{Step 1}.] We show that all bootstrap assumptions are verified  for a sufficiently small  final value $T^*$. 

\item[\textit{Step2}.]  Assuming that  all bootstrap assumptions hold  for   fixed values of $0< T^*\le 1$ and $M$  sufficiently large    we   show  that,  for $\ep>0$ sufficiently small, we may improve on the constant $M$ in our bootstrap assumptions. 

\item[\textit{Step 3}.] Using the  estimates derived in step 2   we can extend  the time of existence $T^*$ to $T^*+\de$  such that  all the bootstrap assumptions remain true.
\end{itemize}

Now, \textit{Step 1} follows from Proposition \ref{prop:continuity}. \textit{Step 2} follows from Proposition \ref{prop:improve1} and Proposition \ref{prop:improve2}. In view of \textit{Step 2}, the estimates corresponding to all bootstrap assumptions of the previous section hold for $0\leq t\leq T^*$ with a universal constant $M$. Thus the conclusion of Proposition \ref{prop:propagreg} holds, and arguing as in the proof of Proposition \ref{prop:continuity}, we obtain \textit{Step 3}. Thus, the bootstrap assumptions hold on $0\leq t\leq 1$ for a universal constant $M$. In particular, this yields together with \eqref{bootk}:
\bea\lab{saionara2}
\norm{\R}_{\lsit{2}}\les \ep\textrm{ and }\norm{k}_{\lsit{2}}\les \ep\textrm{ for all }0\leq t\leq 1.
\eea
In view of  \eqref{bootvolumeradius}, we also obtain the following control on the volume radius:
\bea\lab{bootvolumeradius:bis}
\inf_{0\leq t\leq 1}r_{vol}(\Si_t,1)\geq \frac{1}{4}.
\eea
Furthermore, Proposition \ref{prop:propagreg} yields the following propagation of higher regularity
\bea\lab{gggggggggggggg}
\sum_{|\a|\le m}\|\D^{(\a)} \R\|_{L^\infty_{[0,1]}L^2(\Sigma_t)}  \leq C_m
  \sum_{|i|\le m}\\bigg[\|\nab^{(i)} \textrm{Ric}\|_{L^2(\Sigma_0)} + \|\nab^{(i)} \nab k\|_{L^2(\Sigma_0)}\bigg]
\eea
where $C_m$ only depends on $m$.

\begin{remark}
Note that Proposition \ref{prop:propagreg} only yields the case $m=1$ in \eqref{gggggggggggggg}. The fact that \eqref{gggggggggggggg} also holds for higher derivatives $m\geq 2$ follows from the standard propagation of regularity for the classical local existence result of Theorem \ref{th:classicalwp} and the bound \eqref{gggggggggggggg} with $m=1$ coming from Proposition \ref{prop:propagreg}. 
\end{remark}

Finally, \eqref{saionara2}, the control on the volume radius \eqref{bootvolumeradius:bis} and the propagation of higher regularity \eqref {gggggggggggggg} yield the conclusion of Theorem \ref{th:mainter}. Together with the reduction to small initial data performed in section \ref{sec:reductionsmall}, this concludes the proof of the main Theorem \ref{th:main}.\\

  The rest of the  chapter  deals with the proofs of propositions \ref{prop:improve1} and  \ref{prop:improve2}. 
  The core of the proofs  is to control $A$,  the spatial part of the  connection $\A$. 
  As explained     in  the introduction  we need to project  our equation for   the spatial components 
   $A$  onto divergence free vectorfields. This is needed for two reasons, to eliminate the term 
   $\pr_i (\pr_0 A_0)$ on the left  hand side of \eqref{eq:YM22} and to   obtain, on the right hand side,
   terms which exhibit the crucial null structure we need to implement our proof. Rather than work with 
   the projection $\PP$, which is too complicated, we introduce instead  the new  variable,
   \bea
   B=(-\Delta)^{-1}\curl(A) \label{defineB}
   \eea
    for which we derive a suitable    wave equation.
   Since we have (see Lemma \ref{recoverA})
$A=\curl(B)+l.o.t$
 it suffices   to obtain estimates for $B$ which   lead us  to an  improvement of the bootstrap assumption \eqref{bootA} on $A$.  In section \ref{sec:bobo5},   we derive space-time  estimates  for $\square B$ and its derivatives.     Proposition  \ref{prop:improve1}, which does not require  a parametrix representation, is  proved in \ref{sec:improve1}.   Proposition  \ref{prop:improve2} is proved 
 in sections \ref{sec:improve2} and \ref{sec:improve2B} based on the  representation formula  of theorem \ref{lemma:parametrixconstruction} derived in section \ref{sec:parametrix}. 

\section{Simple  consequences of the bootstrap assumptions}

\subsection{Sobolev embeddings and elliptic estimates on $\Si_t$.}

 The bootstrap assumption \eqref{bootR}  on $R$ and the estimate for $k$ \eqref{bootk} together with the estimates in \cite{param3} (see section 4.4 in that paper) yield the following lapse estimates:
\bea\label{bootn}
\norm{n-1}_{L^\infty(\MM)}+\norm{\nabla n}_{L^\infty(\MM)}+\norm{\nabla^2n}_{\lsit{2}}+\norm{\nabla^2n}_{\lsit{3}}&&\\
\nn+\norm{\nabla(\pr_0n)}_{\lsit{3}}+\norm{\nabla^3n}_{\lsit{\frac{3}{2}}}+\norm{\nabla^2(\pr_0(n))}_{\lsit{\frac{3}{2}}}&\lesssim& M\ep,
\eea
where $\nabla$ denotes the induced covariant derivative on $\Sigma_t$. 

\begin{remark}\lab{rem:tempura}
Recall from \eqref{eq:YM11} that
$(A_0)_{0i}=-n^{-1}\nabla_in.$
Thus, the estimates \eqref{bootn} for $n$ could in principle be deduced from the bootstrap assumptions \eqref{bootA0} for $A_0$. However, notice that $\nabla n\in L^\infty(\MM)$ in view of \eqref{bootn}, while $A_0$ is only in $\lsitt{2}{\infty}$ according to \eqref{bootA0}. This improvement for the components $(A_0)_{0i}$ of $A_0$ will turn out to be crucial (see remark \ref{rem:tempura1}). Its proof is given in section 4.4 of \cite{param3} (see also the discussion in section \ref{ch4:sec:estn}).
\end{remark}

Next, we record   the following Sobolev embeddings and elliptic estimates on $\Si_t$  derived under the assumptions \eqref{bootcurvatureflux} and \eqref{bootR} in \cite{param3} (see sections 3.5 and 4.2 in that paper).
\begin{lemma}[Calculus inequalities on $\Si_t$]\lab{lemma:estimatesit}
Assume that the assumptions \eqref{bootcurvatureflux} and \eqref{bootR} hold, and assume that the volume radius at scales $\leq 1$ on $\Si_0$ is bounded from below by a universal constant. Then, the Sobolev embedding on $\Si_t$ holds for any tensor $F$
\bea\lab{sobineqsit}
\norm{F}_{L^{6}(\Si_t)}\lesssim \norm{\nabla F}_{L^{2}(\Si_t)},
\eea 
Also, we define the operator $(-\Delta)^{-\frac{1}{2}}$ acting on tensors on $\Si_t$ as:
$$(-\Delta)^{-\frac{1}{2}}F=\frac{1}{\Gamma\left(\frac{1}{4}\right)}\int_0^{+\infty}\tau^{-\frac{3}{4}}U(\tau)Fd\tau,$$
where $\Gamma$ is the Gamma function, and where $U(\tau)F$ is defined using the heat flow on $\Si_t$:
$$(\pr_\tau-\Delta)U(\tau)F=0,\, U(0)F=F.$$
We have the following Bochner estimates:
\bea\lab{bochnerestimates}
\norm{\nabla(-\Delta)^{-\frac{1}{2}}}_{\LL(L^2(\Si_t))}\les 1\textrm{ and }\norm{\nabla^2(-\Delta)^{-1}}_{\LL(L^2(\Si_t))}\les 1,
\eea
where $\mathcal{L}(L^2(\Si_t))$ denotes the set of bounded linear operators on $L^2(\Si_t)$. \eqref{bochnerestimates} together with the Sobolev embedding \eqref{sobineqsit} yields:
\bea\lab{sob}
\norm{(-\Delta)^{-\frac{1}{2}}F}_{L^2(\Si_t)}\lesssim \norm{F}_{L^{\frac{6}{5}}(\Si_t)}.
\eea 
\end{lemma}

\subsection{ Elliptic  estimates for $B$}
Here we    record   simple    estimates for $B$, based   the bootstrap
 assumptions \eqref{bootA} \eqref{bootA0} for $A$ and $A_0$ and standard elliptic estimates such
  as the Bochner and Sobolev  inequalities      on $\Si_t$, see  \eqref{bochnerestimates} and   \eqref{sob}.

\begin{proposition}\lab{lemma:estB}
Let $B_i=(-\Delta)^{-1}(\curl(A)_i)$. Then, we have, for each component of $B$:
\bea
\norm{\prb B}_{\lsit{2}}+\norm{\pr^2 B}_{\lsit{2}}+\norm{\pr(\pr_0 B)}_{\lsit{2}}\les M\ep.
\label{squareB2}  
\eea
\end{proposition}

\subsection{Decomposition for $A$}

Recall that $B=(-\Delta)^{-1}(\curl(A))$. We indicate below  how to recover $A$ from $B$:
\begin{lemma}\lab{recoverA}
We have the following estimate:
$$A=\curl(B)+E$$
where $E$ satisfies:
$$\norm{\pr E}_{\lsit{3}}+\norm{\pr^2E}_{\lsit{\frac{3}{2}}}+\norm{E}_{\lsitt{2}{\infty}}\les M^2\ep^2.$$
\end{lemma}

\begin{proof}
We   have, symbolically, 
\beaa
 A=(-\Delta)^{-1}\curl(\curl(A)+(-\Delta)^{-1} (A\pr  A+A^3).
\eeaa
from which,
\beaa
A &=& \curl(B)-(-\Delta)^{-1}[\Delta,\curl]B+(-\Delta)^{-1} (A \pr  A+A^3)\\
\eeaa
The rest of the proof uses elliptic estimates on $\Si_t$, the bootstrap assumptions for $\A$ and $\R$, and the  bootstrap assumption \eqref{bootstrich}. We refer to \cite{KRS} for the details.
\end{proof}

\section{Estimates   for  $\square B $}\lab{sec:bobo5}
 We  outline the proof of two important  propositions concerning  estimates  for $\square \curl A$ and $\square  B$,  with   $B=\Delta^{-1}\curl(A)$. The proofs makes use of the special structure of various  bilinear expressions and thus  is based    not only on the bootstrap assumptions for $A_0, A$, $k$ and $\R$  but also  some of our  bilinear bootstrap assumptions.

 We record first   the following straightforward   commutation lemma, see  \cite{KRS}. 
\begin{lemma}\lab{lemma:commutation}
Let $\phi$ a $so(3,1)$  scalar function on $\MM$. We have, schematically,
\bea\lab{commsquare2}
\pr_j(\square\phi)-\square (\pr_j(\phi))=
2(A^\la)_j\,^{\mu}\,\, \pr_\la   \pr_\mu\phi+\pr_0(A_0)\prb\phi 
+ \A^2\prb\phi.
\eea
We also have:
\bea\lab{commsquare3}
[\square, \Delta]\phi
&=&-4k^{ab}\nabla_a\nabla_b(\pr_0\phi)+4n^{-1}\nabla_bn\nabla_b(\pr_0(\pr_0\phi))
 -2\nabla_0k^{ab}\nabla_a\nabla_b\phi\\
 &+&F^{(1)}\prb^2\phi+F^{(2)}\prb\phi,\nn\\
 F^{(1)}&=& \prb A_0+\A^2,\nn\\
 F^{(2)}&=& \pr\prb A_0+\A\prb\A+\A^3,\nn
\eea
where $\nabla_a$ and $\nabla_b$ denote induced covariant derivatives on $\Si_t$ applied to the scalars $\phi$, $\pr_0\phi$ and $\pr_0(\pr_0\phi)$.
\end{lemma}

The estimates for $\square \curl A$ and $\square  B$ are given by the following propositions.
\begin{proposition}\lab{prop:commutsquarecurl}  We have 
\bea
\sum_{i=1}^3\norm{(-\Delta)^{-\frac12}      \square(\curl  A_i)        }_{L^2(\MM)}\lesssim M^2\ep^2.\label{squareB1}
\eea
\end{proposition}

\begin{proposition}[Estimates  for $\square B$]\lab{prop:waveeqB}
The components  $B_i=(-\Delta)^{-1}(\curl(A)_i)$ verify the following estimate,
\bea\lab{squareB3}
\sum_{i=1}^3\big(\|\square B_i\|_{L^2(\MM)}+ \norm{\pr \,\square B_i}_{L^2(\MM)}  \big)  \lesssim M^2\ep^2.\label{spacetimeB}
\eea
We also have,
\bea
\sum_{i=1}^3\norm{\pr_0\pr_0 B_i}_{L^2(\MM)}\lesssim M\ep.
\eea
\end{proposition}

The proof of Proposition \ref{prop:commutsquarecurl} and Proposition \ref{prop:waveeqB} are similar in spirit. We give below a short outline of the proof of Proposition \ref{prop:waveeqB} which is slightly more difficult.

\begin{proof}
In what follows we   outline the main steps  in the proof of  space-time estimates  \eqref{squareB1}, \eqref{squareB3} for $\square B$ and  $\pr_0^2 B$.
We have:
\beaa
\square(B_i)&=&[\square,(-\Delta)^{-1}](\curl(A)_i)+(-\Delta)^{-1}(\square(\curl(A)_i))\\
&=& -(-\Delta)^{-1}[\square,\Delta](-\Delta)^{-1}(\curl(A)_i)+(-\Delta)^{-1}(\square(\curl(A)_i))\\
&=& -(-\Delta)^{-1}[\square,\Delta](B_i)+(-\Delta)^{-1}(\square(\curl(A)_i)).
\eeaa
Thus,  using the  $L^2$ boundedness of  $\partial (-\De)^{1/2}$ and  result of   proposition \ref{prop:commutsquarecurl},  we obtain:
\bea\lab{csd1}
\norm{\pr\square(B_i)}_{L^2(\MM)} &\les & \norm{(-\Delta)^{-\frac12}[\square,\Delta](B_i)}_{L^2(\MM)}+M^3\ep^2,
\eea
It remains to estimate   $\norm{(-\Delta)^{-\frac12}[\square,\Delta](B_i)}_{L^2(\MM)}$. We rely on the commutator formula \eqref{commsquare3} to write,
\bea\lab{csd2}
[\square, \Delta] B
&=&-4k^{ab}\nabla_a\nabla_b(\pr_0 B)+4n^{-1}\nabla_bn\nabla_b(\pr_0(\pr_0\ B))
 -2\nabla_0k^{ab}\nabla_a\nabla_b B\\
 &+&   F^{(1)}\prb^2B +F^{(2)}\prb B\nn\\
 F^{(1)}&=& \prb(A_0)+\A^2,\nn\\
 F^{(2)}&=& \pr\prb(A_0)+\A\prb\A+\A^3.\nn
\eea
with $B$ any component $(B_l)$, $l=1,2,3$.
  It is easy to check that,
\beaa
\norm{F^{(1)}}_{\lsit{3}}+    \norm{F^{(2)}}_{\lsit{\frac{3}{2}}}\  \les M\ep,
\eeaa
We write,
\bea\lab{csd6}
&&\norm{(-\Delta)^{-\frac12}[\square, \Delta](B_l)}_{L^2(\MM)}\les  N_1+N_2+N_3+N_4\\
\nn\\
&&N_1= \norm{(-\Delta)^{-\frac12}[k^{ab}\nabla_a\nabla_b(\pr_0(B_l))]}_{L^2(\MM)}\nn\\
&&N_2= \norm{(-\Delta)^{-\frac12}[n^{-1}\nabla_bn\nabla_b(\pr_0(\pr_0(B_l)))]}_{L^2(\MM)}\nn\\
&&N_3= \norm{(-\Delta)^{-\frac12}[\nabla_0k^{ab}\nabla_a\nabla_b(B_l)]}_{L^2(\MM)}\nn\\
&&N_4=\norm{F^{(1)}}_{\lsit{3}}\norm{\prb^2(B_l)}_{L^2(\MM)}
+\norm{F^{(2)}}_{\lsit{\frac{3}{2}}}\norm{\prb(B_l)}_{\lsit{6}}\nn
\eea
Using the estimates   \eqref{squareB2} we  easily infer that 
\bea
N_4\les M\ep^2. \label{N4}
\eea
To estimate $N_1$ we proceed as follows, using  the constraint equations \eqref{constk} for $k$,
\beaa
k^{ab}\nabla_a\nabla_b(\pr_0(B_l))&=& \nabla_a[k^{ab}\nabla_b(\pr_0(B_l))]
\eeaa
Together with the Bochner inequality on $\Si_t$ \eqref{bochnerestimates} and the bilinear assumption \eqref{bil5} , we obtain:
\bea\lab{csd7}
 N_1\les \norm{k^{ab}\pr_b(\pr_0(B_l))]}_{L^2(\MM)}\les M^3\ep^2.
\eea
To estimate $N_2$ we write,
$$n^{-1}\nabla_bn\nabla_b(\pr_0(\pr_0(B_l)))=\nabla^b[n^{-1}\nabla_bn\pr_0(\pr_0(B_l))]-(n^{-1}\Delta n-n^{-2}|\nabla n|^2)\pr_0(\pr_0(B_l)).$$
Together with the estimates \eqref{bootn} for the lapse $n$ and the Sobolev embedding on $\Si_t$ \eqref{sob}, this yields:
\bea\lab{csd8}
N_2&\les& \norm{n^{-1}\nabla_bn\pr_0(\pr_0(B_l))}_{L^2(\MM)}+\norm{(n^{-1}\Delta n-n^{-2}|\nabla n|^2)\pr_0(\pr_0(B_l))}_{\lsitt{2}{\frac{6}{5}}}\nn\\
\nn&\les& (\norm{\nabla n}_{L^\infty}+\norm{n^{-1}\Delta n-n^{-2}|\nabla n|^2}_{\lsit{3}})\norm{\pr_0(\pr_0(B_l))}_{L^2(\MM)}\\
&\les& M\ep\norm{\pr_0(\pr_0(B_l))}_{L^2(\MM)}.
\eea

\begin{remark}\lab{rem:tempura1}
Note that there is no room in the estimate \eqref{csd8}. Indeed the sharp   estimate $\norm{\nabla n}_{L^\infty(\MM)}\les M\ep$ given by \eqref{bootn} is   crucial as emphasized in remark \ref{rem:tempura}.      
\end{remark}
Finally, we consider $N_3$.
 Recall from \eqref{eq:structfol1} that the second fundamental form satisfies the following equation:
\bea\lab{csd9}
\nabla_0k_{ab}= E_{ab} -n^{-1}\nabla_a\nabla_bn-k_{ac}k_b\,^c=E_{ab}+\lot
\eea
where $E$ is the 2-tensor on $\Si_t$ defined as
$E_{ab}=\R_{a\,0\,b\,0}$. In view of the estimates \eqref{bootk} for $k$ and \eqref{bootn} for $n$, 
\beaa
\norm{\lot}_{\lsit{3}}\les \norm{\nabla^2n}_{\lsit{3}}+\norm{k}^2_{\lsit{6}}\les M\ep.
\eeaa
Thus  instead of estimating  
$
\nabla_0k^{ab}\nabla_a\nabla_b B $ in the definition of $N_3$    it suffices to estimate the term
 \beaa
 E^{ab}\nabla_a\nabla_b B&=&\nab_a(E^{ab}\nab_b B )- \nab_a E^{ab}\nab_b B
 \eeaa
Using the maximal foliation assumption, the Bianchi identities and the symmetries of $\R$, we   can write, schematically,
$\nabla^aE_{ab}=\A\R$ and therefore,  together with the bootstrap assumptions \eqref{bootA} for $A$ and \eqref{bootA0} for $A_0$, and the bootstrap assumption \eqref{bootR} for $\R$ yields:
\bea\lab{csd11}
\norm{\nabla^aE_{ab}}_{\lsit{\frac{3}{2}}}\les \norm{\A}_{\lsit{6}}\norm{\R}_{\lsit{2}}\les M^2\ep^2.
\eea
We thus have,
\beaa
\nabla_0k_{ab}=\nab_a(E^{ab}\nab_b B )+\lot
\eeaa
Now, in view of the  bilinear assumption \eqref{bil6},
\beaa
\|(-\De)^{1/2}\nab_a(E^{ab}\nab_b B )\|_{L^2(\MM)}&\les&  \norm{ \R_{0a0b}\nab_b B}_{L^2(\MM)}\les M^3 \ep^2
\eeaa
Hence, putting all the above  together we infer that,
\beaa
N_3\les M^3 \ep^2.
\eeaa
 Together with \eqref{csd6},  \eqref{csd7},  \eqref{csd8} and \eqref{N4},  we derive,
 \bea\lab{csd14}
\norm{\pr\, \square B}_{L^2(\MM)}\les  M^3\ep^2+M\ep\norm{\pr_0(\pr_0 B)}_{L^2(\MM)}.
\eea
 To close the estimate for $\norm{\pr\, \square B}_{L^2(\MM) }$  it  remains to control the right-hand side of \eqref{csd14}. This is achieved relying in particular on the following formula
$$\pr_0(\pr_0 B)=-\square  B+\Delta B+n^{-1}\nabla n\cdot\nabla B.$$
\end{proof}

\section{Energy estimate for the wave equation on a curved background}\lab{sec:bobo6}
\label{sect:energy}

Recall that $e_0=T$, the future unit normal to the $\Si_t$ foliation. Let $\pi$ be the deformation tensor of $e_0$, that is the symmetric 2-tensor on $\MM$ defined as:
$$\pi_{\a\b}=\D_\a T_\b+\D_\b T_\a.$$
In view of the definition of the second fundamental form $k$ and the lapse $n$, we have:
\bea\lab{defpi}
\pi_{ab}=-2k_{ab},\, \pi_{a0}=\pi_{0a}=n^{-1}\nabla_an,\,\pi_{00}=0.
\eea
In what follows $\HH$ denotes an arbitrary  weakly regular null  hypersurface\footnote{i.e. it   satisfies assumptions \eqref{assumptionH1} and \eqref{assumptionH2}}  with future normal  $L$  verifying  ${\bf g}(L,T)=-1$.  We denote by $\nabb$     the induced connection on the $2$-surfaces  $\HH\cap \Si_t$.  
\begin{lemma}\lab{lemma:energyestimatebis}
Let $F$ a scalar function on $\MM$, and let $\phi_0$ and $\phi_1$ two scalar functions on $\Si_0$. Let $\phi$ the solution of the following wave equation on $\MM$:
\bea\lab{eq:wave}
\left\{\begin{array}{l}
\square\phi=F,\\
\phi|_{\Si_0}=\phi_0,\, \pr_0(\phi)|_{\Si_0}=\phi_1. 
\end{array}\right.
\eea
Let $\EE_0, \EE_1$ denote the energy quantities,
\beaa
\EE_0[\phi]&:=&\norm{\prb\phi}_{\lsit{2}}+\sup_{\HH}(\norm{\nabb\phi}_{L^2(\HH)}+\norm{L(\phi)}_{L^2(\HH)})\\
\EE_1[\phi]&:=&\norm{\pr(\prb\phi)}_{\lsit{2}}+\norm{\pr_0(\pr_0\phi)}_{L^2(\MM)}+\sup_{\HH}\left(\norm{\nabb(\pr\phi)}_{L^2(\HH)}+\norm{L(\pr\phi)}_{L^2(\HH)}\right)
\eeaa
where the supremum is taken over all  weakly  regular null hypersurfaces $\HH$ (satisfying assumptions \eqref{assumptionH1} and \eqref{assumptionH2}). 
The following estimates hold true, provided that $\ep M^2$ is sufficiently small, 
\bea
\EE_0&\les& \norm{\nabla\phi_0}_{L^2(\Si_0)}+\norm{\phi_1}_{L^2(\Si_0)}+\norm{F}_{L^2(\MM)},\lab{energy0}\\
\EE_1&\les& \norm{\nabla^2\phi_0}_{L^2(\Si_0)}+\norm{\nabla\phi_1}_{L^2(\Si_0)}+\norm{\nabla F}_{L^2(\MM)}.
\lab{energy1}
\eea
\end{lemma}

\begin{proof}
We introduce the energy momentum tensor $Q_{\a\b}$ on $\MM$ given by:
$$Q_{\a\b}=Q_{\a\b}[\phi]=\pr_\a\phi\pr_\b\phi-\frac{1}{2}\g_{\a\b}\left(\g^{\mu\nu}\pr_\mu\phi\pr_\nu\phi\right).$$
In view of the equation \eqref{eq:wave} satisfied by $\phi$, we have:
$$\D^\a Q_{\a\b}=F\pr_\b\phi.$$
Now,  consider  the divergence of the  1-tensor  $P_\a=\Q_{\a\b}e_0^\b=Q_{\a 0},$   
$$\D^\a P_\a=F\pr_0\phi+\frac12 Q_{\a\b}\pi^{\a\b},$$
where $\pi$ is the deformation tensor of $e_0$. Integrating over well-chosen regions of $\MM$, we easily  obtain:
\bea\lab{nrj1}
\EE_0&\les& \norm{\nabla\phi_0}^2_{L^2(\Si_0)}+\norm{\phi_1}^2_{L^2(\Si_0)}+\left|\int_{\MM}F\pr_0\phi d\MM\right|+\left|\int_{\MM}Q_{\a\b}\pi^{\a\b}d\MM\right|\\
\nn&\les& \norm{\nabla\phi_0}^2_{L^2(\Si_0)}+\norm{\phi_1}^2_{L^2(\Si_0)}+\norm{F}_{L^2(\MM)}\norm{\pr_0\phi}_{L^2(\MM)}+\left|\int_{\MM}Q_{\a\b}\pi^{\a\b}d\MM\right|.
\eea

Next, we deal with the last term in the right-hand side of \eqref{nrj1}. In view of \eqref{defpi}  and our maximal foliation assumption, we have:
\beaa
\int_{\MM}Q_{\a\b}\pi^{\a\b}d\MM &=&-2\int_{\MM}\pr_a\phi\pr_b\phi k^{ab}d\MM+\int n^{-1}\nabla^an\pr_a\phi\pr_0\phi d\MM.
\eeaa
 Together with the bilinear bootstrap assumptions \eqref{bil7} and the estimates \eqref{bootn} for the lapse $n$, this yields:
\beaa
\left|\int_{\MM}Q_{\a\b}\pi^{\a\b}d\MM\right|&\les&\norm{k_{a\,\c}\pr^a\phi}_{L^2(\MM)}\norm{\pr\phi}_{L^2(\MM)}+\norm{\nabla n}_{L^\infty(\MM)}\norm{\prb\phi}^2_{L^2(\MM)}\\
&\les& M^2\ep\left(\sup_{\HH}\norm{\nabb\phi}_{L^2(\HH)}\right)\norm{\prb\phi}_{L^2(\MM)}+M\ep\norm{\prb\phi}^2_{L^2(\MM)},
\eeaa
which together with \eqref{nrj1} concludes the proof of the \eqref{energy0}. Though more technical  the proof  of \eqref{energy1} follows the same ideas, and we refer to \cite{KRS} for the details.
\end{proof}

\section{Improvement of the bootstrap assumptions (part 1)}\lab{sec:improve1}

In this section, we discuss the proof of Proposition \ref{prop:improve1}. More precisely, we derive estimates for $\R, A_0$ and $A$ which allow us to improve the basic bootstrap assumptions  \eqref{bootR}, \eqref{bootcurvatureflux}, \eqref{bootA} and \eqref{bootA0}.

\subsection{Curvature estimates} We derive the curvature estimates using the Bell-Robinson tensor,
 \beaa
 Q_{\a\b\ga\de}=\R_\a\,^\la\,\ga\,^\si \R_{\b\,\la\,\de\,\si}+\dual \R_\a\,^\la\,\ga\,^\si\dual \R_{\b\,\la\,\de\,\si}
 \eeaa
 Let 
 $$P_\a=Q_{\a\b\ga\de}e_0^\b e_0^\ga e_0^\de.$$ 
Then, we have:
\bea\lab{lybia}
D^\a P_\a=3Q_{\a\b\ga\de}\pi^{\a\b}e_0^\ga e_0^\de,
\eea
where $\pi$ is the deformation tensor of $e_0$. We introduce the  Riemannian metric,
\bea
h_{\a\b}&=&g_{\a\b}+2(e_0)_\a (e_0)_\b
\eea
 and use it to define  the following  space-time  norm for  tensors $U$:
$$|U|^2=U_{\a_1\cdots\a_k}U_{\a_1'\cdots\a'_k} h^{\a_1\a_1'}\cdots h^{\a_k\a'_k}.$$
Given two space-time tensors $U, V$ we denote by $U\c V$ a given contraction between
the two tensors and by $|U\c V|$ the  norm of the contraction according to the above definition.

Let $\HH$ be a weakly regular  null hypersurface with future normal  $L$,   ${\bf g}(L,T)=-1$.  Integrating \eqref{lybia} on a well-chosen, causal,  space-time region, we have:
$$\int_{\Si_t}|\R|^2+\int_{\HH}|\R\c L|^2\les \norm{\R}^2_{L^2(\Si_0)}+\left|\int_{\MM}Q_{\a\b\ga\de}\pi^{\a\b}e_0^\ga e_0^\de\right|\les \ep^2+\left|\int_{\MM}Q_{\a\b\ga\de}\pi^{\a\b}e_0^\ga e_0^\de\right|.$$
We need to estimate the term in the right-hand side of the previous inequality. Note that since $\pi_{00}=0$, $\pi_{0j}=n^{-1}\nabla_jn$, and $\pi_{ij}=k_{ij}$, the bootstrap assumption \eqref{bootR} for $\R$, and the estimates \eqref{bootn} for $n$ yield:
\beaa
\int_{\Si_t}|\R|^2+\int_{\HH}|\R\c L|^2&\les& \ep^2+\norm{\nabla n}_{L^\infty}\norm{\R}^2_{\lsit{2}}+\left|\int_{\MM}Q_{ij\ga\de} k^{ij}e_0^\ga e_0^\de\right|\\
&\les& \ep^2+(M\ep)^3+\left|\int_{\MM}Q_{ij\ga\de}k^{ij}e_0^\ga e_0^\de\right|.
\eeaa

The term in the right-hand side of the previous inequality is dangerous. Schematically it has the form 
$\left|\int_{\MM}k\R^2\right|.$ 
Typically  this term is  estimated by:
\beaa
\left|\int_{\MM}k\R^2\right|
&\les& \norm{k}_{\lsitt{2}{\infty}} \norm{\R}^2_{\lsit{2}},
\eeaa
which requires   a   Strichartz estimate for $k$ which is false even in flat space.
It is for this reason that we need  the trilinear  bootstrap assumption \eqref{trilinearboot}.
Using it we derive,
\bea\lab{camarche}
\int_{\Si_t}|\R|^2+\int_{\HH}|\R\c L|^2&\les& \ep^2+M^4 \ep^3.
\eea
which, for small $\ep$,   improves  the bootstrap assumptions \eqref{bootR} and  \eqref{bootcurvatureflux}.

\subsection{Improvement of the bootstrap assumption for $A_0$} Recall \eqref{eq:YM21}
\bea\label{tues5}
\De A_0 =\A\pr A+\A\pr A_0+\A^3.
\eea
Then, using \eqref{tues5}, elliptic estimates on $\Si_t$, and commuting  \eqref{tues5} with $\pr_t$ in order to control $\pr_tA_0$, we are able to obtain the improved estimate \eqref{bootA0imp} (see \cite{KRS} for the details).
 
\subsection{Improvement of the bootstrap assumption for $A$}

Using the estimates for $\square B_i$ derived in Lemma \ref{prop:waveeqB},   the estimates for $B$  on the initial slice $\Si_0$, and the energy estimate   \eqref{energy1} derived in Lemma \ref{lemma:energyestimatebis}, we have:
\bea\lab{tuesday2}
\norm{\pr^2B}_{\lsit{2}}\les \ep+M^2\ep^2.
\eea
Using  then \eqref{tuesday2} with Lemma \ref{recoverA}, we obtain:
\bea\lab{tuesday3}
\norm{\pr A}_{\lsit{2}}\les \norm{\pr^2B}_{\lsit{2}}+\norm{\pr E}_{\lsit{2}}\les \ep+M^2\ep^2.
\eea
which  proves corresponding  estimate  in \eqref{bootAimp}.

 To estimate $\pr_0 A$.  we recall that, 
$\pr_0(A_j)=\pr_j(A_0)+\R_{0j\c\c}.$
Thus, we have:
$$\norm{\pr_0 A}_{\lsit{2}}\les \norm{\pr A_0}_{\lsit{2}}+\norm{\R}_{\lsit{2}},$$
which together with the improved estimates for $\R$ and $A_0$ yields:
\bea\lab{tuesday4}
\norm{\pr_0 A}_{\lsit{2}}\les \ep+(M\ep)^{\frac{3}{2}}.
\eea

\section{Parametrix for the wave equation}\lab{sec:parametrix}

Let $u_\pm$ two families, indexed by $\om\in\SSS^2$, of scalar functions on the space-time $\MM$ satisfying the Eikonal equation for each $\om\in\SSS^2$. We also denote $\uom_\pm(t,x)=u_\pm(t,x,\om)$. We have the freedom of choosing $\uom_\pm$ on the initial slice $\Si_0$, and in order for the results in \cite{param2},  \cite{param4} to apply, we need to initialize $\uom_\pm$ on $\Si_0$ as in \cite{param1} (see also Chapter \ref{part:initialu}). 

Let $\HH_{\uom_\pm}$ denote the corresponding null level hypersurfaces. Let $\Lom_\pm$ its normal. $\Lom_\pm$ is null, and we fix it by imposing ${\bf g}(\Lom_\pm, T)=-1$. Let the vectorfield tangent to $\Si_t$ $\Nom_\pm$ be defined such as to satisfy:
$$\Lom_\pm=\pm T+\Nom_\pm.$$
We pick $(\eom_\pm)_A,\, A=1, 2$ vectorfields in $\Si_t$ such that together with $\Nom_\pm$ we obtain an orthonormal basis of $\Si_t$. Finally, we denote by $\nabb_\pm$ derivatives in the directions $(\eom_\pm)_A,\, A=1, 2$.

\begin{remark}
Note that $\HH_{\uom_\pm}$ satisfy assumptions \eqref{assumptionH1} and \eqref{assumptionH2} from the results in \cite{param3} (see Theorem 2.15 and section 3.4 in that paper).
\end{remark}

For any pair of functions $f_\pm$ on $\RRR^3$, we define the following scalar function on $\MM$:
$$\psi[f_+, f_-](t,x)=\int_{\SSS^2} \int_0^\infty  e^{i \la \uom_+(t,x)}f_+(\la\om)\la^2d\la d\om
 +\int_{\SSS^2} \int_0^\infty e^{i\la \uom_-(t,x)}f_-(\la\om)\la^2d\la d\om.$$
We appeal to the following result from \cite{param2} \cite{param4} (see also Chapters \ref{part:paramtime} and \ref{part:paraminit}):
\begin{theorem}\lab{prop:estparam}
Let $\phi_0$ and $\phi_1$ two scalar functions on $\Si_0$. Then, there is a unique pair of functions $(f_+, f_-)$ such that: 
$$\psi[f_+, f_-]|_{\Si_0}=\phi_0\textrm{ and }\pr_0(\psi[f_+, f_-])|_{\Si_0}=\phi_1.$$
Furthermore, $f_\pm$ satisfy the following estimates:
$$\norm{\la f_+}_{L^2(\RRR^3)}+\norm{\la f_-}_{L^2(\RRR^3)}\les \norm{\nabla\phi_0}_{L^2(\Si_0)}+\norm{\phi_1}_{L^2(\Si_0)},$$
and:
$$\norm{\la^2f_+}_{L^2(\RRR^3)}+\norm{\la^2f_-}_{L^2(\RRR^3)}\les \norm{\nabla^2\phi_0}_{L^2(\Si_0)}+\norm{\nabla\phi_1}_{L^2(\Si_0)}.$$
Finally, $\square\psi[f_+, f_-]$ satisfies the following estimates:
$$\norm{\square\psi[f_+, f_-]}_{L^2(\MM)}\les M\ep(\norm{\nabla\phi_0}_{L^2(\Si_0)}+\norm{\phi_1}_{L^2(\Si_0)}),$$
and:
$$\norm{\pr\square\psi[f_+, f_-]}_{L^2(\MM)}\les M\ep(\norm{\nabla^2\phi_0}_{L^2(\Si_0)}+\norm{\nabla\phi_1}_{L^2(\Si_0)}).$$
\end{theorem}

\begin{remark}
The existence of $f_\pm$ and the first two estimates of Theorem \ref{prop:estparam} are proved in \cite{param2} (see also Chapter \ref{part:paraminit}), while the last two estimates in Theorem \ref{prop:estparam} are proved in \cite{param4} (see also Chapter \ref{part:paramtime}).
\end{remark}

We associate to any pair of functions $\phi_0, \phi_1$ on $\Si_0$ the function $\Psi_{om}[\phi_0,\phi_1]$ defined for $(t,x)\in\MM$ as:
$$\Psi_{om}[\phi_0,\phi_1]=\psi[f_+, f_-]$$
where $(f_+, f_-)$ is defined in view of Theorem \ref{prop:estparam} as the unique pair of functions associated to $(\phi_0, \phi_1)$. In particular, we obtain:
$$\norm{\la f_+}_{L^2(\RRR^3)}+\norm{\la f_-}_{L^2(\RRR^3)}\les \norm{\nabla\phi_0}_{L^2(\Si_0)}+\norm{\phi_1}_{L^2(\Si_0)},$$
$$\norm{\la^2f_+}_{L^2(\RRR^3)}+\norm{\la^2f_-}_{L^2(\RRR^3)}\les \norm{\nabla^2\phi_0}_{L^2(\Si_0)}+\norm{\nabla \phi_1}_{L^2(\Si_0)},$$
\bea\lab{par1:bis}
\norm{\square\Psi_{om}[\phi_0,\phi_1]}_{L^2(\MM)}\les M\ep(\norm{\nabla\phi_0}_{L^2(\Si_0)}+\norm{\phi_1}_{L^2(\Si_0)}),
\eea
and:
\bea\lab{par1}
\norm{\pr\square\Psi_{om}[\phi_0,\phi_1]}_{L^2(\MM)}\les M\ep(\norm{\nabla^2\phi_0}_{L^2(\Si_0)}+\norm{\nabla \phi_1}_{L^2(\Si_0)}).
\eea

Next, let $\uoms_\pm$ two families, indexed by $\om\in\SSS^2$ and $s\in \RRR$, of scalar functions on the space-time $\MM$ satisfying the Eikonal equation for each $\om\in\SSS^2$ and $s\in\RRR$. We have the freedom of choosing $\uoms_\pm$ on the slice $\Si_s$, and in order for the results in \cite{param2} \cite{param4} to apply, we need to initialize $\uoms_\pm$ on $\Si_s$ as in \cite{param1}.  
Note that the families $\uom_\pm$ correspond to $\uoms$ with the choice $s=0$. For any pair of functions $f_\pm$ on $\RRR^3$, and for any $s\in\RRR$, we define the following scalar function on $\MM$:
$$\psi_s[f_+, f_-](t,x,s)=\int_{\SSS^2} \int_0^\infty  e^{i \la \uoms_+(t,x)}f_+(\la\om)\la^2d\la d\om
 +\int_{\SSS^2} \int_0^\infty e^{i\la \uoms_-(t,x)}f_-(\la\om)\la^2d\la d\om.$$
We have the following straightforward corollary of Theorem \ref{prop:estparam}:
\begin{corollary}\lab{cor:estparam}
Let $s\in\RRR$. Let $\phi_0$ and $\phi_1$ two scalar functions on $\Si_s$. Then, there is a unique pair of functions $(f_+, f_-)$ such that: 
$$\psi_s[f_+, f_-]|_{\Si_s}=\phi_0\textrm{ and }\pr_0(\psi_s[f_+, f_-])|_{\Si_s}=\phi_1.$$
Furthermore, $f_\pm$ satisfy the following estimates:
$$\norm{\la f_+}_{L^2(\RRR^3)}+\norm{\la f_-}_{L^2(\RRR^3)}\les \norm{\nabla\phi_0}_{L^2(\Si_s)}+\norm{\phi_1}_{L^2(\Si_s)},$$
and:
$$\norm{\la^2f_+}_{L^2(\RRR^3)}+\norm{\la^2f_-}_{L^2(\RRR^3)}\les \norm{\nabla^2\phi_0}_{L^2(\Si_s)}+\norm{\nabla\phi_1}_{L^2(\Si_s)}.$$
Finally, $\square\psi_s[f_+, f_-]$ satisfies the following estimates:
$$\norm{\square\psi_s[f_+, f_-]}_{L^2(\MM)}\les M\ep(\norm{\nabla\phi_0}_{L^2(\Si_s)}+\norm{\phi_1}_{L^2(\Si_s)}),$$
and:
$$\norm{\pr\square\psi_s[f_+, f_-]}_{L^2(\MM)}\les M\ep(\norm{\nabla^2\phi_0}_{L^2(\Si_s)}+\norm{\nabla\phi_1}_{L^2(\Si_s)}).$$
\end{corollary}

Next, for any $s\in \RRR$, we associate to any function $F$ on $\Si_s$ the function $\Psi(t,s)F$ defined for $(t,x)\in\MM$ as:
$$\Psi(t,s)F=\psi_s[f_+, f_-](t)$$
where $(f_+, f_-)$ is defined in view of Corollary \ref{cor:estparam} as the unique pair of functions associated to the choice $(\phi_0, \phi_1)=(0,F)$. In particular, we obtain:
$$\norm{\la f_+}_{L^2(\RRR^3)}+\norm{\la f_-}_{L^2(\RRR^3)}\les \norm{F}_{L^2(\Si_s)},$$
$$\norm{\la^2f_+}_{L^2(\RRR^3)}+\norm{\la^2f_-}_{L^2(\RRR^3)}\les \norm{\nabla F}_{L^2(\Si_s)},$$
\bea\lab{par2:bis}
\norm{\square\Psi(t,s)F}_{L^2(\MM)}\les M\ep\norm{F}_{L^2(\Si_s)},
\eea
and:
\bea\lab{par2}
\norm{\pr\square\Psi(t,s)F}_{L^2(\MM)}\les M\ep\norm{\nabla F}_{L^2(\Si_s)}.
\eea

Now, we are in position to construct an exact  parametrix for the wave equation \eqref{eq:wave}:
\begin{theorem}[Representation formula]\lab{lemma:parametrixconstruction}
Let $F$ a scalar function on $\MM$, and let $\phi_0$ and $\phi_1$ two scalar functions on $\Si_0$. Let $\phi$ the solution of the wave equation \eqref{eq:wave} on $\MM$. Then, there is a sequence $\phi^{(j)}$, $j\ge 0$, of scalar functions  approximations of $\phi$ and  a sequence $F^{(j)}$, $j\ge 0$,  of scalar functions  on $\MM$,   with 
of the form:
\beaa
\phi^{(0)}=\Psi_{om}[\phi_0,\phi_1]+\int_0^t\Psi(t,s)F^{(0)}(s,.)ds,\qquad F^{(0)}=F
\eeaa
and for all $j\geq 1$:
$$\phi^{(j)}=\int_0^t\Psi(t,s)F^{(j)}(s,.)ds,$$
such that,
$$\phi=\sum_{j=0}^{+\infty}\phi^{(j)},$$ 
 and    such that  $\phi^{(j)}$ and $F^{(j)}$ satisfy the following estimates:
$$\norm{\prb\phi^{(j)}}_{\lsit{2}}+\norm{F^{(j)}}_{L^2(\MM)}\les (M\ep)^j(\norm{\nabla\phi_0}_{L^2(\Si_0)}+\norm{\phi_1}_{L^2(\Si_0)}+\norm{F}_{L^2(\MM)}),$$
and:
$$\norm{\pr\prb\phi^{(j)}}_{\lsit{2}}+\norm{\pr F^{(j)}}_{L^2(\MM)}\les (M\ep)^j(\norm{\nabla^2\phi_0}_{L^2(\Si_0)}+\norm{\nabla\phi_1}_{L^2(\Si_0)}+\norm{\pr F}_{L^2(\MM)}),$$

\end{theorem}
\begin{proof}
Let us define:
$$F^{(0)}=F\textrm{ and }\phi^{(0)}=\Psi_{om}[\phi_0,\phi_1]+\int_0^t\Psi(t,s)F^{(0)}(s,.)ds.$$
Then, we define iteratively for $j\geq 1$:
$$F^{(j)}=-\square\phi^{(j-1)}+F^{(j-1)}\textrm{ and }\phi^{(j)}=\int_0^t\Psi(t,s)F^{(j)}(s,.)ds.$$
The proof follows from the estimates \eqref{par1:bis}, \eqref{par1}, \eqref{par2:bis} and \eqref{par2}, together with the energy estimates for the wave equation of Lemma \ref{lemma:energyestimatebis}. We refer to \cite{KRS} for the details.
\end{proof}

\section{Improvement of the bootstrap assumptions (part 2)}\lab{sec:improve2}

The goal of this section  and next  section is to prove Proposition \ref{prop:improve2}. This requires in particular to write $B$ using the representation formula of Theorem \ref{lemma:parametrixconstruction}. In this section  we derive the improved bilinear estimate \eqref{bil1imp}, \eqref{bil5imp}, \eqref{bil6imp}, \eqref{bil7imp} and \eqref{bil8imp} of Proposition \ref{prop:improve2}. We also  derive  the improved trilinear estimate  \eqref{trilinearbootimp}.

\subsection{Improvement of the bilinear bootstrap assumptions I}\lab{sec:proofbil1} 

In this section, we give the main ideas on the  how we  derive the improved bilinear estimate \eqref{bil1imp}, \eqref{bil5imp}, \eqref{bil6imp}, \eqref{bil7imp} and \eqref{bil8imp} of Proposition \ref{prop:improve2}. These bilinear estimates all involve the norm in $L^2(\MM)$ of quantities of the type:
$$\CC(U,\pr\phi),$$
where $\CC(U,\pr\phi)$ denotes a contraction with respect to one index between a tensor $U$ and $\pr\phi$, with $\phi$ being a scalar function which is solution to the wave equation \eqref{eq:wave} with $F, \phi_0$ and $\phi_1$ satisfying the estimate:
$$\norm{\nabla^2\phi_0}_{L^2(\Si_0)}+\norm{\nabla\phi_1}_{L^2(\Si_0)}+\norm{\pr F}_{L^2(\MM)}\les M\ep.$$
In particular, we may use the parametrix constructed in Lemma \ref{lemma:parametrixconstruction} for $\phi$:
$$\phi=\sum_{j=0}^{+\infty}\phi^{(j)},$$
with:
$$\phi^{(0)}=\Psi_{om}[\phi_0,\phi_1]+\int_0^t\Psi(t,s)F(s,.)ds,$$
and for all $j\geq 1$:
$$\phi^{(j)}=\int_0^t\Psi(t,s)F^{(j)}(s,.)ds.$$
Thus, we need to estimate the norm in $L^2(\MM)$ of contractions of quantities of the type:
$$\CC(U, \pr(\Psi_{om}[\phi_0,\phi_1]))+\sum_{j=0}^{+\infty}\int_0^t\CC(U, \pr(\Psi(t,s)F^{(j)}(s,.)))ds.$$
After using the definition of $\Psi_{om}$ and $\Psi(t,s)$, and the estimates for $F^{(j)}$ provided by Lemma \ref{lemma:parametrixconstruction}, this reduces to estimating:
$$\int_{\SSS^2} \int_0^\infty  \CC(U, \pr(e^{i \la \uom_+(t,x)}))f_+(\la\om)\la^2d\la d\om
 +\int_{\SSS^2} \int_0^\infty \CC(U,\pr(e^{i\la \uom_-(t,x)}))f_-(\la\om)\la^2d\la d\om,$$
where $f_\pm$ in view of Theorem \ref{prop:estparam} and the estimates for $F, \phi_0$ and $\phi_1$ 
satisfies:
$$\norm{\la^2f_\pm}_{L^2(\RRR^3)}\les M\ep.$$
Since both half waves parametrices are estimated in the same way, the bilinear estimates \eqref{bil1},  \eqref{bil5}, \eqref{bil6}, \eqref{bil7} and \eqref{bil8} all estimate the norm in $L^2(\MM)$ of  contractions of quantities of the type:
$$\int_{\SSS^2} \int_0^\infty \CC(U,\pr(e^{i \la \uom(t,x)}))f(\la\om)\la^2d\la d\om,$$
where $f$ satisfies:
\bea\lab{shouf}
\norm{\la^2f}_{L^2(\RRR^3)}\les M\ep.
\eea
Furthermore we observe that 
$\pr_j(e^{i \la \uom})=i\la e^{i \la \uom}\pr_j(\uom),$
and  that  the gradient of $\uom$ on $\Si_t$ is given by:
$\nabla(\uom)=\bom^{-1}\Nom,$
with  $\bom=|\nabla(\uom)|^{-1}$ is the null lapse, and 
$\Nom=\frac{\nabla\uom}{|\nabla\uom|}$
is the unit normal to $\HH_{\uom}\cap\Si_t  $  along $\Si_t$. Thus, the bilinear estimates \eqref{bil1},  \eqref{bil5}, \eqref{bil6}, \eqref{bil7} and \eqref{bil8} all  reduce to  $L^2(\MM)$-estimates of  expressions of the form:

\bea
\frak{C}[U,f]:=\int_{\SSS^2} \int_0^\infty e^{i \la \uom(t,x)}\bom^{-1}\CC(U,\Nom)f(\la\om)\la^3d\la d\om,
\eea
where $f$ satisfies \eqref{shouf}. To estimate $\frak{C}[U,f]$  we  follow the strategy of \cite{Kl-R3}. 
\bea\lab{dorbia}
&&\left\|\frak{C}[U,f]\right\|_{L^2(\MM)}\\
\nn&\les&
 \int_{\SSS^2} \left\|  \bom^{-1}\CC(U, \Nom)\left(\int_0^{+\infty}e^{i \la \uom(t,x)}f(\la\om)\la^3d\la\right)\right\|_{L^2(\MM)} d\om\\
\nn&\les & \int_{\SSS^2} \norm{\bom^{-1}}_{L^\infty(\MM)}\norm{\CC(U, \Nom)}_{\luom{2}}\left\|\int_0^{+\infty}e^{i \la \uom(t,x)}f(\la\om)\la^3d\la\right\|_{L^2_{\uom}} d\om\\
\nn&\les & \left(\sup_{\om\in\SSS^2}\norm{\bom^{-1}}_{L^\infty(\MM)}\right)\left(\sup_{\om\in\SSS^2}\norm{\CC(U, \Nom)}_{\luom{2}}\right)\left(\int_{\SSS^2}\norm{\la^3f(\la\om)}_{L^2_\la}d\om\right)\nn\\
\nn&\les & \left(\sup_{\om\in\SSS^2}\norm{\bom^{-1}}_{L^\infty(\MM)}\right)\left(\sup_{\om\in\SSS^2}\norm{\CC(U, \Nom)}_{\luom{2}}\right)\norm{\la^2f}_{L^2(\RRR^3)},
\eea
Now, since $\uom$ has been initialized on $\Si_0$ as in \cite{param1}, and satisfies the Eikonal equation on $\MM$, the results in \cite{param3} (see Theorem 2.15 in that paper, and also \eqref{ch4:estb}) under the assumption of Theorem \ref{th:mainter} imply:
$$\sup_{\om\in\SSS^2}\norm{\bom^{-1}}_{L^\infty(\MM)}\les 1.$$
Together with the fact that $f$ satisfies \eqref{shouf}, and with \eqref{dorbia}, we finally obtain:
\bea\lab{bilproof} 
&&\left\|\int_{\SSS^2} \int_0^\infty  e^{i \la \uom(t,x)}\bom^{-1}\CC(U, \Nom)f(\la\om)\la^3d\la d\om\right\|_{L^2(\MM)}\\
\nn&\les & M\ep\left(\sup_{\om\in\SSS^2}\norm{\CC(U, \Nom)}_{\luom{2}}\right).
\eea
It remains to estimate the right-hand side of \eqref{bilproof} for the contractions appearing in the bilinear estimates \eqref{bil1}, \eqref{bil5}, \eqref{bil6}, \eqref{bil7} and \eqref{bil8}. Since all the estimates in the proof will be uniform in $\om$, we drop the index $\om$ to ease the notations.

\begin{remark}
In the    proof of  bilinear estimates \eqref{bil1imp}, \eqref{bil5imp}, \eqref{bil6imp}, \eqref{bil7imp} and \eqref{bil8imp}, 
the tensor $U$ appearing  in the expression $\CC(U, N)$ is either $\R $ or    derivatives of solutions $\phi$ of a 
a scalar wave equation.    In view of the bootstrap assumption \eqref{bootcurvatureflux} for the curvature flux,  as well  as the first  energy estimate for the wave equation in Lemma \ref{lemma:energyestimatebis}, we can control 
$ \| \CC(U, N)\|_{\lu{2}}$  as long as we can show that $\CC(U, N)$ can be expressed\footnote{In other words, our goal is to check that the term $\CC(U, N)$  does not involve the dangerous terms of the type
$\underline{\a}\textrm{ and }\Lb\phi$,
where $\Lb$ is the vectorfield defined as $\Lb=2T-L$, and $\underline{\a}$ is the two tensor on $\Si_t\cap \HH_u$ defined  by 
$\underline{\a}_{AB}=\R_{\Lb\,A\,\Lb\,B}.$}  in terms of,
$\R\c L,\, \nabb\phi\textrm{ and }L(\phi).$
\end{remark}

 \subsubsection{Proof of   \eqref{bil1imp}}
 Since $A=\curl(B)+E$ in view of Lemma \ref{recoverA} and  bootstrap assumption \eqref{bootA}, we have:
\bea\lab{sodeska}
\norm{A^j\pr_j(A)}_{L^2(\MM)}&\les& \norm{(\curl(B))^j\pr_j(A)}_{L^2(\MM)}+\norm{E}_{\lsitt{2}{\infty}}\norm{\pr A}_{\lsit{2}}\\
\nn&\les& \norm{(\curl(B))^j\pr_j(A)}_{L^2(\MM)}+M^2\ep^2,
\eea
To estimate $\norm{(\curl(B))^j\pr_j(A)}_{L^2(\MM)}$ we write, 
$(\curl(B))^j\pr_j(A)=\in_{jmn}\pr_m(B_n)\pr_j(A).$
We are now ready to apply  the representation theorem \ref{lemma:parametrixconstruction}  to $B$. 
Indeed, according to  Lemma \ref{prop:waveeqB},  and proposition   \ref{lemma:estB},    we have
\bea\lab{reppossible}
\square B=F, \qquad 
\|\pr F\|_{L^2(\MM)}&\les& M^2\ep^2\label{squareB.}\\
\norm{\prb B(0)}_{L^2(\Si_0)}+\norm{\pr^2 B(0)}_{L^2(\Si_0)}+\norm{\pr(\pr_0 B(0))}_{L^2(\Si_0)}&\les& M\ep.\nn
\eea

We are thus in a position to apply  the reduction discussed in the  subsection above and  reduce our desired bilinear estimate to an estimate for,
\beaa
\CC(U, N)&=&\in_{jm\c}N_m\pr_j(A)
\eeaa

Now, we decompose $\pr_j$ on the orthonormal frame $N, f_A, A=1, 2$ of $\Si_t$, where we recall that $f_A, A=1, 2$ denotes an orthonormal basis of $\HH_u\cap \Si_t$. We have schematically:
\bea\lab{decompositionpr}
\pr_j=N_jN+\nabb,
\eea
where $\nabb$ denotes derivatives which are tangent to $\HH_u\cap \Si_t$. Thus, we have:
$$\in_{jm\c}N_m\pr_j(A)=\in_{jm\c}N_mN_j\pr_N(A)+\nabb(A)=\nabb(A),$$
where we used the antisymmetry of $\in_{jm\c}$ in the last equality. Therefore, we obtain in this case:
$$\norm{\CC(U, N)}_{\lu{2}}\les \norm{\nabb(A)}_{\lu{2}}.$$
It remains  to  estimate $\norm{\nabb(A)}_{\lu{2}}$. Since $A=\curl(B)+E$ we have, using  Lemma \ref{recoverA} again, 
followed by  Proposition \ref{prop:waveeqB} and Lemma \ref{lemma:energyestimatebis}:
\beaa
\norm{\nabb(A)}_{\lu{2}}&\les&  \norm{\nabb(\pr B)}_{\lu{2}}+\norm{\nabb(E)}_{\lu{2}}\\
&\les&  \norm{\nabb(\pr B)}_{\lu{2}}+\norm{\pr E}_{\lsit{3}}+\norm{\pr^2E}_{\lsit{\frac{3}{2}}}\\
&\les & \norm{\nabb(\pr B)}_{\lu{2}}+M\ep,\\
&\les& M\ep.
\eeaa
Therefore,
\beaa
\norm{A^j\pr_j(A)}_{L^2(\MM)}&\les& \norm{(\curl(B))^j\pr_j(A)}_{L^2(\MM)}+M^2\ep^2\les M^2 \ep^2,
\eeaa
as desired.

\subsubsection{Proof of \eqref{bil5imp}} 

The proof of \eqref{bil5imp} is similar to the one of \eqref{bil1imp} in view of Lemma \ref{recoverA}. 
 
 \subsubsection{Proof of  \eqref{bil6imp}} Since $B$ satisfies a wave equation in view of Lemma \ref{prop:waveeqB}, the quantity $\CC(U, N)$ is in this case\footnote{Use also $L=T+N$, $\Lb=T-N$ and the symmetries of $\R$},
$$\CC(U, N)=N_j\R_{0\,j\,\c\,\c}=\R_{0\,N\,\c\,\c}=\frac{1}{2}\R_{L\,\Lb\,\c\,\c}$$
which together with the bootstrap assumption for the curvature flux \eqref{bootcurvatureflux} improves the bilinear estimate \eqref{bil6}.

\subsubsection{Proof of  \eqref{bil7imp}} We have $k_{j\,\c}=A^j$ and $A=\curl(B)+E$ in view of Lemma \ref{recoverA}. Arguing as in \eqref{sodeska}, we reduce the proof to the estimate of: 
$$\norm{(\curl B)^j\pr_j\phi}_{L^2(\MM)}.$$
Then, the proof proceeds as the one of \eqref{bil1imp}.

\subsubsection{Proof of \eqref{bil8imp}} 

The proof of \eqref{bil8imp} proceeds as in \eqref{bil7imp}.

\subsection{Improvement of the trilinear estimate}

In this section, we shall derive the improved trilinear estimate \eqref{trilinearbootimp}.
To estimate the trilinear quantity
$
\left|\int_{\MM}Q_{ij\ga\de}k^{ij}e_0^\ga e_0^\de\right|.
$ we first write, according to Lemma \ref{recoverA}, $A=\curl(B)+E$ by. Arguing as in \eqref{sodeska}, 
we reduce the proof   of   \eqref{trilinearbootimp}    to an  estimate for: 
$$\left|\int_{\MM}Q_{\c\,j\ga\de}(\curl(B))_je_0^\ga e_0^\de\right|.$$
Making use of the wave equation \eqref{squareB.} for $B$ 
 we argue as in the beginning of section \ref{sec:proofbil1}  to  reduce the proof to an  estimate of the following:
$$\left|\int_{\MM}\int_{\SSS^2} \int_0^\infty e^{i \la \uom(t,x)}\bom^{-1}\left(\in_{jm\c}\Nom_m Q_{j\,\c\,\c\,\c}\right)f(\la\om)\la^3d\la d\om d\MM\right|$$
where $f$ satisfies:
$$\norm{\la^2f}_{L^2(\RRR^3)}\les M\ep.$$
Arguing  exactly  as in \eqref{dorbia} \eqref{bilproof}, we   can estimate the latter integral  by  the quantity  $\sup_{\om\in\SSS^2}\norm{\in_{jm\c}N_m Q_{j\,\c\,\c\,\c}}_{L^2_{\uom}L^1(\HH_{\uom})}M\ep$.
In other words, 
\bea\lab{kabul}
\left|\int_{\MM}Q_{ij\ga\de}k^{ij}e_0^\ga e_0^\de\right|\les\sup_{\om\in\SSS^2}\norm{\in_{jm\c}N_m Q_{j\,\c\,\c\,\c}}_{L^2_{\uom}L^1(\HH_{\uom})}M\ep +M^3\ep^3.
\eea

Next, we estimate the right-hand side of \eqref{kabul}. Since all the estimates in the proof will be uniform in $\om$, we drop the index $\om$ to ease the notations. The formula \eqref{def:bellrobison} for the Bell-Robinson tensor $Q$ yields:
\beaa
Q_{j\c\c\c}&=&\R_j\,^\la\,\c\c \R_{\c\,\la\,\c\,\c}+dual\\
&=& -\frac{1}{2}\R_{j\,L\,\c\c} \R_{\c\,\Lb\,\c\,\c}-\frac{1}{2}\R_{j\,\Lb\,\c\c} \R_{\c\,L\,\c\,\c}+\R_{j\,A\,\c\c} \R_{\c\,A\,\c\,\c}+dual,
\eeaa
where we used the frame $L, \Lb, f_A, A=1, 2$ in the last equality. Thus, we have schematically:
$$\in_{jm\c}N_m Q_{j\,\c\,\c\,\c}=\R(\R\c L+\in_{jm\c}N_m\R_{j\,A\,\c\c})$$

 Decomposing $e_j$  with respect  to the   orthonormal  frame $N, f_B, B=1, 2$,  we note that:
$$\in_{jm\c}N_m \R_{jA\c\c}=\in_{jm\c}N_jN_m \R_{NA\c\c}+\in_{jm\c}(f_B)_jN_m \R_{BA\c\c}=\R_{BA\c\c}.$$
On the other hand, decomposing  $\R_{BA\c\c}$ further  and using the symmetries of $\R$, one easily checks that $\R_{BA\c\c}$ must contain at least one $L$ so that it is of the type $\R\c\L$.
  Thus, we have schematically:
  \bea
  \in_{jm\c}N_m Q_{j\,\c\,\c\,\c}=\R(\R\c L).
  \eea
Thus,  in view of   \eqref{kabul},   making use of  the bootstrap assumptions \eqref{bootR} on $R$ and \eqref{bootcurvatureflux} on the curvature flux, we deduce,
\beaa
\left|\int_{\MM}Q_{ij\ga\de}k^{ij}e_0^\ga e_0^\de\right|
&\les& (M\ep)^3+M\ep\norm{\R  \R_L}_{L^2_uL^1(\HH_u)}\\
&\les& (M\ep)^3+M\ep\norm{\R}_{L^2(\MM)}\norm{\R_L}_{\lu{2}}\\
&\les& M^3 \ep^3
\eeaa
In other words,
\bea\lab{belrob2}
\left|\int_{\MM}Q_{ij\ga\de}k^{ij}e_0^\ga e_0^\de\right|\les (M\ep)^3.
\eea
 which yields  the desired  improvement of the trilinear estimate \eqref{trilinearboot}.

\section{Improvement of the bootstrap assumptions (part 3)}\lab{sec:improve2B}

In this section, we conclude the proof of Proposition \ref{prop:improve2}. More precisely, we give the main ideas in the improvement of  the bilinear bootstrap assumptions II. We start with a discussion of the sharp $L^4(\MM)$- Strichartz estimate.

\subsection{The sharp Strichartz $L^4(\MM)$ estimate}

To a function $f$ on $\RRR^3$ and a family $\uom$ indexed by $\om\in\SSS^2$ of scalar functions on the space-time $\MM$ satisfying the Eikonal equation for each $\om\in\SSS^2$, we associate a half wave parametrix:
$$\int_{\SSS^2} \int_0^\infty  e^{i \la \uom(t,x)}f(\la\om)\la^2d\la d\om.$$
Let an integer $p$ and a smooth cut-off function $\psi$ on $(0,+\infty)$ supported in a shell. We call a half wave parametrix localized at frequencies of size $\la\sim 2^p$ the following Fourier integral operator:
$$\int_{\SSS^2} \int_0^\infty  e^{i \la \uom(t,x)}\psi(2^{-p}\la)f(\la\om)\la^2d\la d\om.$$
We have the following $L^4(\MM)$ Strichartz estimates localized in frequency for a half wave parametrix which are proved in \cite{bil2} (see also Chapter \ref{part:strich}):
\begin{proposition}[Corollary 2.8 in \cite{bil2}]\lab{prop:L4strichartz}
Let $f$ a function on $\RRR^3$, let $p\in\mathbb{N}$, and let $\psi$ a smooth function on $(0, +\infty)$  compactly supported in the shell $1/2\leq\la\leq 2$. Let $\uom$ a family indexed by $\om\in\SSS^2$ of scalar functions on the space-time $\MM$ satisfying the Eikonal equation for each $\om\in\SSS^2$ and initialized on the initial slice $\Si_0$ as in \cite{param1}. Let $\phi_p$ the scalar function on $\MM$ defined by the following oscillatory integral:
$$\phi_p(t,x)=\int_{\SSS^2} \int_0^\infty  e^{i \la \uom(t,x)}\psi(2^{-p}\la)f(\la\om)\la^2d\la d\om.$$
Then, we have the following $L^4(\MM)$ Strichartz estimates for $\phi_p$:
\bea
\norm{\phi_p}_{L^4(\MM)}&\les& 2^{\frac{p}{2}}\norm{\psi(2^{-p}\la)f}_{L^2(\RRR^3)},\\
\norm{\pr\phi_p}_{L^4(\MM)}&\les& 2^{\frac{3p}{2}}\norm{\psi(2^{-p}\la)f}_{L^2(\RRR^3)},\\
\norm{\pr^2\phi_p}_{L^4(\MM)}&\les& 2^{\frac{5p}{2}}\norm{\psi(2^{-p}\la)f}_{L^2(\RRR^3)}.
\eea
\end{proposition}

Note that this Strichartz estimate is sharp. 

\subsection{Improvement of the non sharp Strichartz estimates} Here, we derive the improved non sharp Strichartz estimates \eqref{bootstrichimp} and  \eqref{bootstrichBimp}. In view of Lemma \ref{recoverA}, \eqref{bootstrichimp} easily follows from \eqref{bootstrichBimp}, so we focus on the later improved estimate. 

\begin{corollary}\lab{cor:strichartzB}
$B$ satisfies the following Strichartz estimate:
$$\norm{\pr B}_{\lsitt{2}{7}}\lesssim M\ep.$$
\end{corollary}

\begin{proof}
Recall \eqref{reppossible} which allows us  to apply  the representation formula of Theorem \ref{lemma:parametrixconstruction}  to $B$. By a straightforward reduction  the proof then reduces to the following  non-sharp    Strichartz estimate  for a half wave parametrix:
\bea\lab{kyoto1bis}
\left\|   \pr \left( \int_{\SSS^2} \int_0^\infty  e^{i \la \uom(t,x)}  f(\la\om)\la^2d\la d\om\right)\right\|_{\lsitt{2}{7}}\les \norm{\la^2 f}_{L^2(\RRR^3)}.
\eea
Then, the proof of Corollary \ref{cor:strichartzB} follows in particular from the sharp Strichartz  estimate of  Proposition \ref{prop:L4strichartz}. We refer to \cite{KRS} for the details.
\end{proof}

\subsection{Improvement of the bilinear bootstrap assumptions II}\lab{sec:proofbil2} 
In this section, we  sketch  the proofs of  the improved bilinear estimates \eqref{bil3imp} and \eqref{bil4imp} of Proposition \ref{prop:improve2}.  Based on  the decomposition $A=\curl(B)+E$ of Lemma \ref{recoverA} it is easy to show that that the proof of the bilinear estimates \eqref{bil3} and \eqref{bil4} reduces to:
\bea\lab{domo}
\norm{(-\Delta)^{-\frac12}(Q_{ij}(\pr B,\pr B))}_{L^2(\MM)}\lesssim M^2\ep^2.
\eea
Decomposing $B$  according to Theorem  \ref{lemma:parametrixconstruction},
\bea\lab{domo0}
\norm{(-\Delta)^{-\frac12}(Q_{ij}(\pr B,\pr B))}_{L^2(\MM)}\leq \sum_{m, n=0}^{+\infty}\norm{(-\Delta)^{-\frac12}(Q_{ij}(\phi^{(m)},\phi^{(n)}))}_{L^2(\MM)}.
\eea
Thus it suffices to prove for all $m, n\geq 0$:
\bea\lab{domo1}
\norm{(-\Delta)^{-\frac12}(Q_{ij}(\phi^{(m)},\phi^{(n)}))}_{L^2(\MM)}\lesssim (M\ep)^{m+1}(M\ep)^{n+1}.
\eea
The estimates in \eqref{domo1} are analogous for all $m, n$, so it suffices to prove \eqref{domo1} in the case $(m,n)=(0,0)$. In view of the definition of $\phi^{(0)}$, the estimates for $B$ on the initial slice $\Si_0$, estimate \eqref{spacetimeB} for $\pr \square B$, and the definition of $\Psi_{om}$ and $\Psi(t,s)$, \eqref{domo1} reduces to the following bilinear estimate for half wave parametrices:
\bea
\left\|(-\Delta)^{-\frac12}\left(Q_{ij}(\phi_{f_1}, \phi_{f_2}\right)\right\|_{L^2(\MM)}&\les& \norm{\la f_1}_{L^2(\RRR^3)}\norm{\la f_2}_{L^2(\RRR^3)} \lab{domo1bis}
\eea
with,
$$\phi_f =\int_{\SSS^2} \int_0^\infty  e^{i \la \uom(t,x)}f(\la\om)\la^2d\la d\om.$$
We then decompose  $f_1, f_2$  with respect to frequency  and reduce the desired estimate to 
 $L^4(\MM)$ Strichartz estimate localized in frequency of Proposition \ref{prop:L4strichartz},  see details in \cite{KRS}. 
 
 This concludes the proof of Proposition \ref{prop:improve2}.

%%%%%%%%%%%%%%%%%%%%%%%%%%%%%%%%%%%%%%%

\chapter{Control of the error term}\lab{part:paramtime}

%% Control of the error term

\renewcommand{\be}{\begin{equation}}
\renewcommand{\ee}{\end{equation}}

\newcommand{\trt}{\textrm{tr}\theta}
\newcommand{\uo}{u(0,x,\o)}
\newcommand{\dd}{{\bf D}}

\renewcommand{\R}{\mathbb{R}}
\renewcommand{\N}{\mathbb{N}}
\renewcommand{\D}{\mathcal{D}}
\newcommand{\rf}{\mathcal{R}}
\renewcommand{\II}{\mathcal{I}_0}
\newcommand{\no}{\mathcal{N}_1}
\newcommand{\noo}{\mathcal{N}_2}
\newcommand{\dcal}{\mathcal{D}_1}
\newcommand{\dcall}{\mathcal{D}_2}
\newcommand{\dcalll}{{}^*\mathcal{D}_1}
\newcommand{\dcallll}{{}^*\mathcal{D}_2}

\newcommand{\rr}{{\bf R}}
\renewcommand{\gg}{{\bf g}}

\renewcommand{\th}{\theta}
\renewcommand{\ep}{\varepsilon}
\newcommand{\hth}{\widehat{\theta}}
\newcommand{\muu}{\mu_u}
\renewcommand{\p}{P_u}
\newcommand{\gn}{{\bf g}(N,N')}
\newcommand{\gl}{{\bf g}(L,L')}
\newcommand{\h}[1]{H^{#1}(\R^3)}
\renewcommand{\l}[2]{L^{#1}_uL^{#2}(\mathcal{H}_u)}
\newcommand{\lprime}[2]{L^{#1}_{u'}L^{#2}(\mathcal{H}_{u'})}
\renewcommand{\ll}[1]{L^{#1}(\mathcal{M})}
\renewcommand{\le}[1]{L^{#1}(\R^3)}
\newcommand{\lp}[1]{L^{#1}(P_u)}
\renewcommand{\lap}{\mbox{$\Delta \mkern-13mu /$\,}}
\newcommand{\lapa}{a^{-1}\mbox{$\Delta \mkern-13mu /$\,($a$)}}
\newcommand{\divb}{\mbox{$\textrm{div} \mkern-13mu /$\,}}
\newcommand{\curlb}{\mbox{$\textrm{curl} \mkern-13mu /$\,}}
\newcommand{\ddb}{{\bf \nab} \mkern-13mu /\,}
\newcommand{\ana}{\mbox{$a^{-1}\nabla \mkern-13mu /\,a$\,}}
\newcommand{\nabn}{\nabla_N}
\renewcommand{\lg}{\log(a)}
\renewcommand{\a}{\alpha}
\renewcommand{\b}{\beta}
\renewcommand{\de}{\delta}

%%%%%%%%%%%%%%%%%

\newcommand{\kep}{\epsilon}
\renewcommand{\Si}{\Sigma}
\newcommand{\Sit}{\Sigma_t}
\renewcommand{\pr}{\partial}
\newcommand{\dmt}{d\mu_{t,u}}
\newcommand{\dmo}{d\mu_{0,u}}
\newcommand{\ptu}{P_{t,u}}
\newcommand{\pou}{P_{0,u}}
\newcommand{\gax}{\ga_{x'}}
\newcommand{\half}{\frac{1}{2}}

\renewcommand{\o}{\omega}
\renewcommand{\S}{\mathbb{S}^2}
\newcommand{\po}{\partial_{\omega}}
\newcommand{\xo}{x\cdot\omega}

\renewcommand{\H}{\mathcal{H}}
\newcommand{\li}[2]{L^{#1}_uL^{#2}(\mathcal{H}_u)}
\newcommand{\lh}[1]{L^{#1}(\mathcal{H}_u)}
\newcommand{\tx}[2]{L^{#1}_tL^{#2}_{x'}}
\newcommand{\xt}[2]{L^{#1}_{x'}L^{#2}_t}
\newcommand{\lpt}[1]{L^{#1}(P_{t,u})}
\newcommand{\lpo}[1]{L^{#1}(P_{0,u})}
\renewcommand{\lsit}[2]{L^{#1}_tL^{#2}(\Sit)}
\newcommand{\ux}[2]{L^{#1}_uL^{#2}_{x'}}

\renewcommand{\a}{\alpha}
\renewcommand{\b}{\beta}
\newcommand{\ab}{\underline{\alpha}}
\renewcommand{\bb}{\underline{\beta}}
\renewcommand{\ga}{\gamma}
\renewcommand{\d}{\delta}
\renewcommand{\r}{\rho}
\newcommand{\s}{\sigma}
\newcommand{\kb}{\overline{k}}
\newcommand{\etah}{\widehat{\eta}}
\newcommand{\kepb}{\overline{\epsilon}}
\newcommand{\db}{\overline{\delta}}

\newcommand{\z}{\zeta}
\newcommand{\zb}{\underline{\zeta}}
\newcommand{\trc}{\textrm{tr}\chi}
\newcommand{\hch}{\widehat{\chi}}
\newcommand{\chb}{\underline{\chi}}
\renewcommand{\trchb}{\textrm{tr}{\underline{\chi}}}
\newcommand{\hchb}{\widehat{\underline{\chi}}}
\renewcommand{\xib}{\underline{\xi}}
\newcommand{\het}{\widehat{\eta}}

In this chapter, we consider the Fourier integral operator $E$ given by \eqref{error-param} in which corresponds to the error term of a plane wave type parametrix. Recall that $E$ is given by:
$$Ef(t,x)=\int_{\S}\int_{0}^{+\infty}e^{i\lambda u(t,x,\o)}\square_{\bf g}u(t,x,\o)f(\lambda\o)\lambda^3 d\lambda d\o,\,(t,x)\in\mathcal{M},$$
where $u(.,.,\o)$ is a solution to the eikonal equation ${\bf g}^{\alpha\beta}\partial_\alpha u\partial_\beta u=0$ on $\mathcal{M}$ such that $u(0,x,\o)\sim \xo$ when $|x|\rightarrow +\infty$ on $\Si_0$ (see section \ref{ch3:sec:defu}). 
The goal of this chapter is to outline the main ideas allowing us to obtain the control for the error term $E$ in \cite{param4}.

\section{Geometric set-up and main results}

\subsection{Geometry of the foliation of $\mathcal{M}$ by $u$}\lab{ch3:sec:defu}

Recall that $u$ is a solution to the eikonal equation $\gg^{\alpha\beta}\partial_\alpha u\partial_\beta u=0$ on $\mathcal{M}$ depending on a extra parameter $\o\in \S$.  The level hypersufaces $u(t,x,\o)=u$  of the optical function $u$ are denoted by  $\H_u$. Let $L'$ denote the space-time gradient of $u$, i.e.:
\be\lab{ch3:def:L'0}
L'=-\gg^{\a\b}\pr_\b u \pr_\a.
\ee
Using the fact that $u$ satisfies the eikonal equation, we obtain:
\be\lab{ch3:def:L'1}
\dd_{L'}L'=0,
\ee
which implies that $L'$ is the geodesic null generator of $\H_u$.

We foliate the space-time $\mathcal{M}$ by space-like hypersurfaces $\Sit$ defined as level hypersurfaces of a time function $t$ and we denote by  $T$ the unit, future oriented, normal to $\Si_t$. We have: 
$$T(u)=\pm |\nab u|$$
where $|\nab u|^2=\sum_{i=1}^3|e_i(u)|^2$ relative to an orthonormal frame $e_i$ on $\Si_t$. Since the sign of $T(u)$ is irrelevant, we choose by convention:
\be\lab{ch3:it1'}
T(u)=|\nab u|.
\end{equation}
We denote by $P_{t,u}$  the surfaces of intersection
between $\Si_t$ and  $\H_u$. 
\begin{definition}[\textit{Canonical null pair}]
 \be\lab{ch3:it2}
L=bL'=T+N, \qquad \lb=2T-L=T-N
\end{equation}
where $L'$ is the space-time gradient of $u$ \eqref{ch3:def:L'0}, $b$  is  the  \textit{lapse of the null foliation} (or shortly null lapse)
\be\lab{ch3:it3}
b^{-1}=-<L', T>=T(u),
\end{equation} 
and $N$ is a unit normal, along $\Si_t$, to the surfaces $P_{t,u}$. Since $u$ satisfies the 
eikonal equation $\gg^{\alpha\beta}\partial_\alpha u\partial_\beta u=0$ on $\mathcal{M}$, 
this yields $L'(u)=0$ and thus $L(u)=0$. In view of the definition of $L$ and \eqref{ch3:it1'}, 
we obtain:
\be\lab{ch3:it3bis}
N=-\frac{\nabla u}{|\nabla u|}.
\end{equation} 
\label{ch3:def:nulllapse}
\end{definition}

\begin{definition} A  null frame   $e_1,e_2,e_3,e_4$ at a point $p\in P_{t,u}$ 
  consists,
in addition to the null pair     $e_3=\lb,
e_4=L$, of {\sl  arbitrary  orthonormal}  vectors  $e_1,e_2$ tangent
to $P_{t,u}$. 
%All the   estimates in this paper  are in fact local and
% independent of the
% choice of a particular frame. We do not need to worry
% that these  frames cannot be  globally defined.
\end{definition}
\begin{definition}[\textit{Second fundamental form}]

 Let  $e_1,e_2,e_3,e_4$ be a null frame on
$P_{t,u}$ as above.   The second fundamental form on  $P_{t,u}$ associated to our canonical null pair is given by 
$$\chi_{AB}=<\dd_A e_4,e_B>.$$

We decompose $\chi$ into its  trace and traceless component.
$$\trc = \gg^{AB}\chi_{AB},\,\,\hch_{AB}=\chi_{AB}-\half \trc \gg_{AB}.$$
\end{definition}
Recall that $\trc$ satisfies a transport equation called the Raychaudhuri equation:
\be\lab{ch3:raychaudhuri}
L(\trc) + \half (\trc)^2 = - |\hch|^2 +\cdots
\ee
(see precise equation in \eqref{ch4:D4trchi}).
%We also recall the transport equation satisfied by the null lapse $b$:
%\be\lab{ch3:D4a}
%L(b) = -\db b.
%\ee

We conclude this section with the identification of the symbol $\square_{\bf g}u$ of the error term. We have (see for example \cite{param3} for a proof):
\begin{equation}\label{ch3:symbolE}
\square_{\bf g}u=b^{-1}\trc.
\end{equation} 
Thus, we may rewrite the error term $E$ as:
\be\lab{ch3:err1} 
Ef(t,x)=\int_{\S}\int_{0}^{+\infty}e^{i\lambda u(t,x,\o)}b^{-1}(t,x,\o)\trc(t,x,\o)f(\lambda\o)\lambda^3 d\lambda d\o. 
\ee

\subsection{Some norms}

We define some norms on $\HH$. For any $1\leq p\leq +\infty$ and for any tensor $F$ on $\HH_u$, we have:
$$\norm{F}_{\lh{p}}=\left(\int_0^1dt\int_{P_{t,u}}|F|^p\dmt\right)^{\frac{1}{p}},$$
where $\dmt$ denotes the area element of $P_{t,u}$.

Let $x'$ a coordinate system on $\pou$. By transporting this coordinate system along the null geodesics generated by $L$, we obtain a coordinate system $(t,x')$ of $\H$. We define the following norms:
$$\norm{F}_{\xt{\infty}{2}}=\sup_{x'\in P_{0,u}}\left(\int_0^1 |F(t,x')|^2dt\right)^{\half},$$
$$\norm{F}_{\xt{2}{\infty}}=\normm{\sup_{0\leq t\leq 1}|F(t,x'))|}_{L^2(\pou)}.$$

\subsection{Estimates for the space-time foliation}\lab{ch3:sec:estneeded}

In this section, we collect the estimates that are needed to follow the discussion of the control of the error term contained in this chapter. An outline of the proof of these estimates will be given in Chapter \ref{part:spacetimeu} (see \cite{param3} for the complete proof).

We start with the regularity in $(t,x)$ of the lapse $b$ and the second fundamental for $\chi$. We need:
\be\lab{ch3:estneeded1}
\norm{\trc}_{L^\infty(\MM)}+\norm{\nabla\trc}_{\xt{\infty}{2}}+\norm{b-1}_{L^\infty(\MM)}+\norm{\nabla b}_{\l{\infty}{2}}+\norm{\hch}_{\xt{\infty}{2}}+\norm{\nabb\hch}_{\li{\infty}{2}}\les\ep.
\ee

\begin{remark}
In this section, all estimates hold for any $\o\in\S$ with the constant in the right-hand side being independent of $\o$. Thus, one may take the supremum in $\o$ everywhere. To ease the notations, we do not explicitly write down this supremum. 
\end{remark}

We also need an estimate for two derivatives of $\trc$ with respect to $\nabn$
\be\lab{ch3:estneeded2}
\normm{\nabn P_j(\nabn\trc)}_{L^2(\H_u)}\les \ep 2^j+ 2^{\frac{j}{2}}\mu(u),
\ee
where $\mu$ in a function satisfying:
$$\norm{\mu}_{L^2(\R)}\les\ep.$$
Next, we consider the regularity with respect to $\o$. We have:
\be\lab{ch3:estricciomega}
\norm{\po b}_{L^\infty(\MM)}\lesssim \ep,
\ee
\begin{equation}\label{ch3:threomega1ter}
|N(t,x,\o)-N(t,x,\o')|\simeq |\o-\o'|,\,\forall (t,x)\in\MM, \o,\o'\in\S,
\end{equation}
and 
\be\lab{ch3:estNomega}
\norm{\po N}_{L^\infty(\MM)}\lesssim 1.
\ee 
Furthermore, we have the following decomposition for $\hch$:
\be\lab{ch3:dechch}
\hch=\chi_1+\chi_2,
\ee
where $\chi_1$ and $\chi_2$ are two symmetric traceless $\ptu$-tangent 2-tensors satisfying in particular, for any $2\leq p<+\infty$:
\be\lab{ch3:dechch2}
\norm{\chi_1}_{\tx{p}{\infty}}+\norm{\po\chi_2}_{\lh{6_-}}\lesssim \ep.
\ee

\begin{remark}
The point of decomposition \eqref{ch3:dechch} is that $\chi_1$ has a better regularity with respect to $(t, x)$ than $\hch$, while $\chi_2$ has a better regularity with respect to $\o$ than $\hch$ (see explanation in section \ref{ch4:sec:firstomegader}).
\end{remark}

Finally, we need to compare quantities evaluated at two angles $\o$ and $\nu$ in $\S$ satisfying $|\o-\nu|\les 2^{-\frac{j}{2}}$. We have the following decomposition for $N(t,x,\o)-N(t,x,\nu)$:
\be\lab{ch3:decNom}
2^{\frac{j}{2}}(N(t,x,\o)-N(t,x,\nu))=F^j_1(t,x,\nu)+F^j_2(t,x,\o,\nu)
\ee
where the tensor $F^j_1$ does not depend on $\o$ and satisfies:
$$\norm{F^j_1}_{L^\infty}\les 1,$$
and the tensor $F^j_2$ satisfies:
$$\norm{F^j_2}_{L^\infty_u\lh{2}}\les 2^{-\frac{j}{2}}.$$
Here $L^\infty_u\lh{2}$ is defined with respect to $u=u(t,x,\o)$. 
We also have the following decomposition for $\trc$:
\be\lab{ch3:dectrcom}
\trc(t,x,\o)=f^j_1(t,x,\nu)+f^j_2(t,x,\o, \nu)
\ee
where the scalar $f^j_1$ does not depend on $\o$ and satisfies:
$$\norm{f^j_1}_{L^\infty}\les \ep,$$
and where the scalar $f^j_2$ satisfies:
$$\norm{f^j_2}_{L^\infty_u\lh{2}}\les \ep 2^{-\frac{j}{2}}.$$

\subsection{Main result}

The main result of this chapter is the following. 
\begin{theorem}\label{ch3:part3:th1}
Let $u$ be a function on $\MM\times\S$ satisfying suitable assumptions (we refer to \cite{param4} for the complete set of assumptions, and to section \ref{ch3:sec:estneeded} for some typical assumptions). Let $E$ the Fourier integral operator with phase $u(t,x,\o)$ and symbol $\square_{\bf g}u$:
\be\lab{ch3:err2} 
Ef(t,x)=\int_{\S}\int_{0}^{+\infty}e^{i\lambda u(t,x,\o)}b^{-1}(t,x,\o)\trc(t,x,\o)f(\lambda\o)\lambda^3 d\lambda d\o. 
\ee
Then, $E$ satisfies the estimate:
\begin{equation}\label{ch3:l2}
\norm{Ef}_{L^2(\mathcal{M})}\lesssim \ep\norm{\lambda f}_{L^2(\mathbb{R}^3)}.
\end{equation}
\end{theorem}

\subsection{Geometric Littlewood-Paley projections on the 2-surfaces $P_{t,u}$}\label{sec:LP}

Throughout the paper, we will use the geometric Littlewood-Paley projections on 2-surfaces ($\ptu$ in our case) constructed in \cite{Kl-R6}. In that paper, the following properties are proved
\begin{theorem}\label{ch3:thm:LP}
 The LP-projections $P_j$ verify the following
 properties:

i)\quad {\sl $L^p$-boundedness} \quad For any $1\leq
p\leq \infty$, and any interval $I\subset \Bbb Z$,
\be\lab{ch3:eq:pdf1}
\|P_IF\|_{\lpt{p}}\lesssim \|F\|_{\lpt{p}}
\end{equation}

ii) \quad  {\sl Bessel inequality} 
$$\sum_j\|P_j F\|_{\lpt{2}}^2\lesssim \|F\|_{\lpt{2}}^2$$

iii)\quad {\sl Finite band property}\quad For any $1\leq p\leq \infty$.
\begin{equation}
\begin{array}{lll}
\|\lap P_j F\|_{\lpt{p}}&\lesssim & 2^{2j} \|F\|_{\lpt{p}}\\
\|P_jF\|_{\lpt{p}} &\lesssim & 2^{-2j} \|\lap F \|_{\lpt{p}}.
\end{array}
\end{equation}

In addition, the $L^2$ estimates
\begin{equation}
\begin{array}{lll}
\|\nabb P_j F\|_{\lpt{2}}&\lesssim & 2^{j} \|F\|_{\lpt{2}}\\
\|P_jF\|_{\lpt{2}} &\lesssim & 2^{-j} \|\nabb F  \|_{\lpt{2}}
\end{array}
\end{equation}
hold together with the dual estimate
$$\| P_j \nabb F\|_{\lpt{2}}\lesssim 2^j \|F\|_{\lpt{2}}$$

iv) \quad{\sl Weak Bernstein inequality}\quad For any $2\leq p<\infty$
\begin{align*}
&\|P_j F\|_{\lpt{p}}\lesssim (2^{(1-\frac 2p)j}+1) \|F\|_{\lpt{2}},\\
&\|P_{<0} F\|_{\lpt{p}}\lesssim \|F\|_{\lpt{2}}
\end{align*}
together with the dual estimates 
\begin{align*}
&\|P_j F\|_{\lpt{2}}\lesssim (2^{(1-\frac 2p)j}+1) \|F\|_{\lpt{p'}},\\
&\|P_{<0} F\|_{\lpt{2}}\lesssim \|F\|_{\lpt{p'}}
\end{align*}
\end{theorem}

\section{Control of the error term}

\subsection{The basic computation}

We start the proof of Theorem \ref{ch3:part3:th1} with the following instructive computation:
\begin{equation}\label{ch3:bisb5}
\begin{array}{ll}
\ds\norm{Ef}_{\ll{2}}\leq &\ds\int_{\S}\normm{b(t,x,\o)^{-1}\trc(t,x,\o)\int_{0}^{+\infty}e^{i\lambda u}f(\lambda\o)\lambda^2 d\lambda}_{\ll{2}}d\o\\
&\ds\leq\int_{\S}\norm{b(t,x,\o)^{-1}\trc(t,x,\o)}_{\li{\infty}{2}}\normm{\int_{0}^{+\infty}e^{i\lambda u}f(\lambda\o)\lambda^2 d\lambda}_{L^2_{u}}d\o\\
&\ds\leq \ep\norm{\lambda^2f}_{L^2(\R^3)},
\end{array}
\end{equation}
where we have used Plancherel with respect to $\la$, Cauchy-Schwarz with respect to $\o$, the estimates \eqref{ch3:estneeded1} for $b$ and $\trc$. \eqref{ch3:bisb5} misses the conclusion \eqref{ch3:l2} of Theorem \ref{ch3:part3:th1} by a power of $\lambda$. Now, assume for a moment that we may replace a power of $\lambda$ by a derivative on $b(t,x,\o)^{-1}\trc(t,x,\o)$. Then, the same computation yields:
\begin{equation}\label{ch3:bisb6}
\begin{array}{ll}
&\ds\normm{\int_{\S}\int_{0}^{+\infty}\nabla(b(t,x,\o)^{-1}\trc(t,x,\o)) e^{i\lambda u}f(\lambda\o)\lambda d\lambda d\o}_{\ll{2}}\\
\ds\leq &\ds \int_{\S}\norm{\nabla(b(t,x,\o)^{-1}\trc(t,x,\o))}_{\li{\infty}{2}}\normm{\int_{0}^{+\infty}e^{i\lambda u}f(\lambda\o)\lambda^2 d\lambda}_{L^2_{ u}}d\o\\
\ds\leq &\ds \ep\norm{\la f}_{L^2(\R^3)},
\end{array}
\end{equation}
where we used the fact that 
\be\lab{ch3:etoile}
\norm{\nabla(b(t,x,\o)^{-1}\trc(t,x,\o))}_{\li{\infty}{2}}\les\ep
\ee 
in view of \eqref{ch3:estneeded1}. Note that the estimate provided by \eqref{ch3:bisb6} is consistent with the control of the error term \eqref{ch3:l2}. This suggests a strategy which consists in making integrations by parts to trade powers of $\lambda$ against derivatives of the symbol $b(t,x,\o)^{-1}\trc(t,x,\o)$. 

\subsection{Structure of the proof of Theorem \ref{ch3:part3:th1}}\lab{ch3:sec:structureproof}

The proof of Theorem \ref{ch3:part3:th1} proceeds in three steps. We first localize in frequencies of size $\la\sim 2^j$. We then localize the angle $\o$ in patches on the sphere $\S$ of diameter $2^{-j/2}$. Finally, we estimate the diagonal terms.

\subsubsection{Step 1: decomposition in frequency}
 
For the first step, we introduce $\varphi$  and $\psi$ two smooth compactly supported functions on $\R$ such that: 
\begin{equation}\label{ch3:bisb7}
\varphi(\lambda)+\sum_{j\geq 0}\psi(2^{-j}\lambda)=1\textrm{ for all }\lambda\in\R.
\end{equation}
We use \eqref{ch3:bisb7} to decompose $Ef$ as follows:
\begin{equation}\label{ch3:bisb8}
Ef(t,x)=\sum_{j\geq -1}E_jf(t,x),
\end{equation}
where for $j\geq 0$:
\begin{equation}\label{ch3:bisb9}
E_jf(t,x)=\int_{\S}\int_{0}^{+\infty}e^{i\lambda u}b(t,x,\o)^{-1}\trc(t,x,\o)\psi(2^{-j}\lambda)f(\lambda\o)\lambda^2 d\lambda d\o,
\end{equation}
and 
\begin{equation}\label{ch3:bisb10}
E_{-1}f(t,x)=\int_{\S}\int_{0}^{+\infty}e^{i\lambda u}b(t,x,\o)^{-1}\trc(t,x,\o)\varphi(\lambda)f(\lambda\o)\lambda^2 d\lambda d\o.
\end{equation}
This decomposition is classical and is known as the first dyadic decomposition (see \cite{St}). The goal of this first step is to prove the following proposition:
\begin{proposition}\label{ch3:bisorthofreq}
The decomposition \eqref{ch3:bisb8} satisfies an almost orthogonality property, from which it follows that:
\begin{equation}\label{ch3:bisorthofreq1}
\norm{Ef}_{\ll{2}}^2\lesssim\sum_{j\geq -1}\norm{E_jf}_{\ll{2}}^2+\ep^2\norm{f}^2_{L^2(\R^3)}.
\end{equation}
\end{proposition}
A discussion of the proof of Proposition \ref{ch3:bisorthofreq} is postponed to section \ref{ch3:bissec:orthofreq}. 

\subsubsection{Step 2: decomposition in angle}

Proposition \ref{ch3:bisorthofreq} enables us to  estimate $\norm{E_jf}_{\ll{2}}$ instead of 
$\norm{Ef}_{\ll{2}}$. The analog of computation \eqref{ch3:bisb5} for $\norm{E_jf}_{\ll{2}}$ yields:
\begin{equation}\label{ch3:bisb5bis}
\begin{array}{l}
\ds\norm{E_jf}_{\ll{2}}\leq \ep\norm{\lambda\psi(2^{-j}\la) f}_{L^2(\s)}\lesssim \ep2^j\norm{\psi(2^{-j}\la) f}_{\le{2}},
\end{array}
\end{equation}
which misses the wanted estimate by a power of $2^j$. We thus need to perform a second dyadic decomposition (see \cite{St}). We introduce a smooth partition of unity on the sphere $\S$:
\begin{equation}\label{ch3:bisb14}
\sum_{\nu\in\Gamma}\eta^\nu_j(\o)=1\textrm{ for all }\o\in\S, 
\end{equation}
where $\Gamma$ is a lattice on $\S$ of size $2^{-\frac{j}{2}}$, where the support of $\eta^\nu_j$ is a patch on $\S$ of diameter $\sim 2^{-j/2}$. We use \eqref{ch3:bisb14} to decompose $E_jf$ as follows:
\begin{equation}\label{ch3:bisb15}
E_jf(t,x)=\sum_{\nu\in\Gamma}E^\nu_jf(t,x),
\end{equation}
where:
\begin{equation}\label{ch3:bisb16}
E^\nu_jf(t,x)=\int_{\S}\int_{0}^{+\infty}e^{i\lambda u}b(t,x,\o)^{-1}\trc(t,x,\o)\psi(2^{-j}\lambda)\eta^\nu_j(\o)f(\lambda\o)\lambda^2 d\lambda d\o.
\end{equation}
We also define:
\begin{equation}\label{ch3:bisdecf}
\begin{array}{l}
\ds\ga_{-1}=\norm{\varphi(\la)f}_{\le{2}},\, \ga_j=\norm{\psi(2^{-j}\la)f}_{\le{2}},\,j\geq 0, \\
\ds\ga^\nu_j=\norm{\psi(2^{-j}\la)\eta^\nu_j(\o)f}_{\le{2}},\,j\geq 0,\,\nu\in\Gamma, 
\end{array}
\end{equation}
which satisfy:
\begin{equation}\label{ch3:bisdecf1}
\norm{f}_{\le{2}}^2=\sum_{j\geq -1}\ga_j^2=\sum_{j\geq -1}\sum_{\nu\in\Gamma}(\ga^\nu_j)^2.
\end{equation}
The goal of this second step is to prove the following proposition:
\begin{proposition}\label{ch3:bisorthoangle}
The decomposition \eqref{ch3:bisb15} satisfies an almost orthogonality property, from which it follows that
\begin{equation}\label{ch3:bisorthoangle1}
\norm{E_jf}_{\ll{2}}^2\lesssim\sum_{\nu\in\Gamma}\norm{E^\nu_jf}_{\ll{2}}^2+\ep^2\ga_j^2.
\end{equation}
\end{proposition}
A discussion of the proof of Proposition \ref{ch3:bisorthoangle} is postponed to section \ref{ch3:bissec:orthoangle}. 

\subsubsection{Step 3: control of the diagonal term}

Proposition \ref{ch3:bisorthoangle} allows us to  estimate $\norm{E^\nu_jf}_{\ll{2}}$ instead of $\norm{E_jf}_{\ll{2}}$. The analog of computation \eqref{ch3:bisb5} for $\norm{E^\nu_jf}_{\ll{2}}$ yields:
\bea\label{ch3:bisb5ter}
&&\norm{E_j^\nu f}_{\ll{2}}\\
\nn&\leq &\ds\int_{\S}\norm{b(t,x,\o)^{-1}\trc(t,x,\o)}_{\l{\infty}{2}}\normm{\int_{0}^{+\infty}e^{i\lambda  u}\psi(2^{-j}\la)\eta^\nu_j(\o)f(\lambda\o)\lambda^2 d\lambda}_{L^2_{ u}}d\o\\
\nn&\leq &\ds 2^j\ep\sqrt{\textrm{vol}(\textrm{supp}(\eta^\nu_j))}\norm{\lambda\psi(2^{-j}\la)\eta^\nu_j(\o)f}_{\le{2}}\\
\nn&\lesssim &\ds \ep2^{j/2}\gamma^\nu_j,
\eea
where the term $\sqrt{\textrm{vol}(\textrm{supp}(\eta^\nu_j))}$ comes from the fact that we apply Cauchy-Schwarz in $\o$. Note that we have used in \eqref{ch3:bisb5ter} the fact that the support of $\eta^\nu_j$ is 2 dimensional and has diameter $2^{-j/2}$ so that:
\begin{equation}\label{ch3:bissuppangle}
\sqrt{\textrm{vol}(\textrm{supp}(\eta^\nu_j))}\lesssim 2^{-j/2}.
\end{equation}
Now, \eqref{ch3:bisb5ter} still misses the wanted estimate by a power of $2^{j/2}$. Nevertheless, using more refined techniques, we are able to estimate the diagonal term:
\begin{proposition}\label{ch3:bisdiiiiagonal}
The diagonal term $E^\nu_jf$ satisfies the following estimate:
\begin{equation}\label{ch3:bisdiiiiagonal1}
\norm{E^\nu_jf}_{\ll{2}}\lesssim \ep\ga^\nu_j.
\end{equation}
\end{proposition}
A discussion of the proof of Proposition \ref{ch3:bisdiiiiagonal} is postponed to section \ref{ch3:bissec:diagonal}. 

\begin{remark}
Note that Proposition \ref{ch3:bisorthoangle} together with Proposition \ref{ch3:bisdiiiiagonal} yields the estimate:
\begin{equation}\label{ch3:bisof13}
\norm{E_jf}_{\ll{2}}\lesssim \ep\ga_j.
\end{equation}
Now, since the proof of Proposition \ref{ch3:bisorthoangle} and the proof of Proposition \ref{ch3:bisdiiiiagonal} do not depend on the proof of Proposition \ref{ch3:bisorthofreq}, we are allowed to use the conclusion of Proposition \ref{ch3:bisorthoangle} and Proposition \ref{ch3:bisdiiiiagonal} in the proof of Proposition \ref{ch3:bisorthofreq}. In particular, the estimate \eqref{ch3:bisof13} will be used for the proof of Proposition \ref{ch3:bisorthofreq}. In the same spirit, since the proof of Proposition \ref{ch3:bisdiiiiagonal} does not depend on the proof of Proposition \ref{ch3:bisorthoangle}, we are allowed to use the conclusion of Proposition \ref{ch3:bisdiiiiagonal} in the proof of Proposition \ref{ch3:bisorthoangle}.
\end{remark}

\noindent{\bf Convention.} In the rest of this chapter, we will use several integration by parts. In turn, these integration by parts will each generate a large number of terms. For the sake of simplicity, we will only discuss few typical terms. We will constantly use the notation "$+\cdots$" in various identities and estimates in order to refer to the additional terms. That is not to say that these additional terms are lower order or estimated in the same way, but simply that the typical terms that we exhibit allow for a simple exposition of the main ideas of the proof. We refer the reader to \cite{param4} for a complete proof which contains the control of the typical terms discussed here as well as the numerous additional terms.\\

\subsubsection{Proof of Theorem \ref{ch3:part3:th1}}

Proposition \ref{ch3:bisorthofreq}, \ref{ch3:bisorthoangle} and \ref{ch3:bisdiiiiagonal} 
immediately yield the proof of Theorem \ref{ch3:part3:th1}. Indeed, \eqref{ch3:bisorthofreq1}, \eqref{ch3:bisdecf1}, \eqref{ch3:bisorthoangle1} and \eqref{ch3:bisdiiiiagonal1} imply:
\begin{equation}\label{ch3:biscclth3}
\begin{array}{ll}
\ds\norm{Ef}_{\ll{2}}^2 & \ds\lesssim\sum_{j\geq -1}\norm{E_jf}_{\ll{2}}^2+\ep^2\norm{f}^2_{\le{2}}\\
& \ds\lesssim\sum_{j\geq -1}\sum_{\nu\in\Gamma}\norm{E^\nu_jf}_{\ll{2}}^2+\ep^2\sum_{j\geq -1}\gamma_j^2+\ep^2\norm{f}^2_{\le{2}}\\
& \ds\lesssim \ep^2\sum_{j\geq -1}\sum_{\nu\in\Gamma}(\gamma_j^\nu)^2+\ep^2\sum_{j\geq -1}\gamma_j^2+\ep^2\norm{f}^2_{\le{2}}\\
& \ds\lesssim \ep^2\norm{f}^2_{\le{2}},
\end{array}
\end{equation}
which is the conclusion of Theorem \ref{ch3:part3:th1}. 

The rest of this chapter is dedicated to a discussion of the proof of Propositions \ref{ch3:bisorthofreq}, \ref{ch3:bisorthoangle} and \ref{ch3:bisdiiiiagonal}. The details of the proofs being very involved, we only give a very sketchy summary of the main ideas. We refer the reader to \cite{param4} for the details.

\section{Almost orthogonality in frequency}\lab{ch3:bissec:orthofreq}

We have to prove \eqref{ch3:bisorthofreq1}:
\begin{equation}\label{ch3:bisof1}
\norm{Ef}_{\ll{2}}^2\lesssim\sum_{j\geq -1}\norm{E_jf}_{\ll{2}}^2+\ep^2\norm{f}^2_{\le{2}}.
\end{equation}
This will result from the following inequality using Shur's Lemma:
\begin{equation}\label{ch3:bisof2}
\left|\int_{\MM}E_jf(t,x)\overline{E_kf(t,x)}d\MM \right| \lesssim \ep^22^{-\frac{|j-k|}{4}}\gamma_j\ga_k\textrm{ for }|j-k|> 2.
\end{equation}
In turn, \eqref{ch3:bisof2} will follow from integrations by parts in $u$.

\subsection{A first integration by parts}\label{ch3:sec:firstibpu} 

From now on, we focus on proving \eqref{ch3:bisof2}. We may assume $j\geq k+3$. We have:
\bea\label{ch3:bisof3}
&&\ds\int_{\MM}E_jf(t,x)\overline{E_kf(t,x)}d\MM \\
\nn&= & \ds\int_{\S}\int_{0}^{+\infty}\int_{\S}\int_{0}^{+\infty}\left(\int_{\MM}e^{i\lambda u-i\la' u'}b(t,x,\o)^{-1}\trc(t,x,\o)\overline{b(x,\o')^{-1}\trc(t,x,\o')}d\MM\right)\\
\nn&& \ds\times\psi(2^{-j}\lambda)f(\lambda\o)\lambda^2 \psi(2^{-k}\lambda')\overline{f(\lambda'\o')}(\lambda')^2 d\lambda d\o d\lambda' d\o'.
\eea

We consider the coordinate system $(t,u,x')$ on $\MM$, and we would like to integrate by parts with respect to $\pr_u$ in \eqref{ch3:bisof3}. Since $\nabla u=b^{-1}N$ and $\nabla u'={b'}^{-1}N'$, we have:
\begin{equation}\label{ch3:bisof4}
e^{i\lambda u-i\la' u'}=-\frac{i}{\la-\la'\frac{b}{b'}g(N,N')}\partial_u(e^{i\lambda u-i\la' u'}),
\end{equation}
where we use the notation $u$ for $u(t,x,\o)$, $b$ for $b(t,x,\o)$, $N$ for $N(t,x,\o)$, $u'$ for $u(t,x,\o')$, $b'$ for $b(t,x,\o')$ and $N'$ for $N(x,\o')$. We will also use the notation $\trc$ for $\trc(t,x,\o)$ and $\trc'$ for $\trc(t,x,\o')$. Using \eqref{ch3:bisof4}, we obtain:
\begin{equation}\label{ch3:bisof5}
\begin{array}{ll}
\ds\int_{\MM}e^{i\lambda u-i\la' u'}b\overline{b'}d\MM = & \ds i\int_{\MM}e^{i\lambda u-i\la' u'}\frac{b^{-1}\partial_{ u}\trc\overline{{b'}^{-1}\trc'}}{\la-\la'\frac{b}{b'}g(N,N')}d\MM\\
&\ds +i\int_{\MM}e^{i\lambda u-i\la' u'}\frac{b^{-1}\trc\partial_{ u}(\overline{{b'}^{-1}\trc'})}{\la-\la'\frac{b}{b'}g(N,N')}d\MM+\cdots,
\end{array}
\end{equation}
where the additional terms in \eqref{ch3:bisof5} arise when $\pr_u$ falls on the volume element of $\MM$ or on the denominator in the right-hand side of \eqref{ch3:bisof4}. Note that:
$$\left|\frac{\la'}{\la}\frac{b}{b'}g(N,N')\right|\leq \frac{\la'}{\la}\left|\frac{b}{b'}\right|\leq \frac{1}{2}+O(\ep)<1,$$
where we used the estimate \eqref{ch3:estneeded1} satisfied by $b$ and $b'$ and the fact that $j\geq k+3$ so that $\la'\leq \la/2$. Thus, we may expand the fraction in \eqref{ch3:bisof5}:
\begin{equation}\label{ch3:bisof6}
\frac{1}{\la-\la'\frac{b}{b'}g(N,N')}=\sum_{p\geq 0}\frac{1}{\la}\left(\frac{\la'\frac{b}{b'}g(N,N')}{\la}\right)^p.
\end{equation}

\begin{remark}\label{ch3:rem:simplification}
The expansion \eqref{ch3:bisof6} generates quantities of the type
$$\int_{0}^{+\infty}e^{i\lambda u}\psi(2^{-j}\lambda)f(\lambda\o)(2^{-j}\la)^{p}\lambda^2 d\lambda.$$
where $p\in\mathbb{Z}$. For simplicity, we omit the index $p$ and denote them by 
\begin{equation}\label{ch3:bisof7}
F_j( u)=\int_{0}^{+\infty}e^{i\lambda u}\psi(2^{-j}\lambda)f(\lambda\o)\lambda^2 d\lambda.
\end{equation}
since they are essentially equivalent. Note that Plancherel yields:
\begin{equation}\label{ch3:bisof24}
\norm{F_j}_{L^2_{\o, u}}\leq\norm{\psi(2^{-j}\la)f(\la\o)\la}_{\le{2}}\lesssim 2^j\ga_j.
\end{equation}
Also, using Cauchy-Schwarz in $\la$, we have
\begin{equation}\label{ch3:bisof24infty}
\norm{F_j}_{L^2_\o L^\infty_u}\leq 2^{\frac{j}{2}}\norm{\psi(2^{-j}\la)f(\la\o)\la}_{\le{2}}\lesssim 2^{\frac{3j}{2}}\ga_j.
\end{equation}
\end{remark}

\eqref{ch3:bisof3}, \eqref{ch3:bisof5} and \eqref{ch3:bisof6} imply:
\bea\label{ch3:bisof8}
&&\int_{\MM}E_jf(t,x)\overline{E_kf(t,x)}d\MM \\
\nn&=&2^{-j}\int_{\MM}\left(\int_{\S}\nabn \trc F_j( u)d\o\right)\overline{\left(\int_{\S}{b'}^{-1}\trc'F_k( u')d\o'\right)}d\MM\\
\nn&& +2^{-j}\int_{\MM}\left(\int_{\S}\trc NF_j( u)d\o\right)\cdot\overline{\left(\int_{\S}\nabla ({b'}^{-1}\trc')F_k( u')d\o'\right)}d\MM+\cdots,
\eea
where we only kept the first term in the expansion \eqref{ch3:bisof6} in order to simplify the exposition\footnote{note that in the last term in the right-hand side of \eqref{ch3:bisof8}, we wrote $\nabla_N({b'}^{-1}\trc')$ as $N\cdot\nabla ({b'}^{-1}\trc')$}.

\begin{remark}
The second term in the right-hand side of \eqref{ch3:bisof8} is easier because the derivative falls on the low frequency term. This is why we estimate this term directly while the other term requires a more elaborate treatment which is explained in section \ref{ch3:sec:morepreciseest}.
\end{remark}

We estimate the second term in the right-hand side of \eqref{ch3:bisof8}. We have:
\bea\lab{ch3:jetlag}
&&\left|2^{-j}\int_{\MM}\left(\int_{\S}\trc NF_j( u)d\o\right)\cdot\overline{\left(\int_{\S}\nabla ({b'}^{-1}\trc')F_k( u')d\o'\right)}d\MM\right|\\
\nn&\les& 2^{-j}\normm{\int_{\S}\trc NF_j( u)d\o}_{L^2(\MM)}\normm{\int_{\S}\nabla ({b'}^{-1}\trc')F_k( u')d\o'}_{L^2(\MM)}.
\eea
We have the following analog of \eqref{ch3:bisof13}:
\begin{equation}\label{ch3:jetlag1}
\normm{\int_{\S}\trc NF_j( u)d\o}_{L^2(\MM)}\les \ep\ga_j.
\end{equation}
Indeed, one can show that the symbol $\trc$ satisfies regularity assumptions which are at least as good as $b^{-1}\trc$ (see \cite{param3}, and also section \ref{sec:mainres}), so that the proof of \eqref{ch3:bisof13} may adapted in a straightforward manner to obtain \eqref{ch3:jetlag1}.

Next, we consider the second term in the right-hand side of \eqref{ch3:jetlag}. Then proceeding as in the basic computation \eqref{ch3:bisb5}, and using the estimate \eqref{ch3:etoile}, we obtain
\be\lab{ch3:jetlag2}
\normm{\int_{\S}\nabla ({b'}^{-1}\trc')F_k( u')d\o'}_{L^2(\MM)}\les \ep\norm{\la\psi(2^{-k}\la)f}_{L^2(\RRR^3)}\les\ep 2^k\ga_k.
\ee
Together with \eqref{ch3:jetlag} and \eqref{ch3:jetlag1}, we finally obtain:
\be\lab{ch3:jetla3}
\left|2^{-j}\int_{\MM}\left(\int_{\S}\trc NF_j( u)d\o\right)\cdot\overline{\left(\int_{\S}\nabla ({b'}^{-1}\trc')F_k( u')d\o'\right)}d\MM\right|\les 2^{-j+k}\gamma_k\ga_j,
\ee
which is consistent with \eqref{ch3:bisof2}.

\begin{remark}
Estimating the first term in the right-hand side of \eqref{ch3:bisof8} in the same way would only yield:
\be\lab{ch3:jetlag5}
\left|2^{-j}\int_{\MM}\left(\int_{\S}\nabn \trc F_j( u)d\o\right)\overline{\left(\int_{\S}{b'}^{-1}\trc'F_k( u')d\o'\right)}d\MM\right|\les \ep^2\ga_j\ga_k,
\ee
which is not sufficient to obtain \eqref{ch3:bisof2}.
\end{remark}

\subsection{A more precise estimate}\lab{ch3:sec:morepreciseest}

In this section, we estimate the first term the right-hand side of \eqref{ch3:bisof8}. Using the geometric Littlewood-Paley projections on the 2-surfaces $\ptu$, we decompose $\nabn\trc$ as:
$$\nabn\trc=P_{\leq \frac{j+k}{2}}(\nabn\trc)+P_{>\frac{j+k}{2}}(\nabn\trc).$$ 
In turn, this yields a decomposition for the first term in the right-hand side of \eqref{ch3:bisof8}:
\begin{equation}\label{ch3:bisof21}
2^{-j}\int_{\MM}\left(\int_{\S}\nabn \trc F_j( u)d\o\right)\overline{\left(\int_{\S}{b'}^{-1}\trc'F_k( u')d\o'\right)}d\MM=A_1+A_2,
\end{equation}
where:
\begin{equation}\label{ch3:bisof22}
\begin{array}{l}
\ds A_1=2^{-j}\int_{\MM}\left(\int_{\S}P_{> \frac{j+k}{2}}(\nabn\trc)F_j( u)d\o\right)\overline{E_kf(t,x)}d\MM ,\\
\ds A_2=2^{-j}\int_{\MM}\left(\int_{\S}P_{\leq \frac{j+k}{2}}(\nabn\trc)F_j( u)d\o\right)\overline{E_kf(t,x)}d\MM .
\end{array}
\end{equation}

We first estimate the easier term $A_1$. The definition of $P_l$ implies $P_l= 2^{-2l}\lap P_l$, and thus
\bee
P_{> \frac{j+k}{2}}(\nabn\trc)&=& \sum_{l> \frac{j+k}{2}}P_l(\nabn\trc)\\
&=& \sum_{l> \frac{j+k}{2}}2^{-2l}\lap P_l(\nabn\trc),
\eee
which yields the following decomposition for $A_1$:
\be\lab{ch3:eq:A101}
A_1 = \sum_{l> \frac{j+k}{2}}A_{1,l}
\ee
where $A_{1,l}$ is given by:
$$A_{1,l} =2^{-j-2l}\int_{\MM}\left(\int_{\S}\lap P_l(\nabn\trc)F_j( u)d\o\right)\overline{E_kf(t,x)}d\MM.$$
Integrating by parts $\lap$ on $\ptu$ and using the fact that $\nabb F_{j, -1}(u)=0$, we obtain:
$$A_{1,l} =-2^{-j-2l}\int_{\S}\int_{t,u}\left(\int_{\ptu}\nabb P_l(\nabn\trc)\nabb(\overline{E_kf(t,x)}b)\dmt \right)F_j( u)du\,dt\,d\o+\cdots,$$
where the additional term corresponds to the case where the derivative falls on the volume element of $\MM$. Next, we apply Cauchy-Schwartz to the integral on $\MM$ and obtain:
\bea\label{ch3:bisof23}
|A_{1,l}| &\leq& 2^{-j-2l}\int_{\S}\norm{\nabb P_l(\nabn\trc)F_j( u)}_{L^2(\MM)}\norm{\nabb E_k}_{L^2(\MM)}d\o\\
\nn&\les& 2^{-j-2l}\int_{\S}\norm{\nabb P_l(\nabn\trc)}_{\li{\infty}{2}}\norm{F_j( u)}_{L^2_u}\norm{\nabb E_k}_{L^2(\MM)}d\o\\
\nn&\les& 2^{-j-l}\int_{\S}\norm{\nabn\trc}_{\li{\infty}{2}}\norm{F_j( u)}_{L^2_u}\norm{\nabb E_k}_{L^2(\MM)}d\o\\
\nn&\les& \ep 2^{-j-l}\int_{\S}\norm{F_j( u)}_{L^2_u}\norm{\nabla E_k}_{L^2(\MM)}d\o,
\eea
where we used the finite band property for $P_l$ and the estimates \eqref{ch3:estneeded1} for $\trc$. In view of \eqref{ch3:bisof23}, we also need to estimate $\norm{\nabla E_k}_{L^2(\MM)}$. We have:
\begin{equation}\label{ch3:bisof26}
\begin{array}{ll}
\ds\nabla E_kf(t,x)= & \ds\int_{\S}\int_0^{+\infty} e^{i\la u}\nabla (b^{-1}\trc)\psi(2^{-k}\la)f(\la\o)\la^2d\la d\o\\
& \ds +i2^k\int_{\S}\int_0^{+\infty} e^{i\la u}  b^{-1}\trc\nabla u\psi(2^{-k}\la)(2^{-k}\la)f(\la\o)\la^2d\la d\o.
\end{array}
\end{equation}
Using the basic computation \eqref{ch3:bisb5} for the first term together with the estimate \eqref{ch3:etoile}, and \eqref{ch3:bisof13} for the second term together with the fact that $\trc L$ satisfies the same regularity assumptions than $b^{-1}\trc$, we obtain:
\begin{equation}\label{ch3:bisof27}
\norm{\nabla E_k}_{\ll{2}}\lesssim \ep2^k\ga_k.
\end{equation}
\eqref{ch3:bisof23}, \eqref{ch3:bisof24}, and \eqref{ch3:bisof27} yield:
$$|A_{1,l}|\lesssim \ep^2 2^{-l+k}\ep^2\ga_j\ga_k.$$
Together with \eqref{ch3:eq:A101}, this yields:
\begin{equation}\label{ch3:bisof28}
|A_1|\lesssim \left(\sum_{l>\frac{j+k}{2}}2^{-l}\right)\ep2^k\ep^2\ga_j\ga_k\les \ep^2 2^{-\frac{j-k}{2}}\ga_j\ga_k,
\end{equation}
which is consistent with \eqref{ch3:bisof2}.

\subsection{A second integration by parts in $u$}\label{ch3:sec:2ibpu}

To estimate $A_2$, we perform a second integration by parts relying again on \eqref{ch3:bisof4}. This leads to:  
\bea\label{ch3:bisof29}
A_2&= &\ds 2^{-2j}\int_{\MM}\int_{\S}\nabn P_{\leq \frac{j+k}{2}}(\nabn\trc)F_j(u)\overline{E_kf(t,x)}d\MM+\cdots,
\eea
where we only keep the worst term, which is the one containing two derivatives of $\trc$. It is at this stage that we need the estimate \eqref{ch3:estneeded2} for $\nabn P_l(\nabn\trc)$ which we recall now. We have:
\be\lab{ch3:estneeded2bis}
\normm{\nabn P_l(\nabn\trc)}_{L^2(\H_u)}\les \ep 2^l+ 2^{\frac{l}{2}}\mu(u),
\ee
where $\mu$ in a function satisfying:
$$\norm{\mu}_{L^2(\R)}\les\ep.$$
In view of the estimate \eqref{ch3:estneeded2bis}, we have:
\bee
\normm{\nabn P_{\leq \frac{j+k}{2}}(\nabn\trc)}_{L^2(\H_u)}&\les & \sum_{l\leq \frac{j+k}{2}}\norm{\nabn P_l(\nabn\trc)}_{L^2(\H_u)}\\
&\les& \sum_{l\leq \frac{j+k}{2}}(2^l\ep+2^{\frac{l}{2}}\mu(u))\\
&\les& \ep 2^{\frac{j+k}{2}}+ 2^{\frac{j+k}{4}}\mu(u).
\eee
In view of \eqref{ch3:bisof29}, this yields after applying Cauchy-Schwartz:
\bea\lab{ch3:elias14}
\ds |A_2|&\les &\ds 2^{-2j}\norm{E_kf}_{L^2(\MM)}\int_{\S}\bigg\|\Big\|\nabn P_{\leq \frac{j+k}{2}}(\nabn\trc)\Big\|_{L^2(\H_u)}F_j( u)\bigg\|_{L^2_u}d\o+\cdots\\
\nn&\les& 2^{-2j}\ep\gamma_k\int_{\S}\normm{(\ep 2^{\frac{j+k}{2}}+\ep 2^{\frac{j+k}{4}}\mu(u))F_j(u)}_{L^2_u}d\o+\cdots\\
\nn&\les& 2^{-2j}\ep\gamma_k\bigg(\ep 2^{\frac{j+k}{2}}\int_{\S}\norm{F_j(u)}_{L^2_u}d\o
+2^{\frac{j+k}{4}}\int_{\S}\norm{\mu}_{L^2(\R)}\norm{F_j(u)}_{L^\infty_u}d\o\bigg)+\cdots\\
\nn&\les& 2^{-\frac{j-k}{4}}\ep^2\gamma_k\gamma_j+\cdots,
\eea
where we used \eqref{ch3:bisof13} for $E_kf$, Cauchy-Schwarz in $\o$, and the estimates \eqref{ch3:bisof24} and \eqref{ch3:bisof24infty} for $F_j(u)$.

\subsection{End of the proof of Proposition \ref{ch3:bisorthofreq}} 

In view of \eqref{ch3:bisof8}, \eqref{ch3:jetla3}, \eqref{ch3:bisof21}, \eqref{ch3:bisof28} and \eqref{ch3:elias14}, we obtain:
\begin{equation}\label{ch3:bisof35}
\left|\int_{\MM}E_jf(t,x)\overline{E_kf(t,x)}d\MM \right|\lesssim \ep^22^{-\frac{|j-k|}{4}}\ga_j\ga_k\textrm{ for }|j-k|>2.
\end{equation}
Finally, \eqref{ch3:bisof35} together with Shur's Lemma yields:
\begin{equation}\label{ch3:bisof36}
\norm{Ef}_{\ll{2}}^2\lesssim\sum_{j\geq -1}\norm{E_jf}_{\ll{2}}^2+\ep^2\norm{f}^2_{\le{2}}.
\end{equation}
This concludes the proof of Proposition \ref{ch3:bisorthofreq}. 

\section{Control of the diagonal term}\lab{ch3:bissec:diagonal}

Since the orthogonality argument in angle is the core of this chapter, we choose to deal first with 
the control of the diagonal term in this section. We will then proceed with the orthogonality argument in angle in the rest of the chapter.

In order to control the diagonal term, we have to prove \eqref{ch3:bisdiiiiagonal1}:
\begin{equation}\label{ch3:bisdiiii1}
\norm{E^\nu_jf}_{\ll{2}}\lesssim \ep\ga^\nu_j.
\end{equation}
Recall that $E^\nu_j$ is given by:
\begin{equation}\label{ch3:bisdiiii2}
E^\nu_jf(t,x)=\int_{\S}b^{-1}(t,x,\o)\trc(t,x,\o)F_j( u)\eta_j^\nu(\o)d\o,
\end{equation}
where $F_j( u)$ is defined by:
\begin{equation}\label{ch3:bisdiiii3}
F_j( u)=\int_0^{+\infty}e^{i\la u}\psi(2^{-j}\la)f(\la\o)\la^2d\la.
\end{equation}
The proof of the estimate \eqref{ch3:bisdiiii1} will proceed in four steps:
\begin{itemize}
\item[Step 1.] We first consider a decomposition roughly of the type:
$$E^\nu_jf(t,x)=b^{-1}(t,x,\nu)\trc(t,x,\nu)\left(\int_{\S}F_j( u)\eta_j^\nu(\o)d\o\right)+\cdots,$$
so that we have to prove estimate \eqref{ch3:bisdiiii1} with $b^{-1}\trc$ replaced by 1.

\item[Step 2.] That estimate is obtained by considering the transport equation along $L_\nu$:
$$L_\nu\left(\int_{\S}F_j( u)\eta_j^\nu(\o)d\o\right)=\cdots.$$

\item[Step 3.] A certain term in the transport equation of Step 2 needs to be estimated using an energy estimate for the wave equation.

\item[Step 4.] We conclude the proof using the estimates obtained in Step 2 and Step 3.
\end{itemize}

\subsection{Step 1: freezing the $\o$ dependance in $b^{-1}\trc$}

In view of the estimate \eqref{ch3:estricciomega} for $\po b$, the estimate \eqref{ch3:estneeded1} for $b$, and the decomposition \eqref{ch3:dectrcom} for $\trc$, we have:
\be\label{ch3:wimby1}
b^{-1}(t,x,\o)\trc(t,x,\o)= f^j_1(t,x,\nu)+f^j_2(t,x,\nu,\o),
\ee
where $f^j_1$ only depends on $(t,x,\nu)$ and satisfies:
\be\label{ch3:wimby2}
\norm{f^j_1}_{L^\infty}\les \ep,
\ee
and where $f^j_2$ satisfies:
\be\label{ch3:wimby3}
\norm{f^j_2}_{L^\infty_u\lh{2}}\les 2^{-\frac{j}{2}}\ep,
\ee
with $u=u(.,\o)$. \eqref{ch3:wimby1} yields the following decomposition for the diagonal term:
\bee
E^\nu_jf(t,x)= f^j_1(t,x,\nu)\int_{\S}F_j( u)\eta_j^\nu(\o)d\o+\int_{\S}F_j( u)f^j_2(t,x,\o,\nu)\eta_j^\nu(\o)d\o,
\eee
which implies:
\bea
\label{ch3:wimby4}&&\normm{E^\nu_jf(t,x)}_{\ll{2}}\\
\nn&\les& \norm{f^j_1}_{\ll{\infty}}\normm{\int_{\S}F_j( u)\eta_j^\nu(\o)d\o}_{\ll{2}}+\int_{\S}\norm{F_j(u)}_{L^2_u}\norm{f^j_2}_{\li{\infty}{2}}\eta_j^\nu(\o)d\o\\
\nn&\les& \ep\normm{\int_{\S}F_j( u)\eta_j^\nu(\o)d\o}_{\ll{2}}+\ep\ga^\nu_j,
\eea
where we used in the last inequality the estimates \eqref{ch3:wimby2} and \eqref{ch3:wimby3}, Cauchy-Schwarz in $\o$, the size of the patch, and the estimate \eqref{ch3:bisof24} for $F_j(u)$.

\begin{remark}
The point of the decomposition \eqref{ch3:wimby1} is to allow us to replace in the diagonal term \eqref{ch3:bisdiiii2} the symbol $b^{-1}\trc$ with 1. An obvious way to achieve this is to write the following decomposition:
\be\lab{ch3:onexplique}
b^{-1}\trc(t,x,\o)=b^{-1}\trc(t,x,\nu)+(b^{-1}\trc(t,x,\o)-b^{-1}\trc(t,x,\nu)).
\ee
The first term clearly satisfies \eqref{ch3:wimby2} in view of the estimate \eqref{ch3:estneeded1} for $b$ and $\trc$. On the other hand, we obtain in \cite{param3} (see also \eqref{ch4:estricciomega}) 
the estimate $\po\trc\in\li{\infty}{2}$ which together with the estimate \eqref{ch3:estricciomega} for $\po b$ yields:
\be\lab{ch3:onexplique1}
\norm{\po(b^{-1}\trc)}_{\li{\infty}{2}}\les\ep.
\ee
Now, we have:
$$b^{-1}\trc(t,x,\o)-b^{-1}\trc(t,x,\nu)=(\o-\nu)\int_0^1\po(b^{-1}\trc)(t,x,\o_\sigma)d\sigma$$
which together with \eqref{ch3:onexplique1} is not enough to conclude since $\li{\infty}{2}$ and $L^\infty_{u_\sigma}L^2(\HH_{u_\sigma})$ are not comparable. 
We refer the reader to \cite{param3} where the decomposition \eqref{ch3:wimby1} as well as several others are proved (see also the discussion in section \ref{ch4:sec:adddec}).
\end{remark}

The following proposition allows us to estimate the right-hand side of \eqref{ch3:wimby4}.
\begin{proposition}\label{ch3:bisdiiiipropppp1}
We have the following bound:
\begin{equation}\label{ch3:bisdiiiipropppp2}
\normm{\int_{\S}F_j( u)\eta_j^\nu(\o)d\o}_{L^2_{u_\nu,x_\nu}L^\infty_t}\lesssim\gamma_j^\nu.
\end{equation}
\end{proposition}

\begin{remark}
In order to control the diagonal term, it suffices to have a bound of the $\ll{2}$ norm for the left-hand side of \eqref{ch3:bisdiiiipropppp2}. The improvement to a bound for the $L^2_{u_\nu,x_\nu}L^\infty_t$ norm will be crucial when proving the almost orthogonality in angle.
\end{remark}

Assuming the result of the proposition, estimates \eqref{ch3:bisdiiiipropppp2} and \eqref{ch3:wimby4} yield:
$$\normm{E^\nu_jf(t,x)}_{\ll{2}}\les\ep\ga^\nu_j,$$
which together with \eqref{ch3:bisdiiii2} and \eqref{ch3:wimby4} implies:
$$\norm{E_j^\nu f}_{\ll{2}}\les \ep\ga^\nu_j.$$
which is the wanted estimate \eqref{ch3:bisdiiii1}. This concludes the proof of Proposition \ref{ch3:bisdiiiiagonal}. 

\subsection{Step 2: A transport equation in the $L_\nu$ direction}

We still need to prove Proposition \ref{ch3:bisdiiiipropppp1}. Note that it suffices to show:
\be\label{ch3:rg1}
\normm{L_\nu\left(\int_{\S}F_j( u)\eta_j^\nu(\o)d\o\right)}_{\ll{2}}\lesssim\gamma_j^\nu.
\ee
Now, since the space-time gradient of $u$ is given by $b^{-1}L$, we have:
$$L_\nu\left(\int_{\S}F_j( u)\eta_j^\nu(\o)d\o\right)=2^j\int_{\S}b^{-1}\gg\Big(L(t,x,\o),L(t,x,\nu)\Big)F_j( u)\eta_j^\nu(\o)d\o,$$
where $F_j$ has been defined in \eqref{ch3:bisof7}. In view of \eqref{ch3:rg2}, we have:
\bea\label{ch3:rg2}
&& L_\nu\left(\int_{\S}F_j( u)\eta_j^\nu(\o)d\o\right)\\
\nn &=& 2^j b^{-1}(t,x,\nu)\int_{\S}\gg(L(t,x,\o),L(t,x,\nu))F_j( u)\eta_j^\nu(\o)d\o\\
\nn &&+2^j\int_{\S}(b^{-1}(t,x,\o)-b^{-1}(t,x,\nu))\gg(L(t,x,\o),L(t,x,\nu))F_j( u)\eta_j^\nu(\o)d\o.
\eea
Next, we estimate the second term in the right-hand side of \eqref{ch3:rg2}. We have:
\be\label{ch3:rg3}
\gg\Big(L(t,x,\o),L(t,x,\nu)\Big)=\gg(N(t,x,\o)-N(t,x,\nu),N(t,x,\o)-N(t,x,\nu)).
\ee
Thus, the estimate \eqref{ch3:estNomega} for $\po N$ and the size of the patch yields:
\be\lab{ch3:rg4}
\norm{\gg(L(t,x,\o),L(t,x,\nu))}_{\lh{\infty}}\les 2^{-j},
\ee
which implies:
\bea\label{ch3:rg5}
&&\normm{\int_{\S}(b^{-1}(t,x,\o)-b^{-1}(t,x,\nu))\gg(L(t,x,\o),L(t,x,\nu))F_j( u)\eta_j^\nu(\o)d\o}_{\ll{2}}\\
\nn&\les& \int_{\S}\norm{b^{-1}(t,x,\o)-b^{-1}(t,x,\nu)}_{\li{\infty}{2}}\norm{\gg(L(t,x,\o),L(t,x,\nu))}_{L^\infty}\norm{F_j( u)}_{L^2_u}\eta_j^\nu(\o)d\o\\
\nn&\les& 2^{-j}\ep\gamma_j^\nu,
\eea
where we used in the last inequality \eqref{ch3:rg4}, the estimate \eqref{ch3:estricciomega} for $\po b$, Cauchy-Schwarz in $\o$, the size of the patch, and \eqref{ch3:bisof24} for $F_j(u)$. \eqref{ch3:rg2} together with \eqref{ch3:rg5} and the estimate \eqref{ch3:estricciomega} for $\po b$ yields:
\bea\label{ch3:rg6}
&&\normm{L_\nu\left(\int_{\S}F_j( u)\eta_j^\nu(\o)d\o\right)}_{\ll{2}}\\
\nn&\les& 2^j\normm{\int_{\S}\gg(L(t,x,\o),L(t,x,\nu))F_j( u)\eta_j^\nu(\o)d\o}_{\ll{2}}+\ep\gamma_j^\nu.
\eea

Next, we estimate the right-hand side of \eqref{ch3:rg6}. Using \eqref{ch3:rg3}, the decomposition \eqref{ch3:decNom} for $N-N'$, and arguing as in \eqref{ch3:wimby4}, we obtain:
$$\normm{\int_{\S}\gg(L(t,x,\o),L(t,x,\nu))F_j( u)\eta_j^\nu(\o)d\o}_{\ll{2}} \les  \normm{\int_{\S}(\o-\nu)^2F_j( u)\eta_j^\nu(\o)d\o}_{\ll{2}}+2^{-j}\gamma_j^\nu.$$
Together with \eqref{ch3:rg6}, this implies:
\be\label{ch3:rg6bis}
\normm{L_\nu\left(\int_{\S}F_j( u)\eta_j^\nu(\o)d\o\right)}_{\ll{2}}\les\normm{\int_{\S}(2^{\frac{j}{2}}(\o-\nu))^2F_j( u)\eta_j^\nu(\o)d\o}_{\ll{2}}+\gamma_j^\nu.
\ee

Finally, we need to estimate the first term in the right-hand side of \eqref{ch3:rg6bis}. We will rely on the energy estimate for the wave equation. 

\subsection{Step3: The energy estimate for the wave equation}

Recall from \eqref{ch3:symbolE} that: 
$$\square_{\bf g}u=b^{-1}\trc.$$
Thus, we have:
\be\lab{ch3:rg7bis}
\square_{\gg}\left(\int_{\S}(\o-\nu)^2F_j( u)\eta_j^\nu(\o)d\o\right)=\int_{\S}b^{-1}(t,x,\o)\trc(t,x,\o)(2^{\frac{j}{2}}(\o-\nu))^2F_j( u)\eta_j^\nu(\o)d\o.
\ee
Arguing as in \eqref{ch3:wimby4}, we may replace $b^{-1}\trc$ by 1:
\bee
&&\normm{\int_{\S}b^{-1}(t,x,\o)\trc(t,x,\o)(2^{\frac{j}{2}}(\o-\nu))^2F_j( u)\eta_j^\nu(\o)d\o}_{\ll{2}}\\
&\les& \ep\normm{\int_{\S}(2^{\frac{j}{2}}(\o-\nu))^2F_j( u)\eta_j^\nu(\o)d\o}_{\ll{2}}+\ep\gamma^\nu_j
\eee
which together with \eqref{ch3:rg7bis} implies:
\bea\lab{ch3:rg8}
&&\normm{\square_{\gg}\left(\int_{\S}(\o-\nu)^2F_j( u)\eta_j^\nu(\o)d\o\right)}_{\ll{2}}\\
\nn&\les& \ep\normm{\int_{\S}(2^{\frac{j}{2}}(\o-\nu))^2F_j( u)\eta_j^\nu(\o)d\o}_{\ll{2}}+\ep\gamma^\nu_j.
\eea
Let $\phi$ be the scalar function in the left-hand side of \eqref{ch3:rg7bis}, i.e.:
$$\phi=\int_{\S}(\o-\nu)^2F_j( u)\eta_j^\nu(\o)d\o.$$
Then, using the energy estimate for the wave equation \eqref{nrj1} we obtain:
\bea\lab{bodega}
\nn\norm{\dd\phi}_{L^\infty_tL^2(\Sigma_t)}^2&\lesssim& \norm{\nabla\phi(0,.)}^2_{L^2(\Si_0)}+\norm{T\phi(0,.)}^2_{L^2(\Si_0)}+\norm{\square_{\gg}\phi}_{L^2(\MM)}\norm{\dd\phi}_{L^2(\MM)}\\
&&+\left|\int_{\MM}Q_{\a\b}\pi^{\a\b}d\MM\right|,
\eea
where $\pi$ is the deformation tensor of $T$
$$\pi_{\a\b}=\dd_\a T_\b+\dd_\b T_\a.$$
and $Q_{\a\b}$ on $\MM$ is the energy momentum tensor associated to $\phi$
$$Q_{\a\b}=Q_{\a\b}[\phi]=\pr_\a\phi\pr_\b\phi-\frac{1}{2}\g_{\a\b}\left(\g^{\mu\nu}\pr_\mu\phi\pr_\nu\phi\right).$$
The control of the parametrix at initial time in \cite{param2} (see also Chapter \ref{part:paraminit}) yields
\be\lab{bodega1}
\norm{\nabla\phi(0,.)}_{L^2(\Si_0)}+\norm{T\phi(0,.)}_{L^2(\Si_0)}\les\gamma^\nu_j.
\ee
Next, we consider the last term in the right-hand side of \eqref{bodega}. From the maximal foliation assumption, $\pi$ is traceless, so that
\bee
Q_{\a\b}\pi^{\a\b}&=&\pi^{\a\b}\pr_\a\phi\pr_\b\phi\\
&=& 2^{2j}\int_{\S}\int_{\S}\pi_{N N'}(\o-\nu)^2(\o'-\nu)^2F_j( u)F_j(u')\eta_j^\nu(\o)\eta_j^\nu(\o')d\o d\o'
\eee
Using \eqref{ch3:estNomega}, one obtains
$$Q_{\a\b}\pi^{\a\b}= 2^{2j}\pi_{N_\nu N_{\nu}}\phi^2+\cdots,$$
and thus
\be\lab{bodega2}
Q_{\a\b}\pi^{\a\b}\simeq \pi_{N_\nu N_{\nu}}(\dd\phi)^2+\cdots.
\ee
It turns out that we have a trace estimate for $\pi_{N_\nu N_\nu}$ (see details in \cite{param4}):
$$\norm{\pi_{N_\nu N_\nu}}_{L^\infty_{u_\nu, {x'}_\nu}L^2_t}\les\ep$$
which together with \eqref{bodega2} implies
\bea\lab{bodega3}
\norm{Q_{\a\b}\pi^{\a\b}}_{L^2(\MM)}&\les& \norm{\pi_{N_\nu N_\nu}}_{L^\infty_{u_\nu, {x'}_\nu}L^2_t}\normm{\dd\phi}_{L^2_{u_\nu,x_\nu}L^\infty_t}\normm{\dd\phi}_{L^2(\MM)}\\
\nn&\les& \ep\normm{\dd\phi}_{L^2_{u_\nu,x_\nu}L^\infty_t}\normm{\dd\phi}_{L^2(\MM)}.
\eea
Finally, \eqref{ch3:rg8}-\eqref{bodega3} implies
$$\norm{\dd\phi}_{L^2(\MM)}\les \ep\norm{\dd\phi}_{L^2_{u_\nu,x_\nu}L^\infty_t}+\gamma^\nu_j,$$
which is equivalent to
\be\lab{ch3:onemoretime}
\normm{\int_{\S}(2^{\frac{j}{2}}(\o-\nu))^2F_j( u)\eta_j^\nu(\o)d\o}_{\ll{2}}\les \ep\normm{\int_{\S}(2^{\frac{j}{2}}(\o-\nu))^2F_j( u)\eta_j^\nu(\o)d\o}_{L^2_{u_\nu,x_\nu}L^\infty_t}+\gamma^\nu_j.
\ee
In view of \eqref{ch3:onemoretime} and \eqref{ch3:rg6bis}, we obtain:
\bee
\normm{L_\nu\left(\int_{\S}F_j( u)\eta_j^\nu(\o)d\o\right)}_{\ll{2}}&\les&\ep\normm{\int_{\S}(2^{\frac{j}{2}}(\o-\nu))^2F_j( u)\eta_j^\nu(\o)d\o}_{L^2_{u_\nu,x_\nu}L^\infty_t}+\gamma_j^\nu,
\eee
and thus:
$$\normm{\int_{\S}F_j( u)\eta_j^\nu(\o)d\o}_{L^2_{u_\nu,x_\nu}L^\infty_t}\les\gamma_j^\nu,$$
which is the desired estimate \eqref{ch3:bisdiiiipropppp2}.

\section{Almost orthogonality in angle}\label{ch3:bissec:orthoangle}

We have to prove \eqref{ch3:bisorthoangle1}:
\begin{equation}\label{ch3:bisoa1}
\norm{E_jf}_{\ll{2}}^2\lesssim\sum_{\nu\in\Gamma}\norm{E^\nu_jf}_{\ll{2}}^2+\ep^2\ga_j^2.
\end{equation}
This will result from an estimate for:
\be\lab{ch3:pc}
\left|\int_{\MM}E^\nu_jf(t,x)\overline{E^{\nu'}_jf(t,x)}d\MM \right|.
\ee

Let us introduce integration by parts first with respect to tangential directions, and then with respect to $L$.

\subsection{Integration by parts}

\subsubsection{Integration by parts in tangential directions}\label{ch3:sec:ibptangential}

By definition of $\nabb$, we have $\nabb h=\nabla h-(\nabn h)N$ for any function $h$ on $\Sigma_t$. In particular, we have $\nabb( u)=0$ and $\nabb( u')={b'}^{-1}N'-{b'}^{-1}\gg(N',N) N$. Now, since $\gg(N'-\gn N, N')=1-\gg(N',N)^2$ and $\nabla u'={b'}^{-1}N'$, we deduce:
\bea\label{ch3:bisoa15} 
e^{i\la u-i\la' u'}&=&\frac{i}{\la'\gg(N'-\gn N, \nabla u')}\nabla_{N'-\gn N}(e^{i\la u-i\la' u'})\\
\nn&=&\frac{ib'}{\la'(1-\gg(N',N)^2)}\nabla_{N'-\gn N}(e^{i\la u-i\la' u'}),
\eea
where we have used the fact that $N'-\gn N$ is a tangent vector with respect of the level surfaces of $ u$. We consider an oscillatory integral of the following form:
$$\int_{\MM}\int_{\S\times\S}h(t,x)F_j(u)F_j(u')\eta_j^\nu(\o)\eta_j^{\nu'}(\o')d\o d\o'd\MM,$$
where $h$ is a scalar function on $\MM$. Integrating by parts once using \eqref{ch3:bisoa15} yields:
\bee
&&\int_{\MM}\int_{\S\times\S}h(t,x)F_j(u)F_j(u')\eta_j^\nu(\o)\eta_j^{\nu'}(\o')d\o d\o'd\MM\\
\nn&=&   -i2^{-j}\int_{\MM}\int_{\S\times\S}\frac{b'}{1-\gn^2}((N'-\gn N)(h)+\cdots)\\
\nn&&\times F_j(u)F_j(u')\eta_j^\nu(\o)\eta_j^{\nu'}(\o')d\o d\o'd\MM,
\eee
where we only kept the term where the derivative falls on $h$, and neglected for the simplicity of the exposition the terms when the derivative falls on the denominator of the right-hand side of \eqref{ch3:bisoa15} or on the volume element of $\MM$. In view of \eqref{ch3:threomega1ter}, we have:
\be\lab{ch3:borekisback}
N'-\gn N\sim N'-N\sim |\o-\o'|\sim |\nu-\nu'|,
\ee
and:
\be\lab{ch3:borek}
1-\gn=\frac{\gg(N-N',N-N')}{2}\sim |\o-\o'|^2\sim |\nu-\nu'|^2,
\ee
and we thus obtain:
\bea\lab{ch3:fete}
&&\int_{\MM}\int_{\S\times\S}h(t,x)F_j(u)F_j(u')\eta_j^\nu(\o)\eta_j^{\nu'}(\o')d\o d\o'd\MM\\
\nn&=&   i\frac{1}{2^j|\nu-\nu'|}\int_{\MM}\int_{\S\times\S}b'\nabb(h) F_j(u)F_j(u')\eta_j^\nu(\o)\eta_j^{\nu'}(\o')d\o d\o'd\MM+\cdots.
\eea

\begin{remark}\lab{ch3:remark:simplify}
In the formula \eqref{ch3:fete}, we neglect two types of terms for the simplicity of the exposition. First, we neglect the term  when the derivative falls on the denominator of the right-hand side of \eqref{ch3:bisoa15} or on the volume element of $\MM$. Next, make the following approximation:
$$\frac{N'-\gn N}{1-\gn^2}\sim \frac{1}{2^j|\nu-\nu'|}.$$
In the actual proof, we use \eqref{ch3:borek} to derive the following expansion:
\be\lab{ch3:nice27}
\frac{1}{1-\gn^2}=\frac{1}{|N_\nu-N_{\nu'}|^2}\left(\sum_{p, q\geq 0}c_{pq}\left(\frac{N-N_\nu}{|N_\nu-N_{\nu'}|}\right)^p\left(\frac{N'-N_{\nu'}}{|N_\nu-N_{\nu'}|}\right)^q\right),
\ee
for some explicit real coefficients $c_{pq}$ such that the series 
$$\sum_{p, q\geq 0}c_{pq}x^py^q$$
has radius of convergence 1. Then, \eqref{ch3:fete} corresponds to the first term in the expansion \eqref{ch3:nice27} with the additional simplification which consists in replacing $|N_\nu-N_{\nu'}|$ with $|\nu-\nu'|$ again in view of \eqref{ch3:borek}. While these approximations greatly simplify the exposition, they still allow us to exhibit typical terms in the proof of the almost orthogonality in angle. 
\end{remark}

\subsubsection{Integration by parts in $L$}

Next, we also introduce integrations by parts with respect to $L$. Since $L(u)=0$ and $L(u')={b'}^{-1}\gg(L,L')$, we have:
\begin{equation}\label{ch3:ibpl} 
e^{i\la u-i\la' u'}=\frac{ib'}{\la'\gg(L,L')}L(e^{i\la u-i\la' u'}).
\end{equation}
We consider an oscillatory integral of the following form:
$$\int_{\MM}\int_{\S\times\S}h(t,x)F_j(u)F_j(u')\eta_j^\nu(\o)\eta_j^{\nu'}(\o')d\o d\o'd\MM,$$
where $h$ is a scalar function on $\MM$. Integrating by parts once using \eqref{ch3:bisoa15} yields:
\bee
&&\int_{\MM}\int_{\S\times\S}h(t,x)F_j(u)F_j(u')\eta_j^\nu(\o)\eta_j^{\nu'}(\o')d\o d\o'd\MM\\
\nn&=&   -i2^{-j}\int_{\MM}\int_{\S\times\S}\frac{b'}{\gl}(L(h)+\cdots)F_j(u)F_j(u')\eta_j^\nu(\o)\eta_j^{\nu'}(\o')d\o d\o'd\MM,
\eee
where we only kept the term where the derivative falls on $h$, and neglected for the simplicity of the exposition the term when the derivative falls on the denominator of the right-hand side of \eqref{ch3:ibpl} or on the volume element. Using the fact that:
\be\lab{ch3:nice24}
\gg(L,L')=-1+\gn
\ee
together with \eqref{ch3:borek}, and keeping only the first term in the expansion \eqref{ch3:nice27}, with the additional simplification which consists in replacing $|N_\nu-N_{\nu'}|$ with $|\nu-\nu'|$, we obtain:
\bea\lab{ch3:fetebis}
&&\int_{\MM}\int_{\S\times\S}h(t,x)F_j(u)F_j(u')\eta_j^\nu(\o)\eta_j^{\nu'}(\o')d\o d\o'd\MM\\
\nn&=&   i\frac{1}{2^j|\nu-\nu'|^2}\int_{\MM}\int_{\S\times\S}b'L(h) F_j(u)F_j(u')\eta_j^\nu(\o)\eta_j^{\nu'}(\o')d\o d\o'd\MM+\cdots.
\eea

\subsection{Presence of a log-loss}\lab{ch3:sec:logloss}

Let us explain why proceeding directly by integration by parts in \eqref{ch3:pc} results in a log-loss. Let us define $\mathcal{E}_{j, \nu, \nu'}$ as: 
$$\mathcal{E}_{j, \nu, \nu'}=\int_{\MM}E^\nu_jf(t,x)\overline{E^{\nu'}_jf(t,x)}d\MM.$$
We have:
\bee
\mathcal{E}_{j, \nu, \nu'}&=&   \int_{\S\times\S}\int_0^{+\infty}\int_0^{+\infty}
\left(\int_{\MM}e^{i\la u-i\la' u'} b^{-1}\trc {b'}^{-1}\trc'd\MM\right)\\
&&\times\eta_j^\nu(\o)\eta_j^{\nu'}(\o')\psi(2^{-j}\la)\psi(2^{-j}\la') f(\la\o)f(\la'\o')\la^2 {\la'}^2d\la d\la' d\o d\o'.
\eee
We integrate by parts tangentially using \eqref{ch3:fete}. Consider the term where the tangential derivative falls on $\trc$, which is of the form:
\bee
\frac{1}{2^j|\nu-\nu'|}&&\int_{\S\times\S}\int_0^{+\infty}\int_0^{+\infty}\left(\int_{\MM}e^{i\la u-i\la' u'} b^{-1}\nabb\trc {b'}^{-1}\trc'd\MM\right)\\
&&\times\eta_j^\nu(\o)\eta_j^{\nu'}(\o')\psi(2^{-j}\la)\psi(2^{-j}\la') f(\la\o)f(\la'\o')\la^2 {\la'}^2d\la d\la' d\o d\o'.
\eee
Since $L\nabb\trc$ is the only derivative of $\nabb\trc$ for which we have an estimate, our next integration by parts must be with respect to $L$, that is we use \eqref{ch3:fetebis}. Consider the term where the $L$ derivative falls on $\trc'$, which is of the form:
\bea\lab{ch3:fof}
\frac{1}{2^{2j}|\nu-\nu'|^3}&&\int_{\S\times\S}\int_0^{+\infty}\int_0^{+\infty}\left(\int_{\MM}e^{i\la u-i\la' u'} b^{-1}\nabb\trc {b'}^{-1}L(\trc')d\MM\right)\\
\nn&&\times\eta_j^\nu(\o)\eta_j^{\nu'}(\o')\psi(2^{-j}\la)\psi(2^{-j}\la') f(\la\o)f(\la'\o')\la^2 {\la'}^2d\la d\la' d\o d\o'.
\eea
Now, note in view of \eqref{ch3:nice24}, \eqref{ch3:borek} and the estimate \eqref{ch3:estNomega} for $\po N$, that:
$$\gg(L,L')\sim |\nu-\nu'|^2,\,\gg(L,e'_A)=\gg(L-L',e'_A)\sim |\nu-\nu'|\textrm{ and }\gg(L,\lb')=-2+\gg(L,L')\sim -2.$$
Thus, decomposing $L$ on the frame $L', N', e'_A$, we obtain:
\be\lab{ch3:encoreuneffort}
L\sim L'+|\nu-\nu'|\nabb'+|\nu-\nu'|^2N'.
\ee
Together with \eqref{ch3:fof}, we finally obtain the sum of three terms:
\bea\lab{ch3:fff}
&&\mathcal{E}_{j, \nu, \nu'}\\
\nn&=&\mathcal{E}_{j, \nu, \nu'}[1]+\mathcal{E}_{j, \nu, \nu'}[2]+\mathcal{E}_{j, \nu, \nu'}[3]\\
\nn&=&\frac{1}{2^{2j}|\nu-\nu'|^3}\int_{\S\times\S}\int_0^{+\infty}\int_0^{+\infty}\left(\int_{\MM}e^{i\la u-i\la' u'} b^{-1}\nabb\trc {b'}^{-1}L'(\trc')d\MM\right)\\
\nn&&\times\eta_j^\nu(\o)\eta_j^{\nu'}(\o')\psi(2^{-j}\la)\psi(2^{-j}\la') f(\la\o)f(\la'\o')\la^2 {\la'}^2d\la d\la' d\o d\o'\\
\nn&&+\frac{1}{2^{2j}|\nu-\nu'|^2}\int_{\S\times\S}\int_0^{+\infty}\int_0^{+\infty}\left(\int_{\MM}e^{i\la u-i\la' u'} b^{-1}\nabb\trc {b'}^{-1}\nabb'(\trc')d\MM\right)\\
\nn&&\times\eta_j^\nu(\o)\eta_j^{\nu'}(\o')\psi(2^{-j}\la)\psi(2^{-j}\la') f(\la\o)f(\la'\o')\la^2 {\la'}^2d\la d\la' d\o d\o'\\
\nn&&+\frac{1}{2^{2j}|\nu-\nu'|}\int_{\S\times\S}\int_0^{+\infty}\int_0^{+\infty}\left(\int_{\MM}e^{i\la u-i\la' u'} b^{-1}\nabb\trc {b'}^{-1}N'(\trc')d\MM\right)\\
\nn&&\times\eta_j^\nu(\o)\eta_j^{\nu'}(\o')\psi(2^{-j}\la)\psi(2^{-j}\la') f(\la\o)f(\la'\o')\la^2 {\la'}^2d\la d\la' d\o d\o'.
\eea
We consider the second term in the right-hand side of \eqref{ch3:fff} which is of the form:
\bee
\mathcal{E}_{j, \nu, \nu'}[2]&=&\frac{1}{2^{2j}|\nu-\nu'|^2}\int_{\MM}\left(\int_{\S}b^{-1}\nabb\trc F_j(u)\eta^\nu_j(\o)d\o\right)\\
&&\times\left(\int_{\S}{b'}^{-1}\nabb'\trc' F_j(u')\eta^{\nu'}_j(\o')d\o'\right)d\MM.
\eee
We claim that such a term leads to a log-loss. Indeed, we have:
\bea
\nn&&\left|\mathcal{E}_{j, \nu, \nu'}[2]\right|\\
\nn&\les& \frac{1}{2^{2j}|\nu-\nu'|^2}\normm{\int_{\S}b^{-1}\nabb\trc F_j(u)\eta^\nu_j(\o)d\o}_{L^2(\MM)}\normm{\int_{\S}{b'}^{-1}\nabb'\trc' F_j(u')\eta^{\nu'}_j(\o')d\o'}_{L^2(\MM)}\\
\nn&\les& \frac{1}{2^{2j}|\nu-\nu'|^2}\left(\int_{\S}\norm{b^{-1}\nabb\trc F_j(u)}_{L^2(\MM)}\eta^\nu_j(\o)d\o\right)\\
\nn&&\times\left(\int_{\S}\norm{{b'}^{-1}\nabb'\trc' F_j(u')}_{L^2(\MM)}\eta^{\nu'}_j(\o')d\o'\right)\\
\nn&\les& \frac{1}{2^{2j}|\nu-\nu'|^2}\left(\int_{\S}\norm{b^{-1}}_{L^\infty}\norm{\nabb\trc}_{\li{\infty}{2}}\norm{F_j(u)}_{L^2_u}\eta^\nu_j(\o)d\o\right)\\
\nn&& \times\left(\int_{\S}\norm{{b'}^{-1}}_{L^\infty}\norm{\nabb'\trc'}_{\li{\infty}{2}} \norm{F_j(u')}_{L^2_{u'}}\eta^{\nu'}_j(\o')d\o'\right)\\
\lab{ch3:chuddington}&\les &\frac{\ep^2\gamma^\nu_j\gamma^{\nu'}_j}{(2^{\frac{j}{2}}|\nu-\nu'|)^2},
\eea
where we used in the last inequality Cauchy-Schwartz in $\o$ and $\o'$ which gains the square root of the volume of the patch, the estimates \eqref{ch3:estneeded1} for $b$ and $\trc$, and the estimate \eqref{ch3:bisof24} for $F_j(u)$ and $F_j(u')$. This leads to a log-loss since we have:
\begin{equation}\label{ch3:bisoa5}
\sup_{\nu}\sum_{\nu'} \frac{1}{(2^{j/2}|\nu-\nu'|)^{2}}\sim j.
\end{equation}
Indeed, note that $\nu'$ runs on a lattice on $\S$ of basic size $2^{-j/2}$ so that \eqref{ch3:bisoa5} corresponds to the sum 
$$\sum_{l\in\mathbb{Z}^2,\, 1\leq |l|\leq 2^{j/2}}\frac{1}{|l|^2}\sim j.$$

\subsection{Strategy of the proof of Proposition \ref{ch3:bisorthoangle}}

Let us explain informally the strategy of the proof. As we noticed in the previous section, the second term in \eqref{ch3:fff} contains a log-loss. Let us start by showing that the first and the third term in the right-hand side of \eqref{ch3:fff} do not contain a log-loss.

\subsubsection{Control of the first term in the right-hand side of \eqref{ch3:fff}} We have
\bee
\mathcal{E}_{j, \nu, \nu'}[1]&=&\frac{1}{2^{2j}|\nu-\nu'|^3}\int_{\MM}\left(\int_{\S}b^{-1}\nabb\trc F_j(u)\eta^\nu_j(\o)d\o\right)\\
&&\times\left(\int_{\S}{b'}^{-1}L'(\trc') F_j(u')\eta^{\nu'}_j(\o')d\o'\right)d\MM.
\eee
In view of the Raychaudhuri equation \eqref{ch3:raychaudhuri}, we have:
$$L'(\trc')=-|\hch'|^2+\cdots,$$
where we keep only the worst term. Thus, we obtain
$$\mathcal{E}_{j, \nu, \nu'}[1]=\frac{1}{2^{2j}|\nu-\nu'|^3}\int_{\MM}\left(\int_{\S}b^{-1}\nabb\trc F_j(u)\eta^\nu_j(\o)d\o\right)\left(\int_{\S}{b'}^{-1}|\hch'|^2 F_j(u')\eta^{\nu'}_j(\o')d\o'\right)d\MM.$$
Let us decompose:
\be\lab{ch3:fff1}
{b'}^{-1}|\hch'|^2=b^{-1}_{\nu'}\hch_{\nu'}\hch+({b'}^{-1}|\hch'|^2-b^{-1}_{\nu'}\hch_{\nu'}\hch),
\ee
and let us assume for the moment that we can control the second term in \eqref{ch3:fff1}. Then, we are led to control:
$$\frac{1}{2^{2j}|\nu-\nu'|^3}\int_{\MM}\left(\int_{\S}b^{-1}\hch\nabb\trc F_j(u)\eta^\nu_j(\o)d\o\right)b^{-1}_{\nu'}\hch_{\nu'}\left(\int_{\S} F_j(u')\eta^{\nu'}_j(\o')d\o'\right)d\MM.$$
We have:
\bea
\nn&&\left|\frac{1}{2^{2j}|\nu-\nu'|^3}\int_{\MM}\left(\int_{\S}b^{-1}\hch\nabb\trc F_j(u)\eta^\nu_j(\o)d\o\right)b^{-1}_{\nu'}\hch_{\nu'}\left(\int_{\S} F_j(u')\eta^{\nu'}_j(\o')d\o'\right)d\MM\right|\\
\nn&\les& \frac{1}{2^{2j}|\nu-\nu'|^3}\normm{\int_{\S}b^{-1}\hch\nabb\trc F_j(u)\eta^\nu_j(\o)d\o}_{L^2(\MM)}\normm{b^{-1}_{\nu'}\hch_{\nu'}\int_{\S} F_j(u')\eta^{\nu'}_j(\o')d\o'}_{L^2(\MM)}\\
\nn&\les& \frac{1}{2^{2j}|\nu-\nu'|^3}\left(\int_{\S}\norm{b^{-1}\hch\nabb\trc F_j(u)}_{L^2(\MM)}\eta^\nu_j(\o)d\o\right)\\
\nn&&\times\normm{b^{-1}_{\nu'}\hch_{\nu'}}_{L^\infty_{u_{\nu'}, {x'}_{\nu'}}L^2_t}\normm{\int_{\S}F_j(u')\eta^{\nu'}_j(\o')d\o'}_{L^2_{u_{\nu'}, {x'}_{\nu'}}L^\infty_t}\\
\nn&\les& \frac{\ep \ga^{\nu'}_j}{2^{2j}|\nu-\nu'|^3}\int_{\S}\norm{b^{-1}}_{L^\infty(\MM)}\norm{\hch}_{\xt{\infty}{2}}\norm{\nabb\trc}_{\xt{\infty}{2}} \norm{F_j(u)}_{L^2_u}\eta^\nu_j(\o)d\o\\
\lab{ch3:fff2}&\les& \frac{\ep^2\ga^\nu_j\ga^{\nu'}_j}{(2^{\frac{j}{2}}|\nu-\nu'|)^3},
\eea
where we used the estimate \eqref{ch3:estneeded1} for $\trc$, $\hch$ and $b$, the estimate \eqref{ch3:bisdiiiipropppp2}, Cauchy-Schwartz in $\o$, the size of the patch, and the estimate \eqref{ch3:bisof24} for $F_j(u)$. Note that the right-hand side of \eqref{ch3:fff2} does not contain a log-loss since:
\begin{equation}\label{ch3:fff3}
\sup_{\nu}\sum_{\nu'} \frac{1}{(2^{j/2}|\nu-\nu'|)^3}\les 1.
\end{equation}

\begin{remark}
While the estimate obtained in \eqref{ch3:fff2} is correct, one has to modify slightly the method leading to it. Indeed, $\hch$ does not have enough regularity with respect to $\o$ to be able to handle the second term in the decomposition \eqref{ch3:fff1}. The way to overcome this is to make use of the decomposition \eqref{ch3:dechch} for $\hch$:
$$\hch=\chi_1+\chi_2.$$
Then, we exploit the fact that, in view of the estimate \eqref{ch3:dechch2}, $\chi_1$ has better regularity than $\hch$ with respect to $(t,x)$, while $\chi_2$ has better regularity than $\hch$ with respect to $\o$. We refer to \cite{param4} for more details.
\end{remark}

\subsubsection{Control of the third term in the right-hand side of \eqref{ch3:fff}}

We have
\bee
\mathcal{E}_{j, \nu, \nu'}[3]&=&\frac{1}{2^{2j}|\nu-\nu'|}\int_{\MM}\left(\int_{\S}b^{-1}\nabb\trc F_j(u)\eta^\nu_j(\o)d\o\right)\\
&&\times\left(\int_{\S}{b'}^{-1}N'(\trc') F_j(u')\eta^{\nu'}_j(\o')d\o'\right)d\MM.
\eee
Now, we have:
\bea
\nn&&\left|\frac{1}{2^{2j}|\nu-\nu'|}\int_{\MM}\left(\int_{\S}b^{-1}\nabb\trc F_j(u)\eta^\nu_j(\o)d\o\right)\left(\int_{\S}{b'}^{-1}N'\trc' F_j(u')\eta^{\nu'}_j(\o')d\o'\right)d\MM\right|\\
\nn&\les& \frac{1}{2^{2j}|\nu-\nu'|}\normm{\int_{\S}b^{-1}\nabb\trc F_j(u)\eta^\nu_j(\o)d\o}_{L^2(\MM)}\normm{\int_{\S}{b'}^{-1}N'\trc' F_j(u')\eta^{\nu'}_j(\o')d\o'}_{L^2(\MM)}\\
\nn&\les& \frac{1}{2^{2j}|\nu-\nu'|}\left(\int_{\S}\norm{b^{-1}\nabb\trc F_j(u)}_{L^2(\MM)}\eta^\nu_j(\o)d\o\right)\\
\nn&&\times\left(\int_{\S}\norm{{b'}^{-1}N'\trc' F_j(u')}_{L^2(\MM)}\eta^{\nu'}_j(\o')d\o'\right)\\
\nn&\les& \frac{1}{2^{2j}|\nu-\nu'|}\left(\int_{\S}\norm{b^{-1}}_{L^\infty}\norm{\nabb\trc}_{\li{\infty}{2}}\norm{F_j(u)}_{L^2_u}\eta^\nu_j(\o)d\o\right)\\
\nn&& \times\left(\int_{\S}\norm{{b'}^{-1}}_{L^\infty}\norm{N'\trc'}_{\li{\infty}{2}} \norm{F_j(u')}_{L^2_{u'}}\eta^{\nu'}_j(\o')d\o'\right)\\
\lab{ch3:fff4}&\les &\frac{\ep^2\gamma^\nu_j\gamma^{\nu'}_j}{2^{\frac{j}{2}}(2^{\frac{j}{2}}|\nu-\nu'|)},
\eea
where we used in the last inequality Cauchy-Schwartz in $\o$ and $\o'$ which gains the square root of the volume of the patch, the estimates \eqref{ch3:estneeded1} for $b$ and $\trc$, and the estimate \eqref{ch3:bisof24} for $F_j(u)$ and $F_j(u')$. Note that the right-hand side of \eqref{ch3:fff4} does not contain a log-loss since we have:
\begin{equation}\label{ch3:fff5}
\sup_{\nu}\sum_{\nu'} \frac{1}{2^{\frac{j}{2}}(2^{\frac{j}{2}}|\nu-\nu'|)}\les 1.
\end{equation}

\subsubsection{A decomposition for $E^\nu_jf$}

To remove the log-loss exhibited in \eqref{ch3:chuddington} \eqref{ch3:bisoa5}, we rely on a decomposition of $\trc$ using the geometric Littlewood-Paley projections $P_j$. We have:
$$\trc=P_{\leq j/2}(\trc)+\sum_{l>j/2}P_l\trc$$
which in turn yields the following decomposition for $E^\nu_jf$:
\begin{equation}\label{ch3:bisoa7}
E^\nu_jf(t,x)=\sum_{l\geq j/2}E^{\nu,l}_jf(t,x),
\end{equation}
where:
\begin{equation}\label{ch3:bisoa8}
E^{\nu,l}_jf(t,x)=\int_{\S}b(t,x,\o)^{-1}P_l\trc(t,x,\o)F_j(u)\eta^\nu_j(\o) d\o,\,\,\,\forall l>\frac{j}{2}
\end{equation}
and:
\begin{equation}\label{ch3:bisoa8:1}
E^{\nu,j/2}_jf(t,x)=\int_{\S}b(t,x,\o)^{-1}P_{\leq j/2}\trc(t,x,\o)F_j(u)\eta^\nu_j(\o) d\o.
\end{equation}
In order to prove almost orthogonality in angle, i.e. \eqref{ch3:bisoa1}, we will estimate:
\be\lab{ch3:fff6}
\left|\sum_{l,m}\int_{\MM}E^{\nu,l}_jf(t,x)\overline{E^{\nu',m}_jf(t,x)}d\MM \right|.
\ee

\subsubsection{The mechanism to remove the log-loss}\lab{sec:logremoval}

In order to explain the mechanism which allows us to remove the log-loss, let us assume for convenience that $m\leq l$ in \eqref{ch3:fff6}. Then, notice first from \eqref{ch3:fff}, \eqref{ch3:bisoa5}, \eqref{ch3:fff3} and \eqref{ch3:fff5} that the only term in the right-hand side of \eqref{ch3:fff} which contains a log-loss is the second one, i.e the term which contains only tangential derivatives. In order to remove the log-loss, our goal will be to always put more tangential derivatives on the lowest frequency, i.e. $P_m\trc'$ (as opposed to the higher frequency $P_l\trc$). This is achieved as follows (see \cite{param4} for the details):
\begin{enumerate}
\item Integrate by parts with respect to $L$ using \eqref{ch3:fetebis}. 

\item One term corresponds to the case where the $L$ derivative falls on the largest frequency $P_l\trc$, while the other term corresponds to the case where $L$ falls on the lowest frequency $P_m\trc'$. For the second term,  decompose the $L$ derivative on the frame $L', N', e'_A$ as in \eqref{ch3:encoreuneffort}. 

\item Notice that the terms involving $L$, $L'$ or $N'$ are estimated in the spirit of \eqref{ch3:fff2} and \eqref{ch3:fff4}, and should in principle contain no log-loss in view of \eqref{ch3:fff3} and \eqref{ch3:fff5}.

\item Finally, the last term is the one containing the $\nabb'$ derivative. This term is the only one which contains the log-loss exhibited in \eqref{ch3:bisoa5}. Now, we have achieved our goal since after integration by parts, the tangential derivative fell on $P_m\trc'$ which is the lowest frequency.
\end{enumerate}
  
\begin{remark}
Due to the decomposition \eqref{ch3:bisoa7}, we now not only need to obtain summability in $(\nu, \nu')$, but also in $(l, m)$. This creates additional difficulties, in particular when estimating the terms $\mathcal{E}_{j, \nu, \nu'}[1]$ and $\mathcal{E}_{j, \nu, \nu'}[3]$ in \eqref{ch3:fff}. We refer to \cite{param4} for more details.
\end{remark}

%%%%%%%%%%%%%%%%%%%%%%%%%%%%%%%%%%%%%%

\chapter{Control of the space-time foliation}\lab{part:spacetimeu}

%%%%%%%%%%%%%%%%%%%%%%%%%%%%%%%%%%%%%%%

%% Control of the space-time foliation

The goal of this chapter is to prove the estimates on the control of the space-time foliation by the optical function $u$ 
which are needed for the proof of Theorem \ref{ch3:part3:th1} (see section \ref{ch3:sec:estneeded}), i.e. for the control of the error term. Here, we outline the main ideas and we refer to \cite{param3} for the details. 

\section{Geometric set-up and main results}

\subsection{Geometry of the foliation of $\mathcal{M}$ by $u$}

Recall from section \ref{sec:maxfoliation} that the space-time $\MM$ is foliated by space-like hypersurfaces $\Sit$ defined as level hypersurfaces of a time function $t$, where $T$ denotes the unit, future oriented, normal to $\Si_t$ and $k$ its second fundamental form. Recall also that $u$ is a solution to the eikonal equation $\gg^{\alpha\beta}\partial_\alpha u\partial_\beta u=0$ on $\mathcal{M}$ depending on a extra parameter $\o\in \S$.  The level hypersufaces $u(t,x,\o)=u$  of the optical function $u$ are denoted by  $\H_u$. Let $L'$ denote the space-time gradient of $u$, i.e.:
\be\lab{ch4:def:L'0}
L'=-\gg^{\a\b}\pr_\b u \pr_\a.
\ee
Using the fact that $u$ satisfies the eikonal equation, we obtain:
\be\lab{ch4:def:L'1}
\dd_{L'}L'=0,
\ee
which implies that $L'$ is the geodesic null generator of $\H_u$.

We have: 
$$T(u)=\pm |\nab u|$$
where $|\nab u|^2=\sum_{i=1}^3|e_i(u)|^2$ relative to an orthonormal frame $e_i$ on $\Si_t$. Since the sign of $T(u)$ is irrelevant, we choose by convention:
\be\lab{ch4:it1'}
T(u)=|\nab u|.
\end{equation}
We denote by $P_{t,u}$  the surfaces of intersection
between $\Si_t$ and  $\H_u$. They play a fundamental role
in our discussion.
\begin{definition}[\textit{Canonical null pair}]
 \be\lab{ch4:it2}
L=bL'=T+N, \qquad \lb=2T-L=T-N
\end{equation}
where $L'$ is the space-time gradient of $u$ \eqref{ch4:def:L'0}, $b$  is  the  \textit{lapse of the null foliation} (or shortly null lapse)
\be\lab{ch4:it3}
b^{-1}=-<L', T>=T(u),
\end{equation} 
and $N$ is a unit normal, along $\Si_t$, to the surfaces $P_{t,u}$. Since $u$ satisfies the 
eikonal equation $\gg^{\alpha\beta}\partial_\alpha u\partial_\beta u=0$ on $\mathcal{M}$, 
this yields $L'(u)=0$ and thus $L(u)=0$. In view of the definition of $L$ and \eqref{ch4:it1'}, 
we obtain:
\be\lab{ch4:it3bis}
N=-\frac{\nabla u}{|\nabla u|}.
\end{equation} 
\label{ch4:def:nulllapse}
\end{definition}

\begin{remark}\lab{ch4:defnullgeod}
$u$ is prescribed on $\Si_0$ as in \cite{param1}. For any $(0,x)$ on $\Si_0$, $L$ is defined as $L=T+N$ where $T$ is the unit normal to $\Si_0$ at $(0,x)$ and $N=-\nabla u/|\nabla u|$ at $(0,x)$, and $b$ is defined as $b^{-1}=|\nabla u|$. Let $\kappa_x(t)$ denote the null geodesic parametrized by $t$ and such that $\kappa_x(0)=(0,x)$ and $\kappa_x'(0)=b^{-1}L$. Then, we claim that 
\be\lab{ch4:defnullgeod1}
\kappa_x'(t)=b(\kappa_x(t))^{-1}L_{\kappa_x(t)}\textrm{ for all }t. 
\ee
Indeed, $L'=b^{-1}L$ is the geodesic null generator of $\H_u$ (see \eqref{ch4:def:L'1}).
\end{remark}

\begin{definition} A  null frame   $e_1,e_2,e_3,e_4$ at a point $p\in P_{t,u}$ 
  consists,
in addition to the null pair     $e_3=\lb,
e_4=L$, of {\sl  arbitrary  orthonormal}  vectors  $e_1,e_2$ tangent
to $P_{t,u}$. 
%All the   estimates in this paper  are in fact local and
% independent of the
% choice of a particular frame. We do not need to worry
% that these  frames cannot be  globally defined.
\end{definition}
\begin{definition}[\textit{Ricci coefficients}]

 Let  $e_1,e_2,e_3,e_4$ be a null frame on
$P_{t,u}$ as above.   The following tensors on  $S_{t,u}$ 
\begin{alignat}{2}
&\chi_{AB}=<\dd_A e_4,e_B>, &\quad 
&\chb_{AB}=<\dd_A e_3,e_B>,\label{ch4:chi}\\
&\z_{A}=\half <\dd_{3} e_4,e_A>,&\quad
&\zb_{A}=\half <\dd_{4} e_3,e_A>,\nn\\
&\xib_{A}=\half <\dd_{3} e_3,e_A>.\nn
\end{alignat}
are called the Ricci coefficients associated to our canonical null pair.

We decompose $\chi$ and $\chb$ into
their  trace and traceless components.
\begin{alignat}{2}
&\trc = \gg^{AB}\chi_{AB},&\quad &\trchb = \gg^{AB}\chb_{AB},
\label{ch4:trchi}\\
&\hch_{AB}=\chi_{AB}-\half \trc \gg_{AB},&\quad 
&\hchb_{AB}=\chb_{AB}-\half \trchb \gg_{AB},
\label{ch4:chih} 
\end{alignat}
\end{definition}

\begin{definition}\label{ch4:def:nullcurv}
The null components of the curvature tensor
$\rr$ of the space-time metric $\gg$ are given
by:
\bea
\a_{ab}&=&\rr(L,e_a,L, e_b)\,,\qquad \b_a=\half \rr(e_a, L,\lb, L) ,\\ \r &=&\frac{1}{4}\rr(\lb, L, \lb,
L)\,,
\qquad
\s=\frac{1}{4}\, ^{\star} \rr(\lb, L, \lb, L)\\
\bb_a&=&\half \rr(e_a,\lb,\lb, L)\,,\qquad \underline{\a}_{ab}=\rr(\lb, e_a,\lb, e_b)
\eea
where $^\star \rr$ denotes the Hodge dual of $\rr$. 
\end{definition}
Observe that all tensors defined above are $\ptu$-tangent.

\begin{remark}\lab{ch4:remark:linkpart1}
Note that $\ab$ is the only null component which does not contain a contraction of $\rr$ with $L$. With the notation of Chapter \ref{part:yangmills} (see for instance \eqref{bootcurvatureflux}), we have:
$$\rr\c L=(\a, \b, \rho, \s, \bb).$$
\end{remark}

\begin{definition}\lab{ch4:def:decompositionk}
We decompose the symmetric traceless 2 tensor $k$ into the scalar $\d$, the $\ptu$-tangent 1-form $\epsilon$, and the $\ptu$-tangent symmetric 2-tensor $\eta$ as follows:
\be\lab{ch4:decompositionk}
\left\{\begin{array}{l}
k_{NN}=\d\\
k_{AN}=\kep_A\\
k_{AB}=\eta_{AB}.
\end{array}\right.
\ee
\end{definition}

The following  \textit{Ricci equations} can be easily derived from the definition  
of $T$, the fact that $L'$ is geodesic \eqref{ch4:def:L'1}, 
and the definition \eqref{ch4:chi} of the Ricci coefficients (see \cite{ChKl} p. 171): 
%They express  the covariant 
%derivatives $\dd$ of the null frame $(e_A)_{A=1,2}, e_3, e_4$ relative to itself.
\begin{alignat}{2}
&\dd_A e_4=\chi_{AB} e_B - \kep_{A} e_4, &\quad 
&\dd_A e_3=\chb_{AB} e_B + \kep_{A} e_3,\nn\\
&\dd_{4} e_4 = -\db   e_4, &\quad 
&\dd_{4} e_3= 2\zb_{A} e_A + \db   e_3, 
  \label{ch4:ricciform} \\
 &\dd_{3} e_4 = 2\z_{A}e_A +
(\d +n^{-1}\nab_Nn)  e_4,&\quad &\dd_{3} e_3 = 2\xib_{A}e_A -
(\d +n^{-1}\nab_Nn) e_3,\nn\\ &\dd_{4} e_A = \ddb_{4} e_A +
\zb_{A} e_4,&\quad &\dd_{3} e_A = \ddb_{3} e_A+\z_A e_3
+ \xib_A e_4,
\nn\\
&\dd_{B}e_A = \nabb_B e_A +\half \chi_{AB}\, e_3 +\half
 \chb_{AB}\, e_4\nn
\end{alignat}
where, $\ddb_{ 3}$, $\ddb_{ 4}$ denote the 
projection on $P_{t,u}$ of $\dd_3$ and $\dd_4$, $\nabb$
denotes the induced covariant derivative on $P_{t,u}$
and $\db, \kepb$ are defined by: 
\be\lab{ch4:newk}
\db=\d-n^{-1}N(n),\,\kepb_A=\kep_A-n^{-1}\nab_A n.
\end{equation}
Also,
\begin{align}
&\chb_{AB}=-\chi_{AB}-2k_{AB},\nn\\
&\zb_A = -\kepb_A,\label{ch4:etab}\\
&\xib_A = \kep_{A}+n^{-1} \nabb_A n-\z_A\nn.
\end{align}

\subsection{Null structure equations}\label{ch4:sec:nullstructure}

Below we  write down our main structure equations (see \cite{ChKl} chapter 7 or \cite{param3} for a proof). 
\begin{proposition}
The components $\trc,\, \hch,\,\z$ and the lapse
$b$ verify the following equations\footnote{which can be interpreted as
transport equations along the  null geodesics 
generated by $L$. Indeed  observe that if a $\ptu$-tangent tensor 
 $\Pi$ satisfies the homogeneous equation  $\ddb_4\Pi=0$ then $\Pi$
is parallel transported along null geodesics. }:
\begin{align}
&L(b) = - b\, \db  , \label{ch4:D4a}\\
&L(\trc) + \half (\trc)^2 = - |\hch|^2 -\db   \trc,
\label{ch4:D4trchi}\\
&\ddb_{4} \hch + \trc \hch = 
-\db   \hch - \a,
\label{ch4:D4chih}\\
&\ddb_{4}\z_A + \half (\trc)\z_A = -(\kepb_B +\z_B)\hch_{AB} -
\half \trc \kepb_A - \b_A,\label{ch4:D4eta}.
\end{align}
\label{ch4:proptransp}
\end{proposition}
\begin{remark}
Equation \eqref{ch4:D4trchi} is known as the Raychaudhuri equation in the relativity
literature.
\end{remark}

To obtain estimates for $\chi$, we may use the transport equations \eqref{ch4:D4trchi} \eqref{ch4:D4chih}. However, this does not allow us to get enough regularity. Instead, we follow \cite{ChKl} \cite{Kl-R2} \cite{Kl-R4} and consider \eqref{ch4:D4trchi} for $\trc$ together with an elliptic system of Hodge type for $\hch$.
\begin{proposition}
The expression $(\divb\hch)_A=\nabb^B \hch_{AB}$ 
verifies the following equation:
\be\label{ch4:Codaz}
(\divb \hch)_A + \hch_{AB}\kep_{B}=\half (\nabb_A \trc + \kep_{A} \trc) - 
\b_A.
\ee
\end{proposition}

See \cite{ChKl} chapter 7 or \cite{param3} for a proof.

Finally, we consider the control of $\z$ and $\lb\trc$. To this end, we follow again \cite{Kl-R2} \cite{Kl-R4}: we derive an elliptic system of Hodge type for $\z$ and a transport equation for $\lb\trc$.
\begin{proposition}
We have:
\be\lab{ch4:D3trc}
\lb(\trc)+\half\trchb\trc= 2\divb\z+(\d +n^{-1}\nab_Nn)\trc-\hch\c\hchb+2\z\c\z+2\r.
\ee
Also, the expressions  $\divb
\z=\nabb^B\z_B$ and \mbox{$\curlb \z=\in^{AB}\nabb_A\z_B$} 
verify the following equations:
\begin{align}
&\divb\,\z = \half\bigg(\mu +\half\trc\trchb +\hch\c\hchb -2|\z|^2\bigg)
 - \rho,\label{ch4:diveta}\\
&\curlb\,\z =  -\half\hch \wedge  \hchb+\s,
\label{ch4:curleta}   
\end{align}
where for $F, G$ symmetric traceless $\ptu$-tangent 2-tensors, we denote by $F\wedge G$ the tensor $F\wedge G_{AB}=\in_{AB}F_{AC}G_{BC}$. Finally, setting, 
\be\lab{ch4:eqmu}
\mu=\lb(\trc) - 
\big (\d +n^{-1}\nab_N n\big )\trc
\end{equation}
we find 
\begin{equation}
\begin{split}
L(\mu) + \trc \mu &=2(\zb-\z)\c\nabb\trc
-2\hch\c\Bigl (\nabb\widehat{\otimes}\z + \z\widehat{\otimes}\z  -\d\hch\Bigr )
\\& -\trc 
\bigg (2\divb\z+2\z\c\z+4 (\kep - \z)\c n^{-1} \nabb n -2\db(\d+n^{-1}\nab_Nn)+
4\rho \\&
-\half\trc\trchb+2|\kep|^2+3|\hch|^2+4\hch\c\etah-2|n^{-1}N(n)|^2\bigg).
\end{split}
\label{ch4:D4tmu}
\end{equation}
\end{proposition}

See \cite{Kl-R2} or \cite{param3} for a proof.

\subsection{Commutation formulas}\label{ch4:sec:commutationf}

We have the following useful commutation formulas (see \cite{ChKl} p. 159):
\begin{lemma}
Let $U_{\und{A}}$ be an m-covariant tensor tangent to the surfaces $P_{t,u}$.
Then,
\bea
\lab{ch4:comm1}\nabb_B \ddb_{4} U_{\und{A}} - \ddb_{4}\nabb_B U_{\und{A}} &=&
\chi_{BC} \nabb_C U_{\und{A}} - n^{-1} \nabb_B n \ddb_{4} U_{\und{A}}\\
&+& \sum_i (\chi_{A_i B} \kepb_C -
\chi_{BC}\kepb_{A_i} -\in_{A_i C}{}^*\b_B) U_{A_1..\Check{C}..A_m},\nn
\eea
\bea\lab{ch4:comm2}
\nn\nabb_B \ddb_{3} U_{\und{A}} - \ddb_{3}\nabb_B U_{\und{A}} &=&
\chb_{BC} \nabb_C U_{\und{A}} - \xib_B \ddb_{4} U_{\und{A}}- b^{-1}\nab_Bb \ddb_{3}U_{\und{A}}+ \sum_i (-\chi_{A_i B} \xib_{C} \\
&+&\chi_{BC}\xib_{A_i}  -\chb_{A_i B} \z_{C} +
\chb_{BC}\z_{A_i } 
+\in_{A_i C}{}^*\bb_B) U_{A_1..\Check{C}..A_m},
\eea
\bea
\lab{ch4:comm3}\ddb_{3}\ddb_{4} U_{\und{A}} - \ddb_{4}\ddb_{3} U_{\und{A}} &=&
 -\db  \ddb_{3} U_{\und{A}}
+ (\d + n^{-1} \nab_Nn)\ddb_{4} U_{\und{A}} +2(\z_{B}-\zb_B) \nabb_B U_{\und{A}} \\ 
&+& 2\sum_i (\zb_{A_i} \z_{C} -
\zb_{C}\z_{A_i} +\in_{A_i C}{}^*\s) U_{A_1..\Check{C}..A_m}.\nn
\eea
\end{lemma}

\subsection{Bianchi identities}

In view of the formulas on p. 161 of \cite{ChKl}, the Bianchi equations for $\a, \b, \rho, \s, \bb$ are:
\bea
\ddb_L(\b)&=&\divb\a-\db  \b+(2\kep-\kepb)\c\a\lab{ch4:bianc1}\\
\ddb_{\lb}(\b)&=&\nabb\r+(\nabb\s)^*+2\hch\c\bb+(\d+n^{-1}\nabla_Nn)\b+\xib\c\a+3(\z\r+{}^*\z\s)\lab{ch4:bianc1bis}\\
L(\rho)&=&\divb\b-\half\hchb\c\a+(\kep-2\kepb)\c\b\lab{ch4:bianc2}\\
\lb(\rho)&=&-\divb\bb-\half\hch\c\ab+2\xib\c\b+(\kep-2\z)\c\bb\lab{ch4:bianc3}\\
L(\s)&=&-\curlb\b+\half\hchb\,{}^*\a+(-\kep+2\kepb)\,{}^*\b\lab{ch4:bianc4}\\
\lb(\s)&=&-\curlb\bb-\half\hch\,{}^*\ab-2\xib\,{}^*\b+(\kep-2\z)\,{}^*\bb\lab{ch4:bianc5}\\
\ddb_L(\bb)&=&-\nabb\r+(\nabb\s)^*+2\hchb\c\b+\db\bb-3(\zb\r-{}^*\zb\s)\lab{ch4:bianc6}
\eea

\subsection{Main results}\lab{ch4:sec:mainres}

We introduce the $L^2$ curvature flux $\rf$ relative to the time foliation: 
\be\lab{ch4:curvflux}
\rf=\left(\norm{\a}^2_{\lh{2}}+\norm{\b}_{\lh{2}}^2+\norm{\r}_{\lh{2}}^2+\norm{\s}_{\lh{2}}^2+\norm{\bb}_{\lh{2}}^2\right)^{\half}.
\ee
In view of Remark \ref{ch4:remark:linkpart1}, we have $\rf=\norm{\rr\c L}_{\lh{2}}$. Thus, the we may rewrite the bootstrap assumptions of Chapter \ref{part:yangmills} on $\rr$ as:
$$\norm{\rr}_{\lsit{\infty}{2}}\leq M\ep,\,\sup_u\rf\leq M\ep.$$
To ease the notations, we drop the bootstrap constant $M$:
\be\lab{ch4:curvflux1}
\norm{\rr}_{\lsit{\infty}{2}}\leq\ep,\,\sup_u\rf\leq\ep.
\ee
The goal of this part is to control the geometry of the null hypersurfaces $\H_u$ of $u$ up to time $t=1$ when only assuming the smallness assumption \eqref{ch4:curvflux1}. 

\begin{remark}
In the rest of the chapter, all inequalities, except the ones of Theorem \ref{ch4:thregx1} below, hold for any $u$ with the constant in the right-hand side being independent of $u$. Thus, one may take the supremum in $u$ in these inequalities. To ease the notations, we do not explicitly write down the supremum in $u$ in these estimates. 
\end{remark}

$u$ is a solution to the eikonal equation $\gg^{\alpha\beta}\partial_\alpha u\partial_\beta u=0$ on $\mathcal{M}$ depending on a extra parameter $\o\in \S$. Now, for $u$ to be uniquely defined, we need to prescribe it on $\Si_0$ (i.e. at $t=0$). This issue has been settled in \cite{param1} (see also Chapter \ref{part:initialu}). From now on, we assume that $u$ is the solution to the eikonal equation $\gg^{\alpha\beta}\partial_\alpha u\partial_\beta u=0$ on $\mathcal{M}$ which is prescribed on $\Si_0$ as in \cite{param1}.

\begin{remark}
In the rest of the chapter, all inequalities hold for any $\o\in\S$ with the constant in the right-hand side being independent of $\o$. Thus, one may take the supremum in $\o$ everywhere. To ease the notations, we do not explicitly write down this supremum. 
\end{remark}

We define some norms on $\H_u$. For any $1\leq p\leq +\infty$ and for any tensor $F$ on $\H_u$, we have:
$$\norm{F}_{\lh{p}}=\left(\int_0^1dt\int_{P_{t,u}}|F|^p\dmt\right)^{\frac{1}{p}},$$
where $\dmt$ denotes the area element of $P_{t,u}$.
We also introduce the following norms:
$$\no(F)=\norm{F}_{\lh{2}}+\norm{\nabb F}_{\lh{2}} +\norm{\ddb_LF}_{\lh{2}},$$
$$\noo(F)=\no(F)+\norm{\nabb^2 F}_{\lh{2}} +\norm{\nabb\ddb_LF}_{\lh{2}}.$$
%The structure equations involve transport equations along the null geodesics generated by $L$ (see for example \eqref{ch4:D4a} and \eqref{ch4:D4trchi}). 
Let $x'$ a coordinate system on $\pou$. By transporting this coordinate system along the null geodesics generated by $L$, we obtain a coordinate system $(t,x')$ of $\H$. We define the following norms:
$$\norm{F}_{\xt{\infty}{2}}=\sup_{x'\in P_{0,u}}\left(\int_0^1 |F(t,x')|^2dt\right)^{\half},$$
$$\norm{F}_{\xt{2}{\infty}}=\normm{\sup_{0\leq t\leq 1}|F(t,x'))|}_{L^2(\pou)}.$$

The following theorem investigates the regularity of $u$ with respect to $(t,x)$:
\begin{theorem}\lab{ch4:thregx}
Assume that $u$ is the solution to the eikonal equation $\gg^{\alpha\beta}\partial_\alpha u\partial_\beta u=0$ on $\mathcal{M}$ such that $u$ is prescribed on $\Si_0$ as in \cite{param1}. Assume also that the estimate \eqref{ch4:curvflux1} is satisfied. Then, null geodesics generating $\H_u$ do not have conjugate points and distinct null geodesics do not intersect. Furthermore, the following estimates are satisfied:
\be\lab{ch4:estn}
\norm{n-1}_{L^\infty}+\norm{\nabla n}_{\tx{\infty}{2}}+\norm{\nabla^2n}_{\tx{\infty}{2}}+\norm{\nabla \dd_Tn}_{\tx{\infty}{2}}\lesssim\ep,
\ee
\be\lab{ch4:estk}
\no(k)+\norm{\ddb_{\lb}\kep}_{\lh{2}}+\norm{\lb(\d)}_{\lh{2}}+\norm{\kepb}_{\xt{\infty}{2}}+\norm{\db}_{\xt{\infty}{2}}\lesssim\ep,
\ee
\be\lab{ch4:estb}
\norm{b-1}_{L^\infty}+\noo(b)+\norm{\lb(b)}_{\xt{2}{\infty}}\lesssim\ep,
\ee
\be\lab{ch4:esttrc}
\norm{\trc}_{L^\infty}+\norm{\nabb\trc}_{\xt{2}{\infty}}+\norm{\lb\trc}_{\xt{2}{\infty}}
\lesssim\ep,
\ee
\be\lab{ch4:esthch}
\norm{\hch}_{\xt{2}{\infty}}+\no(\hch)+\norm{\ddb_{\lb}\hch}_{\lh{2}}
\lesssim\ep,
\ee
\be\lab{ch4:estzeta}
\norm{\z}_{\xt{2}{\infty}}+\no(\z)\lesssim\ep.
\ee
\end{theorem} 

We introduce the family of intrinsic Littlewood-Paley projections
$P_j$ which have been constructed in \cite{Kl-R6} using the heat flow 
on the surfaces $P_{t,u}$ (see also section \ref{sec:LP}). This allows us to state our second theorem 
which investigates the regularity of $\lb\lb\trc$ and $\ddb_\lb\z$.
\begin{theorem}\lab{ch4:thregx1}
Assume that $u$ is the solution to the eikonal equation $\gg^{\alpha\beta}\partial_\alpha u\partial_\beta u=0$ on $\mathcal{M}$ such that $u$ is prescribed on $\Si_0$ as in \cite{param1}. Assume also that the assumption \eqref{ch4:curvflux1} is satisfied. Then, there exists a function $\lambda$ in $L^2(\R)$ such that for all $j\geq 0$, we have:
\be\lab{ch4:estlblbtrc}
\norm{P_j\lb\lb\trc}_{\lh{2}}\lesssim 2^j\ep+2^{\frac{j}{2}}\lambda(u),
\ee
and
\be\label{ch4:estlbzeta}
\norm{P_j\ddb_{\lb}(\z)}_{\lh{2}}\lesssim \ep+2^{-\frac{j}{2}}\lambda(u).
\ee
\end{theorem}

The following theorem investigates the regularity with respect to the parameter $\o\in\S$.
\begin{theorem}\lab{ch4:thregomega}
Assume that $u$ is the solution to the eikonal equation $\gg^{\alpha\beta}\partial_\alpha u\partial_\beta u=0$ on $\mathcal{M}$ such that $u$ is prescribed on $\Si_0$ as in \cite{param1}. Assume also that the estimate \eqref{ch4:curvflux1} is satisfied. Then, we have the following estimates:
\be\lab{ch4:estNomega}
\norm{\po N}_{L^\infty}\lesssim 1,
\ee
\be\lab{ch4:estricciomega}
\norm{\dd\po N}_{\xt{2}{\infty}}+\norm{\po b}_{L^\infty}+\norm{\nabb\po b}_{\xt{2}{\infty}}+\norm{\po\chi}_{\xt{2}{\infty}}+\norm{\po\z}_{\xt{2}{\infty}}\lesssim \ep.
\ee
Furthermore, we have the following decomposition for $\hch$:
\be\lab{ch4:dechch}
\hch=\chi_1+\chi_2,
\ee
where $\chi_1$ and $\chi_2$ are two symmetric traceless $\ptu$-tangent 2-tensors satisfying: 
\be\lab{ch4:dechch1}
\no(\chi_1)+\norm{\ddb_{\lb}\chi_1}_{\lh{2}}+\norm{\po\chi_1}_{\tx{\infty}{2}}+\no(\chi_2)+\norm{\ddb_{\lb}\chi_2}_{\lh{2}}+\norm{\po\chi_2}_{\tx{\infty}{2}}\lesssim \ep
\ee
and for any $2\leq p<+\infty$, we have:
\be\lab{ch4:dechch2}
\norm{\chi_1}_{\tx{p}{\infty}}+\norm{\po\chi_2}_{\tx{p}{4_-}}+\norm{\po\chi_2}_{\lh{6_-}}+\norm{\nabb\po\chi_2}_{\lh{2}}\lesssim \ep,
\ee
where for any real number $a$, $a_-=a-\delta$ for any $\delta>0$.
\end{theorem}

\begin{remark}\lab{ch4:rem:dehch}
Notice from \eqref{ch4:dechch1} that $\chi_1$ and $\chi_2$ have at least the same regularity as $\hch$. Now, the point of the decomposition \eqref{ch4:dechch} is that both $\chi_1$ and $\chi_2$ have better regularity properties than $\hch$. Indeed, in view of \eqref{ch4:dechch2}, $\chi_1$ has better regularity with respect to $(t,x)$ while $\chi_2$ has better regularity with respect to $\o$.
\end{remark}

Next, the following theorem contains estimates for second order derivatives with respect to $\o$.
\begin{theorem}\lab{ch4:thregomega2}
Assume that $u$ is the solution to the eikonal equation $\gg^{\alpha\beta}\partial_\alpha u\partial_\beta u=0$ on $\mathcal{M}$ such that $u$ is prescribed on $\Si_0$ as in \cite{param1}. Assume also that the estimate \eqref{ch4:curvflux1} is satisfied. Then, we have the following estimates:
\be\lab{ch4:estNomega2}
\norm{\po^2 N}_{\xt{2}{\infty}}\lesssim 1,
\ee
\be\lab{ch4:estNomega2bis}
\norm{\ddb_{L}\Pi(\po^2 N)}_{\lh{2}}\lesssim \ep.
\ee
\be\lab{ch4:estricciomega2}
\norm{P_j\ddb_{\lb}\Pi(\po^2 N)}_{\tx{p}{2}}+\norm{P_j\Pi(\po^2\chi)}_{\tx{\infty}{2}}+\norm{P_j\Pi(\po^2\z)}_{\tx{p}{2}}\lesssim 2^j\ep,
\ee
where $p$ is any real number such that $2\leq p<+\infty$, and where $\Pi$ denotes the projection on $\ptu$-tangent tensors. 
\end{theorem}

Finally, we need to compare quantities evaluated at two angles $\o$ and $\nu$. The following decompositions are used in sections \ref{ch3:bissec:diagonal} and \ref{ch3:bissec:orthoangle}

\begin{theorem}\lab{ch4:cor:xx1bis}
Let $\o$ and $\nu$ in $\S$ such that $|\o-\nu|\les 2^{-\frac{j}{2}}$. Let $u=u(.,\o)$, $N=N(.,\o)$ and $N_\nu=N(.,\nu)$. For any $j\geq 0$, we have the following decomposition for $N-N_\nu$:
\be\lab{ch4:decNom}
2^{\frac{j}{2}}(N-N_\nu)=F^j_1+F^j_2
\ee
where the tensor $F^j_1$ does not depend on $\o$ and satisfies:
$$\norm{F^j_1}_{L^\infty}\les 1,$$
and where the tensor $F^j_2$ satisfies:
$$\norm{F^j_2}_{L^\infty_u\lh{2}}\les 2^{-\frac{j}{2}}.$$
We also have following decomposition for $\trc$:
\be\lab{ch4:dectrcom}
\trc=f^j_1+f^j_2
\ee
where the scalar $f^j_1$ does not depend on $\o$ and satisfies:
$$\norm{f^j_1}_{L^\infty}\les \ep,$$
and where the scalar $f^j_2$ satisfies:
$$\norm{f^j_2}_{L^\infty_u\lh{2}}\les \ep 2^{-\frac{j}{2}}.$$
\end{theorem}

Let us conclude this section by mentioning several ingredients of \cite{param3} that have been omitted here for the 
sake of simplicity: 
\begin{itemize}
\item estimates for the transport equations along $L$, and the elliptic systems of Hodge type on $\ptu$ involved in the null structure equations

\item embeddings on $\H_u$, $\Sigma_t$ and $\ptu$

\item geometric Littlewood-Paley projections and Besov spaces on $\Sigma_t$

\item control of the Gauss curvature of $\ptu$ 

\item Bochner inequalities on $\Sigma_t$ and $\ptu$ 

\item estimates for various commutator terms of the type: $[\dd_L, \nabb]$, $[\dd_{\lb}, \nabb]$, $[\dd_L, P_j]$, $[\dd_{\lb}, P_j]$, ...
\end{itemize}

\section{Regularity of the foliation with respect to $(t,x)$}

In this section, we outline the main ideas of the proof of Theorem \ref{ch4:thregx}. We assume the following bootstrap assumptions:
\be\lab{ch4:boot1}
\norm{n-1}_{\lh{\infty}}+\norm{b-1}_{\lh{\infty}}\leq\frac{1}{10},
\ee
\be\lab{ch4:boot2}
\norm{\nabla n}_{\tx{\infty}{2}}+\norm{\nabla^2n}_{\tx{\infty}{2}}+\norm{\nabla \dd_Tn}_{\tx{\infty}{2}}+\noo(b)+\norm{\lb(b)}_{\xt{2}{\infty}}\leq D\ep,
\ee
\be\lab{ch4:boot3}
\no(k)+\norm{\ddb_{\lb}\kep}_{\lh{2}}+\norm{\dd_{\lb}\d}_{\lh{2}}+\norm{\kepb}_{\xt{\infty}{2}}+\norm{\db}_{\xt{\infty}{2}}\leq D\ep,
\ee
\be\lab{ch4:boot4}
\norm{\trc}_{\lh{\infty}}+\norm{\nabb\trc}_{\xt{2}{\infty}}+\norm{\lb\trc}_{\xt{2}{\infty}}
\leq D\ep,
\ee
\be\lab{ch4:boot5}
\norm{\hch}_{\xt{2}{\infty}}+\no(\hch)+\norm{\ddb_{\lb}\hch}_{\lh{2}}
\leq D\ep,
\ee
\be\lab{ch4:boot6}
\norm{\z}_{\xt{2}{\infty}}+\no(\z)\leq D\ep,
\ee
where $D>0$ is a large enough constant. We will improve on these estimates.

\subsection{Non intersection of null geodesics on $\H_u$}\lab{ch4:sec:nointersec}

The control we obtain on the geometric quantities associated to our foliation is only valid as long as there are no conjugate points and null geodesics do not intersect. The goal of  this section is to prove that this holds at least until $t=1$. In addition to the bound \eqref{ch4:curvflux1} on the curvature tensor $\rr$ of $\gg$, we make the following regularity assumption on $\gg$. There exists a coordinate chart such that 
\be\lab{ch4:grossereg}
\norm{\gg}_{C^2(\mathcal{M})}\leq M,
\ee
where $M$ is a very large constant. 

\begin{remark}
The assumption \eqref{ch4:grossereg} is only used to prove the absence of caustic and that null geodesics do not intersect at least until $t=1$, which is a qualitative property. On the other hand, we only rely on the bound \eqref{ch4:curvflux1} on $\rr$ to prove the various quantitative bounds of Theorems \ref{ch4:thregx}, \ref{ch4:thregx1}, \ref{ch4:thregomega} and \ref{ch4:thregomega2}.
\end{remark}

For $(0,x)$ in $\Si_0$, recall the definition in Remark \ref{ch4:defnullgeod} of the null geodesic $\kappa_x(t)$. For all $0\leq t\leq 1$, let $\Phi_t:\Si\rightarrow\Sit$ defined by $\Phi_t(0,x)=\kappa_x(t)$. We have $\Phi_0(0,x)=(0,x)$ on $\Si_0$. We define $t_0\geq 0$ as the supremum of $0\leq t\leq 1$ such that $\Phi_t$ is bijective from $\Si_0$ to $\Sit$. 

\begin{remark}\label{ch4:rem:ctoutbon}
As long as $0\leq t<t_0$, there are no conjugate points and no distinct null geodesic intersections. Thus, we may assume that the u-foliation exists and satisfies the bounds \eqref{ch4:boot1}-\eqref{ch4:boot6} given by the bootstrap assumptions. Furthermore, we may assume the identity \eqref{ch4:defnullgeod1} for the null geodesics $\kappa_x(t)$.
\end{remark}

Our goal is to show that we have in fact $t_0=1$. We proceed in three steps (see \cite{param3} for the details):
\begin{itemize}
\item[\textit{Step 1.}] As noticed in Remark \ref{ch4:rem:ctoutbon}, the $L^\infty$ bound for $\trc$ given by \eqref{ch4:boot4} holds for $0\leq t<t_0$. Furthermore, using the Raychaudhuri equation \eqref{ch4:D4trchi} and the bound \eqref{ch4:grossereg}, we obtain the existence of a constant $\delta>0$ depending on $M$ such that the $L^\infty$ bound for $\trc$ given by \eqref{ch4:boot4} holds for $0\leq t<t_0+\delta$. This control for $\trc$ allows us to prove that there are no conjugate points on $0\leq t<t_0+\delta$. 

\item[\textit{Step 2.}] Next, we prove that $\Sigma_t=\cup_u\ptu$ for $0\leq t\leq t_0+\delta$ where $\delta>0$ is a constant depending on $M$. This requires the bound \eqref{ch4:grossereg}, and the control it induces on forward and backward light cones for small time intervals with a size depending on $M$.  
 
\item[\textit{Step 3.}] Assume now that $0<t_0<1$. In view of \textit{Step 1} and \textit{Step 2}, the only thing that can go wrong at $t=t_0$ is that two distinct null geodesics intersect in $\Sigma_{t_0}$. Assume by contradiction that this is indeed the case so that there exists $(0,x_1)\neq (0,x_2)$ two points in $\Si_0$ such that $\kappa_{x_1}(t_0)=\kappa_{x_2}(t_0)=(t_0,x_0)$. Since
$$\kappa'_{x_j}(t)=b(\kappa_{x_j(t)})^{-1}L_{\kappa_{x_j(t)}},\,\, j=1, 2,$$ 
in view of Remark \ref{ch4:defnullgeod}, the regularity of $b$ and $L$ yields $\kappa_{x_1}'(t_0)=\kappa_{x_2}'(t_0)$.    From the classical uniqueness result for ODEs, we deduce that $\kappa_{x_1}(t)=\kappa_{x_2}(t)$ for all $t$. In particular, taking $t=0$, we obtain $(0,x_1)=(0,x_2)$ which yields a contradiction. 
\end{itemize}

Finally, Steps 1, 2 and 3, yield $t_0\geq 1$. In particular, we have:
\be\lab{ch4:conclusionfoliation}
\begin{array}{l}
\textrm{On }0\leq t\leq 1, \textrm{ there are no conjugate points and no intersection of distinct}\\ 
\textrm{null geodesics. In particular, }u\textrm{ exists on }0\leq t\leq 1\textrm{ and the bootstrap}\\ 
\textrm{assumptions \eqref{ch4:boot1}-\eqref{ch4:boot6} hold. Furthermore, }\Sit=\cup_u\ptu\textrm{ for all }0\leq t\leq 1. 
\end{array}
\ee

\subsection{Lower bound on the volume radius of $\Sit$}\label{ch4:sec:lowerboundvolrad}

%\subsection{Coordinate systems on $\Sit$ and $\ptu$}

In this section, we prove the lower bound on the volume radius of $\Sit$ given by the estimate \eqref{bootvolumeradius}. We use the global coordinate system $x'=(x^1,x^2)$ on $\pou$  which has been constructed in \cite{param1} (see also Proposition \ref{gl0}). Transporting this coordinate system along the null geodesics generated by $L$, we obtain a coordinate system $x'$ of $\ptu$, which in particular satisfies
\be\lab{ch4:eq:coordchartbis}
(1-O(\ep))|\xi|^2\leq \gamma_{AB}(p)\xi^A\xi^B\leq (1+O(\ep))|\xi|^2, \qquad \mbox{uniformly for  all }
\,\, p\in\ptu,
\ee
where $\gamma$ is the metric induces by $\g$ on $\ptu$. We denote by $x'$ this global coordinate system on $\ptu$. 

Next, we obtain a global coordinate system on $\Sit$ as follows. First, recall from \eqref{ch4:conclusionfoliation} that $\Sit=\cup\ptu$ so that $u$ is defined on $\Sit$. To any  $p\in\Sit$, we associate the coordinates $(u(p),x'(p))$ where $u(p)$ is the value of the optical function $u$ at $p$, and $x'(p)$ are the coordinate of $p$ in the coordinate system of $\ptu$. In this coordinate system, the metric $g_t$ on $\Sit$ (i.e. the restriction of $\gg$ on $\Sit$) takes the following form:
\be\label{ch4:eq:coordchartsit0}
g_t=\left(\begin{array}{ll}
b^{-2} & 0\\
0& \ga
\end{array}\right),
\ee
where $\ga$ is the induced metric on $\ptu$. Together with the estimate \eqref{ch4:boot1} for $b$ and \eqref{ch4:eq:coordchartbis} for $\ga$, we obtain the following lower bound on the volume radius of $\Sit$ at scales $\leq 1$:
\be\label{ch4:volradbound}
r_{vol}(\Sit,1)\geq \frac{1}{4},
\ee
which is the estimate \eqref{bootvolumeradius}.

\subsection{Estimates for the second fundamental form $k$ and the lapse $n$}\label{ch4:sec:estn}

We first estimate $k$ on $\Sit$. $k$ satisfies the following symmetric Hodge system on $\Sit$:
\be\lab{ch4:eqksit}
\left\{\begin{array}{l}
\textrm{curl} k_{ij}={}^*\rr_{\mu i\nu j}T^\mu T^\nu,\\
\nabla^jk_{ij}=0,\\
\textrm{tr}k=0,
\end{array}\right.
\ee
where curl$k_{ij}=\frac{1}{2}(\in_i^{lm}\nabla_lk_{mj}+\in_j^{lm}\nabla_lk_{mi})$ and tr$k=g^{ij}k_{ij}$. Using an elliptic estimate for the Hodge system \eqref{ch4:eqksit}, we easily obtain:
\be\lab{ch4:eqksit3}
\norm{\nabla k}_{\lsit{\infty}{2}}\lesssim \ep.
\ee

Recall from \eqref{eqlapsen} that the lapse $n$ satisfies the following elliptic equation on $\Sit$:
\be\lab{ch4:lapsen1bis}
\Delta n=|k|^2n.
\ee
Using \eqref{ch4:lapsen1bis} and \eqref{ch4:eqksit3}, together with elliptic estimates on $\Sit$, we improve the estimate for $n$ in the bootstrap assumptions \eqref{ch4:boot1} \eqref{ch4:boot2}. We also prove the following estimate which is needed for the estimate \eqref{bootn}
\be\label{ch4:estnmoreinvolved}
\norm{\nabla n}_{L^\infty(\MM)}\lesssim \ep.
\ee
Using \eqref{ch4:lapsen1bis} and \eqref{ch4:eqksit3} together  with  the Sobolev embedding on the three dimensional riemannian manifold $\Sit$ yields $\Delta n\in \lsit{\infty}{3}$. Together with elliptic estimates, this implies $\nabla^2n\in\lsit{\infty}{3}$, and thus $\nabla n$ misses  to be in $L^\infty(\MM)$ by a log divergence.  However, one can overcome this loss by exploiting the Besov improvement with respect to the Sobolev embedding on $\Sit$. This requires to introduce a geometric Littlewood-Paley theory on $\Sit$\footnote{Note that we use a geometric construction based on the heat flow on $\Sigma_t$ since we don't have enough regularity for the metric in order to use a coordinate dependent Littlewood-Paley decomposition}. We refer the reader to  section 4.4 in \cite{param3} for the details.

Finally, we estimate $k$ on $\HH_u$. To this end, we use the decomposition of $k$ \eqref{ch4:decompositionk} in $\d, \kep$ and $\eta$, and obtain a Hodge system for $\d, \kep$ and $\eta$ on $\HH_u$. This allows us to derive the following estimate
\be\lab{ch4:hodgk8}
\no(k)\les\no(\eta)+\no(\kep)+\no(\d)\lesssim \ep.
\ee
Deriving an estimate for $T(\d)$ and $\ddb_T\kep$, together with \eqref{ch4:hodgk8}, then yields
\be\lab{ch4:hodgk10}
\norm{\dd_{\lb}\d}_{\li{\infty}{2}}+\norm{\ddb_{\lb}\kep}_{\li{\infty}{2}}\lesssim \ep.
\ee

\subsection{Time foliation versus geodesic foliation}\label{ch4:sec:tfvsgf}

While we work with a time foliation, we recall that the estimates corresponding to the bootstrap assumptions on $\chi$ and $\z$ have already been proved in the context of a geodesic foliation in \cite{Kl-R4} \cite{Kl-R6} \cite{Kl-R5}. One may reprove these estimates by adapting the proofs to the context of a time foliation. However, this would be rather lengthy and we suggest a more elegant solution which consists in translating certain estimates from the geodesic foliation to the time foliation, and in obtaining directly the rest of the estimates. More precisely, we wish to obtain the $L^\infty$ bound from $\trc$, and the trace bounds for $\hch$ and $\z$ by exploiting the corresponding estimates in the geodesic foliation. We will obtain the trace bounds for $\d$ and $\kep$ by reducing to estimates in the geodesic foliation in section \ref{ch4:sec:tracenormk}. Finally, these trace bounds and the null structure equations will allow us to get all the remaining estimates in section \ref{ch4:sec:remainest}. 
We start by recalling some of the results obtained in the context of the geodesic foliation in \cite{Kl-R4} \cite{Kl-R6} \cite{Kl-R5}.

\subsubsection{The case of the geodesic foliation}\lab{ch4:sec:defgeodfol}

Recall that $L'=-\gg^{\a\b}\pr_\b u \pr_\a$ is the geodesic null generator of $\H_u$. Let $s$ denote its affine parameter, i.e. $L'(s)=1$. We denote by $P'_{s,u}$  the level surfaces of $s$ in $\H_u$. 

\begin{definition} A  null frame   $e'_1,e'_2,e'_3,e'_4$ at a point $p\in P'_{s,u}$ consists,
in addition to $e'_4=L'$, of {\sl  arbitrary  orthonormal}  vectors  $e'_1,e'_2$ tangent
to $P'_{s,u}$ and the unique vectorfield $e'_3=\lb'$ satisfying the relations:
$$\gg(e'_3,e'_4)=-2,\,\gg(e'_3,e'_3)=0,\,\gg(e'_3,e'_1)=0,\,\gg(e'_3,e'_2)=0.$$
\end{definition}

\begin{definition}[\textit{Ricci coefficients in the geodesic foliation}]
 Let  $e'_1,e'_2,e'_3,e'_4$ be a null frame on
$P'_{s,u}$ as above.   The following tensors on  $P'_{s,u}$ 
\be\label{ch4:chip}
\begin{array}{ll}
\chi'_{AB}=<\dd_{e'_A} e'_4,e'_B>, & \chb'_{AB}=<\dd_{e'_A} e'_3,e'_B>,\\
\z'_{A}=\half <\dd_{e'_A} e'_4,e'_3> &
\end{array}
\ee
are called the Ricci coefficients associated to the geodesic foliation.

We decompose $\chi'$ and $\chb'$ into
their  trace and traceless components.
\begin{alignat}{2}
&\trc' = \gg^{AB}\chi'_{AB},&\quad &\trchb' = \gg^{AB}\chb'_{AB},
\label{ch4:trchip}\\
&\hch'_{AB}=\chi'_{AB}-\half \trc' \gg_{AB},&\quad 
&\hchb'_{AB}=\chb'_{AB}-\half \trchb' \gg_{AB}.
\label{ch4:chihp} 
\end{alignat}
\end{definition}

\begin{definition}
The null components of the curvature tensor
$\rr$ of the space-time metric $\gg$ in the geodesic foliation are given
by:
\bea
\a'_{ab}&=&\rr(L',e'_a,L', e'_b)\,,\qquad \b'_a=\half \rr(e'_a, L',\lb', L') ,\\ \r' &=&\frac{1}{4}\rr(\lb', L', \lb',
L')\,,
\qquad
\s'=\frac{1}{4}\, ^{\star} \rr(\lb', L', \lb', L')\\
\bb'_a&=&\half \rr(e'_a,\lb',\lb', L')\,,\qquad \underline{\a'}_{ab}=\rr(\lb', e'_a,\lb', e'_b)
\eea
where $^\star \rr$ denotes the Hodge dual of $\rr$. 
\end{definition}

We now recall the main estimates obtained in \cite{Kl-R4} \cite{Kl-R6} \cite{Kl-R5}. We have:
\be\lab{ch4:estgeodesicfol}
\norm{\trc'}_{\lh{\infty}}+\norm{\hch'}_{L^{2}_{x'}L^{\infty}_s}+\norm{\z'}_{L^{2}_{x'}L^{\infty}_s}\lesssim\ep
\ee
and 
\be\lab{ch4:estgeodesicfolbis}
\norm{\chb'}_{L^2_{x'}L^\infty_s}+\no'(\chi')+\no'(\z')\lesssim\ep,
\ee
where the norm $\no'$ is given by 
$$\no'(F)=\norm{F}_{\lh{2}}+\norm{\nabb' F}_{\lh{2}} +\norm{\ddb_{L'}F}_{\lh{2}}.$$

\begin{remark}\lab{ch4:remark:equivnorm}
Note that the norm $\lh{\infty}$ does not depend on the particular foliation. Now, this is also the case for  the trace norm $L^{2}_{x'}L^{\infty}_s$. Indeed, recall the definition of the null geodesic $\kappa_x$ in Remark \ref{ch4:defnullgeod}. Then, we have:
\bee
\norm{F}_{\xt{\infty}{2}}^2=\sup_{(0,x)\in\Si_0}\int_0^1|F(\kappa_x(t))|^2dt=\sup_{(0,x)\in\Si_0}\int_0^1|F(\kappa_x(s))|^2n^{-1}b^{-1}ds\sim \norm{F}_{L^{\infty}_{x'}L^{2}_s}^2
\eee
where we used the fact that $\frac{dt}{ds}=n^{-1}b^{-1}$ and the fact that $nb\sim 1$ by the bootstrap assumption \eqref{ch4:boot1}.
\end{remark}

In the next section, we will obtain the estimates corresponding to \eqref{ch4:estgeodesicfol} in the time foliation. For now, we conclude this section by recalling the definition and some properties of the Besov spaces constructed in \cite{Kl-R4} \cite{Kl-R6} \cite{Kl-R5}. For $P'_{s,u}$-tangent tensors $F$ on $\H_u$, $0\leq a\leq 1$,  we introduce the Besov norms:
\bea
\|F\|_{{\BB'}^a}&=&\sum_{j\geq 0} 2^{ja} \sup_{0\leq s\leq 1}\|
P'_jF\|_{L^2(P'_{s,u})} +  \sup_{0\leq s\leq 1}\|
P'_{<0}F\|_{L^2(P'_{s,u})}    ,\label{ch4:eq:Besovnorm}\\
\|F\|_{{\PP'}^a}&=&\sum_{j\geq 0} 2^{ja} \| P'_jF\|_{L^2(\H_u)} +
\| P'_{<0}F\|_{L^2(\H_u)}\label{ch4:eq:Penrosenorm}
\eea
where $P'_j$ are the geometric Littlewood-Paley projections on the 2-surfaces $P'_{s,u}$. Using the definition of these Besov spaces, we have (see \cite{Kl-R4} \cite{Kl-R6} \cite{Kl-R5})
\be\lab{ch4:estgeodesicfolter}
\norm{\chb'}_{{\BB'}^0}\lesssim\ep.
\ee
Furthermore, we have for scalar functions on  $\H_u$ (see \cite{Kl-R4} section 5):
\be\lab{ch4:eq:Besov-Sobolev-BB0}
\|f\|_{\lh{\infty}}\lesssim \|f\|_{{\BB'}^1}\lesssim \|f\|_{L_s^\infty L_{x'}^2}+\|\nabb' f\|_{{\BB'}^0}.
\end{equation}
Finally, we have the following version of the sharp classical trace theorem (see Corollary 4.21 in \cite{param3} for a proof). 

\begin{proposition} Assume $F$ is an $P'_{s,u}$-tangent tensor which
admits a decomposition
of the form, $\nabb' F=A\ddb_{L'}P+E$.
Then,
\be\lab{ch4:eq:funny-classical-trace}
\|F\|_{L_{x'}^\infty L_s^2}\lesssim \no'(F)+\no'(P)(\norm{A}_{L^\infty}+\norm{\nabb'A}_{L^2_{x'}L^\infty_s}+\norm{\ddb_{L'}A}_{L^2_{x'}L^\infty_s})
+\|E\|_{{\PP'}^0}.
\end{equation}
\label{ch4:corr:funny-classical-trace}
\end{proposition}

\subsubsection{Estimates in the time foliation}

In this section, we obtain the $L^\infty$ bound for $\trc$, and the trace bounds for $\hch$ and $\z$ by relying on the corresponding estimates in the geodesic foliation \eqref{ch4:estgeodesicfol}. We start by establishing the relation between the Ricci coefficients in the time and in the geodesic foliation. Recall from \eqref{ch4:it2} that $L=bL'$. Since $(e_1, e_2)$ and $(e'_1, e'_2)$ are both orthonormal vectors in the tangent space of $\H_u$ which are both orthogonal to $L$, we may chose these vectors such that there is a tensor $F'$ on $P'_{s,u}$ satisfying:
$$e_A=e'_A+F'_AL',\, A=1, 2.$$
We then easily express $\lb$ in the frame $(L', \lb', e_A')$. Finally, we have the following relations:
\be\lab{ch4:comptg}
\begin{array}{l}
L=bL',\\
e_A=e'_A+F'_AL',\, A=1, 2,\\
\lb=b^{-1}\lb'+2b^{-1}F'_Ae'_A+b^{-1}|F'|^2L'.
\end{array}
\ee

Next, using the definition \eqref{ch4:chi} and \eqref{ch4:chip} of the Ricci coefficients respectively in the time and geodesic foliation, and the identities \eqref{ch4:comptg}, we easily obtain
\be\lab{ch4:comptg0}
\chi=b\chi',\,\trc=b\trc',\,\hch=b\hch',\, \z_A=\z'_A+\chi'_{AC}F'_C.
\ee
\eqref{ch4:comptg0} together with the bootstrap assumption \eqref{ch4:boot1} and the estimate \eqref{ch4:estgeodesicfol} yields:
\be\lab{ch4:comptg1}
\begin{array}{l}
\norm{\trc}_{\lh{\infty}}\leq\norm{b}_{\lh{\infty}}\norm{\trc'}_{\lh{\infty}}\lesssim\ep,\\
\norm{\hch}_{\xt{\infty}{2}}\leq\norm{b}_{\lh{\infty}}\norm{\hch'}_{L^{2}_{x'}L^{\infty}_s}\lesssim\ep,\\
\norm{\z}_{\xt{\infty}{2}}\lesssim \norm{\z'}_{L^\infty_{x'}L^2_s}+\norm{\chi'}_{L^\infty_{x'}L^2_s}\norm{F'}_{L^\infty}\lesssim\ep+\ep\norm{F'}_{L^\infty},
\end{array}
\ee
where we have used the fact that the trace norms $L^{2}_{x'}L^{\infty}_t$ and $L^{2}_{x'}L^{\infty}_s$ are equivalent by Remark \ref{ch4:remark:equivnorm}. 

In view of the trace estimate for $\z$ given by \eqref{ch4:comptg1}, we need to estimate $\norm{F'}_{L^\infty}$. To this end, we  estimate $\nabb'F'$. Using the definition \eqref{ch4:chi} of $\chb$ and \eqref{ch4:chip} of $\chb'$, and the identities \eqref{ch4:comptg}, we obtain:
$$\gg(\dd_{e'_A}F',e'_B)= -\half\chb'_{AB}+\cdots,$$
where we only kept the main term. Together with the estimate for $\chb'$ \eqref{ch4:estgeodesicfolter}, this yields
$$\norm{\nabb'F'}_{{\BB'}^0}\lesssim D\ep$$
which together with \eqref{ch4:eq:Besov-Sobolev-BB0} implies:
\be\lab{ch4:comptg15}
\norm{F'}_{L^\infty}\lesssim D\ep.
\ee
In particular, \eqref{ch4:comptg1} and \eqref{ch4:comptg15} imply:
\be\lab{ch4:comptg16}
\norm{\z}_{\xt{\infty}{2}}\lesssim \ep.
\ee
Note that \eqref{ch4:comptg1} and \eqref{ch4:comptg16} are improvements of the corresponding estimates in the bootstrap assumptions \eqref{ch4:boot4}-\eqref{ch4:boot6}.

\subsection{Trace norm bounds for $\db$ and $\kepb$}\lab{ch4:sec:tracenormk}

The goal of this section is to improve the estimate for $\norm{\db}_{\xt{\infty}{2}}$ and $\norm{\kepb}_{\xt{\infty}{2}}$ given by the bootstrap assumption \eqref{ch4:boot3}, where $\db$ and $\kepb$ are defined in \eqref{ch4:newk}. Let us first define $k_{LL}$ and $k_{LA}$:
\be\lab{ch4:tracek}
k_{LL}=-\gg(\dd_LT,L),\,k_{LA}=-\gg(\dd_LT,e_A),\,A=1,2.
\ee
Then, using in particular the definition \eqref{ch4:newk}, we have:
\be\lab{ch4:tracek1}
\db = k_{LL}\textrm{ and }\kepb_{A}=k_{LA}.
\ee
We also define $k_{L'L'}$ and $k_{L'A}$:
\be\lab{ch4:tracek3}
k_{L'L'}=-\gg(\dd_{L'}T,L'),\,k_{L'A}=-\gg(\dd_{L'}T,e'_A),\,A=1,2.
\ee
Then, the relations \eqref{ch4:comptg} between $L, e_1, e_2$ and $L', e'_1, e'_2$ together with the definitions \eqref{ch4:tracek} and \eqref{ch4:tracek3} yield:
\be\lab{ch4:tracek4}
k_{LL}=b^2k_{L'L'}\textrm{ and }k_{LA}=bk_{L'A}+bF'_Ak_{L'L'}.
\ee
Thus, \eqref{ch4:tracek1} and \eqref{ch4:tracek4} imply:
\bea
\lab{ch4:tracek5}\norm{\db  }_{\xt{\infty}{2}}&\lesssim& \norm{bk_{L'L'}}_{L^\infty_{x'}L^2_s}\lesssim \norm{k_{L'L'}}_{L^\infty_{x'}L^2_s}\\
\nn\norm{\kepb}_{\xt{\infty}{2}}&\lesssim& \norm{bk_{L'A}}_{L^\infty_{x'}L^2_s}+\norm{bF'_Ak_{L'L'}}_{L^\infty_{x'}L^2_s}\lesssim \norm{k_{L'L'}}_{L^\infty_{x'}L^2_s}+\norm{k_{L'A}}_{L^\infty_{x'}L^2_s}
\eea
where we used the bootstrap assumption \eqref{ch4:boot1}, the $L^\infty$ bound for $F'$ \eqref{ch4:comptg15} and Remark \ref{ch4:remark:equivnorm}.

In view of \eqref{ch4:tracek5}, it is enough to bound the trace norms $\norm{k_{L'L'}}_{L^\infty_{x'}L^2_s}$ and $\norm{k_{L'A}}_{L^\infty_{x'}L^2_s}$. To this end, we would like to apply  the trace estimate \eqref{ch4:eq:funny-classical-trace}, which requires to show that $\nabb'k_{L'L'}$ and $\nabb'k_{L'A}$ admit a decomposition of the form, $A\ddb_{L'}P+E$. We only discuss the estimate for $k_{L'L'}$, and we refer the reader to \cite{param3} for $k_{L'A}$. We have:
$$\nabb'_{e'_A}k_{L'L'}=-\dd_{e'_A}\gg(\dd_{L'}T,L')=-\gg(\dd_{e'_A}\dd_{L'}T,L')-\gg(\dd_{L'}T,D_{e'_A}L').$$
Introducing the commutator term $[\dd_{e'_A},\dd_{L'}]$, and decomposing the corresponding component of $\rr$, we obtain
\be\label{ch4:tracek11}
\nabb'_{e'_A}k_{L'L'}=-b^{-1}F'_B\a'_{AB}+\cdots,
\ee
where we only kept a typical term for simplicity. Relying on the Bianchi identities, the following decomposition for $\a'$ was obtained in \cite{Kl-R4}:
\be\lab{ch4:tracek12}
\a'=\ddb_{L'}(P)+E,
 \ee
where $P=\mathcal{D'}_2^{-1}\b'$, and 
\be\lab{ch4:tracek13}
\no'(P)+\norm{E}_{{\PP'}^0}\lesssim\ep.
\ee
Together with \eqref{ch4:tracek11}, we obtain a decomposition of the following form
\be\lab{ch4:tracek15}
\nabb'k_{L'L'}=A_1\ddb_{L'}P_1+E_1,
\ee
where 
\be\lab{ch4:tracek16}
\norm{A_1}_{L^\infty}+\norm{\nabb' A_1}_{L^2_{x'}L_s^\infty}+\norm{\ddb_{L'}A_1}_{L^2_{x'}L_s^\infty}+\no'(P_1)+\norm{E_1}_{{\PP'}^0}\lesssim\ep.
\ee
Using \eqref{ch4:tracek15}, \eqref{ch4:tracek16} and the trace estimate \eqref{ch4:eq:funny-classical-trace}, we deduce
$$\norm{k_{L'L'}}_{L^\infty_{x'}L^2_s}\lesssim\ep,$$
which together with \eqref{ch4:tracek5} allows us to improve the estimate for $\norm{\db}_{\xt{\infty}{2}}$ and $\norm{\kepb}_{\xt{\infty}{2}}$ given by the bootstrap assumption \eqref{ch4:boot3}.

\subsection{Remaining estimates for $\trc$, $\hch$, $\z$ and $b$}\lab{ch4:sec:remainest}

In order to improve the remaining estimates in the bootstrap assumptions \eqref{ch4:boot1}-\eqref{ch4:boot6}, we use the null structure equation of section \ref{ch4:sec:nullstructure}, which consists of transport equations along $L$ and Hodge systems on $\ptu$. We refer the reader to section 4.8 in \cite{param3}, where using the $L^\infty$ bound of $\trc$, the trace estimates for $\hch, \db$ and $\kepb$, and the estimates for the lapse $n$, we easily obtain the remaining estimates. Thus, there exists a universal constant $D>0$ such that \eqref{ch4:boot1}-\eqref{ch4:boot6} hold. This yields \eqref{ch4:estn}-\eqref{ch4:estzeta} which concludes the proof of Theorem \ref{ch4:thregx}.

\section{An estimate for \underline{L}\underline{L}tr$\chi$}

In this section, we outline the main ideas of the proof of Theorem \ref{ch4:thregx1}. Let $\mu_1=b\lb(\mu)$. Then, we first derive a transport equation for $\mu_1$, and a Hodge system for $\ddb_{\lb}\z$. For simplicity, we only discuss the transport equation for $\mu_1$. We differentiate the transport equation \eqref{ch4:D4tmu} satisfied by $\mu$ with respect to $\lb$ and multiply it by $nb$. We also use commutator formulas of section \ref{ch4:sec:commutationf}, the Bianchi identity \eqref{ch4:bianc3} for $\rho$, the curvature bound \eqref{ch4:curvflux1} and the estimates \eqref{ch4:estn}-\eqref{ch4:estzeta} obtained in Theorem \ref{ch4:thregx}. We obtain 
\bee
nL(\mu_1)+n\trc\mu_1&=&-2bn\ddb_{\lb}(\z)\c\nabb\trc
-2bn\hch\c\bigg(\nabb\widehat{\otimes}\ddb_{\lb}(\z)+b^{-1}\nabb b\ddb_{\lb}(\z)+2\ddb_{\lb}\z\widehat{\otimes}\z\bigg)\\
&&+2n\trc bn^{-1}\nabb n\c\ddb_{\lb}(\z)+\divb(F_1)+f_2,
\eee
where the tensor $F_1$ and the scalar $f_2$ satisfy
$$\norm{F_1}_{\lh{2}}+\norm{f_2}_{\lh{1}}\lesssim\ep.$$
This yields:
\be\lab{ch4:lbt2}
\norm{P_j(\mu_1)}_{\lh{2}}\les 2^{\frac{j}{2}}\lambda(u)\ep+2^j\ep+\normm{P_j\left(\int_0^t(bn\hch\c(\nabb\widehat{\otimes}\ddb_{\lb}(\z))d\tau\right)}_{\lh{2}}+\cdots.
\ee
Here, the term $2^j\ep$ comes from the estimate for $F_1$ and $f_2$ together with Bernstein and the finite band property for $P_j$, and the term $2^{\frac{j}{2}}\lambda(u)\ep$ comes from the initial data term for the transport equation - i.e. $\mu_1$ at $t=0$ which is estimated in \cite{param1} - together with Bernstein for $P_j$. We have only kept one typical term in the right-hand side of \eqref{ch4:lbt2} for the sake of simplicity. 

In view of the desired estimate \eqref{ch4:estlbzeta}, we are bootstrapping an estimate of the type 
$$\norm{P_j\ddb_{\lb}(\z)}_{\lh{2}}\lesssim D\ep+D 2^{-\frac{j}{2}}\lambda(u)$$
for some large enough bootstrap constant $D$, and where $\lambda$ is a function in $L^2(\mathbb{R})$. Thus, 
estimating directly the term in the right-hand side of \eqref{ch4:lbt2} would yield an upper bound of the type  
\be\lab{jfkwaiting}
\sum_{l,q}2^j2^{-\frac{|q-l|}{2}}\gamma^{(1)}_q\gamma^{(2)}_l,\textrm{ where }\gamma^{(1)}_q\in\ell^2(\NNN)\textrm{ and }\gamma^{(2)}_l\in\ell^\infty(\NNN)
\ee
which is not summable.  Instead, we rely on the following decomposition for $bn\hch$:
\be\lab{jfkwainting1}
\nabb (bn\hch)=\ddb_{nL}P+E
\ee
where $P$, $E$ are $\ptu$-tangent tensors, and $P$, $E$ satisfy:
$$\no(P)+\norm{E}_{\PP^0}\les\ep.$$

\begin{remark}
A similar decomposition has been proved in the geodesic foliation in \cite{Kl-R4}, and adapted to the time foliation 
in the spirit of section \ref{ch4:sec:tfvsgf}. In order to obtain \eqref{jfkwainting1}, we use the fact that the proof in the geodesic foliation relies on a specific structure of certain commutators and of the Bianchi identities, which can be recovered in the time foliation. We refer to \cite{param3} for the details.
\end{remark}

Using the decomposition\eqref{jfkwainting1}, we decompose the term in the right-hand side of \eqref{ch4:lbt2} in a sum of two terms which are estimated as follows (see \cite{param3} for the details):
\begin{itemize}
\item For the term involving $\ddb_{nL}P$, we integrate by parts in $\ddb_{nL}$, and consider the term where the $L$ derivative falls on $\ddb_{\lb}\z$. Differentiating the transport equation \eqref{ch4:D4eta} satisfied by $\z$ with respect to $\ddb_{\lb}$, commutators formula, and the Bianchi identity \eqref{ch4:bianc1bis}, we obtain
$$\ddb_{nL}\ddb_{\lb}\z=\ddb_{nL}\mathcal{D}^0(\bb)+\cdots$$ 
for some elliptic operator of order 0 $\mathcal{D}^0$ on $\ptu$. We then integrate by parts the $L$ derivatives, and  obtain for this term an upper bound of the type
$$2^j\no(P)\norm{\bb}_{\lh{2}}+\cdots.$$
Then, using the estimate for $P$ and the the curvature bound \eqref{ch4:curvflux1} for $\bb$, this is enough to bound  the term involving $\ddb_{nL}P$ in the right-hand side of \eqref{ch4:lbt2}.

\item For the term involving $E$, we we rely on the Besov improvement for $E$, and we derive an upper bound of the form 
$$\sum_{l,q}2^j2^{-\frac{|q-l|}{2}}\gamma^{(1)}_q\gamma^{(2)}_l,\textrm{ where }\gamma^{(1)}_q\in\ell^1(\NNN)\textrm{ and }\gamma^{(2)}_l\in\ell^\infty(\NNN),$$
which is summable unlike \eqref{jfkwaiting}. This is enough to bound the term involving $E$ in the right-hand side of \eqref{ch4:lbt2}.
\end{itemize}

\begin{remark}
The reader may wonder why the estimate for \underline{L}\underline{L}tr$\chi$ is not better than $L^2$ with respect to the variable $t$ - instead of $L^\infty$ as one should expect since we rely on a transport equation for the  corresponding quantity $\mu_1$. The reason is that boundary terms arise from several integration by parts in $\ddb_L$ in the course of the proof. The point is that we do not have better estimates than $L^2$ with respect to the variable $t$ for these terms. 
\end{remark}

\section{Regularity of the foliation with respect to $\omega$}

\subsection{First order derivatives with respect to $\omega$}\label{ch4:sec:firstomegader}

In this section, we outline the main ideas of the proof of Theorem \ref{ch4:thregomega}. Let us first explain how to control $\po\trc$. Differentiating the  Raychaudhuri equation \eqref{ch4:D4trchi} with respect to $\o$, we obtain
$$L(\po\trc)=[L, \po]\trc+\cdots.$$
Now, we have 
$$[L, \po]\trc=-\po N(\trc)=-\nabb_{\po N}\trc$$
where we used the fact that $\gg(N, N)=1$, which differentiated with respect to $\o$ implies that $\po N$ is a vectorfield tangent to $\ptu$. Thus, we obtain
$$L(\po\trc)=-\nabb_{\po N}\trc+\cdots$$
which together with the estimate \eqref{ch4:esttrc} for $\nabb\trc$ immediately yields
$$\norm{\po\trc}_{\xt{2}{\infty}}\les\ep.$$

\begin{remark}
In view of the commutator
$$[L, \po]=-\nabb_{\po N},$$
a derivative with respect to $\o$ has essentially the same regularity as a $\nabb$-derivative. 
\end{remark}

The estimates in \eqref{ch4:estNomega} and \eqref{ch4:estricciomega} are obtained in the same way, i.e. by differentiating 
the ricci equations \eqref{ch4:ricciform} and the transport equations \eqref{ch4:D4a} \eqref{ch4:D4trchi} \eqref{ch4:D4chih} and \eqref{ch4:D4eta} with respect to $\o$, computing the commutators $[\po, \dd_L]$ and $[\po, \ddb_L]$, and estimating the corresponding transport equations (see \cite{param3} for the details). 

Next, let us explain how to derive the decomposition \eqref{ch4:dechch} for $\hch$:
$$\hch=\chi_1+\chi_2,$$
where $\chi_1$ and $\chi_2$ are two symmetric traceless $\ptu$-tangent 2-tensors satisfying the estimates \eqref{ch4:dechch1} and \eqref{ch4:dechch2}. Recall the Codazzi type equation \eqref{ch4:Codaz} satisfied by $\hch$:
$$\divb \hch =\half \nabb \trc  - \b+\cdots.$$
This is an elliptic system on $\ptu$, and we may write formally
$$\hch= \half\mathcal{D}^{-1}\nabb\trc-\mathcal{D}^{-1}\b+\cdots,$$
where $\mathcal{D}^{-1}$ is a pseudodifferential operator of order -1 on $\ptu$. This allows us to define 
$\chi_1$ and $\chi_2$ as
$$\chi_1=\half\mathcal{D}^{-1}\nabb\trc+\cdots\textrm{ and }\chi_2=-\mathcal{D}^{-1}\b.$$
The estimate \eqref{ch4:dechch1} corresponds to the estimate \eqref{ch4:esthch} for $\hch$ and is prove similarly, so we 
focus on the estimate \eqref{ch4:dechch2}. For the sake of clarity, we only explain why, compared to $\hch$, $\chi_1$ has better regularity with respect to $(t,x)$ while $\chi_2$ has better regularity with respect to $\o$, which is the point of the decomposition \eqref{ch4:dechch} (see Remark \ref{ch4:rem:dehch}). Indeed, since the estimate \eqref{ch4:esttrc} for $\trc$ is better than the estimate \eqref{ch4:esthch} for $\hch$, and since $\nabb\mathcal{D}^{-1}$ is a pseudodifferential operator of order 0 on $\ptu$, we are able to obtain better regularity in $(t, x)$ for $\chi_1$ compared to $\hch$. Next, we focus on $\chi_2$. Now, note that the curvature tensor $\rr$ does not depend on $\o$. Thus, when differentiating $\b$ with respect to $\o$, the $\o$ derivative falls on the frame $(L, \lb, e_A)$, and we obtain schematically
$$\po\b=(\a+\r+\s)\po N.$$
In particular, we have
$$\norm{\po\b}_{\li{\infty}{2}}\les \norm{\po N}_{L^\infty}(\norm{\a}_{\li{\infty}{2}}+\norm{\r}_{\li{\infty}{2}}+\norm{\s}_{\li{\infty}{2}})\les\ep,
$$
where we used the curvature bound \eqref{ch4:curvflux1} for $\a, \r$ and $\s$, and the estimate \eqref{ch4:estNomega} for $\po N$. Thus, $\po\b$ has the same regularity with respect to $(t, x)$ than $\b$. In view of the definition of $\chi_2$, we obtain that $\po\chi_2$ has essentially the same regularity as $\chi_2$, while the estimate \eqref{ch4:estricciomega} for $\po\hch$ looses one $\nabb$-derivative with respect to the estimate \eqref{ch4:esthch} for $\hch$. Thus, the regularity of $\chi_2$ with respect to $\o$ is better than the corresponding regularity for $\hch$.

\subsection{Second order derivatives with respect to $\o$}

In this section, we outline the main ideas of the proof of Theorem \ref{ch4:thregomega2}. We focus on the estimate \eqref{ch4:estricciomega2} for $\po^2\z$ which is typical. Differentiating twice with respect to $\o$ the transport equation \eqref{ch4:D4eta} for $\z$, and computing the commutator $[\ddb_L, \po^2]$, we obtain
\be\lab{ch4:poo32}
\ddb_L(\Pi(\po^2\z))=-\chi\c\Pi(\po^2\z)+\nabb(F_1)+F_2+\cdots,
\ee
where the $\ptu$-tangent tensors $F_1$ and $F_2$ satisfy
$$\norm{F_1}_{\lh{2}}+\norm{F_2}_{\xt{1}{2}}\les\ep.$$
We first get rid of the first term in the right-hand side of \eqref{ch4:poo32} which is troublesome. To this end, we use the following lemma.
\begin{lemma}\lab{ch4:lemma:poo3}
Let $\ga$ denotes the metric induced by $\gg$ on $\ptu$. Let $M$ the $\ptu$-tangent 2-tensor defined as the solution of the following transport equation:
\be\lab{ch4:poo35}
\ddb_LM_{AB}=M_{AC}\chi_{CB},\,M_{AB}=\ga_{AB}\textrm{ on }\pou,
\ee
Then, $M_{AB}$ satisfies the following estimate:
\be\lab{ch4:poo36}
\norm{M-\ga}_{L^\infty}+\norm{\nabb M}_{\BB^0}\les\ep.
\ee
\end{lemma}

Using the transport equation \eqref{ch4:poo32} for $\Pi(\po^2\z)$ and the transport equation \eqref{ch4:poo35}, for $M$ allows us to get rid of the troublesome term $\chi\c\Pi(\po^2\z)$:
$$\ddb_L(M\c \Pi(\po^2\z))=\nabb(M\c F_1)-\nabb(M)\c F_1+M\c F_2+\cdots.$$
Together with the finite band property and the Bernstein inequality for $P_j$, the estimates for $F_1$ and $F_2$, and the estimate \eqref{ch4:poo36} for $M$, we obtain for $M\c \Pi(\po^2\z)$ the estimate corresponding to \eqref{ch4:estricciomega2}. Then, we obtain the wanted estimate \eqref{ch4:estricciomega2} for $\Pi(\po^2\z)$ by proving that the estimate \eqref{ch4:poo36} for $M$ is enough to ensure that the multiplication by $M^{-1}$ preserves the estimate \eqref{ch4:estricciomega2}.

\section{Additional decompositions}\lab{ch4:sec:adddec}

In this section, we outline the main ideas of the proof of Theorem \ref{ch4:cor:xx1bis}. We need to compare $N$ and $\trc$ at two different angles $\o$ and $\nu$. The basic tool is the following lemma. 
\begin{lemma}\lab{ch4:lemma:xx2}
Let $\o$ and $\o'$ in $\S$. Let $u=u(t,x,\o)$ and $u'=u(t,x,\o')$. Then, 
for any tensor $F$, we have:
$$
\norm{F}_{L^\infty_{u'}L^2(\H_{u'})}\les \norm{F}_{L^\infty_uL^2(\H_u)}+|\o-\o'|^{\frac{1}{4}}\norm{F}^{\frac{1}{2}}_{L^\infty_uL^2(\H_u)}\left(\sup_u\left(\int_u^{u+|\o-\o'|}\norm{\dd F}^2_{L^2(\H_{\tau})}d\tau\right)\right)^{\frac{1}{4}}.
$$
\end{lemma}

In order to compare the norms $L^\infty_{u'}L^2(\H_{u'})$ and $L^\infty_uL^2(\H_u)$, we need coordinate systems. We define $\Phi_{t,\o}:\Si_t\rightarrow \R^3$ defined by:
\begin{equation}\label{ch4:gl1}
\Phi_{t,\o}(t,x):=u(t,x,\o)\o+\po u(t,x,\o).
\end{equation}
Then we claim that $\Phi_{t,\o}$ is a global $C^1$ diffeomorphism from $\Si_t$ to $\R^3$ and therefore provides a global coordinate system on $\Si_t$ (see Proposition \ref{gl20} for a related result on $\Sigma_0$). Next, we prove that 
\bea\lab{ch4:lxx2:3}
\norm{F}^2_{L^2(\H_{u})}\simeq\int_0^1\int_{\RRR^2} |F(\Phi^{-1}_{t,u,\o}(y'))|^2dy'dt.
\eea
This formula allows us to compare the norms $L^\infty_{u'}L^2(\H_{u'})$ and $L^\infty_uL^2(\H_u)$. In turn, one needs to evaluate 
$$|F(\Phi^{-1}_{t,u,\o}(y'))|^2-|F(\Phi^{-1}_{t,u',\o'}(y'))|^2$$
In particular, we need to estimate
$$\normm{\po\left[\Phi^{-1}_{t,u,\o}(y')\right]}_{L^\infty}.$$
We refer the reader to \cite{param3} for details on the proof of Lemma \ref{ch4:lemma:xx2}. 

Using Lemma \ref{ch4:lemma:xx2} as well as commutator estimates for $[\dd, P_l]$ among others, we may prove the following corollary. 

\begin{corollary}\lab{ch4:lemma:xx3}
Let $f$ a scalar function and $\o, \o'$ in $\S$. Then, for any $l\geq 0$, we have:
$$\norm{P_lf}_{L^\infty_{u'}L^2(\H_{u'})}\les (2^{-l}+|\o-\o'|^{\half}2^{-\frac{l}{2}})(\norm{f}_{L^\infty_uL^2(\H_u)}+\norm{\dd f}_{L^\infty_uL^2(\H_u)}),$$
and
\bee
&&\norm{P_{\leq l}f}_{L^\infty_{u'}L^2(\H_{u'})}\\
&\les& (1+|\o-\o'|^{\half}2^{\frac{l}{2}})\norm{f}_{L^\infty_uL^2(\H_u)}+|\o-\o'|^{\frac{1}{4}}\norm{f}^{\half}_{L^\infty_uL^2(\H_u)}\\
&&\times\left(\sup_u\sum_{q\leq l}\int_u^{u+|\o-\o'|}(\norm{P_q(nL(f))}_{L^2(\H_{\tau})}^2+\norm{P_q(bN(f))}_{L^2(\H_{\tau})}^2)d\tau\right)^{\frac{1}{4}}.
\eee
\end{corollary}

We also need the following non sharp commutator lemma.
\begin{lemma}\lab{ch4:lemma:xx5}
Let $f$ a scalar function and $\o, \o'$ in $\S$. Then, for any $l\geq 0$, we have:
$$\norm{[\po,P_{\leq l}]f}_{L^\infty_{u'}L^2(\H_{u'})}\les \norm{\dd f}_{L^\infty_uL^2(\H_u)}.$$
\end{lemma}

Using Corollary \ref{ch4:lemma:xx3} and Lemma \ref{ch4:lemma:xx5} together with the estimates \eqref{ch4:esttrc} and \eqref{ch4:estricciomega} for $\trc$ and the fact that $|\o-\nu|\les 2^{-\frac{j}{2}}$, we are able to prove the decomposition \eqref{ch4:dectrcom} for $\trc$. 

Next, we consider $N-N_\nu$. We have
$$2^{\frac{j}{2}}(N-N_\nu)=\int_{[\o, \nu]}\po N(.,\o'')d\o'' (2^{\frac{j}{2}}(\o-\nu)),$$
where $[\o, \nu]$ denotes the arc of $\S$ joining $\o$ and $\nu$. Since $|\o-\nu|\les 2^{-\frac{j}{2}}$, we want to proceed as for the decomposition of $\trc$. More precisely, we 
want to use Corollary \ref{ch4:lemma:xx3} and Lemma \ref{ch4:lemma:xx5} together with the estimates  
 \eqref{ch4:estNomega}, \eqref{ch4:estricciomega}, \eqref{ch4:estNomega2}, \eqref{ch4:estNomega2bis} and \eqref{ch4:estricciomega2} for $\po N$, in order to prove the decomposition \eqref{ch4:decNom} for $N-N_\nu$. Now, unlike $\trc$ which is a scalar,  $\po N$ is a tensor. Since Corollary \ref{ch4:lemma:xx3} and Lemma \ref{ch4:lemma:xx5} only apply to scalars, we need one last ingredient to prove the decomposition \eqref{ch4:decNom} for $N-N_\nu$ and conclude the proof of Theorem \ref{ch4:cor:xx1bis}. Namely we need to scalarize $\po N$ using a basis of the tangent space of $\Sit$ which does not depend on $\o$. We refer to \cite{param3} for the details.

%%%%%%%%%%%%%%%%%%%%%%%%%%%%%%%%%%%%%%%

\chapter{Construction and control of the parametrix at initial time}\lab{part:paraminit}

%%%%%%%%%%%%%%%%%%%%%%%%%%%%%%%%%%%%%%%

\renewcommand{\l}[2]{L^{#1}_uL^{#2}(P_u)}
\renewcommand{\s}{\Sigma}

In this chapter, and the next one, we will only consider the leave $\Sigma_0$ of the foliation $\Sigma_t$ of $\MM$, and we denote it by $\Sigma$ for simplicity. Recall the plane wave type parametrix given by \eqref{parametrixinit.intr}\footnote{This is actually a half wave parametrix. See \eqref{ch5:param1} below for the full parametrix} 
$$\int_{\SSS^2} \int_0^\infty e^{i\la \uom(t,x)} \,  f(\la\om) \la^2 d\la  d\om$$
where $u(.,.,\o)$ is a solution to the eikonal equation ${\bf g}^{\alpha\beta}\partial_\alpha u\partial_\beta u=0$ on $\mathcal{M}$ such that $u(0,x,\o)\sim \xo$ when $|x|\rightarrow +\infty$ on $\Si$. The goal of this chapter is to outline the main ideas allowing us to obtain the control for that parametrix restricted to $\Sigma$ in \cite{param4}.

\section{Geometric set-up and main results}

\subsection{Presentation of the parametrix}\label{ch5:sec:fullparam}

In this section, we construct a parametrix for the following homogeneous wave equation:
\begin{equation}\label{ch5:waveeq}
\left\{\begin{array}{l}
\ds\square_{\bf g}\phi=0\textrm{ on }\mathcal{M},\\
\ds\phi_{|_\Sigma}=\phi_0,\, T(\phi)_{|_\Sigma}=\phi_1,
\end{array}\right.
\end{equation}
where $\phi_0$ and $\phi_1$ are two given functions on $\Sigma$ and $T$ is the future oriented unit normal to $\Sigma$ in the space-time $\mathcal{M}$.

We recall the plane wave representation of the solution of the flat wave equation. This corresponds to the case where ${\bf g}$ is the Minkowski metric. \eqref{ch5:waveeq} becomes:
\begin{equation}\label{ch5:flatwaveeq}
\left\{\begin{array}{l}
\ds\square\phi=0\textrm{ on }\R^{1+3},\\
\ds\phi(0,.)=\phi_0,\, \partial_t\phi(0,.)=\phi_1\textrm{ on }\R^3.
\end{array}\right.
\end{equation}
The plane wave representation of the solution $\phi$ of \eqref{ch5:flatwaveeq} is given by:
\begin{equation}\label{ch5:flatparam}
\begin{array}{l}
\ds\int_{\S}\int_0^{+\infty}e^{i(-t+\xo)\la}\frac{1}{2}\left(\mathcal{F}\phi_0(\la\o)+i\frac{\mathcal{F}\phi_1(\la\o)}{\la}\right)d\la d\o\\
\ds +\int_{\S}\int_0^{+\infty}e^{i(t+\xo)\la}\frac{1}{2}\left(\mathcal{F}\phi_0(\la\o)-i\frac{\mathcal{F}\phi_1(\la\o)}{\la}\right)d\la d\o,
\end{array}
\end{equation}
where $\mathcal{F}$ denotes the Fourier transform on $\R^3$.

We would like to construct a parametrix in the curved case similar to \eqref{ch5:flatparam}. We introduce two solutions $u_\pm$ of the eikonal equation
\begin{equation}\label{ch5:eikonal}
{\bf g}^{\alpha\beta}\partial_\alpha u_\pm\partial_\beta u_\pm=0\textrm{ on }\mathcal{M},
\end{equation}
such that:
\begin{equation}\label{ch5:eikonal1}
T(u_\pm)=\mp |\nabla u_\pm| =\mp a^{-1}_\pm\textrm{ on }\Sigma,
\end{equation}
where $\nabla$ is the covariant derivative on $\Sigma$ associated to the metric $g$ induced by $\g$ on $\Sigma$, $|\c|$ is the length associated to $g$ for vectorfields on $\Sigma$, and $a_\pm$ is the lapse of $u_\pm$ on $\Sigma$. We look for a parametrix for \eqref{ch5:waveeq} of the form:
\be\lab{ch5:param1}
S(t,x)=S_+f_+(t,x)+S_-f_-(t,x),
\ee
where
\be\lab{ch5:param1bis}
\ds S_{\pm}f_{\pm}(t,x)=  \ds\int_{\S}\int_{0}^{+\infty}e^{i\lambda u_\pm(t,x,\o)}f_\pm(\lambda\o)\lambda^2 d\lambda d\o.
\ee
In the next two sections, we specify the parametrix \eqref{ch5:param1} by prescribing $u_\pm$ on $\Sigma$ and by making our choice for $f_\pm$ explicit.

\subsubsection{Prescription of $u_+$ and $u_-$ on $\Sigma$}\lab{sec:prescriptionu}

\eqref{ch5:eikonal} and \eqref{ch5:eikonal1} are not enough to define $u_\pm$ in a unique manner. Indeed, we 
still need to prescribe $u_\pm$ on $\Sigma$. To motivate our choice, we need to introduce some geometric 
objects connected to $u_\pm$. Let $N_\pm$ the vectorfield on $\Sigma$ defined by:
\begin{equation}\label{ch5:eikonal2}
N_\pm=\frac{\nabla u_\pm}{|\nabla u_\pm|}=a_\pm\nabla u_\pm,
\end{equation}
and $L_\pm$ the vectorfield on $\mathcal{M}$ which is given on $\Sigma$ by: 
\begin{equation}\label{ch5:eikonal3}
L_\pm=a_\pm{\bf g}^{\alpha\beta}\partial_\alpha u_\pm\partial_\beta=a_\pm(-T(u_\pm)T+\nabla u_\pm)=\pm T+N_\pm.
\end{equation}
Let $P_{u_\pm}=\{x\in\Sigma/\,u_\pm(x)=u_\pm\}$ denote the level surfaces of $u_\pm$ in $\Sigma$. Since $N_\pm$ is the unit normal to $P_{u_\pm}$, the second fundamental form of $P_{u_\pm}$ in $\Sigma$ is given by:
\begin{equation}\label{ch5:eikonal4}
\th_\pm(e^\pm_A,e^\pm_{B})=g(D_{e^\pm_A}N_\pm,e^\pm_{B}),\,A, B=1,2,
\end{equation}
where $(e^\pm_1,e^\pm_2)$ is an arbitrary orthonormal frame of $TP_{u_\pm}$. Let$$\mathcal{H}_{u_\pm}=\{(t,x)\in\mathcal{M}/\,u_\pm(t,x)=u_\pm\}$$ 
denote the null level hypersurfaces of $u_\pm$ in $\mathcal{M}$. Since $L_\pm$ is null and orthogonal to $P_{u_\pm}$ in $\mathcal{H}_{u_\pm}$, the null second fundamental form $\chi_\pm$ is given on $P_{u_\pm}$ by:
\begin{equation}\label{ch5:eikonal5}
\chi_\pm(e^\pm_A,e^\pm_{B})=g(\dd_{e^\pm_A}L_\pm,e^\pm_{B}),\,A, B=1,2.
\end{equation}
Taking the trace in \eqref{ch5:eikonal4} and \eqref{ch5:eikonal5}, and using \eqref{ch5:eikonal3} and the fact that $k$ is the second fundamental form of $\Sigma$, we obtain:
\begin{equation}\label{ch5:eikonal6}
\textrm{tr}\chi_\pm=\pm\textrm{tr}k+\textrm{tr}\th_\pm.
\end{equation}
Note that $\tr_g k=\textrm{tr}k+k_{NN}$, where $\tr_g$ denotes the trace for 2-tensors on $\Sigma$. In view of the maximal foliation assumption \eqref{maxfoliation}, we have $\tr_g k=0$. Together with \eqref{ch5:eikonal6}, this yields:
\begin{equation}\label{ch5:eikonal7}
\textrm{tr}\chi_\pm=\mp k_{N_\pm N_\pm}+\textrm{tr}\th_\pm.
\end{equation}
Now, in \cite{param3} (see also Theorem \ref{ch4:thregx}), we prove that $\textrm{tr}\chi_\pm$ belongs to $L^\infty(\mathcal{M})$ using a transport equation (the Raychaudhuri equation, see \eqref{ch4:D4trchi}) provided that it belongs to $L^\infty(\Sigma)$ at $t=0$. Thus, one needs the following estimate
\begin{equation}\label{ch5:eikonal9}
\textrm{tr}\chi_\pm\in L^\infty(\Sigma),
\end{equation}
which in view of \eqref{ch5:eikonal7} is equivalent to:
\begin{equation}\label{ch5:eikonal10}
\mp k_{N_\pm N_\pm}+\textrm{tr}\th_\pm\in L^\infty(\Sigma).
\end{equation}
We construct in \cite{param1} (see also Chapter \ref{part:initialu}) a function $u(x,\o)$ on $\Sigma\times\S$ such that 
\begin{equation}\label{ch5:eikonal11}
- k_{NN}+\textrm{tr}\th\in L^\infty(\Sigma).
\end{equation}
Note that $-u(x,-\o)$ satisfies:
\begin{equation}\label{ch5:eikonal12}
k_{NN}+\textrm{tr}\th\in L^\infty(\Sigma).
\end{equation}
Thus, in view of \eqref{ch5:eikonal10}, \eqref{ch5:eikonal11} and \eqref{ch5:eikonal12}, we initialize $u_\pm$ on $\Sigma$ by:
\begin{equation}\label{ch5:eikonal13}
u_+(0,x,\o)=u(x,\o)\textrm{ and }u_-(0,x,\o)=-u(x,-\o)\textrm{ for }(x,\o)\in\Sigma\times\S.
\end{equation}

\begin{remark}
Note that in the particular case where $k\equiv 0$ - the so-called time symmetric case-, we may take 
$$u_+(0,x,\o)=u_-(0,x,\o)=u(x,\o)\textrm{ for }(x,\o)\in\Sigma\times\S.$$
In particular, we have $u_+(0,x,\o)=u_-(0,x,\o)=\xo$ in the flat case.
\end{remark}

\subsubsection{The choice of $f_+$ and $f_-$}\label{ch5:sec:choicef}

Having defined $u_\pm$, we still need to define $f_\pm$ in the parametrix \eqref{ch5:param1}. According to \eqref{ch5:waveeq}, the half wave parametrix $S_+$ and $S_-$ should satisfy on $\Sigma$:
\begin{equation}\label{ch5:choicef}
\left\{\begin{array}{l}
\ds S_+f_+(0,x)+S_-f_-(0,x)=\phi_0(x),\\
\ds T(S_+f_+)(0,x)+T(S_-f_-)(0,x)=\phi_1(x).
\end{array}\right.
\end{equation}
Let us introduce the following operators acting on functions of $\R^3$:
\begin{equation}\label{ch5:choicef1}
M_\pm f(x)=\int_{\S}\int_{0}^{+\infty}e^{\pm i\lambda u(x,\pm\o)}f(\lambda\o)\lambda^2 d\lambda d\o
\end{equation}
and 
\begin{equation}\label{ch5:choicef2}
Q_\pm f(x)=\int_{\S}\int_{0}^{+\infty}e^{\pm i\lambda u(x,\pm\o)}a(x,\pm\o)^{-1} f(\lambda\o)\lambda^2 d\lambda d\o,
\end{equation}
where $a(x,\o)=|\nabla u(x,\o)|^{-1}$ is the lapse of $u$. 
Using \eqref{ch5:eikonal1}, the definition of $S_\pm$ in \eqref{ch5:param1}, \eqref{ch5:eikonal13}, the definition \eqref{ch5:choicef1} of $M_\pm$ and the definition \eqref{ch5:choicef2} of $Q_\pm$, we may rewrite \eqref{ch5:choicef} as: 
\begin{equation}\label{ch5:choicef3}
\left\{\begin{array}{l}
M_+f_++M_-f_-=\phi_0,\\
Q_+(\la f_+)-Q_-(\la f_-)=i\phi_1.
\end{array}\right.
\end{equation}
The goal of this chapter will be to show that there exist a unique $(f_+,f_-)$ satisfying \eqref{ch5:choicef3}, 
and that $(f_+,f_-)$ satisfies the following estimate:
\begin{equation}\label{ch5:choicef4}
\norm{\la f_+}_{L^2(\R^3)}+\norm{\la f_-}_{L^2(\R^3)}\lesssim \norm{\nabla\phi_0}_{L^2(\Sigma)}+\norm{\phi_1}_{L^2(\Sigma)}.
\end{equation}

\begin{remark}\lab{rem:theflatcase}
In the case of the flat wave equation \eqref{ch5:flatwaveeq}, we have $(\Sigma,g, k)=(\R^3,\de, 0)$, $u_\pm(t,x,\o)=\mp t+\xo$, $u(x,\o)=\xo$ and $a(x,\o)= 1$. In particular, the operators $M_\pm$ and $Q_\pm$ defined respectively by \eqref{ch5:choicef1} and \eqref{ch5:choicef2} all coincide with the inverse Fourier transform. Then, the system \eqref{ch5:choicef3} admits the following solutions:
$$f_\pm (\la\o) =\frac{1}{2}\left(\mathcal{F}\phi_0(\la\o)\pm i\frac{\mathcal{F}\phi_1(\la\o)}{\la} \right),$$
which clearly satisfy the estimate \eqref{ch5:choicef4}.
\end{remark}

\subsection{Geometric set-up}

We define the lapse $a(x,\o)=|\nabla u(x,\o)|^{-1}$, and the unit vector $N$ such that $\nabla u(x,\o)=a(x,\o)^{-1}N(x,\o)$. We also define the level surfaces $\p=\{x\,/\,u(x,\o) =u\}$ so that $N$ is the normal to $\p$. 

For $1\leq p,q\leq +\infty$, we define the spaces $\l{p}{q}$ using the norm 
$$\norm{F}_{\l{p}{q}}=\left(\int_{u}\norm{F}^p_{L^q(\p)}du\right)^{1/p}.$$
We assume that $1/2\leq a(x)\leq 2$ for all $x\in\s$ (see \eqref{ch5:thregx1} below) so that $\l{p}{p}$ coincides with $L^p(\Sigma)$ for all $1\leq p\leq +\infty$. We denote by $\gamma$ the metric induced by $g$ on $\p$, and by $\nabb$ the induced covariant derivative. 

Before stating precisely the main results of this chapter, we first record the regularity obtained for the 
phase $u(x,\o)$ constructed in \cite{param1} (see also Chapter \ref{part:initialu}). 

\subsection{Regularity assumptions on the phase $u(x,\o)$}\lab{ch5:sec:estneeded}

In this section, we collect the estimates  for the phase $u(x,\o)$ of our Fourier integral operators that are needed to follow the discussion of the control of the parametrix at initial time contained in this chapter. An outline of the proof of these estimates will be given in Chapter \ref{part:initialu}  (see \cite{param1} for the complete proof).

We start with the regularity in $x$ of the lapse $a$. We need:
\begin{equation}\label{ch5:thregx1}
\ds\norm{\nabla a}_{\l{\infty}{2}}+\norm{a-1}_{L^\infty(\Sigma)}+\norm{\nabb\nabla a}_{L^2(\Sigma)}\lesssim\ep.
\end{equation}
We also need a decomposition for $\nabn a$. For all $j\geq 0$, there are scalar functions $a^j_1$ and $a^j_2$ such that\footnote{we choose $a_1^j=P_{>j/2}(\nabn a)$ and $a_2^j=P_{\leq j/2}(\nabn a)$, and then obtain \eqref{ch5:cordecfr1} using \eqref{ch5:thregx1} and an estimate for $\nabla_N^2a$ (see \cite{param1} for the details)}:
\begin{equation}\label{ch5:cordecfr1}
\begin{array}{l}
\ds\nabn a=a^j_1+a^j_2\textrm{ where }\norm{a^j_1}_{L^2(\Sigma)}\lesssim 2^{- j/2}\ep,\,\norm{a^j_2}_{\l{\infty}{2}}\lesssim \ep\\
\ds\textrm{and }\norm{\nabn a^j_2}_{L^2(\Sigma)}+\norm{a^j_2}_{\l{2}{\infty}}\lesssim 2^{j/2}\ep.
\end{array}
\end{equation}
Next, we consider the regularity with respect to $\o$. We have:
\begin{equation}\label{ch5:threomega1}
\ds\norm{\po a}_{L^2(\Sigma)}+\norm{\nabla\po a}_{L^2(\Sigma)}\lesssim\ep,
\end{equation}
\begin{equation}\label{ch5:threomega3}
\norm{\po^{\a}a}_{L^\infty(\Sigma)}\lesssim 1\textrm{ for some }0<\a<1,
\end{equation}
where \eqref{ch5:threomega3} should be understood in the H\"older sense,
\begin{equation}\label{ch5:threomega1bis}
\norm{\po N}_{L^\infty(\Sigma)}\lesssim 1,
\end{equation}
and 
\begin{equation}\label{ch5:threomega3bis}
\norm{\po^3u}_{L^\infty_{\textrm{loc}}(\Sigma)}\lesssim 1.
\end{equation}
We will need the following global change of variable on $\Sigma$. Let $\o\in\S$. Let $\phi_\o:\Sigma\rightarrow\R^3$ defined by:
\begin{equation}\label{ch5:gl21}
\phi_\o(x):=u(x,\o)\o+\po u(x,\o).
\end{equation}
Then $\phi_\o$ is a bijection, and the determinant of its Jacobian satisfies the following estimate:
\begin{equation}\label{ch5:gl22}
\norm{|\det(\textrm{Jac}\phi_\o)|-1}_{L^\infty(\Sigma)}\lesssim\ep.
\end{equation}
Finally, we can compare $u(x,\o)$ with a phase linear in $\o$. Let $\nu\in\S$ and $\phi_\nu$ the map defined in \eqref{ch5:gl21}. Then, we have:
\begin{equation}\label{ch5:gl23}
\begin{array}{l}
\ds u(x,\o)-\phi_\nu(x)\c\o=O(\ep|\o-\nu|^2),\\
\ds \po u(x,\o)-\po(\phi_\nu(x)\c\o)=O(\ep|\o-\nu|),\\
\ds \po^2 u(x,\o)-\po^2(\phi_\nu(x)\c\o)=O(\ep).
\end{array}
\end{equation}

\begin{remark}
In \eqref{ch5:thregx1}-\eqref{ch5:gl23}, all inequalities hold for any $\o\in\S$ with the constant in the right-hand side being independent of $\o$. Thus, one may take the supremum in $\o$ everywhere. To ease the notations, we do not explicitly write down this supremum. 
\end{remark}

\begin{remark}
In the case of the flat wave equation \eqref{ch5:flatwaveeq}, we have $(\Sigma,g)=(\R^3,\de)$, $u(x,\o)=\xo$, $a= 1$, $N=\o$ and $\phi_\o=I\! d_{\R^3}$. Thus, \eqref{ch5:thregx1}-\eqref{ch5:gl23} are clearly satisfied with $\ep=0$.
\end{remark}

\begin{remark}
Recall that the lapse $a$ is at the level of one derivative of $u$ with respect to $x$. Thus, we obtain from \eqref{ch5:thregx1} that some components of $\nabla^3u$ are in $L^2(\Sigma)$. Note that this is not true for all components since \eqref{ch5:cordecfr1} does not allow us to control $\nabn^2a$ in $L^2(\Sigma)$. In fact, \eqref{ch5:cordecfr1} is consistent with $3/2$ derivatives of $a$ with respect to $N$ in $L^2$. 
\end{remark}

\subsection{Main results}

We first state a result of boundedness on $L^2$ for Fourier integral operators with phase $u(x,\o)$. 
\begin{theorem}\label{ch5:th1}
Let $u$ be a function on $\Sigma\times\S$ satisfying suitable assumptions (we refer to \cite{param2} for the complete set of assumptions, and to section \ref{ch5:sec:estneeded} for some typical assumptions). Let $U$ the Fourier integral operator with phase $u(x,\o)$ and symbol $b(x,\o)$:
\begin{equation}\label{ch5:fio}
Uf(x)=\int_{\S}\int_{0}^{+\infty}e^{i\lambda u(x,\o)}b(x,\o)f(\lambda\o)\lambda^2 d\lambda d\o.
\end{equation}
Let $D>0$. We assume furthermore that $b(x,\o)$ satisfies:
\begin{equation}\label{ch5:thregx1s}
\norm{b}_{L^\infty(\Sigma)}+\norm{\nabla b}_{\l{\infty}{2}}+\norm{\nabb\nabla b}_{L^2(\Sigma)}\lesssim D,
\end{equation}
\begin{equation}\label{ch5:threomega1s}
\norm{\po b}_{L^2(\Sigma)}+\norm{\nabla\po b}_{L^2(\Sigma)}\lesssim D,
\end{equation}
and
\begin{equation}\label{ch5:cordecfr1s}
\begin{array}{l}
\ds\nabn b=b^j_1+b^j_2\textrm{ where }\norm{b^j_1}_{L^2(\Sigma)}\lesssim 2^{-\frac{j}{2}}D, \norm{b^j_2}_{\l{\infty}{2}}\lesssim D,\\
\ds\textrm{ and }\norm{\nabn b^j_2}_{L^2(\Sigma)}+\norm{b^j_2}_{\l{2}{\infty}}\lesssim 2^{\frac{j}{2}}D.
\end{array}
\end{equation}
Then, $U$ is bounded on $L^2$ and satisfies the estimate:
\begin{equation}\label{ch5:l2}
\norm{Uf}_{L^2(\Sigma)}\lesssim D\norm{f}_{L^2(\R^3)}.
\end{equation}
\end{theorem}

\begin{remark}
We intend to apply Theorem \ref{ch5:th1} to the Fourier integral operators $M_\pm$ and $Q_\pm$ introduced in section \ref{ch5:sec:choicef} whose symbol are respectively 1 and $a^{-1}$. Thus, our assumptions on the regularity of the symbol $b(x,\o)$ are consistent with the assumptions on the regularity of $a(x,\o)$ given by \eqref{ch5:thregx1}-\eqref{ch5:threomega3bis}.  
\end{remark}

\begin{remark}
Under the additional assumption \eqref{ch5:gl23} on $u$, and under some restrictions on the constant $D$ appearing in \eqref{ch5:thregx1s}, \eqref{ch5:threomega1s}, and \eqref{ch5:cordecfr1s}, we may prove the opposite of \eqref{ch5:l2}:
$$\norm{f}_{\le{2}}\lesssim\norm{Uf}_{L^2(\Sigma)}$$
(see Proposition \ref{ch5:re2}). This will be a major ingredient in the proof of Theorem \ref{ch5:th2} below, and in particular of \eqref{ch5:choicef4bis}.
\end{remark}

Recall the definition of the Fourier integral operators $M_\pm$ and $Q_\pm$ introduced in section \ref{ch5:sec:choicef}: 
\begin{equation}\label{ch5:choicef1bis}
M_\pm f(x)=\int_{\S}\int_{0}^{+\infty}e^{\pm i\lambda u(x,\pm\o)}f(\lambda\o)\lambda^2 d\lambda d\o,
\end{equation}
and 
\begin{equation}\label{ch5:choicef2bis}
Q_\pm f(x)=\int_{\S}\int_{0}^{+\infty}e^{\pm i\lambda u(x,\pm\o)}a(x,\pm\o)^{-1} f(\lambda\o)\lambda^2 d\lambda d\o.
\end{equation}
The following theorem is the main result of this chapter.

\begin{theorem}\label{ch5:th2}
Let $u$ be a function on $\Sigma\times\S$ satisfying suitable assumptions (we refer to \cite{param2} for the complete set of assumptions, and to section \ref{ch5:sec:estneeded} for some typical assumptions). Then, there exist a unique $(f_+,f_-)$ satisfying:
\begin{equation}\label{ch5:choicef3bis}
\left\{\begin{array}{l}
M_+f_++M_-f_-=\phi_0,\\
Q_+(\la f_+)-Q_-(\la f_-)=i\phi_1.
\end{array}\right.
\end{equation}
Furthermore, $(f_+,f_-)$ satisfies the following estimate:
\begin{equation}\label{ch5:choicef4bis}
\norm{\la f_+}_{L^2(\R^3)}+\norm{\la f_-}_{L^2(\R^3)}\lesssim \norm{\nabla\phi_0}_{L^2(\Sigma)}+\norm{\phi_1}_{L^2(\Sigma)}.
\end{equation}
\end{theorem}

Proving the estimates \eqref{ch5:l2} and \eqref{ch5:choicef4bis} for the Fourier integral operators $U$, $M_\pm$ and $Q_\pm$ will require taking several integrations by parts. The main difficulty in proving Theorem \ref{ch5:th1} and Theorem \ref{ch5:th2} will be to perform these integrations by parts within the very low level of regularity for the phase $u(x,\o)$ given by \eqref{ch5:thregx1}-\eqref{ch5:gl23} and for the symbol $b(x,\o)$ given by \eqref{ch5:thregx1s} \eqref{ch5:threomega1s} \eqref{ch5:cordecfr1s}. The proof will rely both on harmonic analysis decompositions and the geometry of the foliation of $\Sigma$ by $u$. Theorem \ref{ch5:th1} will be reviewed in section \ref{ch5:sec:th1} and Theorem \ref{ch5:th2} will be reviewed in section \ref{ch5:sec:th2}. 

\section{Control of Fourier integral operators}\lab{ch5:sec:th1}

\subsection{Structure of the proof of Theorem \ref{ch5:th1}}

The proof of Theorem \ref{ch5:th1} proceeds in three steps. We first localize in frequencies of size $\la\sim 2^j$. We then localize the angle $\o$ in patches on the sphere $\S$ of diameter $2^{-j/2}$. Finally, we estimate the diagonal terms. 

\begin{remark}
Note that the structure of the proof is analogous to the one on the control of the error term in Chapter \ref{part:paramtime} (see section \ref{ch3:sec:structureproof}). However, the proof each step (almost orthogonality in frequency, almost orthogonality in angle, and control of the diagonal term) is different, more particularly the last two steps. 
\end{remark}

\subsubsection{Step 1: decomposition in frequency}
 
For the first step, we introduce $\varphi$  and $\psi$ two smooth compactly supported functions on $\R$ such that: 
\begin{equation}\label{ch5:bisb7}
\varphi(\lambda)+\sum_{j\geq 0}\psi(2^{-j}\lambda)=1\textrm{ for all }\lambda\in\R.
\end{equation}
We use \eqref{ch5:bisb7} to decompose $Uf$ as follows:
\begin{equation}\label{ch5:bisb8}
Uf(x)=\sum_{j\geq -1}U_jf(x),
\end{equation}
where for $j\geq 0$:
\begin{equation}\label{ch5:bisb9}
U_jf(x)=\int_{\S}\int_{0}^{+\infty}e^{i\lambda u}b(x,\o)\psi(2^{-j}\lambda)f(\lambda\o)\lambda^2 d\lambda d\o,
\end{equation}
and 
\begin{equation}\label{ch5:bisb10}
U_{-1}f(x)=\int_{\S}\int_{0}^{+\infty}e^{i\lambda u}b(x,\o)\varphi(\lambda)f(\lambda\o)\lambda^2 d\lambda d\o.
\end{equation}
The goal of this first step is to prove the following proposition:
\begin{proposition}\label{ch5:bisorthofreq}
The decomposition \eqref{ch5:bisb8} satisfies an almost orthogonality property:
\begin{equation}\label{ch5:bisorthofreq1}
\norm{Uf}_{L^2(\Sigma)}^2\lesssim\sum_{j\geq -1}\norm{U_jf}_{L^2(\Sigma)}^2+D^2\norm{f}^2_{\le{2}}.
\end{equation}
\end{proposition}
The proof of Proposition \ref{ch5:bisorthofreq} is postponed to section \ref{ch5:bissec:orthofreq}. 

\subsubsection{Step 2: decomposition in angle}

Proposition \ref{ch5:bisorthofreq} allows us to  estimate $\norm{U_jf}_{L^2(\Sigma)}$ instead of 
$\norm{Uf}_{L^2(\Sigma)}$. We perform a second dyadic decomposition. We introduce a smooth partition of unity on the sphere $\S$:
\begin{equation}\label{ch5:bisb14}
\sum_{\nu\in\Gamma}\eta^\nu_j(\o)=1\textrm{ for all }\o\in\S, 
\end{equation}
where $\Gamma$ is a lattice on $\S$ of size $2^{-\frac{j}{2}}$, where the support of $\eta^\nu_j$ is a patch on $\S$ of diameter $\sim 2^{-j/2}$. We use \eqref{ch5:bisb14} to decompose $U_jf$ as follows:
\begin{equation}\label{ch5:bisb15}
U_jf(x)=\sum_{\nu\in\Gamma}U^\nu_jf(x),
\end{equation}
where:
\begin{equation}\label{ch5:bisb16}
U^\nu_jf(x)=\int_{\S}\int_{0}^{+\infty}e^{i\lambda u}b(x,\o)\psi(2^{-j}\lambda)\eta^\nu_j(\o)f(\lambda\o)\lambda^2 d\lambda d\o.
\end{equation}
We also define:
\begin{equation}\label{ch5:bisdecf}
\begin{array}{l}
\ds\ga_{-1}=\norm{\varphi(\la)f}_{\le{2}},\, \ga_j=\norm{\psi(2^{-j}\la)f}_{\le{2}},\,j\geq 0, \\
\ds\ga^\nu_j=\norm{\psi(2^{-j}\la)\eta^\nu_j(\o)f}_{\le{2}},\,j\geq 0,\,\nu\in\Gamma, 
\end{array}
\end{equation}
which satisfy:
\begin{equation}\label{ch5:bisdecf1}
\norm{f}_{\le{2}}^2=\sum_{j\geq -1}\ga_j^2=\sum_{j\geq -1}\sum_{\nu\in\Gamma}(\ga^\nu_j)^2.
\end{equation}
The goal of this second step is to prove the following proposition:
\begin{proposition}\label{ch5:bisorthoangle}
The decomposition \eqref{ch5:bisb15} satisfies an almost orthogonality property:
\begin{equation}\label{ch5:bisorthoangle1}
\norm{U_jf}_{L^2(\Sigma)}^2\lesssim\sum_{\nu\in\Gamma}\norm{U^\nu_jf}_{L^2(\Sigma)}^2+D^2\ga_j^2.
\end{equation}
\end{proposition}
The proof of Proposition \ref{ch5:bisorthoangle} is postponed to section \ref{ch5:bissec:orthoangle}. 

\subsubsection{Step 3: control of the diagonal term}

Proposition \ref{ch5:bisorthoangle} allows us to  estimate $\norm{U^\nu_jf}_{L^2(\Sigma)}$ instead of $\norm{U_jf}_{L^2(\Sigma)}$. The diagonal term is estimated as follows.
\begin{proposition}\label{ch5:bisdiagonal}
The diagonal term $U^\nu_jf$ satisfies the following estimate:
\begin{equation}\label{ch5:bisdiagonal1}
\norm{U^\nu_jf}_{L^2(\Sigma)}\lesssim D\ga^\nu_j.
\end{equation}
\end{proposition}
The proof of Proposition \ref{ch5:bisdiagonal} is postponed to section \ref{ch5:bissec:diagonal}. 

\subsubsection{Proof of Theorem \ref{ch5:th1}}

Proposition \ref{ch5:bisorthofreq}, \ref{ch5:bisorthoangle} and \ref{ch5:bisdiagonal} 
immediately yield the proof of Theorem \ref{ch5:th1}. Indeed, \eqref{ch5:bisorthofreq1}, \eqref{ch5:bisdecf1}, \eqref{ch5:bisorthoangle1} and \eqref{ch5:bisdiagonal1} imply:
\begin{equation}\label{ch5:biscclth3}
\begin{array}{ll}
\ds\norm{Uf}_{L^2(\Sigma)}^2 & \ds\lesssim\sum_{j\geq -1}\norm{U_jf}_{L^2(\Sigma)}^2+D^2\norm{f}^2_{\le{2}}\\
& \ds\lesssim\sum_{j\geq -1}\sum_{\nu\in\Gamma}\norm{U^\nu_jf}_{L^2(\Sigma)}^2+D^2\sum_{j\geq -1}\gamma_j^2+D^2\norm{f}^2_{\le{2}}\\
& \ds\lesssim D^2\sum_{j\geq -1}\sum_{\nu\in\Gamma}(\gamma_j^\nu)^2+D^2\sum_{j\geq -1}\gamma_j^2+D^2\norm{f}^2_{\le{2}}\\
& \ds\lesssim D^2\norm{f}^2_{\le{2}},
\end{array}
\end{equation}
which is the conclusion of Theorem \ref{ch5:th1}. 

The remainder of section \ref{ch5:sec:th1} is dedicated to the proof of Proposition \ref{ch5:bisorthofreq}, \ref{ch5:bisorthoangle} and \ref{ch5:bisdiagonal}.

\subsection{Proof of Proposition \ref{ch5:bisorthofreq} (almost orthogonality in frequency)}\label{ch5:bissec:orthofreq}

We have to prove \eqref{ch5:bisorthofreq1}:
\begin{equation}\label{ch5:i:bisof1}
\norm{Uf}_{L^2(\Sigma)}^2\lesssim\sum_{j\geq -1}\norm{U_jf}_{L^2(\Sigma)}^2+D^2\norm{f}^2_{\le{2}}.
\end{equation}
This will result from the following inequality using Shur's Lemma:
\begin{equation}\label{ch5:i:bisof2}
\left|\int_{\Sigma}U_jf(x)\overline{U_kf(x)}d\Sigma \right| \lesssim D^22^{-\frac{|j-k|}{2}}\gamma_j\ga_k\textrm{ for }|j-k|> 2.
\end{equation}

We consider a coordinate system $(u,x')$ on $\Sigma$ where $x'$ denotes a coordinate system on $P_u$, and we would like to integrate by parts with respect to $\pr_u$. Since $\nabla u=a^{-1}N$ and $\nabla u'={a'}^{-1}N'$, we have:
\begin{equation}\label{ch5:i:bisof4}
e^{i\lambda u-i\la' u'}=-\frac{i}{\la-\la'\frac{a}{a'}g(N,N')}\partial_u(e^{i\lambda u-i\la' u'}),
\end{equation}
where we use the notation $u$ for $u(x,\o)$, $a$ for $a(x,\o)$, $N$ for $N(x,\o)$, $u'$ for $u(x,\o')$, $a'$ for $a(x,\o')$ and $N'$ for $N(x,\o')$. Then, the proof of \eqref{ch5:i:bisof2} is analogous to the proof of \eqref{ch3:bisof2}, so we skip it for the sake of simplicity. Let us just say that, as for the proof of \eqref{ch3:bisof2}, most terms require one integration by parts using \eqref{ch5:i:bisof4} and the estimate \eqref{ch5:thregx1s} for $b$ (as in section \ref{ch3:sec:firstibpu}), while one term requires a second integration by parts using \eqref{ch5:i:bisof4} and the decomposition \eqref{ch5:cordecfr1s} for $\nabn b$ (as in sections \ref{ch3:sec:morepreciseest} and \ref{ch3:sec:2ibpu}).

\subsection{Proof of Proposition \ref{ch5:bisorthoangle} (almost orthogonality in angle)}\label{ch5:bissec:orthoangle}

We have to prove \eqref{ch5:bisorthoangle1}:
\begin{equation}\label{ch5:bisoa1}
\norm{U_jf}_{L^2(\Sigma)}^2\lesssim\sum_{\nu\in\Gamma}\norm{U^\nu_jf}_{L^2(\Sigma)}^2+D^2\ga_j^2.
\end{equation}

\subsubsection{Presence of a log-loss}

Integrating by parts twice in $\int_{\Sigma}U^\nu_jf(x)\overline{U^{\nu'}_jf(x)}d\Sigma $ would ultimately imply:
\begin{equation}\label{ch5:bisoa4}
\left|\int_{\Sigma}U^\nu_jf(x)\overline{U^{\nu'}_jf(x)}d\Sigma \right|\lesssim \frac{D^2\ga_j^\nu\ga_j^{\nu'}}{(2^{j/2}|\nu-\nu'|)^{2}},\,|\nu-\nu'|\neq 0.
\end{equation}
This yields to a log-loss since we have:
\begin{equation}\label{ch5:bisoa5}
\sup_{\nu}\sum_{\nu'} \frac{1}{(2^{j/2}|\nu-\nu'|)^{2}}\sim j.
\end{equation}

\begin{remark}
Recall that there is an analogous log-loss in the almost orthogonality argument in angle for the error term (see section \ref{ch3:sec:logloss}). In section \ref{ch3:bissec:orthoangle}, we removed the log-loss in particular by using integration by parts with respect to the null vectorfield $L$. On the other hand, we work here on $\Sigma$ which is Riemannian, so there is no equivalent of  the null vectorfield $L$. Instead, we will use a second decomposition in $\la$ (see section \ref{ch5:sec:iremovelogloss}). Note that such a strategy can not be used to control the error term (see Remark \ref{ch5:rem:seppi}). 
\end{remark}

To avoid the log-loss present in \eqref{ch5:bisoa4}, we will instead derive the following inequality:
\begin{equation}\label{ch5:bisoa2}
\left|\int_{\Sigma}U^\nu_jf(x)\overline{U^{\nu'}f(x)}d\Sigma \right|\lesssim \frac{D^2\ga_j^\nu\ga_j^{\nu'}}{2^{j\a/2}(2^{j/2}|\nu-\nu'|)^{2-\a}}+\frac{D^2\ga_j^\nu\ga_j^{\nu'}}{(2^{j/2}|\nu-\nu'|)^{3}},\,|\nu-\nu'|\neq 0,
\end{equation}
where $\a>0$. Indeed, since $\S$ is 2 dimensional and $1\leq 2^{j/2}|\nu-\nu'|\leq 2^{j/2}$ for $\nu, \nu'\in\Gamma$ and $\nu\neq\nu'$, we have:
\begin{equation}\label{ch5:bisoa3}
\sup_{\nu}\sum_{\nu'} \frac{1}{(2^{j/2}|\nu-\nu'|)^{3}}\leq C<+\infty,
\end{equation}
and 
\begin{equation}\label{ch5:bisoa3b}
\sup_{\nu}\sum_{\nu'} \frac{1}{2^{j\a/2}(2^{j/2}|\nu-\nu'|)^{2-\a}}\leq C_\a<+\infty\,\forall\a>0.
\end{equation}
Thus, \eqref{ch5:bisoa2}, \eqref{ch5:bisoa3} and \eqref{ch5:bisoa3b} together with Shur's Lemma imply \eqref{ch5:bisoa1}.

\subsubsection{A second decomposition in frequency}\lab{ch5:sec:iremovelogloss}

To avoid the log-loss present in \eqref{ch5:bisoa4}, we do a second decomposition in frequency. $\la$ belongs to the interval $[2^{j-1},2^{j+1}]$ which we decompose in intervals $I_k$:
\begin{equation}\label{ch5:bisoa6}
[2^{j-1},2^{j+1}]=\bigcup_{1\leq k\leq |\nu-\nu'|^{-\a}}I_k\textrm{ where }\textrm{diam}(I_k)\sim 2^j|\nu-\nu'|^\a.
\end{equation}
Let $\phi_k$ a partition of unity of the interval $[2^{j-1},2^{j+1}]$ associated to the $I_k$'s. We decompose $U^\nu_jf$ as follows:
\begin{equation}\label{ch5:bisoa7}
U^\nu_jf(x)=\sum_{1\leq k\leq |\nu-\nu'|^{-\a}}U^{\nu,k}_jf(x),
\end{equation}
where:
\begin{equation}\label{ch5:bisoa8}
U^{\nu,k}_jf(x)=\int_{\S}\int_{0}^{+\infty}e^{i\lambda u}b(x,\o)\psi(2^{-j}\lambda)\eta^\nu_j(\o)\phi_k(\la)f(\lambda\o)\lambda^2 d\lambda d\o.
\end{equation}

\begin{remark}
The point of this additional decomposition is to exploit the volume in $\la$. Indeed, after performing Cauchy-Schwarz in $\la$, we obtain
$$\sqrt{|I_k|}\sim 2^{\frac{j}{2}}|\nu-\nu'|^{\frac{\a}{2}}$$
which displays the crucial gain $|\nu-\nu'|^{\frac{\a}{2}}$. 
\end{remark}

We also define:
\begin{equation}\label{ch5:bisoa9}
\begin{array}{l}
\ds\ga^{\nu,k}_j=\norm{\psi(2^{-j}\la)\eta^\nu_j(\o)\phi_k(\la)f}_{\le{2}},\,j\geq 0,\,\nu\in\Gamma,1\leq k\leq |\nu-\nu'|^{-\a}, 
\end{array}
\end{equation}
which satisfy:
\begin{equation}\label{ch5:bisoa10}
(\gamma^{\nu}_j)^2=\sum_{1\leq k\leq |\nu-\nu'|^{-\a}}(\ga^{\nu,k}_j)^2.
\end{equation}

\subsubsection{The two key estimates}\label{ch5:bissec:keyest} 

We will prove the following two estimates:
\begin{equation}\label{ch5:bisoa11}
\begin{array}{r}
\ds\left|\int_{\Sigma}U^{\nu,k}_jf(x)\overline{U^{\nu',k}_jf(x)}d\Sigma \right|\lesssim \frac{D^2\ga_j^{\nu,k}\ga_j^{\nu',k}}{2^{j\a/2}(2^{j/2}|\nu-\nu'|)^{2-\a}}+\frac{D^2\ga_j^{\nu,k}\ga_j^{\nu',k}}{(2^{j/2}|\nu-\nu'|)^{3}}\\
\ds\textrm{for }|\nu-\nu'|\neq 0,\,1\leq k\leq |\nu-\nu'|^{-\a},
\end{array}
\end{equation}
and
\begin{equation}\label{ch5:bisoa12}
\begin{array}{r}
\ds\left|\int_{\Sigma}U^{\nu,k}_jf(x)\overline{U^{\nu',k'}_jf(x)}d\Sigma \right|\lesssim \frac{D^2\ga_j^{\nu,k}\ga_j^{\nu',k'}}{|k-k'|2^{j/2(1-4\a)}(2^{j/2}|\nu-\nu'|)^{1+4\a}},\\
\ds\textrm{ for }|\nu-\nu'|\neq 0,\,1\leq k,k'\leq |\nu-\nu'|^{-\a},\,k\neq k'.
\end{array}
\end{equation}
\eqref{ch5:bisoa11} and \eqref{ch5:bisoa12} imply:
\bea\label{ch5:bisoa13}
&&\left|\int_{\Sigma}U^{\nu}_jf(x)\overline{U^{\nu'}_jf(x)}d\Sigma \right| \\
\nn&\leq & \sum_{1\leq k\leq |\nu-\nu'|^{-\a}}\left|\int_{\Sigma}U^{\nu,k}_jf(x)\overline{U^{\nu',k}_jf(x)}d\Sigma \right|+\sum_{1\leq k\neq k'\leq |\nu-\nu'|^{-\a}}\left|\int_{\Sigma}U^{\nu,k}_jf(x)\overline{U^{\nu',k'}_jf(x)}d\Sigma \right|\\
\nn& \lesssim & \ds\sum_{1\leq k\leq |\nu-\nu'|^{-\a}}\frac{D^2\ga_j^{\nu,k}\ga_j^{\nu',k}}{2^{j\a/2}(2^{j/2}|\nu-\nu'|)^{2-\a}} +\sum_{1\leq k\leq |\nu-\nu'|^{-\a}}\frac{D^2\ga_j^{\nu,k}\ga_j^{\nu',k}}{(2^{j/2}|\nu-\nu'|)^{3}}\\
\nn& &\ds +\sum_{1\leq k\neq k'\leq |\nu-\nu'|^{-\a}}\frac{D^2\ga_j^{\nu,k}\ga_j^{\nu',k'}}{|k-k'|2^{\frac{j}{2}(1-4\a)}(2^{j/2}|\nu-\nu'|)^{1+4\a}}\\
\nn& \lesssim & \ds\frac{D^2\ga_j^{\nu}\ga_j^{\nu'}}{2^{j\a/2}(2^{j/2}|\nu-\nu'|)^{2-\a}}+\frac{D^2\ga_j^{\nu}\ga_j^{\nu'}}{(2^{j/2}|\nu-\nu'|)^{3}},
\eea
where we have used \eqref{ch5:bisoa10} and the fact that we may choose $0<\a<1/5$, together with the fact that:
\begin{equation}\label{ch5:bisoa14}
\sup_{1\leq k\leq |\nu-\nu'|^{-\a}}\sum_{1\leq k'\leq |\nu-\nu'|^{-\a},\,k'\neq k}\frac{1}{|k-k'|}\lesssim \a |\log(|\nu-\nu'|)|.
\end{equation}
Since \eqref{ch5:bisoa13} yields the wanted estimate \eqref{ch5:bisoa2}, we are left with proving \eqref{ch5:bisoa11} and \eqref{ch5:bisoa12}. The discussion in the following section will be very informal for the sake of simplicity. We refer to \cite{param2} for the details.

\subsubsection{Proof of \eqref{ch5:bisoa11}} The estimate \eqref{ch5:bisoa11} will result of two integrations by parts with respect to tangential derivatives (in the spirit of section \ref{ch3:sec:ibptangential}). Let us consider for instance the case where the two tangential derivatives fall on the symbol $b$ of $U^{\nu,k}_jf$ defined in \eqref{ch5:bisoa8}. This yields a term of the form
\be\lab{ch5:seppi}
\int_{\S}\nabb^2b\left(\int_{0}^{+\infty}e^{i\lambda u}\psi(2^{-j}\lambda)\phi_k(\la)f(\lambda\o)\lambda^2d\lambda\right)\eta_j^\nu(\o)d\o.
\ee
Then, in view of the estimate \eqref{ch5:thregx1s} for $b$, we have in particular $\nabb^2b\in L^2(\Sigma)$ which will force us to estimate the $\la$ integral in \eqref{ch5:seppi} in $L^\infty_u$. To this end, we do Cauchy-Schwarz and obtain in particular the square root of the diameter of $I_k$. Due to \eqref{ch5:bisoa6}, we thus gain an additional factor of $|\nu-\nu'|^\a$ with respect to \eqref{ch5:bisoa4}, which yields \eqref{ch5:bisoa11}. 

\begin{remark}\lab{ch5:rem:seppi}
Here, the log-loss is removed by exploiting the size of the diameter of $I_k$. This is possible since we estimate the $\la$ integral of \eqref{ch5:seppi} in $L^\infty_u$ using Cauchy Schwartz. In turn, this is a consequence of our estimate for $\nabb^2b$ in $L^2(\Sigma)$. This explains why this method cannot be used to remove the log-loss of the error term in Chapter \ref{part:paramtime}. Indeed, our estimates for the space-time foliation in Chapter \ref{part:spacetimeu} are typically of the type $L^\infty_uL^2(\HH_u)$, so that the integral in $\la$ is estimated in $L^2_u$ using Plancherel. In turn, this does not allow us to see the size of the localization in $\la$, so that a second decomposition in frequency of the type \eqref{ch5:bisoa7} would be useless in that case.
\end{remark}

\subsubsection{Proof of \eqref{ch5:bisoa12}} Note that we not only need to gain summability in $(\nu, \nu')$ for this term, but also in $(k, k')$. This is achieved through the presence of the additional gain $k-k'$ in the right-hand side of \eqref{ch5:bisoa12}. The estimate \eqref{ch5:bisoa12} will result of two integrations by parts, one with respect to the normal derivative $N$, and one with respect to tangential derivatives. We obtain a term analogous to \eqref{ch5:seppi} 
\be\lab{ch5:seppi1}
\int_{\S}\nabb\nabn b\left(\int_{0}^{+\infty}e^{i\lambda u}\psi(2^{-j}\lambda)\phi_k(\la)f(\lambda\o)\lambda^2d\lambda\right)\eta_j^\nu(\o)d\o.
\ee
In view of the estimate \eqref{ch5:thregx1s} for $b$, we have in particular $\nabb\nabn b\in L^2(\Sigma)$. Thus, the log-loss of the summation in $(\nu, \nu')$ is removed as in the previous section, in particular using the size of the diameter of $I_k$. Note also that the gain $k-k'$ in the right-hand side of \eqref{ch5:bisoa12} comes from the integration by parts in $N$. Indeed, we use
\begin{equation}\label{ch5:bisoa43b}
e^{i\la u-i\la' u'}=-\frac{ia}{\la-\la'\frac{a}{a'}g(N,N')}\nabn(e^{i\la u-i\la' u'}).
\end{equation}
Now, since $\la\in I_k$, $\la\in I_{k'}$, we have in view of \eqref{ch5:bisoa6}, the assumption \eqref{ch5:threomega3} for $\po^\a a$, and the assumption \eqref{ch5:threomega1bis} for $\po N$
$$\left|\la-\la'\frac{a}{a'}g(N,N')\right|\sim |k-k'|2^j|\nu-\nu'|^\a,$$
which yields the gain $k-k'$ in the right-hand side of \eqref{ch5:bisoa12}. 

\subsection{Proof of Proposition \ref{ch5:bisdiagonal} (control of the diagonal term)}\label{ch5:bissec:diagonal}

We have to prove \eqref{ch5:bisdiagonal1}:
\begin{equation}\label{ch5:bisdi1}
\norm{U^\nu_jf}_{L^2(\Sigma)}\lesssim D\ga^\nu_j.
\end{equation}
Recall that $U^\nu_j$ is given by:
\begin{equation}\label{ch5:bisdi2}
U^\nu_jf(x)=\int_{\S}bF_j( u)\eta_j^\nu(\o)d\o,
\end{equation}
where $F_j( u)$ is defined by:
\begin{equation}\label{ch5:bisdi3}
F_j( u)=\int_0^{+\infty}e^{i\la u}\psi(2^{-j}\la)f(\la\o)\la^2d\la.
\end{equation}
We decompose $U^\nu_j$ in the sum of two terms:
$$U^\nu_jf(x)=b(x,\nu)\int_{\S}F_j( u)\eta_j^\nu(\o)d\o+\int_{\S}(b(x,\o)-b(x,\nu))F_j( u)\eta_j^\nu(\o)d\o.$$
Then, using in particular the assumption \eqref{ch5:thregx1s} for $b$ and the assumption \eqref{ch5:threomega1s} for $\po b$, we 
obtain
\be\lab{ch5:bisdi4}
\norm{U^\nu_jf}_{L^2(\Sigma)}\lesssim D\normm{\int_{\S}F_j( u)\eta_j^\nu(\o)d\o}_{L^2(\Sigma)}+D\ga^\nu_j.
\ee
In order to estimate the right-hand side of \eqref{ch5:bisdi4}, we use the following proposition.
\begin{proposition}\label{ch5:bisdiprop1}
We have the following bound:
\begin{equation}\label{ch5:bisdiprop2}
\normm{\int_{\S}F_j( u)\eta_j^\nu(\o)d\o}_{L^2(\Sigma)}\lesssim\gamma_j^\nu.
\end{equation}
\end{proposition}
The proof of Proposition \ref{ch5:bisdiprop1} is postponed to the next section. Finally, \eqref{ch5:bisdi4} and \eqref{ch5:bisdiprop2} 
yield  the wanted estimate \eqref{ch5:bisdi1} which concludes the proof of Proposition \ref{ch5:bisdiagonal}

\subsubsection{Proof of Proposition \ref{ch5:bisdiprop1}}\label{ch5:sec:bisdiprop1}

Recall that $\int_{\S}F_j( u)\eta_j^\nu(\o)d\o$ is given by:
\begin{equation}\label{ch5:di10}
\int_{\S}F_j( u)\eta_j^\nu(\o)d\o=\int_{\S}\int_0^{+\infty}e^{i\la u}\psi(2^{-j}\la)\eta_j^\nu(\o)f(\la\o)\la^2d\la d\o.
\end{equation}
Relying on the classical $TT^*$ argument, \eqref{ch5:bisdiprop2} is equivalent to proving the boundedness on $L^2(\Sigma)$ of the operator whose kernel $K$ is given by:
\begin{equation}\label{ch5:di11}
K(x,y)=\int_{\S}\int_0^{+\infty}e^{i\la u(x,\o)-i\la u(y,\o)}\psi(2^{-j}\la)\eta_j^\nu(\o)\la^2d\la d\o,\,x, y\in\s.
\end{equation}
The decay satisfied by this kernel is stated in the following proposition.
\begin{proposition}\label{ch5:di12}
The kernel $K$ defined in \eqref{ch5:di11} satisfies the following decay estimate for all $x, y$ in $\Sigma$:
\begin{equation}\label{ch5:di13}
\begin{array}{ll}
\ds |K(x,y)|\lesssim & \ds\frac{2^j}{(1+|2^j|u(x,\nu)-u(y,\nu)|-2^{j/2}|\po u(x,\nu)-\po u(y,\nu)||)^2}\\
& \ds\times\frac{2^j}{(1+2^{j/2}|\po u(x,\nu)-\po u(y,\nu)|)^3}.
\end{array}
\end{equation}
\end{proposition}
The proof of Proposition \ref{ch5:di12} is postponed to section \ref{ch5:sec:di20}. In the rest of this section, we show how \eqref{ch5:di13} implies Proposition \ref{ch5:bisdiprop1}. According to Schur's Lemma, the operator whose kernel is $K$ is bounded on $L^2(\Sigma)$ provided we can prove the following bound:
\begin{equation}\label{ch5:di14}
\sup_{x\in\Sigma}\int_{\Sigma} |K(x,y)|dy<+\infty,\,\sup_{y\in\Sigma}\int_{\Sigma} |K(x,y)|dx<+\infty.
\end{equation}
Due to the symmetry of $K$ in $x, y$, the two bounds in \eqref{ch5:di14} are obtained in the same way. We 
focus on establishing the first bound. Using in particular \eqref{ch5:di13} and the global change of variable on $\Sigma$ given by \eqref{ch5:gl21}\footnote{using also the bound on the Jacobian \eqref{ch5:gl22}}, we are able to obtain:
\begin{equation}\label{ch5:di18}
\ds \int_{\Sigma} |K(x,y)|dy\lesssim\ds\int_{\R^3}\frac{2^j}{(1+|2^j|\underline{y}\c\nu|-2^{j/2}|\underline{y}'||)^2}\frac{2^j}{(1+2^{j/2}|\underline{y}'|)^3}d\underline{y},
\end{equation}
where $y=y\c\nu+y'$ and $y'\c\nu=0$. Making the change of variable $y\rightarrow z$ where $z$ is defined by $z\c\nu=2^j\underline{y}\c\nu$ and $z'=2^{j/2}\underline{y}'$ in the right-hand side of \eqref{ch5:di18}, and remarking that $z\cdot\nu$ is one dimensional, and $z'$ is two dimensional, we obtain:
\begin{equation}\label{ch5:di19}
\ds \int_{\Sigma} |K(x,y)|dy\lesssim\ds\int_{\R^3}\frac{dz}{(1+||z\c\nu|-|z'||)^2(1+|z'|)^3}\lesssim 1.
\end{equation}
\eqref{ch5:di19} implies the first bound in \eqref{ch5:di14}. $K$ being symmetric with respect to $x, y$, the second bound in \eqref{ch5:di14} is also true. Thus, the operator whose kernel is $K$ is bounded on $L^2(\Sigma)$ which concludes the proof of Proposition \ref{ch5:bisdiprop1}. 

\subsubsection{Proof of Proposition \ref{ch5:di12}}\label{ch5:sec:di20}

Recall the definition of $K$:
\begin{equation}\label{ch5:di21}
K(x,y)=\int_{\S}\int_0^{+\infty}e^{i\la u(x,\o)-i\la u(y,\o)}\psi(2^{-j}\la)\eta_j^\nu(\o)\la^2d\la d\o,\,x, y\in\s.
\end{equation}
We need to prove that $K$ satisfies the following decay estimate for all $x, y$ in $\Sigma$:
\begin{equation}\label{ch5:di22}
\begin{array}{ll}
\ds |K(x,y)|\lesssim & \ds\frac{2^j}{(1+|2^j|u(x,\nu)-u(y,\nu)|-2^{j/2}|\po u(x,\nu)-\po u(y,\nu)||)^2}\\
& \ds\times\frac{2^j}{(1+2^{j/2}|\po u(x,\nu)-\po u(y,\nu)|)^3}.
\end{array}
\end{equation}
For the sake of simplicity, let us just describe the general strategy of the proof of Proposition \ref{ch5:di12}. In view of the regularity for $u(x,\o)$ with respect to $\o$ provided by \eqref{ch5:thregx1} and \eqref{ch5:threomega1}-\eqref{ch5:threomega3bis}, we have
\begin{equation}\label{ch5:reguomega}
|u(x,\o)|+|\po u(x,\o)|+|\po^2u(x,\o)|+|\po^3u(x,\o)|\lesssim 1+|x|,\,\forall x\in\Sigma,\,\forall \o\in\S.
\end{equation}
This regularity allows us to integrate by part three times with respect to $\o$, while we may integrate as much as we want with respect to $\la$. The estimate \eqref{ch5:di22} is then obtained after performing in \eqref{ch5:di21} three integrations  by parts with respect to $\o$ and two integrations by parts with respect to $\la$.

\section{Control of the parametrix at initial time}\label{ch5:sec:th2}

In this section, we discuss the proof of Theorem \ref{ch5:th2}. To this end, we first show that the Fourier integral operator $U$ of Theorem \ref{ch5:th1} almost preserves the $L^2$ norm provided we make additional assumptions on its symbol. We then use this observation to prove the estimate \eqref{ch5:choicef4bis}. Finally, we conclude the proof of Theorem \ref{ch5:th2} by establishing the existence and uniqueness of $(f_+,f_-)$ solution of the system \eqref{ch5:choicef3bis}. 

\subsection{A refinement of Theorem \ref{ch5:th1}}

In Theorem \ref{ch5:th1}, we have proved that the Fourier integral operator $U$ with phase $u$ and symbol 
$b$ is bounded on $L^2(\Sigma)$ provided $u$ satisfies \eqref{ch5:thregx1}, \eqref{ch5:threomega1}-\eqref{ch5:threomega3bis} and \eqref{ch5:gl21}-\eqref{ch5:gl22}, and the symbol $b$ satisfies \eqref{ch5:thregx1s} \eqref{ch5:threomega1s}. We now would like to prove that $U$ satisfies the following bound from below:
\begin{equation}\label{ch5:re1}
\norm{f}_{\le{2}}\lesssim\norm{Uf}_{L^2(\Sigma)},
\end{equation}
provided $u$ also satisfies \eqref{ch5:gl23} and under additional assumptions on the symbol $b$. This is the aim of the following proposition.

\begin{proposition}\label{ch5:re2}
Let $u$ be a function on $\Sigma\times\S$ satisfying suitable assumptions (we refer to \cite{param2} for the complete set of assumptions, and to section \ref{ch5:sec:estneeded} for some typical assumptions). Let $U$ the Fourier integral operator with phase $u(x,\o)$ and symbol $b(x,\o)$:
\begin{equation}\label{ch5:fio1}
Uf(x)=\int_{\S}\int_{0}^{+\infty}e^{i\lambda u(x,\o)}b(x,\o)f(\lambda\o)\lambda^2 d\lambda d\o.
\end{equation}
We assume furthermore that $b(x,\o)$ satisfies:
\begin{equation}\label{ch5:re3}
\norm{\po b}_{L^2(\Sigma)}+\norm{\nabla\po b}_{L^2(\Sigma)}\lesssim 1,
\end{equation}
\begin{equation}\label{ch5:re4}
\norm{b-1}_{\ll{\infty}}+\norm{\nabla b}_{\l{\infty}{2}}+\norm{\nabb\nabla b}_{L^2(\Sigma)}\lesssim \ep,
\end{equation}
and
\begin{equation}\label{ch5:re5}
\begin{array}{l}
\ds\nabn b=b^j_1+b^j_2\textrm{ where }\norm{b^j_1}_{L^2(\Sigma)}\lesssim 2^{-\frac{j}{2}}\ep,\,\norm{b^j_2}_{\l{\infty}{2}}\lesssim\ep\\
\ds\textrm{and }\norm{\nabn b^j_2}_{L^2(\Sigma)}+\norm{b^j_2}_{\l{2}{\infty}}\lesssim 2^{\frac{j}{2}}\ep.
\end{array}
\end{equation}
Then, $U$ is bounded on $L^2$ and satisfies the estimate:
\begin{equation}\label{ch5:re6}
\norm{f}_{\le{2}}\lesssim\norm{Uf}_{L^2(\Sigma)}.
\end{equation}
\end{proposition}

\begin{remark}
Notice that the only difference in the assumptions with respect to Theorem \ref{ch5:th1} lies in the fact that $u$ also satisfies \eqref{ch5:gl23} and in the constant $D$ which has been replaced by 1 in \eqref{ch5:re3} and by $\ep$ in \eqref{ch5:re4} \eqref{ch5:re5}.
\end{remark}

The proof of Proposition \ref{ch5:re2} uses the decomposition in frequency and angle of the operator $U$ introduced in section \ref{ch5:sec:th1}. In order to control the diagonal term in a third step (see next section), we have to modify slightly the size of the support of our partition of unity $\eta_j^\nu$ on $\S$ introduced in \eqref{ch5:bisb14}. Let $\de>0$ such that:
\begin{equation}\label{ch5:re12}
0<\sqrt{\ep} <\!\!< \de <\!\!< 1.
\end{equation}
We now require that the support of $\eta^\nu_j$ is a patch on $\S$ of diameter $\sim \de 2^{-j/2}$. With this modification, the assumptions for $b$ in Proposition \ref{ch5:re2}, and by carefully tracking the size of the various terms in the almost orthogonality argument in frequency  and angle, and the control of the diagonal term, we obtain
\begin{equation}\label{ch5:re13}
\norm{Uf}_{L^2(\Sigma)}^2=\sum_{|j-l|\leq 2}\sum_{|\nu-\nu'|\leq 2\de 2^{-j/2}}\int_{\Sigma}S_j^\nu f(x)\overline{S_l^{\nu'}f(x)}d\Sigma+O\left(\frac{\ep}{\de^2}+\delta\right)\norm{f}^2_{\le{2}},
\end{equation}
where the operator $S_j^\nu$ is defined on $\Sigma$ by:
\begin{equation}\label{ch5:re14}
S_j^\nu f(x)=\int_{\Sigma}\int_{0}^{+\infty}e^{i\la u(x,\o)}\psi(2^{-j}\la)\eta_j^\nu(\o) f(\la\o)\la^2 d\la d\o.
\end{equation}

Recall \eqref{ch5:gl21}-\eqref{ch5:gl22} which states that the map $\phi_\nu:\Sigma\rightarrow\R^3$ defined by:
\begin{equation}\label{ch5:re21}
\phi_\nu(x):=u(x,\nu)\nu+\po u(x,\nu),
\end{equation}
is a bijection, such that the determinant of its Jacobian satisfies the following estimate:
\begin{equation}\label{ch5:re22}
\norm{|\det(\textrm{Jac}\phi_\nu)|-1}_{\ll{\infty}}\lesssim\ep.
\end{equation}
Let us note $\FF^{-1}$ the inverse Fourier transform on $\R^3$. We introduce the operator $\widetilde{S}_j^\nu$ on $\Sigma$ defined by:
\begin{equation}\label{ch5:re24}
\widetilde{S}_j^\nu f(x)=\FF^{-1}(\psi(2^{-j}\c)\eta_j^\nu f)(\phi_\nu(x))=\int_{\R^3}e^{i\la\phi_\nu(x)\c\o}\psi(2^{-j}\la)\eta_j^\nu(\o)f(\la\o)\la^2 d\la d\o.
\end{equation}
The following proposition shows that the term $\int_{\Sigma}S_j^\nu f(x)\overline{S_l^{\nu'}f(x)}d\Sigma$ is close to the term $\int_{\Sigma}\widetilde{S}_j^\nu f(x)\overline{\widetilde{S}_l^{\nu'}f(x)}d\Sigma$.
\begin{proposition}\label{ch5:re25}
We have the following bound:
\begin{equation}\label{ch5:re26}
\norm{S_j^\nu f-\widetilde{S}_j^\nu f}_{L^2(\Sigma)}\lesssim\de^{\frac{1}{2}}\gamma_j^\nu.
\end{equation}
\end{proposition}
The proof of Proposition \ref{ch5:re25} relies on the classical $TT^*$ argument, and the comparison between $u(x,\o)$ and $\phi_\nu(x)\cdot\o$ provided by \eqref{ch5:gl23} (see \cite{param2} for the details). Now, \eqref{ch5:re13} and \eqref{ch5:re26} yield:
\begin{equation}\label{ch5:re27}
\norm{Uf}_{L^2(\Sigma)}^2=\sum_{|j-l|\leq 2}\sum_{|\nu-\nu'|\leq 2\de 2^{-j/2}}\int_{\Sigma}\widetilde{S}_j^\nu f(x)\overline{\widetilde{S}_l^{\nu'}f(x)}d\Sigma+O\left(\frac{\ep}{\de^2}+\de^{\frac{1}{2}}\right)\norm{f}^2_{\le{2}}.
\end{equation}
Making the change of variable $y=\phi_\nu(x)$ in $\int_{\Sigma}\widetilde{S}_j^\nu f(x)\overline{\widetilde{S}_l^{\nu'}f(x)}d\Sigma$ and using \eqref{ch5:re22} and \eqref{ch5:re24} implies:
\begin{equation}\label{ch5:re28}
\begin{array}{ll}
&\ds\sum_{|j-l|\leq 2}\sum_{|\nu-\nu'|\leq 2\de 2^{-j/2}}\int_{\Sigma}\widetilde{S}_j^\nu f(x)\overline{\widetilde{S}_l^{\nu'}f(x)}d\Sigma\\
\ds = &\ds\sum_{|j-l|\leq 2}\sum_{|\nu-\nu'|\leq 2\de 2^{-j/2}}\int_{\R^3}\FF^{-1}(\psi(2^{-j}\c)\eta_j^\nu f)(y)\overline{\FF^{-1}(\psi(2^{-l}\c)\eta_j^{\nu'} f)(y)}dy\\
&\ds +O(\ep)\norm{f}^2_{\le{2}}\\
\ds = &\ds\sum_{|j-l|\leq 2}\sum_{|\nu-\nu'|\leq 2\de 2^{-j/2}}\int_{\R^3}\psi(2^{-j}\la)\eta_j^\nu(\o) f(\la\o)\overline{\psi(2^{-l}\la)\eta_j^{\nu'}(\o) f(\la\o)}dy\\
&\ds +O(\ep)\norm{f}^2_{\le{2}},
\end{array}
\end{equation}
where we have used the fact that $\FF^{-1}$ is an isomorphism on $\le{2}$ in the last equality of \eqref{ch5:re28}. 
Now, we have:
\begin{equation}\label{ch5:re29}
\ds\sum_{|j-l|\leq 2}\sum_{|\nu-\nu'|\leq 2\de 2^{-j/2}}\int_{\R^3}\psi(2^{-j}\la)\eta_j^\nu(\o) f(\la\o)\overline{\psi(2^{-l}\la)\eta_j^{\nu'}(\o) f(\la\o)}dy=\norm{f}^2_{\le{2}},
\end{equation}
which  together with \eqref{ch5:re27} and \eqref{ch5:re28} yields:
\begin{equation}\label{ch5:re30}
\norm{Uf}_{L^2(\Sigma)}^2=\norm{f}_{\le{2}}^2+O\left(\frac{\ep}{\de^2}+\de^{\frac{1}{2}}\right)\norm{f}^2_{\le{2}}.
\end{equation}
Choosing $\de^{\frac{1}{2}}$ and $\ep\de^{-2}$ small enough, we deduce from \eqref{ch5:re30}:
\begin{equation}\label{ch5:re31}
\norm{f}_{\le{2}}\lesssim\norm{Uf}_{L^2(\Sigma)},
\end{equation}
which is the wanted estimate. This conclude the proof of Proposition \ref{ch5:re2}. 

\subsection{Proof of Theorem \ref{ch5:th2}}

The symbol of the Fourier integral operators $M_\pm$ and $Q_\pm$ are respectively given by 1 and $a(x,\pm\o)^{-1}$. Thus, they clearly satisfy the assumptions of Proposition \ref{ch5:re2}. Relying on Proposition \ref{ch5:re2}, we are then able to prove the estimate \eqref{ch5:choicef4bis}. We refer to \cite{param2} for the details.

The uniqueness of $(f_+,f_-)$ solution of the system \eqref{ch5:choicef3bis} is an immediate consequence of the estimate \eqref{ch5:choicef4bis}, so there only remains to prove the existence of $(f_+,f_-)$ to conclude the proof 
of Theorem \ref{ch5:th2}. Recall that the phase $u(x,\o)$ of our Fourier integral operators has been constructed in \cite{param1} (see also Chapter \ref{part:initialu}) on $\Sigma\times\S$ under the assumption that $(\Sigma,g,k)$ satisfies the following bounds consistent with the bounded $L^2$ curvature conjecture:
\be\lab{ch5:l2bounds}
\norm{R}_{L^2(\Sigma)}\leq\ep,\,\norm{\nabla k}_{L^2(\Sigma)}\leq\ep,
\ee
where the fact that we may take $\ep>0$ small comes from a reduction to the small data case. $(\Sigma,g,k)$ also satisfies the constraint equations and the maximal foliation assumption
\be\lab{ch5:const1}
\left\{\begin{array}{l}
\nabla^j k_{ij}=0,\\
 R=|k|^2,\\
 \textrm{Tr}k=0.
\end{array}\right.
\ee
We introduce two sets $V$ and $W$:
\be\lab{ch5:defv}
V=\{(\Sigma,g,k)\textrm{ such that \eqref{ch5:l2bounds} and \eqref{ch5:const1} are satisfied}\},
\ee
and 
\be\lab{ch5:defw}
\begin{array}{ll}
\ds W=\{ & \ds(\Sigma,g,k)\in V\textrm{ such that }(f_+,f_-)\textrm{ solution of }\eqref{ch5:choicef3bis}\textrm{ exist for all }(\phi_0,\phi_1)\\
& \ds\textrm{ such that }\nabla\phi_0\in L^2(\Sigma)\textrm{ and }\phi_1\in L^2(\Sigma)\}.
\end{array}
\ee
Not first that $W$ is not empty since $(\Sigma,g,k)=(\R^3,\de,0)$ belongs to $W$ in view of Remark \ref{rem:theflatcase}. We then show that $V$ is connected and $W$ is both open and closed in $V$ for a suitable topology (see \cite{param2} for the details). We infer $W=V$. This proves the existence of $(f_+,f_-)$ solution of \eqref{ch5:choicef3bis} and concludes the proof of Theorem \ref{ch5:th2}.

%%%%%%%%%%%%%%%%%%%%%%%%%%%%%%%%%%%%%%%

\chapter{Control of the foliation at initial time}\lab{part:initialu}

%%%%%%%%%%%%%%%%%%%%%%%%%%%%%%%%%%%%%%%

%% Control of the foliation at initial time

In this chapter, we will only consider $u(t,x,\o)$ and $\Sigma_t$ at $t=0$. Thus, for simplicity, we denote in the rest of this chapter $u(0, x,\o)$ by $u(x,\o)$ and $\Sigma_0$ by $\Sigma$. The goal of this chapter is to prove the estimates on the control of the foliation of $\Sigma$ by $u(x,\o)$ which are needed for the proof of Theorem \ref{ch5:th2} (see section \ref{ch5:sec:estneeded}), i.e. for the control of the parametrix at initial time. The estimates obtained for $u(x,\o)$ in this chapter must also be consistent with the control on $\MM$ for $u(t,x,\o)$ obtained in Chapter \ref{part:spacetimeu} (see section \ref{ch4:sec:mainres}). Here, we outline the main ideas and we refer to \cite{param1} for the details. 

\section{Geometric set-up and main results}

\subsection{Reduction to small data}\label{reducsmall}

Recall from section \ref{sec:reductionsmall} that we have reduced ourselves to an asymptotically flat initial data set $(\Sigma, g,k)$ solution to the constraint equations which satisfies the bounds
\be\lab{small1}
\norm{R}_{\lli{2}}\leq\varepsilon,\,\norm{\nabla k}_{\lli{2}}\leq\varepsilon, 
\ee
and is smooth outside of a small neighborhood $U$. In order to construct $u(x,\o)$ satisfying the asymptotic behavior 
$u(x,\o)\sim \xo$ when $|x|\rightarrow +\infty$ on $\Sigma$, we need to modify $(\Sigma, g,k)$ outside of $U$. We can glue it in a trivial way to $(\R^3,\delta,0)$ so that the new initial data set is still smooth outside of $U$, satisfies \eqref{small1}, and coincides with $(\R^3,\delta,0)$ outside of a slightly larger neighborhood. We still denote this initial data set $(\Sigma, g,k)$. Of course, $(\Sigma, g,k)$ does not satisfies the constraint equations in the annulus where the gluing takes place. However, for the construction of $u(x,\o)$, we only require $(\Sigma, g,k)$ to satisfy the constraint equations in $U$. Outside of $U$, $(\Sigma, g,k)$ is smooth, so things are much easier.

\subsection{Geometry of the foliation of $\s$ by a scalar function $u$}\label{sec:foliation}

We define the lapse $a=|\nabla u|^{-1}$, and the unit vector $N$ such that $\nabla u=a^{-1}N$. We also define the level surfaces $P_{u_0}=\{x\,/\,u=u_0\}$ so that $N$ is the normal to $\p$. The second fundamental form $\th$ of $\p$ is defined by 
\be\lab{eq:I6}
\th(X,Y)=g(\nabla_{X}N,Y)
\end{equation}
with $X,Y$ arbitrary vectorfields tangent to the  $u$-foliation $P_{u}$ of $\s$ and where $\nabla$ denotes the covariant differentiation with respect to $g$. We denote by $\trt$ the trace of $\th$,
i.e. $\trt=\delta^{AB}\th_{AB}$ where $\th_{AB}$ are the components of $\th$ relative to an
orthonormal frame $(e_A)_{A=1,2}$ on $P_u$. 

\subsection{Structure equations of the foliation of $\s$ by a scalar function $u$}

We recall some of the structure equations of the foliation of $\s$ by a scalar function $u$ which will be needed for the discussion of the main result of this chapter (see \cite{ChKl} for a proof). 
\begin{proposition}
The orthonormal frame frame $N, e_A, A=1, 2$ of $\s$ satisfies the following system:
\be\lab{frame}
\left\{\begin{array}{l}
\nabla_AN=\th_{AB}e_B,\\[1mm]
\nabn N=-a^{-1}\nabb a.
\end{array}\right.
\ee
Also, the lapse $a$ satisfies the following equation
\be\lab{struct}
\lapa=-\nabn\trt-|\th|^2+R_{NN},
\ee
where $\lap$ is the Laplace-Beltrami for the metric $\gamma$ on $\p$ induced by $g$. Finally, the second fundamental form $\th$ satisfies the following Codazzi equation
\be\lab{codazzith}
\nabb^B\th_{AB}=\nabb_A\trt+R_{NA}.
\ee
\end{proposition}

%\subsection{Commutation formulas}

%We have the following useful commutation formulas between $\nabb$ and $\nabn$ (see \cite{ChKl} page 64).
%\begin{lemma}
%For any $\p$-tangent tensor $F$ on $\Sigma$, we have schematically:
%\be\lab{commut}
%[\nabn,\nabb]F=\ana\nabn F-\theta\nabb F+R_{N.}F+\theta\ana F.
%\ee
%In particular, we obtain for any scalar $f$ on $\Sigma$:
%\be\lab{scommut}
%[\nabn,\nabb]f=\ana\nabn f-\th\nabb f,
%\ee
%and:
%\begin{equation}\label{ad12}
%\begin{array}{ll}
%\ds [\nabn,\lap]f = & \ds -\trt\lap f -2\hth\nabb^2f+2a^{-1}\nabb a\nabb\nabn f+a^{-1}\lap a \nabn f-2R_{N.}\nabb f\\
%&\ds -\nabb\trt\nabb f +2\hth a^{-1}\nabb a\nabb f.
%\end{array}
%\end{equation}
%\end{lemma}

\subsection{Choice of $u(x,\o)$}\label{sec:choice}

We look for $u(x,\o)$ satisfying the three following conditions:
{\em\begin{enumerate}
\item[{\bf (a)}] $u(x,\o)\sim x.\o$ when $|x|\rightarrow +\infty$ on $\s$

\item[{\bf (b)}] The regularity of $u(x,\o)$ with respect to $x$ and $\o$ is consistent with the regularity of $u(t,x,\o)$ with respect to $(t,x)$ and $\o$ obtained in Chapter \ref{part:spacetimeu} (see section \ref{ch4:sec:mainres}). In particular, we have $\trt - k_{NN}\in\lli{\infty}$ (see the discussion in section \ref{sec:prescriptionu}) 

\item[{\bf (c)}] $u(x,\o)$ has as enough regularity in $x$ and $\o$ to control the parametrix at initial time, i.e. to obtain the conclusion of Theorem \ref{ch5:th2} 
\end{enumerate}
}
\noindent where the initial data set $(\Sigma, g,k)$ satisfies
\be\lab{const1}
\left\{\begin{array}{l}
\nabla^j k_{ij}=0,\\
 R=|k|^2,\\
\tr_g k=0, 
\end{array}\right.
\ee
in $U$ (see section \ref{reducsmall}), and where $R$ and $\nabla k$ are in $\lli{2}$ and satisfy the smallness assumption \eqref{small1}.

In order to motivate our choice of $u(x,\o)$, we investigate the regularity of the lapse $a$, which by \eqref{struct} satisfies the following equation:
\be\lab{eqlapse}
\lapa=-\nabn\trt-|\th|^2-R_{NN}.
\ee
Since $R$ is in $\lli{2}$, \eqref{eqlapse} implies that $a$ has at most two derivatives in $\lli{2}$. Thus, $u(x,\o)$ has at most three derivatives with respect to $x$ in $\lli{2}$. This is not enough to satisfy {\bf (c)}. In fact, the classical $T^*T$ argument (see for example \cite{St}) relies on integrations by parts in $x$ and would require at least one more derivative since $\s$ has dimension 3. 

Alternatively, we could try to use the $TT^*$ argument which relies on integrations by parts in $\o$. Indeed, $R$ being independent of $\o$, one would expect the regularity of $u(x,\o)$ with respect to $\o$ to be better. Differentiating \eqref{eqlapse} with respect to $\o$, we obtain:
\be\lab{diffom}
a^{-1}\lap(\partial_\o a)=2\nabb\nabn a+\cdots,
\ee
where the term on the right-hand side comes from the commutator $[\partial_\o,\lap]$ (see section \ref{regomega}). Thus, obtaining an estimate for $\partial_\o a$ from \eqref{diffom} requires to control $\nabn a$. Unfortunately, \eqref{eqlapse} seems to give control of tangential derivatives of $a$ only. This is where the specific choice of $u(x,\o)$ comes into play. 

Having in mind the equation of minimal surfaces (i.e. $\trt=0$), condition {\bf (b)} suggest the choice $\trt - k_{NN}=0$. Unfortunately, this equation together with \eqref{eqlapse} does not provide any control of $\nabn a$. We might propose as a second natural guess to take instead $\trt -k_{NN}=\nabn a$. Plugging in \eqref{eqlapse} yields an elliptic equation for $a$: $\nabn^2a+\lapa=-|\th|^2-\nabn k_{NN}-R_{NN}$. This allows us to control $\nabn^2a$ in $\lli{2}$. However, $\trt-k_{NN}=\nabn a$, and $\nabn a$ is at most in $H^1(\Sigma)$ which does not embed in $\lli{\infty}$ - since $\s$ has dimension 3 - so that condition {\bf (b)} is not satisfied. To sum up, the first guess $\trt-k_{NN}=0$ satisfies {\bf (b)}, but not {\bf (c)}, whereas the second guess $\trt-k_{NN}=\nabn a$ might satisfy {\bf (c)}, but does not satisfy {\bf (b)}. 

The correct choice is the intermediate one
\be\lab{choice}
\trt-k_{NN}=1-a.
\ee
We will see in section \ref{regx} that $a-1$ belongs to $\lli{\infty}$ so that {\bf (b)} is satisfied. Also, plugging \eqref{choice} in \eqref{eqlapse} yields the parabolic equation:
\be\lab{eqlapse1}
\nabn a-\lapa=|\th|^2+\nabn k_{NN}+R_{NN}.
\ee
This  will allow us to control normal derivatives of $a$. In turn, we will control derivatives of $a$ with respect to $\o$ using \eqref{diffom}. Ultimately, we will prove enough regularity with respect to both $x$ and $\o$, such that {\bf (c)} is  satisfied. 

\subsection{Main results}\label{sec:mainres}

For $1\leq p,q\leq +\infty$, we define the spaces $\l{p}{q}$ for tensors $F$ on $\s$ using the norm: 
$$\norm{F}_{\l{p}{q}}=\left(\int_u\norm{F}^p_{L^q(\p)}du\right)^{1/p}.$$

\begin{remark}
In the rest of the paper, all inequalities hold for any $\o\in\S$ with the constant in the right-hand side being independent of $\o$. Thus, one may take the supremum in $\o$ everywhere. To ease the notations, we do not explicitly write down this supremum. 
\end{remark}

We first state a result of existence and regularity with respect to $x$ for $u$.
\begin{theorem}\label{thregx}
Let $(\Sigma, g,k)$ chosen as in section \ref{reducsmall}. There exists a scalar function $u$ on $\s\times\S$ satisfying assumption {\bf (a)} and such that:
\begin{equation}\label{thregx1} 
\begin{array}{l}
\norm{a-1}_{\l{\infty}{2}}+\norm{\nabla a}_{\l{\infty}{2}}+\norm{a-1}_{\lli{\infty}}+\norm{\nabb\nabla a}_{\lli{2}}\lesssim\ep,\\
\norm{\trt-k_{NN}}_{\lli{\infty}}+\norm{\nabla\th}_{\lli{2}}\lesssim\ep,
\end{array}
\end{equation}
where $\p$, $a$, $N$ and $\th$ are associated to $u$ as in section \ref{sec:foliation}. 
\end{theorem}

Notice that condition {\bf (b)} is included in \eqref{thregx1}. In order to state our second result, we introduce fractional Sobolev spaces $H^b(\p)$ on the surfaces $\p$ for any $b\in\mathbb{R}$ (see \cite{param1} for their precise definition). We have the following estimate for $\nabn^2a$, and improved estimate for $\nabn a$.  

\begin{theorem}\label{thnabn2a}
Let $(\Sigma, g,k)$ chosen as in section \ref{reducsmall}. Let $u$ the scalar function on $\s\times\S$ constructed in theorem \ref{thregx}, and let $\p$, $a$ and $N$ be associated to $u$ as in section \ref{sec:foliation}. We have:
\be\lab{nabn2a1}
\norm{\nabn a}_{\l{\infty}{4}}+\norm{\nabn^2a}_{\lhs{2}{-\frac{1}{2}}}\les \ep.
\ee
\end{theorem}

The third theorem investigates the regularity of $u$ with respect to $\o$:
\begin{theorem}\label{thregomega}
Let $(\Sigma, g,k)$ chosen as in section \ref{reducsmall}. Let $u$ the scalar function on $\s\times\S$ constructed in theorem \ref{thregx}, and let $\p$, $a$, $N$ and $\th$ be associated to $u$ as in section \ref{sec:foliation}. We have:
\begin{equation}\label{threomega1}
\begin{array}{r}
\ds\norm{\po a}_{\lli{\infty}}+\norm{\nabb\po a}_{\l{\infty}{2}}+\norm{\nabb^2\po a}_{\lli{2}}+\norm{\nabn\po a}_{L^2_u H^{\frac{1}{2}}(\p)}\\
\ds +\norm{\nabn^2\po a}_{L^2_u H^{-\frac{3}{2}}(\p)}+\norm{\nabla\po\th}_{\lli{2}}\lesssim\ep,\,\,\,\,\norm{\po N}_{\lli{\infty}}\lesssim 1,
\end{array}
\end{equation}
\begin{equation}\label{threomega2}
\begin{array}{r}
\ds\norm{\po^2a}_{L^2_u H^{\frac{3}{2}}(\p)}+\norm{\po^2a}_{L^\infty_u H^{\frac{1}{2}}(\p)}+\norm{\nabn\po^2a}_{L^2_u H^{-\frac{1}{2}}(\p)}+\norm{\nabla\po^2\th}_{\lli{2}}\lesssim\ep,\\
\ds\norm{\po^2 N}_{\lli{\infty}}\lesssim 1,
\end{array}
\end{equation}
and
\begin{equation}\label{threomega3}
\norm{\po^3u}_{L^{\infty}_{\textrm{loc}}(\Sigma)}\lesssim 1.
\end{equation}
\end{theorem}

\subsection{Additional results}

The following proposition establishes the existence of a global coordinate system on $\Sigma$.
\begin{proposition}\label{gl20}
Let $\o\in\S$. Let $\phi_\o:\s\rightarrow\R^3$ defined by:
\begin{equation}\label{gl21}
\phi_\o(x):=u(x,\o)\o+\po u(x,\o).
\end{equation}
Then $\phi_\o$ is a bijection, and the determinant of its Jacobian satisfies the following estimate:
\begin{equation}\label{gl22}
\norm{|\det(\textrm{Jac}\phi_\o)|-1}_{\lli{\infty}}\lesssim\ep.
\end{equation}
\end{proposition}

Below, we state several additional estimates. We start with a first proposition.  
\begin{proposition}\lab{prop:estimatesadded}
Let $(\Sigma, g,k)$ chosen as in section \ref{reducsmall}. Let $u$ the scalar function on $\s\times\S$ constructed in theorem \ref{thregx}. Let $\nu\in\S$ and $\phi_\nu$ the map defined in \eqref{gl21}. Then, we have for all $x\in\s$ and $\o\in\S$
\begin{equation}\label{gogol2}
\begin{array}{l}
\ds u(x,\o)-\phi_\nu(x)\c\o=O(\ep|\o-\nu|^2),\\
\ds \po u(x,\o)-\po(\phi_\nu(x)\c\o)=O(\ep|\o-\nu|),\\
\ds \po^2 u(x,\o)-\po^2(\phi_\nu(x)\c\o)=O(\ep).
\end{array}
\end{equation}
\end{proposition}

Using the geometric Littlewood Paley projections $P_j$ on $P_u$ constructed in \cite{Kl-R6} (see section \ref{sec:LP}) together with the estimates for $\nabn a$ in \eqref{thregx1}, and the estimate for $\nabn^2a$ in \eqref{nabn2a1}, we obtain the following proposition:
\begin{proposition}\label{cordecfr}
Let $(\Sigma, g,k)$ chosen as in section \ref{reducsmall}. Let $u$ the scalar function on $\s\times\S$ constructed in theorem \ref{thregx}, and let $a$ and $N$ be associated to $u$ as in section \ref{sec:foliation}. For all $j\geq 0$, there are scalar functions $a^j_1$ and $a^j_2$ such that:
\begin{equation}\label{cordecfr1}
\nabn a=a^j_1+a^j_2\textrm{ where }\norm{a^j_1}_{\lli{2}}\lesssim 2^{-\frac{j}{2}}\ep\textrm{ and }\norm{\nabn a^j_2}_{\lli{2}}\lesssim 2^{\frac{j}{4}}\ep.
\end{equation}
\end{proposition} 

\begin{remark}\label{paslinfty}
Recall from section \ref{sec:choice} that we do not have enough regularity in $x$ to apply the $T^*T$ method. Alternatively, we could try the $TT^*$ method which relies on integration by parts in $\o$. But $\po^3u\in\lli{\infty}$ is also not enough and we would need at least one more derivative in $\o$. Nevertheless, using the regularity in $x$ and $\o$ obtained for $u$ in the present and the previous section, we are able to control the parametrix at initial time (see Chapter \ref{part:paraminit}).
\end{remark}

Let us conclude this section by mentioning several ingredients of \cite{param1} that have been omitted here for the 
sake of simplicity, and that have to be proved by relying on low regularity assumptions for $u$ which are consistent with the results stated in this section and the previous one: 
\begin{itemize}
\item estimates for the parabolic operator $(\nabn-a^{-1}\lap)$

\item estimates for $\th$ and $N$

\item product estimates in the Sobolev spaces $H^b(\p)$

\item embeddings on $\Sigma$ and $\p$

\item a control of the Gauss curvature of $\p$ 

\item Bochner inequalities on $\p$ 

\item estimates for various commutator terms of the type: $[\nabn, \nabb]$, $[\nabn, P_j]$, ...
\end{itemize}

The rest of this chapter is as follows. In section \ref{regx}, we discuss the proof of Theorem \ref{thregx}. In section \ref{sec:nabla2a}, we discuss the proof of Theorem \ref{thnabn2a}.  In section \ref{regomega}, we discuss the proof of Theorem \ref{thregomega}. In  section \ref{sec:globalcoord}, we discuss the proof of Proposition \ref{gl0} and Proposition \ref{gl20}. Finally, Proposition \ref{prop:estimatesadded} and Proposition \ref{cordecfr} are discussed in section \ref{sec:addition}.

\section{Construction of the foliation and regularity with respect to $x$}\label{regx}

In this section, we discuss the proof of Theorem \ref{thregx}. By section \ref{reducsmall}, we may assume that $(\Sigma, g,k)$ coincides with $(\R^3,\delta,0)$ outside of a compact, say $|x|\geq 1$. Notice that in $|x|\geq 1$ and for all $\o\in\S$, the scalar function $x.\o$ satisfies the equation \eqref{choice} and the estimate \eqref{thregx1}, since $a\equiv 1, \th\equiv 0$ and $N\equiv\o$ in this region. Thus, we would like to construct a function $u$ solution of \eqref{choice} satisfying \eqref{thregx1} in a region containing $|x|\leq 2$ and to glue it to $x.\o$ in $1\leq |x|\leq 2$. Now, \eqref{choice} is of parabolic type - see \eqref{eqlapse1} - where $u$ plays the role of time. For each fixed $\o\in\S$, we start with $u=-2$ on $x\c\omega=-2$. Then, we propagate with the parabolic equation \eqref{eqlapse1}, coupled with the equation for $\th$ \eqref{codazzith}, to the strip $S=\{x\in\s\textrm{ such that }-2< u(x,\o)< 2\}$. This strip covers the entire region $|x|\leq 1$, and we then glue $u$ to $x\c\omega$ outside of $|x|\leq 1$ (see section \ref{sec:thisistheend}). In the next section, we prove a priori estimates consistent with the estimate \eqref{thregx1} and valid on $-2<u<2$ for the solution $u$ of:
\begin{equation}\label{choice1}
\trt-k_{NN}=1-a,\textrm{ on }-2<u<2,
\end{equation}
where $u$ is initialized on $x.\o=-2$ by:
\begin{equation}\label{init}
u(x,\o)=-2\textrm{ on }x.\o=-2.
\end{equation}
Note that the first equation of \eqref{struct1}, \eqref{init} and the fact that $(g,k,\s)$ coincides with $(\delta,0,\R^3)$ for $|x|\geq 2$ yields:
\begin{equation}\label{init1}
\nabla^p(a-1)=0,\,\nabla^p\theta=0,\,\nabla^p(N-\o)=0\textrm{ for all }p\in\N\textrm{ on }u=-2.
\end{equation} 

\subsection{A priori estimates for lower order derivatives}\lab{sec:loworderder}

Let $(\Sigma, g,k)$ chosen as in section \ref{reducsmall}. In particular, we assume:
\begin{equation}\label{small2}
\norm{\nabla k}_{L^2(\Sigma)}+\norm{R}_{L^2(\Sigma)}\leq\ep.
\end{equation}
Let $u$ a scalar function on $\s\times\S$, and let $\p$, $a$, $N$ and $\th$ be associated to $u$ as in section \ref{sec:foliation}. Assume that $u$ satisfies the additional equation \eqref{choice}, which we recall below together with  \eqref{frame} and \eqref{struct}:
\be\lab{frame1}
\left\{\begin{array}{l}
\nabla_AN=\th_{AB}e_B,\\[1mm]
\nabn N=-\nabb\lg,
\end{array}\right.
\ee
and
\be\lab{struct1}
\left\{\begin{array}{l}
\trt-k_{NN}=1-a,\\[1mm]
\nabn a-\lapa=|\th|^2+\nabn k_{NN}+R_{NN}.
\end{array}\right.
\ee
In this section, we establish a priori estimates for $a$, $N$ and $\th$ corresponding to \eqref{thregx1} in the region $S$ of $\s$ between $P_{-2}$ and $P_2$ (i.e. $S=\{x\,/\,-2< u(x,\o)< 2\}$) where $u$ is initialized on $x.\o=-2$ by \eqref{init}. In particular, we have \eqref{init1}, so that the subsequent integrations by parts will not create boundary terms at $u=-2$.

For the sake of simplicity, let us just discuss the estimate \eqref{thregx1} for the lapse $a$. We rewrite the second equation of \eqref{struct1} as:
\begin{equation}\label{r11}
(\nabn-a^{-1}\lap)(a-1)=h,
\end{equation}
where $h$ is given by:
\begin{equation}\label{r12}
h=\nabn k_{NN}+R_{NN}+\cdots.
\end{equation}
Using in particular \eqref{small2}, we obtain:
$$\norm{h}_{\lli{2}}\lesssim \ep.$$
Together with \eqref{r11} and an $L^2$ parabolic estimate for the  operator $(\nabn-a^{-1}\lap)$, we obtain
\begin{equation}\label{r19}
\ds\norm{a-1}_{\l{\infty}{2}}+\norm{\nabb a}_{\l{\infty}{2}}+\norm{\nabn a}_{\lli{2}}+\norm{\nabb^2a}_{\lli{2}}\lesssim \ep.
\end{equation}

In order to obtain estimates for $\nabb\nabn a$ and $\nabn^2a$, we differentiate the second equation of \eqref{struct1} by $\nabn$:
\begin{equation}\label{r20}
(\nabn - a^{-1}\lap)\nabn a=\nabn^2k_{NN}+\nabn R_{NN}+\cdots,
\end{equation}
where we only kept two typical terms. Note that $\nabn^2k_{NN}$ and $\nabn R_{NN}$ are dangerous terms which cannot be estimated directly. We first need to trade a $\nabn$ derivative with a $\nabb$ derivative using Bianchi identities and the constraints equations. We use the twice-contracted Bianchi identity on $\s$ 
\begin{equation}\label{bianchi}
\nabla^jR_{ij}=\frac{1}{2}\nabla_iR.
\end{equation}
In particular, using also the constraint equations \eqref{const1}, we have 
$$\nabn R_{NN}=-\nabla_AR_{AN}+\nabn |k|^2.$$
Also, the constraint equations \eqref{const1} yield
$$\nabn^2k_{NN}=-\nabn\nabla_Ak_{AN}+\cdots=-\nabla_A\nabn k_{AN}+\cdots.$$
Together with \eqref{r20}, we obtain
\begin{equation}\label{r21}
(\nabn - a^{-1}\lap)\nabn a=\divb(H)+h_1, 
\end{equation}
where
\begin{equation}\label{r23bis}
H=-\nabn k_{.N}-R_{.N}, 
\end{equation}
and $h_1$ satisfies
\begin{equation}\lab{bazot}
\norm{h_1}_{\l{2}{\frac{4}{3}}}\lesssim\ep.
\end{equation}
Using the smallness assumption \eqref{small2} and the definition of $H$ \eqref{r23bis}, we have
\be\lab{bazot1}
\norm{H}_{\lli{2}}\les\ep
\ee
which together with \eqref{bazot}, \eqref{r21}, and an $L^2$ parabolic estimate for the  operator $(\nabn-a^{-1}\lap)$, yields
\begin{equation}\label{r28}
\norm{\nabn a}_{\l{\infty}{2}}+\norm{\nabb\nabn a}_{\lli{2}}\lesssim \ep.
\end{equation}
Finally, \eqref{r19} and \eqref{r28} yield the wanted estimate \eqref{thregx1}. 

\subsection{End of  the proof of Theorem \ref{thregx}}\lab{sec:thisistheend}

We briefly sketch the rest of the proof of Theorem \ref{thregx}, and we refer to \cite{param1} for the details. In \eqref{r19} and \eqref{r28}, we have obtained a priori estimates consistent with the estimate \eqref{thregx1} and valid on $-2<u<2$ for the solution $u$ of \eqref{choice1}. Then, we also prove on $-2<u<2$ a priori estimates for higher derivatives of the solution $u$ of \eqref{choice1}. We then use the existence of $u$ solution to\footnote{this local existence result could be proved either using a Nash Moser procedure or a combination of Cauchy-Kowalewska and enhanced a priori estimates for all derivatives}:
\begin{equation}\label{choice2}
\left\{\begin{array}{l}
\trt-k_{NN}=1-a,\textrm{ on }\alpha<u<\alpha+T,\\
u=\alpha\textrm{ on }\underline{u}=\alpha,
\end{array}\right.
\end{equation}
where $-2\leq \alpha\leq 2$, $\underline{u}$ is smooth, and $T>0$ is small enough. Together with the a  priori estimates, this allows us to control the solution of \eqref{choice2} on $-2+kT<u<-2+(k+1)T$ uniformly with respect to $k=0,\dots,[4/T]$ in order to obtain a solution $u$ of \eqref{choice1} on $-2<u<2$. Finally, we conclude the proof of Theorem \ref{thregx} by showing how to glue the solution $u$ of \eqref{choice1} to $x.\o$ in $1\leq |x|\leq 2$ in order to obtain a solution on $\s$ satisfying \eqref{thregx1}.

\section{Estimates for $\nabn a$ and $\nabn^2a$}\label{sec:nabla2a}

In this section, we discuss the proof of Theorem \ref{thnabn2a}. Recall the decomposition \eqref{r21}, \eqref{r23bis}, and the estimate \eqref{bazot}. We introduce the scalar functions on $S$ $a_1$ and $a_2$ solutions of:
\begin{equation}\label{zoc3}
(\nabn - a^{-1}\lap)a_1=h_1\textrm{ on S},\, a_1(-2,.)=0,
\end{equation}
and:
\begin{equation}\label{zoc4}
(\nabn - a^{-1}\lap)a_2=\divb(H)\textrm{ on S},\, a_2(-2,.)=0,
\end{equation}
which yields, in view of \eqref{r21} and \eqref{init1}, the decomposition:
\be\lab{zoc5}
\nabn a=a_1+a_2.
\ee

The idea behind the decomposition \eqref{zoc5} is to take advantage of the better regularity of $h_1$ for $a_1$ (see \eqref{bazot} compared to \eqref{bazot1}), and to use the structure of $\divb(H)$ to obtain a useful equation for $\nabn a_2$. Indeed, in view of the equation \eqref{zoc4} satisfied by $a_2$, $\nabn a_2$ satisfies:
\be\lab{zoc14}
(\nabn-a^{-1}\lap)(\nabn a_2)=\nabn(\divb(H))+\cdots,
\ee
and using in particular the twice-contracted Bianchi identity on $\s$, the constraint equations in the maximal foliation \eqref{const1}, and \eqref{zoc14}, we obtain
\begin{equation}\label{zoc20}
(\nabn - a^{-1}\lap)\nabn a_2=\divb\divb(H_1)+\divb(H_2)+h_2+\cdots, 
\end{equation}
where the tensors $H_1, H_2$ and the scalar $h_2$ satisfy
$$\norm{H_1}_{L^2(S)}+\norm{H_2}_{\l{2}{\frac{4}{3}}}+\norm{h_2}_{L^1(S)}\les \ep.$$
These ideas allow us to derive the following two propositions (see \cite{param1} for the detailed proof of these propositions). 

\begin{proposition}\label{prop:zoc1}
Let $a_1$ be the solution of \eqref{zoc3}, where $h_1$ satisfies \eqref{bazot}. Then, we have:
\begin{equation}\label{zoc6}
\norm{a_1}_{\l{\infty}{4}}\lesssim\ep,
\end{equation}
and:
\be\lab{zoc7}
\sum_{j\geq 0}2^{-j}\norm{P_j(\nabn a_1)}_{\lli{2}}^2\les\ep^2,
\ee 
\end{proposition}

\begin{proposition}\label{prop:zoc2}
Let $a_2$ be the solution of \eqref{zoc4}, where $H$ is defined in \eqref{r23bis}. Then, we have:
\begin{equation}\label{zoc8}
\norm{a_2}_{\l{\infty}{4}}\lesssim\ep,
\end{equation}
and:
\be\lab{zoc10}
\sup_{j\geq 0}\norm{P_j(\nabn a_2)}_{\lli{2}}\les\ep.
\ee
\end{proposition}

In view of the decomposition \eqref{zoc5} for $\nabn a$, the estimates \eqref{zoc6} \eqref{zoc7} for $a_1$, and the estimates \eqref{zoc8} \eqref{zoc10} for $a_2$, we immediately obtain the estimate \eqref{nabn2a1} for $\nabn a$ and $\nabn^2a$. This concludes the proof of Theorem  \ref{thnabn2a}. 

\section{Regularity of the foliation with respect to $\omega$}\label{regomega}

Let $u(x,\o)$ the function constructed in section \ref{regx}. In this section, we discuss the proof of Theorem \ref{thregomega} which deals with the control of the derivatives with respect to $\o$ of the foliation 
of $\s$ provided by $u(x,\o)$. Recall that $(\Sigma, g,k)$ coincides with $(\R^3,\de,0)$ in $|x|\geq 2$. Also, $u(x,\o)$ coincides with $x.\o$ in $|x|\geq 2$, and so $a\equiv 1$, $N\equiv\o$ and $\th\equiv 0$ in this region. Thus, $u$ clearly satisfies the estimates of Theorem \ref{thregomega} in $|x|\geq 2$ and it is enough to control the derivatives with respect to $\o$ of the function $u(x,\o)$ solution to: 
\begin{equation}\label{choice6}
\left\{\begin{array}{l}
\trt-k_{NN}=1-a,\textrm{ on }-2<u<2,\\
u(.,\o)=-2\textrm{ on }x.\o=-2,
\end{array}\right.
\end{equation}
in the strip $S=\{x/\,-2<u<2\}$.
%In section \ref{regx}, we have constructed a smooth solution $u(.,\o)$ of \eqref{choice6} using 
%a priori estimates for higher order derivatives together with a Nash Moser procedure. We could 
%easily modify the proof to show the smoothness of our solution with respect to $\o$ as  well. Thus, 
%in this section, we only need to prove a priori estimates for the solution $u(x,\o)$ of \eqref{choice6} consistent with the estimates of Theorem \ref{thregomega}. 

\subsection{First order derivatives with respect to $\o$}\lab{sec:regomega1} 

The goal of this section is to prove \eqref{threomega1}. For the sake of simplicity, we only outline of the proof of the estimate for $\po a$. Differentiating the second equation of \eqref{struct1} with respect to $\o$, we obtain:
\begin{equation}\label{diffom1}
(\nabn-a^{-1}\lap)\po a = 2\nabb\nabn a+2R_{N \po N}+\cdots
\end{equation}
where the first term on the right-hand side comes from the commutator $[\po,\lap]$ (see \cite{param1}).  
Since $\nabb\nabn a$ and $R$ are in $\lli{2}$ respectively by \eqref{thregx1} and \eqref{small2}, we obtain using in particular an $L^2$ parabolic estimate for the  operator $(\nabn-a^{-1}\lap)$
\begin{equation}\label{diffom2}
\norm{\nabn\po a}_{\lli{2}}+\norm{\nabb\po a}_{\l{\infty}{2}}+\norm{\nabb^2\po a}_{\lli{2}} \lesssim \ep.
\end{equation}
Next, we differentiate \eqref{diffom1} with respect to $\nabn$. We obtain:
\begin{equation}\label{diffom3}
(\nabn-a^{-1}\lap)\nabn\po a = 2\nabb\nabn^2a+2\nabn R_{N \po N}+\cdots.
\end{equation} 
The term $\nabn R_{N \po N}$ may be treated using the contracted Bianchi identity for $R$ - as we did for $\nabn R_{NN}$ in section \ref{sec:loworderder} - and turns out to be in $\lhs{2}{-1}$. On the other hand, in view of the estimate \eqref{nabn2a1} for $\nabn^2a$, $\nabb\nabn^2a$ belongs to $\lhs{2}{-\frac{3}{2}}$. We obtain using in particular a refined parabolic estimate for the operator $(\nabn-a^{-1}\lap)$ 
\begin{equation}\label{diffom4bis}
\norm{\nabn\po a}_{\lhs{2}{\frac{1}{2}}}+\norm{\nabn^2\po a}_{\lhs{2}{-\frac{3}{2}}} \lesssim \ep.
\end{equation}
Finally, by interpolation between \eqref{diffom2} and \eqref{diffom4bis}, we obtain $\po a$ in $\lhs{\infty}{\frac{5}{4}}$ which embeds in $\lli{\infty}$ since $\p$ has dimension 2. Together with \eqref{diffom2} and \eqref{diffom4bis}, we obtain the estimate corresponding to $\po a$ in \eqref{threomega1}. 

\subsection{Second order derivatives with respect to $\o$}\lab{sec:regomega2}

The goal of this section is to prove \eqref{threomega2}. For the sake of simplicity, we only outline of the proof of the estimate for $\po^2a$. Differentiating the equation \eqref{diffom1} for $\po a$ with respect to $\o$, we obtain:
\begin{equation}\label{diffomm1}
(\nabn-a^{-1}\lap)\po^2a = 2\nabn^2a+\nabb\nabn\po a+2R_{\po N \po N}+\cdots
\end{equation}
where the first two terms on the right-hand side come respectively from the commutators $[\po,\nabb]$ and $[\po,\lap]$. Since $R$ is in $\lli{2}$ by \eqref{small2}, $\nabn^2a$ is in $\lhs{2}{-\frac{1}{2}}$ by \eqref{nabn2a1}, and $\nabn\po a $ is in $\lhs{2}{\frac{1}{2}}$ by \eqref{threomega1}, the right-hand side of \eqref{diffomm1} belongs to $\lhs{2}{-\frac{1}{2}}$. Using in particular estimates for the parabolic operator $(\nabn-a^{-1}\lap)$, we deduce 
\begin{equation}\label{diffomm2}
\norm{\po^2a}_{\lhs{2}{\frac{3}{2}}}+\norm{\po^2a}_{\lhs{\infty}{\frac{1}{2}}}+\norm{\nabn\po^2a}_{\lhs{2}{-\frac{1}{2}}} \lesssim \ep,
\end{equation}
which is the estimate corresponding to $\po^2a$ in \eqref{threomega2}.

\begin{remark}\lab{theend}
Note that we may not differentiate the equation \eqref{diffomm1} for $\po^2a$ with respect to $\nabn$. Indeed, 
the term $\nabn R_{\po N \po N}$ has no structure: unlike $R_{NN}$ and $R_{N \po N}$ which were involved in the equation for $a$ and $\po a$, $R_{\po N \po N}$ does not contain any contraction with $N$ since $\po N$ is tangent to $\p$. Thus, unlike $\nabn R_{NN}$ and $\nabn R_{N \po N}$, we can not write $\nabn R_{\po N\po N}$ as a tangential derivative using the contracted Bianchi identities for $R$. Consequently, we can not obtain any estimate for $\nabn^2\po^2a$. 
\end{remark}

\subsection{Third order derivatives with respect to $\o$}\label{thirdorom}

The goal of this section is to prove \eqref{threomega3}. Recall that  
div$(N)=\trt$, $N=\nabla u/|\nabla u|$, $a=1/|\nabla u|$ and $\trt=1-a+k_{NN}$, so that:
\begin{equation}\label{diffommm1}
\textrm{div}\left(\frac{\nabla u}{|\nabla u|}\right)=1-\frac{1}{|\nabla u|}+k_{NN}.
\end{equation}
Differentiating \eqref{diffommm1} three times with respect to $\o$ yields:
\be\lab{diffommm2}
(\nabn-a^{-1}\lap)\po^3u=\nabb\po^2\log(a)+\cdots.
\ee
Using in particular \eqref{diffommm2}, the estimate \eqref{threomega2} for $\po^2a$ and refined parabolic estimates for the operator $(\nabn-a^{-1}\lap)$ we obtain
$$\norm{\po^3u}_{\lhs{2}{\frac{5}{2}}}+\norm{\po^3u}_{\lhs{\infty}{\frac{3}{2}}}+\norm{\nabn\po^3u}_{\lhs{2}{\frac{1}{2}}}\les 1.$$
Now, since $\po^3u\in \lhs{\infty}{\frac{3}{2}}$ and $\p$ is 2-dimensional, we obtain that $\po^3u$ belongs to $L^\infty_{loc}(\Sigma)$, which is the desired estimate \eqref{threomega3}.

\section{A global coordinate system on $\p$ and $\Sigma$}\label{sec:globalcoord}

The goal of this section is to discuss the proof of Proposition \ref{gl20}. We start by constructing a global coordinate system on $\p$.

\subsection{A global coordinate system on $\p$} We have the following proposition 
\begin{proposition}\label{gl0}
Let $\o\in\S$. Let $\Phi_u:\p\rightarrow T_\o\S$ defined by:
\begin{equation}\label{gl1}
\Phi_u(x):=\po u(x,\o),
\end{equation}
where $T_\o\S$ is the tangent space to $\S$ at $\o$. Then $\Phi_u$ is a global $C^1$ diffeomorphism from $\p$ to $T_\o\S$.
\end{proposition}

For the sake of simplicity, we only briefly sketch the proof. We start by showing that $\Phi_u$ is a local $C^1$ diffeomorphism. We have
\begin{displaymath}
\begin{array}{l}
(\textrm{Jac}\Phi_u)^T\textrm{Jac}\Phi_u=a^{-2}\left(\begin{array}{cc}
g(\partial_\varphi N,\partial_\varphi N) &  g(\partial_\psi N,\partial_\varphi N)\\
g(\partial_\psi N,\partial_\varphi N) & g(\partial_\psi N,\partial_\psi N)
\end{array}\right),
\end{array}
\end{displaymath}
where $(\varphi,\psi)$ denotes the usual spherical coordinates  on $\S$. Using the estimates \eqref{thregx1} and \eqref{threomega1}, we are able to derive the following estimate
\begin{equation}\label{gl5}
\norm{(\textrm{Jac}\Phi_u)^T\textrm{Jac}\Phi_u-I}_{\lli{\infty}}\lesssim\ep,
\end{equation}
so that $\Phi_u$ is a $C^1$ local diffeomorphism. In turn, this yields:
\begin{equation}\label{gl6}
\norm{|\det(\textrm{Jac}\Phi_u)|-1}_{\lli{\infty}}\lesssim\ep.
\end{equation}
 
It remains to show that $\Phi_u$ is onto and one-to-one. The proof relies on the estimate \eqref{gl5} for the Jacobian of $\Phi_u$, the fact that $u$ coincides with $x.\o$ in the region $|x|\geq 2$ and geometric considerations on the level sets of $\partial_\varphi u$ and $\partial_\psi u$. We refer to \cite{param1} for the details.
 
\subsection{Proof of Proposition \ref{gl20}}

Let $\o\in\S$. Recall the definition \eqref{gl21} of $\phi_\o:\s\rightarrow\R^3$:
$$\phi_\o(x):=u(x,\o)\o+\po u(x,\o)=u(x,\o)\o+\Phi_u(x),$$
where $\Phi_u$ has been defined in Proposition \ref{gl0}. The fact that $\phi_\o$ is a bijection is an easy consequence of the fact that $\Phi_u$ is a bijection for all $u$. Then, it remains to prove \eqref{gl22}. We are able to obtain
\begin{displaymath}
\begin{array}{l}
(\textrm{Jac}\phi_\o)^T\textrm{Jac}\phi_\o=a^{-2}\\
\times\left(\begin{array}{ccc}
1 & -\partial_\varphi\log(a) & -\partial_\psi\log(a)\\
-\partial_\varphi\log(a) & (\partial_\varphi\log(a))^2+g(\partial_\varphi N,\partial_\varphi N) & \partial_\varphi\log(a)\partial_\psi\log(a)+ g(\partial_\psi N,\partial_\varphi N)\\
-\partial_\psi\log(a) & \partial_\varphi\log(a)\partial_\psi\log(a)+ g(\partial_\psi N,\partial_\varphi N) & (\partial_\psi\log(a))^2+g(\partial_\psi N,\partial_\psi N)
\end{array}\right).
\end{array}
\end{displaymath}
Taking the determinant yields:
\begin{equation}\label{gl27}
\det((\textrm{Jac}\phi_\o)^T\textrm{Jac}\phi_\o)=a^{-2}\det((\textrm{Jac}\Phi_u)^T\textrm{Jac}\Phi_u),
\end{equation}
which together with \eqref{gl5} and the estimate \eqref{thregx1} for $a$ implies:
\begin{equation}\label{gl28}
\norm{\det((\textrm{Jac}\phi_\o)^T\textrm{Jac}\phi_\o)-1}_{\lli{\infty}}\lesssim\ep.
\end{equation}
\eqref{gl28} yields \eqref{gl22}. This concludes the proof of Proposition \ref{gl20}.

\section{Additional estimates}\label{sec:addition} 

The proof of Proposition \ref{prop:estimatesadded} and Proposition \ref{cordecfr} follow from the estimates of Theorem \ref{thregx}, Theorem \ref{thnabn2a}, and Theorem \ref{thregomega} using also for some estimates the fact that $u$ coincides with $x.\o$ in the region $|x|\geq 2$ or the properties of the Littlewood-Paley projections $P_j$. For the sake of simplicity, we skip these proofs and refer the reader to \cite{param1} for the details.

%%%%%%%%%%%%%%%%%%%%%%%%%%%%%%%%%%%%%%%

\chapter{The Strichartz estimates}\lab{part:strich}

Recall Steps A, B, C and D introduced in section \ref{sec:strategyproof}. In this chapter, we perform Step D, i.e. we prove Proposition \ref{prop:L4strichartz}. More precisely, let $j\geq 0$, and let $\psi$ a smooth function on $\RRR^3$ supported in 
$$\frac 1 2 \leq |\xi|\leq 2.$$
Let $\varphi_j$ the parametrix \eqref{parametrix.intr} with an additional frequency localization $\la\sim 2^j$
\begin{equation}\lab{chap7:paraml}
\varphi_j(t,x)=\int_{\SSS^2} \int_0^\infty  e^{i \la u(t,x,\o)}\psi(2^{-j}\la)f(\la\om)\la^2d\la d\om,
\end{equation}
where $u(.,.,\o)$ is a solution to the eikonal equation ${\bf g}^{\alpha\beta}\partial_\alpha u\partial_\beta u=0$ which depends on an extra parameter $\o\in\SSS^2$. Assume that the space-time $\mathcal{M}$ is foliated by space-like hypersurfaces $\Sit$ defined as level hypersurfaces of a time function $t$. Let $(p,q,r)$ such that $p, q\geq 2$, $q<+\infty$, and 
$$\frac{1}{p}+\frac{1}{q}\leq\frac{1}{2},\,r=\frac{3}{2}-\frac{1}{p}-\frac{3}{q}.$$
In this chapter, we outline the proof of the following sharp\footnote{Note in particular that the corresponding estimates in the flat case are sharp.} Strichartz estimates
\be\lab{chap7:strichgeneralintro}
\norm{\varphi_j}_{L^p_{[0,1]}L^q(\Sigma_t)}\les 2^{jr}\norm{\psi(2^{-j}\la)f}_{L^2(\R^3)}.
\ee
The proof of Proposition \ref{prop:L4strichartz} will then be a simple consequence of \eqref{chap7:strichgeneralintro} with the choice  $p=q=4$.

\begin{remark}
Even though we only need $L^4(\MM)$ Strichartz estimates - which corresponds to $p=q=4$ in \eqref{chap7:strichgeneralintro} - to prove Proposition \ref{prop:L4strichartz}, it turns out that this particular case is not easier to prove than the general case.
\end{remark}

\section{Assumptions on the phase $u(t,x,\omega)$ and main results}

\subsection{Time foliation on $\mathcal{M}$}

We foliate the space-time $\mathcal{M}$ by space-like hypersurfaces $\Sit$ defined as level hypersurfaces of a time function $t$. We  assume $0\leq t\leq 1$ so that
\be\lab{chap7:decopmpoM}
\MM=\bigcup_{0\leq t\leq 1}\Sigma_t.
\ee
We denote by 
$T$ the unit, future oriented, normal to $\Sigma_t$. We also define the lapse $n$ as 
\be\lab{chap7:lapsen}
n^{-1}=T(t).
\ee
Note that we have the following identity between the volume element of $\MM$ and the volume element corresponding to the induced metric on $\Sigma_t$
\be\lab{chap7:compvolume}
d\MM=n\,d\Sigma_t\,dt.
\ee
We will assume the following assumption on $n$
\be\lab{chap7:assonn}
\frac{1}{2}\leq n\leq 2
\ee
which together with \eqref{chap7:compvolume} yields
\be\lab{chap7:compvolume1}
d\MM\simeq d\Sigma_t\,dt.
\ee

\begin{remark}
The assumption \eqref{chap7:assonn} is very mild. In particular, it is compatible with the estimates for $n$ derived in \cite{param3} (see also \eqref{ch4:estn}).
\end{remark}

\subsection{Geometry of the foliation generated by $u$ on $\mathcal{M}$}

Recall that $u$ is a solution to the eikonal equation $\gg^{\alpha\beta}\partial_\alpha u\partial_\beta u=0$ on $\mathcal{M}$ depending on an extra parameter $\o\in \S$.  The level hypersufaces $u(t,x,\o)=u$  of the optical function $u$ are denoted by  $\H_u$. Let $L'$ denote the space-time gradient of $u$, i.e.:
\be\lab{chap7:def:L'0}
L'=\gg^{\a\b}\pr_\b u \pr_\a.
\ee
Using the fact that $u$ satisfies the eikonal equation, we obtain:
\be\lab{chap7:def:L'1}
\dd_{L'}L'=0,
\ee
which implies that $L'$ is the geodesic null generator of $\H_u$.

We have: 
$$T(u)=\pm |\nab u|$$
where $|\nab u|^2=\sum_{i=1}^3|e_i(u)|^2$ relative to an orthonormal frame $e_i$ on $\Sigma_t$. Since the sign of $T(u)$ is irrelevant, we choose by convention:
\be\lab{chap7:it1'}
T(u)=-|\nab u|
\end{equation}
so that $u$ corresponds to $-t+x\c\o$ in the flat case.

Let
 \be\lab{chap7:it2}
L=bL'=T+N,
\end{equation}
where $L'$ is the space-time gradient of $u$ \eqref{chap7:def:L'0}, $b$  is  the  \textit{lapse of the null foliation} (or shortly null lapse)
\be\lab{chap7:it3}
b^{-1}=-<L', T>=-T(u),
\end{equation} 
and $N$ is a unit vectorfield given by
\be\lab{chap7:it3bis}
N=\frac{\nabla u}{|\nabla u|}.
\end{equation}

Note that we have the following identities.
\begin{lemma}
\be\lab{chap7:identities}
L(u)=0,\, L(\po u)=0
\ee
and
\be\lab{chap7:identities1}
\gg(N, \po N)=0.
\ee
\end{lemma}

The proof is elementary and can be found in \cite{bil2}.

\subsection{Regularity assumptions for $u(t,x,\o)$}

We now state our assumptions for the phase $u(t,x,\o)$. These assumptions are compatible with the regularity obtained for the function $u(t,x,\o)$ constructed in \cite{param3} (see also \eqref{ch4:estb}, \eqref{ch4:estNomega}, \eqref{ch4:estricciomega}). Let $0<\ep<1$ a small enough universal constant. $b$ and $N$ satisfy
\be\lab{chap7:regb}
\norm{b-1}_{L^\infty}+\norm{\po b}_{L^\infty}\les\ep.
\ee

\be\lab{chap7:regpoN}
\norm{\gg(\po N,\po N)-I_2}_{L^\infty}\les \ep.
\ee

\be\lab{chap7:ad1}
|N(., \o)-N(., \o')|=|\o-\o'|(1+O(\ep)).
\ee

\begin{remark}
In the flat case, we have $\mathcal{M}=(\R^{1+3},{\bf m})$,  where ${\bf m}$ is the Minkowski metric, $u(t,x,\o)=-t+\xo$, $b= 1$, $N=\o$ and $L=\partial_t+\o\cdot\partial_x$. Thus, the assumptions \eqref{chap7:regb} \eqref{chap7:regpoN} \eqref{chap7:ad1}  are clearly satisfied with $\ep=0$.
\end{remark}

\begin{remark}
In terms of the regularity of $u(t,x,\o)$, the assumptions \eqref{chap7:regb} \eqref{chap7:regpoN} correspond to 
$$\nabla u\in L^\infty\textrm{ and }\nabla\po u\in L^\infty$$
which is very weak. In particular, the classical proof for obtaining Strichartz estimates for the wave equation relies on the stationary phase for an oscillatory integral involving $u$ as a phase, and typically requires at the least one more derivative for $u$ (see Remark \ref{chap7:rem:compstatphase}).
\end{remark}

\subsection{A global coordinate system on $\Sigma_t$}\lab{chap7:sec:assumptioncoord}

For all $0\leq t\leq 1$, and for all $\o\in\S$, $(u(t,x,\o), \po u(t,x,\o))$ is a global coordinate system on $\Sigma_t$. Furthermore, the volume element is under control in the sense that in this coordinate system, we have
\be\lab{chap7:assglobalcoordvol}
\frac{1}{2}\leq \sqrt{\det g}\leq 2
\ee
where $g$ is the induced metric on $\Sigma_t$, and where $\det g$ denotes the determinant of the matrix of the coefficients of $g$. 

%% Sur P_{t,u}, se deduit de cas t=0 sur P_{0,u} grace a \eqref{chap7:identities}. Puis sur $\Sigma_t$ car $\Sigma_t=\cup P_{t,u}$ avec union disjointe 

\begin{remark}
In the flat case, we have $\Sigma_t=\{t\}\times\R^3$ and $u(t,x,\o)=-t+\xo$ so that $(u(t,x,\o), \po u(t,x,\o))$ is clearly a global coordinate system on $\Sigma_t$ and $\det g=1$ in this case. These assumptions are also satisfied by the function $u(t,x, \o)$ constructed in \cite{param3}. 
\end{remark}

\subsection{Main results}

We next state the main result of this chapter concerning general Strichartz inequalities in mixed space-time norms
of the form $L^p_{[0,1]}L^q(\Sigma_t)$  defined as follows,
$$\|F\|_{L^p_{[0,1]}L^q(\Sigma_t)}=\left(\int_0^1 \|F(t,\cdot)\|_{L^p(\Si_t)}^pdt\right)^{\frac{1}{p}}.$$
 
\begin{theorem}\lab{chap7:mainth}
Let $(p,q)$ such that $p, q\geq 2$, $q<+\infty$, and 
$$\frac{1}{p}+\frac{1}{q}\leq\frac{1}{2}.$$
Let $r$ defined by
$$r=\frac{3}{2}-\frac{1}{p}-\frac{3}{q}.$$
Then, the parametrix localized at frequency $j$ defined in \eqref{chap7:paraml} satisfies under the assumptions \eqref{chap7:assonn}, 
\eqref{chap7:regb}, \eqref{chap7:regpoN}, \eqref{chap7:ad1} and the assumptions in section \ref{chap7:sec:assumptioncoord} the following Strichartz inequalities
\be\lab{chap7:strichgeneral}
\norm{\varphi_j}_{L^p_{[0,1]}L^q(\Sigma_t)}\les 2^{jr}\norm{\psi(2^{-j}\la)f}_{L^2(\R^3)}.
\ee
\end{theorem}

Proposition \ref{prop:L4strichartz} - which corresponds to Corollary 2.8 in \cite{bil2} - is then a simple consequence of Theorem \ref{chap7:mainth}, see \cite{bil2} for the details. 

\vspace{0.3cm}

The rest of the chapter is organized as follows. In section \ref{chap7:sec:proofmainthbis}, we use the standard $TT^*$ argument to reduce the proof of Theorem \ref{chap7:mainth} to an upper bound on the kernel   $K$ of a certain operator. This kernel is an oscillatory integral with a phase $\phi$. In section \ref{chap7:sec:estkernel}, we prove the upper bound on the kernel $K$ provided we have a suitable lower bound on $\phi$. Finally, in section \ref{chap7:sec:lowerb}, we prove the lower bound for $\phi$ used in section \ref{chap7:sec:estkernel}. 

\section{Proof of the Strichartz estimates}\lab{chap7:sec:proofmainthbis} 
 
The goal of this section is to prove Theorem \ref{chap7:mainth}. We start with the following remark.
\begin{remark}\lab{chap7:rem:globalcoordsimple}
Fixing  a global system of coordinates  $x=(x^1, x^2, x^3)$ in $\Si_t$, such as the one described in section  \ref{chap7:sec:assumptioncoord}, we note in view of \eqref{chap7:assglobalcoordvol} that \eqref{chap7:strichgeneral} is equivalent with the same inequality where the norm $L^q(\Sigma_t)$ on the  left-hand side is replaced by the corresponding  euclidean norm in the  given coordinates. More precisely we can assume from now on that
$$\norm{F}_{L^p_{[0,1]}L^q(\Sigma_t)}=\left(\int_0^1\left(\int_{\R^3}|F(t,x)|^q dx\right)^{\frac{p}{q}}dt\right)^{\frac{1}{q}}$$
which we will denote by a slight abuse of notation by
$$\norm{F}_{L^p_{[0,1]}L^q(\R^3)}.$$
Note also that in the $(t,x)$  coordinates $\MM=[0,1]\times \R^3$. 
\end{remark}

For convenience, let us introduce the operator $T_j$ acting on functions $f\in L^2(\R^3)$,
\bea
\lab{chap7:eq:op-Tj}
T_jf(t,x)= \int_{\S} \int_0^\infty  e^{i \la u(t,x,\o)}\psi(2^{-j}\la)f(\la\o)\la^2d\la d\o.
\eea
Note in particular that 
\bea\lab{chap7:idvarphijTj}
T_jf=\varphi_j
\eea
where $\varphi_j$ is the parametrix localized at frequency $2^j$ defined in \eqref{chap7:paraml}. To prove Theorem \ref{chap7:mainth}, we rely on  the standard  $TT^*$ argument for
 the Fourier integral operator \eqref{chap7:eq:op-Tj}. Note that  the operator  $T_j^*  $  takes real valued  functions $h$ on $\MM$ to complex valued functions on $\R^3$
$$T_j^*h(\la\o)=\psi(2^{-j} \la) \int_\MM  e^{-i\la u(s,y, \o)} h(s,y) ds dy.$$
Therefore, the operator $U_j:=T_j T^*_j$ is given by the formula,
$$U_j h(t,x)=\int_{\S} \int_0^\infty \int_\MM e^{i \la u(t,x,\o)-i\la u(s,y,\o)}\psi(2^{-j}\la)^2h(s,y)\la^2d\la d\o   ds dy.$$
Note, in view of Remark \ref{chap7:rem:globalcoordsimple} and \eqref{chap7:idvarphijTj}, that \eqref{chap7:strichgeneral} is equivalent to the following estimate
\be\lab{chap7:ttsineq}
\norm{ U_j h}_{L^p_{[0,1]}L^q(\R^3)}\les 2^{2jr}\norm{h}_{L^{p'}_{[0,1]}L^{q'}(\R^3)},
\ee
where $p'$ (resp. $q'$) is the conjugate exponent to $p$ (resp. $q$).
Observe that,
$$U_jh\left(\frac{t}{2^j},\frac{x}{2^j}\right) = 2^{-j}\int_{\S} \int_0^\infty \int_{2^j\MM} e^{i \la 2^j u\left(\frac{t}{2^j},\frac{x}{2^j},\o\right)-i\la 2^j u\left(\frac{s}{2^j},\frac{y}{2^j},\o\right)}\psi(\la)^2h\left(\frac{s}{2^j},\frac{y}{2^j}\right)\la^2d\la d\o  ds dy$$
with $2^j\MM=[0,  2^j]\times \R^3$ relative to the rescaled  variables $(s,y)$.  Thus,
 setting,
$$Ah(t,x) := \ds\int_{\S} \int_0^\infty \int_{2^j\MM} e^{i \la 2^j u\left(\frac{t}{2^j},\frac{x}{2^j},\o\right)-i\la 2^j u\left(\frac{s}{2^j},\frac{y}{2^j},\o\right)}\psi(\la)^2h(s,y)\la^2d\la d\o  ds dy$$
we have 
$$U_jh\left(\frac{t}{2^j},\frac{x}{2^j}\right)=2^{-j}A h_j(t,x),\,\,\, h_j(s,y)=h\left(\frac{s}{2^j},\frac{y}{2^j}\right).$$ 
We easily infer that \eqref{chap7:ttsineq} is equivalent to  the estimate,
\be\lab{chap7:ttsineq1}
\norm{Ah}_{L^p_{[0,2^j]}L^q(\R^3)}\les \norm{h}_{L^{p'}_{[0,2^j]}L^{q'}(\R^3)}.
\ee

We introduce the kernel $K$ of $A$
\bea\lab{chap7:defkernel}
K(t,x,s,y)=\int_{\S} \int_0^\infty e^{i \la 2^j u\left(\frac{t}{2^j},\frac{x}{2^j},\o\right)-i\la 2^j u\left(\frac{s}{2^j},\frac{y}{2^j},\o\right)}\psi(\la)^2\la^2d\la d\o.
\eea

\begin{remark}
In the flat case, we have $u(t,x,\o)=-t+x\cdot\o$ so that 
$$2^ju\left(\frac{t}{2^j},\frac{x}{2^j},\o\right)=u(t,x,\o).$$
In particular, $K$ is independent of $j$ 
$$K(t,x,s,y)=\int_{\S} \int_0^\infty e^{i \la u(t,x,\o)-i\la u(s,y,\o)}\psi(\la)^2\la^2d\la d\o.$$
\end{remark}

We have the following proposition.
\begin{proposition}\lab{chap7:prop:estkernel}
The kernel $K$ of the operator $A$ satisfies the  dispersive estimates,
\be\lab{chap7:estkernel} 
|K(t,x,s,y)|\les \frac{1}{|t-s|},\,\,\,\forall (t,x)\in 2^j\MM,\,\,\, \forall (s,y)\in 2^j\MM. 
\ee
\end{proposition}
The proof of Proposition \ref{chap7:prop:estkernel} is postponed to section \ref{chap7:sec:estkernel}. We now conclude the proof of Theorem \ref{chap7:mainth}. \eqref{chap7:ttsineq1} follows from \eqref{chap7:estkernel} using interpolation and the Hardy-Littlewood inequality according to the standard procedure, see for example \cite{sog} and \cite{St}. Finally, in view of the discussion above, \eqref{chap7:ttsineq1} yields \eqref{chap7:ttsineq} which in turn implies \eqref{chap7:strichgeneral} in view of \eqref{chap7:idvarphijTj}. This concludes the proof of Theorem \ref{chap7:mainth}.

\section{Upper bound on the kernel $K$}\lab{chap7:sec:estkernel}

The goal of this section is to prove Proposition \ref{chap7:prop:estkernel}. Let $\phi$ the scalar function on $\MM\times\MM\times\S$ defined as
\be\lab{chap7:defphi}
\phi(t,x,s,y,\o)=u(t,x,\o)-u(s,y,\o).
\ee
In view of \eqref{chap7:defkernel}, we may rewrite $K$ as 
$$K(t,x,s,y)=\int_{\S} \int_0^\infty e^{i \la 2^j \phi\left(\frac{t}{2^j},\frac{x}{2^j},\frac{s}{2^j},\frac{y}{2^j},\o\right)}\la^2d\la d\o.$$
After integrating by parts twice in $\la$, and using the size of the support of $\psi$, this yields
\be\lab{chap7:ibpkernel}
|K(t,x,s,y)|\les \int_{\S}\frac{1}{1+2^{2j}\phi\left(\frac{t}{2^j},\frac{x}{2^j},\frac{s}{2^j},\frac{y}{2^j},\o\right)^2}d\o.
\ee
The next section is dedicated to the obtention of a lower bound on $|\phi|$ which will allow us to deduce \eqref{chap7:estkernel} from \eqref{chap7:ibpkernel}.

\begin{remark}\lab{chap7:rem:compstatphase}
It is at this stage that we depart from the standard strategy for proving Strichartz estimates. Indeed, the usual method  consists in using the stationary phase method to derive \eqref{chap7:estkernel}. To this end, one considers the neighborhood in $\S$ of stationary points $\o_0$, i.e. such that $\po\phi_{|_{\o=\o_0}}=0$. One then needs an identity of the type
\be\lab{chap7:usualstatmeth}
\phi=(s-t)A(\o-\o_0)\c (\o-\o_0)+o\left((s-t)(\o-\o_0)^2\right)
\ee 
for $\o$ in the neighborhood of $\o_0$ and for some $3\times 3$ invertible matrix $A$. \eqref{chap7:usualstatmeth} then allows to perform a change of variables in $\o$ which ultimately leads to \eqref{chap7:estkernel}. In particular, the standard method requires at the least\footnote{One also needs to take care of the contribution to $K$ of the angles $\o\in\S$ corresponding to the exterior of the neighborhood of stationary points which may increase the needed regularity.} 
$\partial_{t,x}\po^2 u\in L^\infty$
just to derive \eqref{chap7:usualstatmeth}. 

Our assumptions correspond only to $\partial_{t,x}\po u\in L^\infty$. Thus, in order to obtain \eqref{chap7:estkernel}, we instead integrate by parts in $\la$ to obtain \eqref{chap7:ibpkernel}, and then look for a suitable lower bound on $|\phi|$. In particular, we obtain lower bounds of the following type (see details in Lemma \ref{chap7:lemma:key})
\be\lab{chap7:usualstatmeth1}
|\phi|\gtrsim |s-t||\o-\o_0|^2
\ee 
for $\o$ in the neighborhood of some $\o_0\in\S$. The fundamental observation is that, as it turns out, the inequality \eqref{chap7:usualstatmeth1} requires less regularity than the equality \eqref{chap7:usualstatmeth}.
\end{remark}

\subsection{The key lemma}

Let $(t,x)$ and $(s,y)$ in $\MM$, and let $\o\in\S$. In this section, we obtain a lower bound on $\phi(t,x,s,y,\o)$. We may assume
$$0\leq t<s\leq 1.$$
\begin{definition}
For any $\o\in\S$ and $\sigma\in\R$, let $\gamma_{\o}(\sigma)$ denote the null geodesic parametrized by the time function and with initial data
$$\gamma_\o(0)=(t,x),\, \gamma_\o'(0)=b^{-1}(t,x,\o)L(t,x,\o).$$
\end{definition}

\begin{definition}
Let us define the subset $S$ of $\Sigma_s$ as
\be\lab{chap7:def:S}
S=\bigcup_{\o\in\S}\{\ga_\o(s-t)\}.
\ee
\end{definition}

We also define for all $(s,z)\in\Sigma_s$
\be\lab{chap7:def:m}
m(s,z)=\max_{\o\in\S}(u(s,z,\o)-u(t,x,\o)).
\ee
We have the following lemma characterizing the zeros of $m$ (see \cite{bil2} for a proof).

\begin{lemma}\lab{chap7:lemma:S}
We have
$$S=\{p\in\Sigma_s,\,/\, m(p)=0\}.$$
\end{lemma}

Next, we define the following two  subsets of $\Sigma_s$
\be\lab{chap7:defAintAext}
A_{int}=\{p\in\Sigma_s\,/\, m(p)<0\},\, A_{ext}=\{p\in\Sigma_s\,/\, m(p)>0\}.
\ee
Note in view of Lemma \ref{chap7:lemma:S} that 
\be\lab{chap7:union}
\Sigma_s=S\sqcup  A_{int}\sqcup A_{ext}.
\ee

\begin{remark}
In the flat case, the picture is the following:
\begin{enumerate}
\item The null geodesics\footnote{which are straight lines in this case} $\ga_\o$ span the light cone from $(t,x)$. In particular, the null geodesics $\ga_\o$ do not intersect except at $(t,x)$.

\item $S$ is the intersection\footnote{$S$ is a sphere in this case} of the forward light cone from $(t,x)$ with $\{s\}\times\R^3$.

\item $A_{int}$ and $A_{ext}$ correspond respectively to the interior and the exterior of $S$.
\end{enumerate}
Note that we do not need to prove these statements in our case. This is fortunate since these statements - while probably true in our general setting - would be delicate to establish (see for instance \cite{Kl-R7} for a proof of (1) on a space-time $(\MM,\gg)$ with limited regularity).
\end{remark}

Next, we introduce some further notations. First, we denote by $m_0$ the value of $m$ at $(s,y)$, i.e.
\be\lab{chap7:def:m0}
m_0=\max_{\o\in\S}(u(s,y,\o)-u(t,x,\o)).
\ee
We also denote by $\o_0$ an angle in $\S$ where the maximum in \eqref{chap7:def:m0} is achieved, i.e.
\be\lab{chap7:def:om0}
m_0=u(s,y,\o_0)-u(t,x,\o_0).
\ee

\begin{remark}
In the flat case, $\o_0$ is unique and corresponds to the angle of the projection of $(s,y)$ on $S$. Again, while this may be also true in our general setting, we do not need to prove this statement in our case.
\end{remark}

Note that if $(s,y)\in A_{ext}$, the function $u(s,y,\o)-u(t,x,\o)$ may change sign as $\o$ varies on $\S$.
We define
\be\lab{chap7:eo1}
D=\{\o\in\S\,/\, u(t,x,\o)=u(s,y,\o)\}.
\ee
The following lemma gives a precise description of $D$ (see \cite{bil2} for a proof).
\begin{lemma}\lab{chap7:lemma:D}
Let $(s,y)\in A_{ext}$. Let $D$ defined as in \eqref{chap7:eo1}. Let $(\theta, \varphi)$ denote the spherical coordinates with axis $\o_0$. Then, there exists a $C^1$ $2\pi$-periodic function  
$$\theta_1:[0,2\pi)\rightarrow (0,\pi)$$
such that in the coordinate system $(\theta, \varphi)$, $D$ is parametrized by
$$D=\{\theta=\theta_1(\varphi),\, 0\leq\varphi<2\pi\}.$$
\end{lemma}

\begin{remark}
In the flat case, recall that $u(t,x,\o)=-t+x\c\o$. In this case, one easily checks that $D$ is a circle of axis $\o_0$ 
on the sphere $S$ which is generated by the tangents to $S$ through $y$ (see figure \ref{chap7:figbis}).
\end{remark}

Let $\o\in\S$. According to Lemma \ref{chap7:lemma:D}, the great half circle on $\S$ originating at $\o_0$ and containing $\o$ intersects  $D$ at a fixed point $\o_1$. Let $\th$ and $\th_1$ respectively denote  the positive angles between $\o_0$ and $\o$  (resp. $\o_0$ and $\o_1$). \\

In order to obtain a lower bound for $|\phi|$, we will argue differently according to whether $(s,y)$ belongs to the region $S$, $A_{int}$ or $A_{ext}$.
\begin{lemma}[Key lemma]\lab{chap7:lemma:key}
$|\phi|$ satisfies the following lower bounds
\begin{enumerate}
\item If $(s,y)\in S$, we have
\be\lab{chap7:keylowb1}
|\phi(t,x,s,y,\o)|\geq \frac{1}{4}|t-s||\o-\o_0|^2.
\ee

\item If $(s,y)\in A_{int}$, we have
\be\lab{chap7:keylowb2}
|\phi(t,x,s,y,\o)|\geq \frac{1}{8}|t-s||\o-\o_0|^2.
\ee

\item If $(s,y)\in A_{ext}$ and $\theta_1\leq\theta\leq \pi$, we have
\be\lab{chap7:keylowb3}
|\phi(t,x,s,y,\o)|\geq \frac{1}{4}|t-s||\o-\o_1|^2.
\ee

\item If $(s,y)\in A_{ext}$ and $0\leq\theta\leq\theta_1$, we have
\be\lab{chap7:keylowb4}
|\phi(t,x,s,y,\o)| \gtrsim \sqrt{\frac{1-\cos(\theta-\theta_1)}{1-\cos(\theta_1)}}m_0
\ee
\end{enumerate}
\end{lemma}

The proof of Lemma \ref{chap7:lemma:key} is postponed to section \ref{chap7:sec:lowerb}. 

\begin{figure}[t]
\vspace{-4cm}
\begin{center}
\hspace{-4.0005cm}\includegraphics[width=20cm, height=13cm]{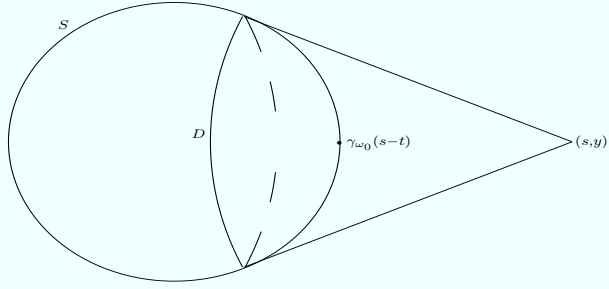}
\vspace{-6.5cm}
\caption{Representation of $D$ in the flat case}\lab{chap7:figbis}
\end{center}
\end{figure}

\begin{remark}\lab{chap7:reminiscentoverlap}
The proof of Lemma \ref{chap7:lemma:key} is inspired by the overlap estimates for wave packets derived in \cite{Sm} and \cite{Sm-Ta} in the context of Strichartz estimates respectively for $C^{1,1}$ and $H^{2+\ep}$ metrics. Note however that the estimates in these papers rely heavily on a direct comparison of various quantities with the corresponding ones in the flat case. Such direct comparisons do not hold in our framework. Here, the closeness to the flat case manifests itself in the small constant $\ep$ in the right-hand side of \eqref{chap7:regb}, \eqref{chap7:regpoN} and \eqref{chap7:ad1}, and in the existence of the global coordinates systems of section \ref{chap7:sec:assumptioncoord}.
\end{remark}

\subsection{Proof of Proposition \ref{chap7:prop:estkernel}}

Recall that we need to show that the kernel $K$ defined in \eqref{chap7:defkernel} satisfies the upper bound \eqref{chap7:estkernel}. To this end, we will use the estimate \eqref{chap7:ibpkernel} for $K$ together with the estimates provided by Lemma \ref{chap7:lemma:key}. We argue differently according according to whether $(s,y)$ belongs to $S$, $A_{int}$ or $A_{ext}$. 

If $(s,y)$ belongs to $S$, we have the lower bound \eqref{chap7:keylowb1} for $|\phi|$
$$|\phi(t,x,s,y,\o)|\geq \frac{1}{4}|t-s||\o-\o_0|^2,$$
where $\o_0\in\S$ is an angle satisfying \eqref{chap7:def:om0}. Then, we deduce
$$2^j\left|\phi\left(\frac{t}{2^j},\frac{x}{2^j},\frac{s}{2^j},\frac{y}{2^j},\o\right)\right|\geq \frac{1}{4}|t-s||\o-\o_0|^2.$$
Together with \eqref{chap7:ibpkernel}, this yields
$$|K(t,x,s,y)|\les \int_{\S}\frac{d\o}{1+|t-s|^2|\o-\o_0|^4}.$$
Using the spherical coordinates $(\theta, \varphi)$ with axis $\o_0$, we obtain
$$|K(t,x,s,y)|\les \int_0^{\pi}\frac{\sin(\theta)d\theta}{1+|t-s|^2(1-\cos(\theta))^2}.$$
Performing the change of variables
$$z=|t-s|(1-\cos(\theta))$$
we obtain
$$|K(t,x,s,y)|\les \frac{1}{|t-s|}\int_0^{+\infty}\frac{dz}{1+z^2}.$$
This implies
\be\lab{chap7:ok}
|K(t,x,s,y)|\les \frac{1}{|t-s|},\,\,\,\forall (t,x)\in 2^j\MM,\,\,\, \forall \left(\frac{s}{2^j},\frac{y}{2^j}\right)\in S 
\ee
which is the desired estimate.

The estimates corresponding to the cases where  $(s,y)$ belongs to $A_{int}$ or $A_{ext}$ are similar (see \cite{bil2} for the details). This concludes the proof of Proposition \ref{chap7:prop:estkernel}.

\section{Lower bound for $|\phi|$}\lab{chap7:sec:lowerb}

The goal of this section is to prove Lemma \ref{chap7:lemma:key}. The main ingredients of the proof are already present in the flat case. Thus, to simplify the exposition, we will prove Lemma \ref{chap7:lemma:key} for the phase function $u=-t+x\c\o$ of the flat case. We will then explain what are the modifications in the general case (see Remark \ref{chap7:rem:generalcasemodif}). We refer to \cite{bil2} for the proof in the general case. 

\subsection{A lower bound for $|\phi|$ when $(s,y)\in S$ (proof of \eqref{chap7:keylowb1})}\lab{chap7:sec:lowerbinS}

 In view of the definition of $m_0$ in \eqref{chap7:def:m0} and $\om_0$ in \eqref{chap7:def:om0}, we have in the flat case
\begin{equation}\lab{chap7:flatcasevalues}
u(t,x,\o)=-t+x\c\o,\, \o_0=\frac{y-x}{|y-x|},\, m_0=-(s-t)+|y-x|.
\end{equation}
If $(s,y)\in S$, we have $|y-x|=s-t$. Together with \eqref{chap7:flatcasevalues}, this yields
\bea
\nn u(s,y,\o)-u(t,x,\o) &=&-s+t+  (y-x)\c\om\\
\nn& =& -(s-t)+(s-t) \om_0\c\om\\
\lab{estfundch7:1}&=&- \frac 1 2 (s-t)|\om-\om_0|^2
\eea
which is the desired estimate \eqref{chap7:keylowb1}.

\subsection{A lower bound for $|\phi|$ when $(s,y)\in A_{int}$ (proof of \eqref{chap7:keylowb2})}\lab{chap7:sec:cccc}

\eqref{chap7:flatcasevalues} yields
 \bea
\nn u(s,y,\o)-u(t,x,\o) &=&-(s-t)+(y-x)\c\om\\
\nn& =& -(s-t)+|y-x|\om\c\om_0\\
 &=&-(s-t)+|y-x| -   \frac 1 2 |x-y||\om-\om_0|^2.\lab{estfundch7:2}
 \eea
 Now,
 if          $|x-y|\leq\frac 1 4  (s-t)$ we have,
 \beaa
 u(s,y,\o)-u(t,x,\o) &\leq &- \frac 3 4 (s-t)+       \frac 1 4 (s-t)   \frac 1 2    |\om-\om_0|^2  \leq- \frac 1 2 (s-t).
 \eeaa
 On the other hand,  if $|x-y|\ge\frac   1 4  (s-t)$
 \beaa
  u(s,y,\o)-u(t,x,\o) \leq- \frac 1 2 |x-y|        |\om-\om_0|^2\leq - \frac 1 4( s-t)|\om-\om_0|^2.
 \eeaa
 Thus,
 in both cases,
  \beaa
  u(s,y,\o)-u(t,x,\o) \leq- \frac 1 2 |x-y|        |\om-\om_0|^2\leq - \frac 1 4( s-t)|\om-\om_0|^2
 \eeaa
which is the desired estimate \eqref{chap7:keylowb2}.

\subsection{A lower bound for $|\phi|$ when $(s,y)\in A_{ext}$ (proof of \eqref{chap7:keylowb3} \eqref{chap7:keylowb4})}

\eqref{chap7:flatcasevalues} yields 
\bea
\nn u(s,y,\o)-u(t,x,\o) &=&-(s-t)+(y-x)\c\om\\
& =& -(s-t)  +|x-y| \om\c\om_0\nn\\
 &=&-\frac 1 2 (s-t)|\om-\om_0|^2+ m_0 \om\c\om_0. \label{eq:form0}
 \eea
Recall  that we have defined the set $D$ by
\beaa
D=\{\o\in\S\,/\, u(t,x,\o)=u(s,y,\o)\}.
\eeaa
Also, for fixed $\om_0, \om$  we  defined $\om_1\in D$   to lie   on the same plane great circle of $ \SSS^2$ as $\om_0, \om$. Clearly,    since $\om_1\in D$, and in view of \eqref{chap7:flatcasevalues}, \eqref{eq:form0} and the definition of $D$, we have
\bea
\om_1\c\om_0=\frac{(s-t)}   {|x-y|} \label{angle-0m10m0}
\eea

Fix now $\om_1\in D$ and let $z=\gamma_{\o_1}(s-t)$, i.e.  
\bea
z=x+(s-t)\om_1\in S.\label{def:z}
\eea
Note that   in view of the definition of 
$D$,  
$$(y-z)\c\om_1=-(s-t)+(y-x)\c\om_1=u(s,y,\om_1)-u(t,x,\om_1)=0.$$
Hence, with the notation 
$$v_0=y-z,$$
we obtain
\bea
v_0\c\om_1=0.
\eea
Now, we have
\bea
\nn u(s,y,\o)-u(t,x,\o) &=&u(s,y,\o)-u(s, z,\o) +u(s, z,\o) -u(t,x,\o)\\
\nn &=&v_0\c\om+u(s, z,\o) -u(t,x,\o)\\
\lab{estfundch7:3} &=&v_0\c(\om-\om_1)+u(s, z,\o) -u(t,x,\o).
\eea
Note that  since  $z\in S$ we can apply the estimate   obtained  in section \ref{chap7:sec:lowerbinS}. Since the maximum in $m(s,z)$ is attained at $\o=\o_1$, we have 
\bea\lab{estfundch7:4}
u(s, z,\o) -u(t,x,\o)=- \frac 1 2 (s-t)|\om-\om_1|^2
\eea
and we infer that,
\bea
 u(s,y,\o)-u(t,x,\o) &=&- \frac 1 2 (s-t)|\om-\om_1|^2+v_0\c(\om-\om_1).\label{eq:dif-om}
 \eea
 
Recall  that we   have also   denoted   by $\th$  and $ \th_1$  the        positive  angles between $\om_0$, $\om$  and respectively  
 $\om_0$, $\om_1$. If $\th_1\leq\th\leq \pi$ - which corresponds to $v_0\c(\om-\om_1)\leq 0$ - we have in view of \eqref{eq:dif-om} 
 \bea
  u(s,y,\o)-u(t,x,\o) &\leq &- \frac 1 2 (s-t)|\om-\om_1|^2
 \eea
 which is the desired estimate \eqref{chap7:keylowb3}. 
 
 The delicate  case is when $0\leq\th<\th_1$ which corresponds to
 \bea
 v_0\c(\om-\om_1)>0.\label{angle1}
 \eea
 In the rest of the proof, we assume \eqref{angle1}, and we focus on the remaining estimate \eqref{chap7:keylowb4}. In view of the definition of $\om_0$ and \eqref{eq:dif-om}, we have
 \beaa
 -(s-t)+|x-y|&=&u(s,y,\o_0)-u(t,x,\o_0) \\&=&- \frac 1 2 (s-t)|\om_0-\om_1|^2+v_0\c(\om_0-\om_1).
 \eeaa
 Thus,
 \bea
 m_0= -(s-t)+|x-y|&=&- \frac 1 2 (s-t)|\om_0-\om_1|^2+v_0\c(\om_0-\om_1).
 \eea
 Since $m_0>0$ we deduce,
 \bea
 v_0\c(\om_0-\om_1)> 0. \label{angle2}
 \eea
 Let $\alpha$, $\alpha_1$  be the positive angles  between  $v_0$ and $ (\om-\om_1)$ and, respectively
 $v_0$ and $ (\om_0-\om_1)$. In view of   \eqref{angle1} and \eqref{angle2} we infer that
 \bea
 0<\alpha, \alpha_1<\pi/2.\label{eq:ass-angles}
 \eea
Also, in view  of  \eqref{angle-0m10m0}  we  have,
 \beaa
 0<\th_1<\pi/2.
 \eeaa
 Simple considerations on angles imply\footnote{Let $\varphi_1$ the angle defined on figure \ref{chap7:fig6}. Then $2\varphi_1+\th_1=\pi$, and $\varphi_1+\alpha_1=\frac{\pi}{2}$. Hence $\th_1=2\alpha_1$} (see figure \ref{chap7:fig6}),
 \bea
 \th_1=2\a_1.\label{eq:angle1}
 \eea
  \begin{figure}[t]
\vspace{-4cm}
\begin{center}
\hspace{-4.0005cm}\includegraphics[width=20cm, height=14.5cm]{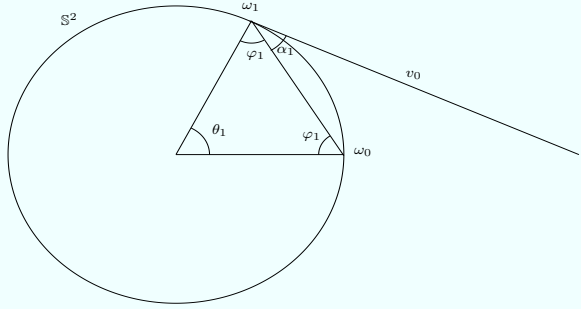}
\vspace{-8cm}
\caption{Definition of the angles $\th_1$ and $\alpha_1$}\lab{chap7:fig6}
\end{center}
\end{figure}
Therefore,
 \beaa
 m_0 &=&- \frac 1 2 (s-t)|\om_0-\om_1|^2+      |z-y|       |\om_0-\om_1|\cos\left(\frac{\th_1}{2}\right)
 \eeaa
 and
 \bea
 |v_0|=\frac{m_0+\frac 1 2 (s-t)|\om_0-\om_1|^2}{|\om_0-\om_1|\cos\left(\frac{\th_1}{2}\right)}.\label{eq:|z-y|}
 \eea
 Using the same type of argument\footnote{Let $\varphi$ the angle defined on figure \ref{chap7:fig7}. Then $2\varphi+|\th_1-\th|=\pi$, and $\varphi+\alpha=\frac{\pi}{2}$. Hence $|\th_1-\th|=2\alpha$} as  in \eqref{eq:angle1} we also deduce (see figure \ref{chap7:fig7})
 \bea
 \label{eq:angle2}
 \a=\frac{\th_1-\th}{2}.
 \eea 
 \begin{figure}[t]
\vspace{-3.75cm}
\begin{center}
\hspace{-4.0005cm}\includegraphics[width=20cm, height=15cm]{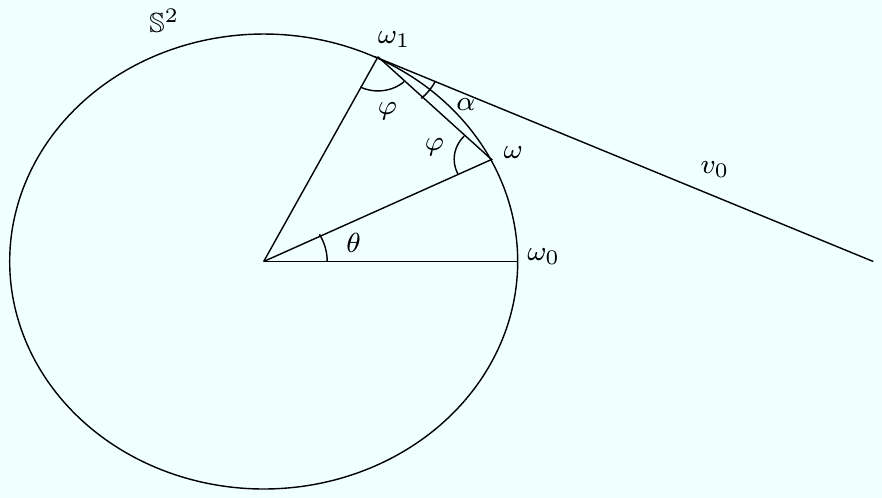}
\vspace{-8.3cm}
\caption{Definition of the angles $\th$ and $\alpha$}\lab{chap7:fig7}
\end{center}
\end{figure}
Therefore,  according to  \eqref{eq:dif-om}, \eqref{eq:|z-y|} and \eqref{eq:angle2}, we obtain
 \bea
  u(s,y,\o)-u(t,x,\o) &=&- \frac 1 2 (s-t)|\om-\om_1|^2+|v_0| \, |\om-\om_1| \cos\left(\frac{\th_1-\th}{2}\right)\nn\\
  &=&  \frac{|\o-\o_1|}{|\o_0-\o_1|}\frac{\cos\left(\frac{\theta_1-\theta}{2}\right)}{\cos\left(\frac{\theta_1}{2}\right)}m_0+\frac{1}{2}(s-t)|\o-\o_1| A(\o), \label{oo13}
\eea
where $A(\o)$ is given by
\be\lab{oo14}
A(\o)=-|\o-\o_1|+\frac{\cos\left(\frac{\theta_1-\theta}{2}\right)}{\cos\left(\frac{\theta_1}{2}\right)}|\o_0-\o_1|.
\ee

We have the following lemma (see \cite{bil2} for a proof).
\begin{lemma}
For all $0\leq\th\leq\th_1$, we have
\bea
A(\om)\geq 0.    \label{eq:angle3}
\eea
\end{lemma}
  
Back to  \eqref{oo13}, we thus   derive,
 \beaa
   u(s,y,\o)-u(t,x,\o) &\ge& \frac{|\o-\o_1|}{|\o_0-\o_1|}\frac{\cos\left(\frac{\theta_1-\theta}{2}\right)}{\cos\left(\frac{\theta_1}{2}\right)}m_0.
   \eeaa
Using our angle restriction
$$0\leq \theta\leq \theta_1  <\frac{\pi}{2},$$
we deduce
\be\lab{presque}
u(s,y,\o)-u(t,x,\o)\geq \frac{\sqrt{2}}{2}\frac{|\o-\o_1|}{|\o_0-\o_1|}m_0.
\ee   
Since $\theta$ is the angle between $\o$ and $\o_0$, and $\theta_1$ is the angle between $\o_1$ and $\o_0$, we have
\be\lab{oo14bis}
|\o-\o_1|=\sqrt{2}\sqrt{1-\cos(\theta_1-\theta)},\, |\o_1-\o_0|=\sqrt{2}\sqrt{1-\cos(\theta_1)}.
\ee
In view of \eqref{oo14bis}, we  can rewrite \eqref{presque}
in the form,
\be\lab{oo15}
\phi(t,x,s,y,\o) \gtrsim  m_0 \sqrt{\frac{1-\cos(\theta-\theta_1)}{1-\cos(\theta_1)}}
\ee
which is the desired estimate \eqref{chap7:keylowb4}. This concludes the proof of Lemma \ref{chap7:lemma:key} in the flat case.

\begin{remark}\lab{chap7:rem:generalcasemodif}
Let us indicate how to prove Lemma \ref{chap7:lemma:key} in the general case. The whole point is to realize that the only estimates for which the precise regularity of $u$ matters are the ones corresponding to \eqref{estfundch7:1}, \eqref{estfundch7:2}, \eqref{eq:form0}, \eqref{estfundch7:3} and \eqref{estfundch7:4}. Indeed, once this has been achieved, the rest of the argument is then essentially the one of the flat case. 

Now, to prove the estimates corresponding to \eqref{estfundch7:1}, \eqref{estfundch7:2}, \eqref{eq:form0}, \eqref{estfundch7:3} and \eqref{estfundch7:4} in the general case, one needs the following two additional ingredients (see \cite{bil2} for the details):
\begin{enumerate}
\item These estimates are obtained by using the following standard identity
\begin{equation}\lab{thekoikoi}
u(\eta(1),\o)=u(\eta(0),\o)+\int_0^1 \gg(L_{\eta(\sigma)},\eta'(\sigma))d\sigma,
\end{equation}
where $\eta$ is a curve in $\MM$, and where $L$ denotes the space-time gradient of $u$. It turns out that one may chose suitable curves\footnote{In the flat case, the corresponding curves $\eta$ are straight lines.} $\eta$ allowing us to deduce from \eqref{thekoikoi} the estimates corresponding to \eqref{estfundch7:1}, \eqref{estfundch7:2}, \eqref{eq:form0}, \eqref{estfundch7:3} and \eqref{estfundch7:4} under our assumptions  \eqref{chap7:regb}, \eqref{chap7:regpoN} and \eqref{chap7:ad1}. This changes the constants in the inequalities due to the presence of additional $O(\ep)$ terms, but does not change the nature of the estimates for $\ep>0$ small enough.

\item The above mentioned curves $\eta$ start on  $S$, and a crucial point is to check that such curves end up exactly at $(s,y)$. To this end, one uses the global coordinate system $(u(t,x,\o_0), \po u(t,x,\o_0))$ of section \ref{chap7:sec:assumptioncoord} on $\Sigma_s$ for a well-chosen angle $\o_0\in\S$, which allows us to identify $(s,y)$ as the unique point $p$ on $\Sigma_s$ such that 
$$u(p,\o_0)=u(s,y,\o_0)\textrm{ and }\po u(p,\o_0)=\po u(s,y,\o_0).$$ 
\end{enumerate}
\end{remark}


\begin{thebibliography}{99}
\bibitem{An}
M.~T. Anderson,
\newblock Cheeger-{G}romov theory and applications to general relativity.
\newblock In {\em The {E}instein equations and the large scale behavior of
  gravitational fields}, pages 347--377. Birkh\"auser, Basel, 2004.
  
\bibitem{Ba-Ch1}
H.~Bahouri, J.-Y. Chemin,
\newblock {\em \'Equations d'ondes quasilin\'eaires et estimation de
  Strichartz\/}.
\newblock Amer. J. Math., \textbf{121}, 1337--1777, 1999.

\bibitem{Ba-Ch2}
H.~Bahouri, J.-Y. Chemin,
\newblock {\em \'Equations d'ondes quasilin\'eaires et effet dispersif\/}.
\newblock IMRN, \textbf{21}, 1141--1178, 1999.
    
\bibitem{BCS}
R. Beig, P.~T. Chru{\'s}ciel, R.  Schoen,
\newblock {\em K{ID}s are non-generic}.
\newblock Ann. Henri Poincar\'e, \textbf{6} (1), 155--194, 2005.

\bibitem{Br} 
Y. C. Bruhat, \textit{Theoreme d'Existence pour certains
systemes d'equations aux derivees partielles nonlineaires.}, Acta Math. 
\textbf{88}, 141--225, 1952.

\bibitem{Ca} E. Cartan, \textit{Sur une g\'en\'eralisation de la notion de courbure de Riemann et les espaces ˆ torsion.} C. R. Acad. Sci. (Paris), {\bf 174}, 593--595, 1922.

\bibitem{Ch1} 
D.  Christodoulou, \textit{Bounded variation solutions of the spherically symmetric einstein-scalar field equations}, Comm. Pure and Appl. Math,  {\bf 46}, 1131--1220, 1993.

\bibitem{Ch2} 
D.  Christodoulou, \textit{The instability of naked singularities in
the gravitational collapse of a scalar field}, Ann. of Math., {\bf 149}, 183--217,1999.

\bibitem{ChKl}
D. Christodoulou, S. Klainerman.
\newblock {\em The global nonlinear stability of the {M}inkowski space},
  volume~41 of {\em Princeton Mathematical Series}.
\newblock Princeton University Press, Princeton, NJ, 1993.

\bibitem{Cor}
J. Corvino, {\em Scalar curvature deformation and a gluing construction for the Einstein constraint equations}, Commun. Math. Phys.  \textbf{214}, 137--189, 2000.

\bibitem{CorSch}
J. Corvino, R. Schoen, On the asymptotics for the vacuum Einstein constraint equations, Jour. Diff. Geom. \textbf{73}, 185--217, 2006.  

\bibitem{FM} 
A. Fischer, J. Marsden, \textit{The Einstein evolution equations as a first-order quasi-linear symmetric hyperbolic system.  I}, Comm. Math. Phys. \textbf{28}, 1--38, 1972.

\bibitem{Gl} 
J. Glimm, \textit{Solutions in the large for nonlinear hyperbolic systems of equations},
Comm. Pure and Appl. Math. {\bf 18}, 697--715, 1965.

\bibitem{HKM} 
T. J. R. Hughes, T. Kato, J. E. Marsden, {\it Well-posed quasi-linear second-order hyperbolic systems
 with applications to nonlinear elastodynamics and general 
relativity}, Arch. Rational Mech. Anal. {\bf 63}, 273--394, 1977.

\bibitem{Kl-Ma1} 
S. Klainerman, M. Machedon,  {\it Space-time estimates for null forms and
the local existence theorem}, Communications on Pure and Applied
Mathematics, \textbf{46}, 1221--1268, 1993.

\bibitem{Kl-Ma2} 
S. Klainerman, M. Machedon, {\it Finite energy solutions of the Maxwell-Klein-Gordon equations}, Duke Math. J. \textbf{74}, 19--44, 1994.

\bibitem{Kl-Ma3} 
S. Klainerman, M. Machedon, {\it Finite Energy Solutions for the Yang-Mills
Equations in ${\Bbb R}^{1+3}$}, Annals of Math. \textbf{142}, 39--119, 1995.

\bibitem{PDE} 
S. Klainerman, \textit {PDE as a unified subject}, Proceeding of Visions in Mathematics,
GAFA 2000(Tel Aviv 1999). Geom Funct. Anal.  2000, Special Volume , Part 1, 279--315.  

\bibitem{Kl-R1} 
S. Klainerman, I. Rodnianski, \textit{Improved local well-posedness for quasi-linear wave equations in dimension three}, Duke Math. J.  \textbf{117} (1), 1--124, 2003.

\bibitem{Kl-R2} 
S. Klainerman, I. Rodnianski, \textit{Rough solutions to the Einstein vacuum equations}, Annals of Math. \textbf{161}, 1143--1193, 2005.

\bibitem{Kl-R3} 
S. Klainerman, I. Rodnianski, {\it Bilinear estimates on curved space-times},  J.  Hyperbolic Differ.  Equ. \textbf{2} (2), 279--291, 2005.

\bibitem{Kl-R4} 
S. Klainerman, I. Rodnianski, {\it Casual geometry of Einstein vacuum space-times with finite curvature flux}, {\it Inventiones} {\bf 159}, 437--529, 2005. 

\bibitem{Kl-R5} 
S. Klainerman,  I. Rodnianski,  {\it Sharp trace theorems on null hypersurfaces}, GAFA \textbf{16} (1), 164--229, 2006. 

\bibitem{Kl-R6} 
S. Klainerman,  I. Rodnianski, {\it A geometric version of Littlewood-Paley theory}, GAFA {\bf  16} (1), 126--163, 2006. 

\bibitem{Kl-R7} 
S. Klainerman,  I. Rodnianski,  {\it On the radius of injectivity of Null Hypersurfaces}, J. Amer. Math. Soc. \textbf{21}, 775--795, 2008. 

\bibitem{conditional} 
S. Klainerman, I. Rodnianski, \textit{On a  break-down  criterion in General Relativity}, J. Amer. Math. Soc. \textbf{23}, 345--382, 2010. 
  
\bibitem{KRS} 
S. Klainerman,  I. Rodnianski, J. Szeftel, \textit{The Bounded $L^2$ Curvature Conjecture}, arXiv:1204.1767, 91 p, 2012.
 
\bibitem{Kr-S} 
J. Krieger, W. Schlag, \textit{Concentration compactness for critical wave maps}, Monographs of the European Mathematical Society, 2012.

\bibitem{L} 
H. Lindblad, \textit{Counterexamples to local existence for quasilinear wave equations}, Amer. J. Math. \textbf{118} (1), 1--16, 1996.

\bibitem{wNC}    
H. Lindblad, I. Rodnianski, \textit{The weak null condition for the Einstein vacuum equations}, C. R. Acad. Sci. \textbf{336}, 901--906,  2003.

\bibitem{Pa} 
D. Parlongue, \textit{An integral breakdown criterion for Einstein vacuum equations in the case of asymptotically flat spacetimes}, arXiv:1004.4309, 88 p, 2010.

\bibitem{Pe}
P. Petersen, 
\newblock Convergence theorems in {R}iemannian geometry.
\newblock In {\em Comparison geometry ({B}erkeley, {CA}, 1993--94)}, volume~30
  of {\em Math. Sci. Res. Inst. Publ.}, pages 167--202. Cambridge Univ. Press,
  Cambridge, 1997.

\bibitem{PlRo} 
F. Planchon, I. Rodnianski, {\em Uniqueness in general relativity}, preprint.

\bibitem{Po-Si} 
G. Ponce, T. Sideris, \textit{Local regularity of nonlinear wave equations in three space dimensions},
Comm. PDE \textbf{17}, 169--177, 1993.

\bibitem{Sm}
H. F. Smith, {\em A parametrix construction for wave equations with $C^{1,1}$ coefficients}, Ann. Inst. Fourier (Grenoble) \textbf{48},  797--835, 1998.

\bibitem{Sm-Ta} 
H.F. Smith, D. Tataru, \textit{Sharp local well-posedness results for the nonlinear
wave equation}, Ann. of Math. \textbf{162}, 291--366, 2005.

\bibitem{Sob} 
S. Sobolev, {\it Methodes nouvelle a resoudre le probleme de Cauchy 
pour les equations lineaires hyperboliques normales}, 
 Matematicheskii Sbornik, \textbf{1 (43)}, 31--79, 1936.

\bibitem{sog} C. D. Sogge, {\em Lectures on non-linear wave equations}, International Press, Boston, MA, 2008.

\bibitem{St} 
E. Stein, \textit{Harmonic Analysis}, Princeton University Press, 1993.

\bibitem{St-Ta1} 
J.   Sterbenz, D. Tataru, \textit{Regularity of Wave-Maps in dimension $2 + 1$}, Comm. Math. Phys. \textbf{298} (1), 231--264, 2010. 

\bibitem{St-Ta2}  
J. Sterbenz, D. Tataru, \textit{Energy dispersed large data wave maps in $2 + 1$ dimensions}, Comm. Math. Phys. \textbf{298} (1),  139--230, 2010.
 
\bibitem{param1} 
J. Szeftel, \textit{Parametrix for wave equations on a rough background I: Regularity of the phase at initial time}. arXiv:1204.1768, 145 p, 2012.

\bibitem{param2} 
J. Szeftel, \textit{Parametrix for wave equations on a rough background II: Construction of the parametrix and control at initial time}. arXiv:1204.1769, 84 p, 2012.

\bibitem{param3} 
J. Szeftel, \textit{Parametrix for wave equations on a rough background III: Space-time regularity of the phase}. arXiv:1204.1770, 276 p, 2012.

\bibitem{param4} 
J. Szeftel, \textit{Parametrix for wave equations on a rough background IV: Control of the error term}. arXiv:1204.1771, 284 p, 2012.

\bibitem{bil2} 
J. Szeftel, {\em Sharp Strichartz estimates for the wave equation on a rough background}. arXiv:1301.0112, 30 p, 2013.

\bibitem{Tao} 
T. Tao, \textit{Global regularity of wave maps I--VII}, preprints.

\bibitem{Ta3} 
D. Tataru, \textit{Local and global results for  Wave Maps I}, Comm. PDE \textbf{23}, 1781--1793, 1998.

\bibitem{Ta1}
D.~Tataru.
\newblock {\em Strichartz estimates for operators with non smooth coefficients
  and the nonlinear wave equation\/}.
\newblock Amer. J. Math. \textbf{122}, 349--376, 2000.

\bibitem{Ta2}
D.~Tataru, 
\newblock {\em Strichartz estimates for second order hyperbolic operators with
  non smooth coefficients\/}, J.A.M.S. {\bf 15} (2), 419--442, 2002.

\bibitem{Uhl}
K. Uhlenbeck, {\em Connections with $L^p$ bounds on curvature}, Commun. Math. Phys. \textbf{83}, 31--42, 1982.  

\bibitem{W} Q. Wang, \textit{Improved breakdown criterion for Einstein vacuum equation in CMC gauge}, Comm. Pure Appl. Math. \textbf{65} (1), 21--76, 2012.
\end{thebibliography}
\end{document}